\documentclass[11pt, twoside]{article}

\usepackage{amssymb}
\usepackage{amsmath}
\usepackage{amsthm}
\usepackage{color}
\usepackage{mathrsfs}

\allowdisplaybreaks

\def\rr{{\mathbb R}}
\def\rn{{{\rr}^n}}
\def\zz{{\mathbb Z}}
\def\cc{{\mathbb C}}
\def\nn{{\mathbb N}}

\def\ca{{\mathcal A}}
\def\cb{{\mathcal B}}

\def\cd{{\mathcal D}}
\def\ce{{\mathcal E}}

\def\cl{{\mathcal L}}
\def\cm{{\mathcal M}}
\def\cn{{\mathcal N}}

\def\cq{{\mathcal Q}}

\def\cs{{\mathcal S}}

\def\cw{{\mathcal W}}
\def\mi{{\mathrm I}}

\def\fz{\infty}
\def\az{\alpha}
\def\bz{\beta}
\def\dz{\delta}
\def\bdz{\Delta}
\def\ez{\epsilon}
\def\gz{{\gamma}}

\def\kz{\kappa}
\def\lz{\lambda}

\def\oz{{\omega}}
\def\boz{{\Omega}}
\def\tz{\theta}
\def\sz{\sigma}
\def\vz{\varphi}

\def\lf{\left}
\def\r{\right}

\def\la{\langle}
\def\ra{\rangle}
\def\hs{\hspace{0.25cm}}
\def\ls{\lesssim}
\def\gs{\gtrsim}

\def\pa{\partial}

\def\noz{\nonumber}
\def\wz{\widetilde}
\def\wh{\widehat}
\def\st{\subset}

\def\dis{\displaystyle}

\def\supp{\mathop\mathrm{\,supp\,}}

\def\loc{{\mathop\mathrm{\,loc\,}}}

\def\lp{{L^p(\rn)}}

\newcommand{\ala}{{{A}_{\cl,q,a}^{w,\tau}(\rn)}}
\newcommand{\alw}{{{A}_{\cl,q}^{w,\tau}(\rn)}}

\newcommand{\fla}{{{F}_{\cl,q,a}^{w,\tau}(\rn)}}

\newcommand{\ela}{{{\ce}_{\cl,q,a}^{w,\tau}(\rn)}}

\newcommand{\N}{{\mathbb N}}
\newcommand{\R}{{\mathbb R}}
\newcommand{\C}{{\mathbb C}}
\newcommand{\Z}{{\mathbb Z}}

\def\qn{|\!|\!|}

\newtheorem{theorem}{\sc Theorem}[section]
\newtheorem{lemma}[theorem]{\sc Lemma}
\newtheorem{corollary}[theorem]{\sc Corollary}
\newtheorem{proposition}[theorem]{\sc Proposition}
\theoremstyle{definition}
\newtheorem{remark}[theorem]{\sc Remark}
\newtheorem{definition}[theorem]{\sc Definition}
\newtheorem{assumption}[theorem]{\sc Assumption}
\newtheorem{example}[theorem]{\sc Example}

\numberwithin{equation}{section}

\begin{document}

\arraycolsep=1pt

\title{\bf\Large A new framework for generalized
Besov-type and Triebel-Lizorkin-type spaces
\footnotetext{\hspace{-0.35cm} 2010 {\it
Mathematics Subject Classification}. Primary 46E35;
Secondary 42B35, 42B25, 42B15, 42C40.
\endgraf {\it Key words and phrases}.
Besov space, Triebel-Lizorkin space, atom, molecule,
difference, oscillation, wavelet, embedding, multiplier, pseudo-differential operator.
\endgraf This work is partially supported by the GCOE program
of Kyoto University and 2010 Joint Research Project Between
China Scholarship Council and German Academic Exchange
Service (PPP) (Grant No. LiuJinOu [2010]6066). Yoshihiro Sawano is supported
by Grant-in-Aid for Young Scientists (B) (Grant No. 21740104) of
Japan Society for the Promotion of Science.
Dachun Yang is supported by the National
Natural Science Foundation (Grant No. 11171027) of China
and Program for Changjiang Scholars and Innovative
Research Team in University of China. Wen
Yuan is supported by the National
Natural Science Foundation (Grant No. 11101038) of China}}
\author{Yiyu Liang, Yoshihiro Sawano, Tino Ullrich,
Dachun Yang\,\footnote{Corresponding author}\ \ and Wen Yuan}
\date{}
\maketitle

\begin{center}
\begin{minipage}{13.5cm}\small
{{\bf Abstract.}
In this paper, the authors propose a new framework
under which a theory of generalized
Besov-type and Triebel-Lizorkin-type function spaces is developed.
Many function spaces appearing in harmonic analysis
fall under the scope of this new framework.
Among others, the boundedness of the Hardy-Littlewood
maximal operator or the related vector-valued maximal function
on any of these function spaces is not required to construct these generalized
scales of smoothness spaces.
Instead of this, a key idea used in this framework is an application
of the Peetre maximal function. This idea originates from recent findings in
the abstract coorbit space theory obtained by Holger Rauhut and Tino Ullrich.
Under this new setting, the authors establish the
boundedness of pseudo-differential operators based on atomic and molecular
characterizations
and also the boundedness of the Fourier multipliers.
The characterizations of these function
spaces by means of differences and oscillations are also established.
As further applications of this new framework,
the authors reexamine and polish some
existing results for many different scales of function spaces.
}
\end{minipage}
\end{center}

\newpage

\tableofcontents

\newpage

\section{Introduction}\label{s1}

Different types of smoothness spaces play an important role in
harmonic analysis, partial differential
equations as well as in approximation theory.
For example, Sobolev spaces are widely used
for the theory of elliptic partial differential equations.
However, there are several partial differential equations
on which the scale of Sobolev spaces is no longer sufficient.
A proper generalization is given by the classical
Besov and Triebel-Lizorkin function spaces.
In recent years, it turned out to be necessary to generalize even further
and replace the fundamental space $L^p(\rr^n)$ by something more general,
like a Lebesgue space with variable exponents (\cite{d1, DHR})
or, more generally, an Orlicz space.
Another direction is pursued via replacing the space $L^p(\rr^n)$ by
the Morrey space ${\mathcal M}^p_u(\rr^n)$; see \cite{MO,An1,An2}, or generalizations thereof
\cite{ky,st07,syy,tx,x1,x2,yy1,yy2,ysy-textbook}. Therefore,
the theory of function spaces has become more and more complicated
due to their definitions. Moreover, results on atomic or molecular decompositions
were often developed from scratch again and again for different scales.

A nice approach to unify
the theory was performed by Hedberg and Netrusov in \cite{hn07}.
They developed an axiomatic approach
to function spaces of Besov-type and Triebel-Lizorkin-type, in which the
underlying function space is a quasi-normed space $E$ of sequences of
Lebesgue measurable functions on $\rn$, satisfying some additional
assumptions. The key property assumed in this approach
is that the space $E$ satisfies a vector-valued
maximal inequality of Fefferman-Stein type, namely,
for some $r\in(0,\fz)$ and $\lz\in [0,\fz)$, there exists a
positive constant $C$ such that,
for all $\{f_i\}_{i=0}^\fz\subset E$,
\begin{eqnarray*}
\|\{ M_{r,\lz} f_i\}_{i=0}^\fz\|_E\le C\|\{ f_i\}_{i=0}^\fz\|_E
\end{eqnarray*}
(see \cite[Definition 1.1.1(b)]{hn07}), where
\begin{eqnarray*}
M_{r,\lz}f(x):=\sup_{R>0}
\left\{\frac{1}{R^n}\int_{|y|<R}|f(x+y)|^r(1+|y|)^{-r\lambda}\,dy
\right\}^{1/r}\quad \mbox{for all}\quad x\in\rn.
\end{eqnarray*}
Related to \cite{hn07},
Ho \cite{ho10} also developed a theory of function spaces on $\rn$
under the additional assumption that the Hardy-Littlewood maximal operator $M$
is bounded on the corresponding fundamental function space.

Another direction towards a unified
treatment of all generalizations has been developed
by Rauhut and Ullrich \cite{ru11} based
on the generalized abstract coorbit space theory. 
The coorbit space theory was originally developed by Feichtinger and Gr{\"o}chenig
\cite{FeGr86, Gr88, Gr91} with the aim of providing a unified approach
for describing
function spaces and their atomic decompositions. The classical theory uses locally compact groups
together with integrable group representations as key ingredients. Based on the idea to measure smoothness
via decay properties of an abstract wavelet transform one can particularly recover homogeneous Besov-Lizorkin-Triebel spaces
as coorbits of Peetre spaces $\mathcal{P}^s_{p,q,a}(\rn)$. The latter fact  was observed recently by Ullrich in \cite{u10}.
In the next step Fornasier and Rauhut
\cite{fora05} observed that a locally compact group structure is not needed
at all to develop a coorbit space theory. While the theory in \cite{fora05}
essentially applies only to coorbit spaces with respect to weighted Lebesgue
spaces, Rauhut and Ullrich \cite{ru11}  extended
this abstract theory in order to treat a wider variety of coorbit spaces. The main motivation was
to cover inhomogeneous Besov-Lizorkin-Triebel spaces and generalizations
thereof. Indeed, the Besov-Lizorkin-Triebel type spaces appear as
coorbits of Peetre type spaces $\mathcal{P}^{w}_{p,\mathcal{L},a}(\rn)$; see \cite{ru11}.

All the aforementioned theories are either not complete or in some situations too restrictive.
Indeed, the boundedness of
maximal operators of Hardy-Littlewood type or the related
vector-valued maximal functions is always required and,
moreover, the Plancherel-Polya-Nikolskij inequality
(see Lemma \ref{l1.1} below)
and the Fefferman-Stein vector-valued inequality
had been a key tool in order to develop a theory of function spaces
of Besov and Triebel-Lizorkin type.

Despite the fact that the generalized coorbit space theory \cite{ru11}
so far only works for Banach spaces we mainly borrow techniques from there and
combine them with recent ideas from the theory of Besov-type and
Triebel-Lizorkin-type spaces (see \cite{st07,syy,tx, yy1,yy2,yy3,ysy-textbook})
to build up our theory for quasi-normed spaces in the present paper.
In order to be applicable also in microlocal analysis,
we even introduce these spaces directly
in the weighted versions.
The key idea, used in this new framework, is some delicate application of the
sequence of the Peetre maximal functions
\begin{equation}\label{1.3}
(\vz_j^*f)_a(x):=
\begin{cases}
\dis\sup_{y\in\rn}\frac{|\Phi\ast f(x+y)|}{(1+|y|)^a},&j=0;\\
\dis\sup_{y\in\rn}\frac{|\vz_j\ast f(x+y)|}{(1+2^j|y|)^a},&j \in {\mathbb N},
\end{cases}
\end{equation}
for all $f\in S'(\rn)$, where $\Phi$ and $\vz$ are,
respectively, as in \eqref{1.1} and \eqref{1.2} below,
and $\vz_j(\cdot)=2^{jn}\vz(2^j\cdot)$ for all $j\in\nn$.
Instead of the pure convolution $\vz_j\ast f$
involved in the definitions of the classical
Besov and Triebel-Lizorkin spaces,
we make use of the Peetre maximal function $(\vz_j^*f)_a$ already
in the definitions of the spaces considered in the present paper. The second main
feature, what concerns generality, is the fundamental space $\mathcal{L}(\rr^n)$
involved in the definition (instead of $L^p(\R^n)$). This space is given in
Section \ref{s2} via a list of fundamental assumptions $(\mathcal{L}1)$ through
$(\mathcal{L}6)$. The key assumption is $(\mathcal{L}6)$,
which originates in \cite{ru11} (see (\ref{eq:L6}) below). The most
important advantage of the Peetre maximal function in this framework is the fact
that
$(\vz_j^*f)_a$ can be pointwise controlled by
a linear combination of some other Peetre maximal functions $(\psi_k^*f)_a$,
whereas in the classical setting, $\vz_j\ast f$ can only be dominated by a
linear combination of the Hardy-Littlewood maximal function $M(|\psi_k\ast f|)$
of $\psi_k\ast f$ (see \eqref{hlmf} below). This simple fact illustrates
quite well that the boundedness of $M$ on $\mathcal{L}(\R^n)$ is not required
in the present setting. It represents the key advantage of our theory since, according
to Example \ref{e1.1} and Section \ref{s10}, we are now able to deal with a
greater variety of spaces. However, we
do not define abstract coorbit spaces here.
Compared with the results in \cite{ru11},
the approach in the present paper admits the following additional features:
\begin{itemize}
 \item Extension of the decomposition results to quasi-normed spaces
(Section \ref{s4});
 \item Sharpening the conditions on admissible atoms, molecules, and wavelets
(Section \ref{s4});
 \item Intrinsic characterization for the respective spaces on domains
(Section \ref{s4.5});
 \item Boundedness of pseudo-differential operators acting between two spaces
(Section \ref{s5});
 \item Direct Characterizations via differences and oscillations
(Section \ref{s7}).
\end{itemize}
\noindent Let us describe the organization of the present paper.
In Section \ref{s2},
we describe the {\it new} setting we propose, which consists of a list of
assumptions $(\cl1)$ through $(\cl6)$ on the fundamental space
$\cl(\rn)$. Also several important consequences and further inequalities are
provided.

In Section \ref{s3}, based on $\cl(\rn)$,
we introduce two sorts of generalized Besov-type and Triebel-Lizorkin-type
spaces, respectively (see Definition \ref{d3.1} below). We justify
these definitions by proving some properties, such as completeness
(without assuming $\mathcal{L}(\R^n)$ being complete!), the Schwartz space
$\mathcal{S}(\R^n)$ being contained, and the embedding into the distributions
$\mathcal{S}'(\R^n)$. An analogous statement holds true with the classical
2-microlocal space $B^w_{1,1,a}(\R^n)$ as test functions and its dual, the
space $B^{1/w}_{\infty,\infty,a}(\R^n)$, as distributions, which is an
important observation for the characterization with wavelets in Section
\ref{s4}. Therefore, the latter spaces which have been studied intensively by
Kempka \cite{k10-1, k10-2}, appear naturally in our context.

In Section \ref{s4}, we establish atomic and molecular
decomposition characterizations (see Theorem \ref{t4.1} below),
which are further used in Section \ref{s5}
to obtain the boundedness of some pseudo-differential operators
from the H\"ormander class $S^0_{1,\mu}(\rn)$, with $\mu\in [0,1)$ (see
Theorems \ref{t5.1} and \ref{t5.2} below). In addition, characterizations using
biorthogonal wavelet bases are given (see Theorem \ref{t4.3} below).
Appropriate wavelets (analysis
and synthesis) must be sufficiently smooth, fast decaying and provide enough
vanishing moments. The precise conditions on these three issues are provided in
Subsection 4.4 and allow for the selection of particular
biorthogonal wavelet bases according to the well-known construction by Cohen,
Daubechies and Feauveau \cite{codafe92}. Characterizations via orthogonal wavelets
are contained in this setting.

Section \ref{s4.5} considers pointwise multipliers and the restriction of
our function spaces to Lipschitz domains $\Omega$ and provides
characterizations from inside the domain (avoiding extensions).

Section \ref{s5} considers Fourier multipliers and
pseudo-differential operators,
which supports that our new framework works.

In Section
\ref{s6},
we obtain a sufficient condition
for which the function spaces consist of continuous functions (see Theorem
\ref{t6.1} below).
This is a preparatory step for Section \ref{s7}, where we deal with
differences and oscillations. Another issue of Section \ref{s6} is a further
interesting application of the atomic decomposition result from Theorem
\ref{t4.1}. Under certain conditions on the involved scalar parameters (by
still using a general fundamental space $\mathcal{L}(\R^n)$), our spaces
degenerate to the well-known classical 2-microlocal Besov spaces
$B^w_{\infty,\infty}(\R)$.

In Section \ref{s7},
we obtain a direct characterization in terms of differences
and oscillations of these generalized Besov-type and
Triebel-Lizorkin-type spaces (see Theorems \ref{t7.1} and \ref{t7.2} below).
Also, under some mild condition,
$\cl(\rn)$ is shown to fall under our new framework
(see Theorem \ref{t8.30} below).

The Peetre maximal construction in the present paper makes it necessary to deal
with a further parameter $a\in(0,\fz)$ in the definition of the function spaces.
However, this new parameter $a$ does not seem to play a significant role
in generic setting, although we do have an example showing that the space may
depend upon $a$ (see Example \ref{e3.2}).
We present several sufficient conditions in Section \ref{s8}
which allow to remove the parameter $a$ from the function spaces
(see Assumption 8.1 below).

Homogeneous counterparts of the above are
available and we describe them in Section \ref{s9}.
Finally, in Section \ref{s10} we present some well-known function
spaces as examples
of our abstract results and compare them
with earlier contributions. We reexamine and polish some
existing theories for these known function spaces.

Next we clarify some conventions on the notation and review some basic
definitions. In what follows, as usual, we use $\cs(\rn)$ to denote
the \emph{classical topological vector space of all Schwartz
functions on $\rn$} and $\cs'(\rn)$ its \emph{topological
dual space} endowed with weak-$\ast$ topology.
For any $\varphi \in \cs(\rn)$, we use
$\widehat{\vz}$ to denote its \emph{Fourier transform}, namely, for all
$\xi\in\rn$, $\widehat{\vz}(\xi):=\int_\rn e^{-i\xi
x}\vz(x)\,dx$. We denote \emph{dyadic dilations} of a given function $\varphi \in
\mathcal{S}(\R^n)$ by $\varphi_j(x):= 2^{jn}\varphi(2^jx)$ for all
$j\in\mathbb{Z}$ and $x\in\rn$. Throughout the whole paper we permanently
use a system $(\Phi,\varphi)$ of Schwartz functions satisfying
\begin{equation}\label{1.1}
\supp\widehat{\Phi}\subset \{\xi\in\rn:\ |\xi|\le 2\}\ \mathrm{and}\
|\widehat{\Phi}(\xi)|\ge C>0\ \mathrm{if}\ |\xi|\le 5/3
\end{equation}
and
\begin{equation}\label{1.2}
\supp\widehat{\vz}\subset \{\xi\in\rn:\,1/2\le|\xi|\le 2\}\ \mathrm{and}\
|\widehat{\vz}(\xi)|\ge C>0\ \mathrm{if}\ 3/5\le|\xi|\le 5/3.
\end{equation}
\noindent The \emph{space} $L^1_\loc(\rn)$ denotes
the set of all locally integrable functions on $\rn$, the \emph{space}
$L^\eta_\loc(\rn)$ for any $\eta\in(0,\fz)$ the set of all
measurable functions on $\rn$ such that
$|f|^\eta\in L^1_\loc(\rn)$, and the \emph{space $L^\fz_\loc(\rn)$} the set of
all
locally essentially bounded functions on $\rn$.
We also let $M$ denote the \emph{Hardy-Littlewood maximal operator}
defined by setting, for all $f\in L^1_\loc(\rn)$,
\begin{equation}\label{hlmf}
Mf(x)=
M(f)(x):=\sup_{r>0}\frac{1}{r^n}\int_{|z-x|<r}|f(z)|\,dz \quad\mbox{for all}\quad x\in\rn.
\end{equation}
One of the main tools in the classical theory of function spaces
is the boundedness of
the Hardy-Littlewood maximal operator $M$ on a space of functions, say
$L^p(\R^n)$ or its vector-valued extension
$L^p(\ell^q)$, in connection with the
Plancherel-Polya-Nikolskij inequality connecting the Peetre maximal function
and the Hardy-Littlewood maximal operator.

\begin{lemma}[{\rm \cite[p.\,16]{t83}}]
\label{l1.1}
Let $\eta\in(0,1]$, $R\in(0,\fz)$ and
$f \in {\mathcal S}'(\R^n)$ be such that
$\supp \widehat{f}\subset Q(0,R):=\{x\in\rn:\,
|x|<R\}$. Then there exists a positive constant $c_\eta$ such that,
for all $x\in\rn,$
\[
\sup_{y \in \R^n}
\frac{|f(x-y)|}{(1+R|y|)^{n/\eta}}
\le c_\eta
\left[M(|f|^\eta)(x)\right]^{\frac{1}{\eta}}.
\]
\end{lemma}

\noindent The following examples show situations when the boundedness of $M$ can be
achieved and when we can not expect it.

\begin{example}\label{e1.1}
i) Let $p\in (1,\fz)$.
It is known that the Hardy-Littlewood operator $M$
is not bounded on the weighted Lebesgue space $L^p(w)$
unless $w \in A_p(\rn)$,
where $A_p(\rn)$ is the class of \emph{Muckenhoupt weights}
(see, for example, \cite{gr85,st93} for their definitions
and properties) such that
\[
A_p(w):=\sup_{Q \in \cq}
\left[\frac{1}{|Q|}\int_Q w(x)\,dx\right]
\left[\frac{1}{|Q|}\int_Q [w(x)]^{-1/(p-1)}\,dx\right]^{p-1}<\infty.
\]
Also observe that there exists a positive constant $C_{p,q}$ such that
\[
\left\{
\int_{{\mathbb R}^n}
\left(\sum_{j=1}^\infty [Mf_j(x)]^q\right)^{q/p}w(x)\,dx
\right\}^{1/p}
\le C_{p,q}
\left\{
\int_{{\mathbb R}^n}
\left[\sum_{j=1}^\infty |f_j(x)|^q\right]^{q/p}w(x)\,dx
\right\}^{1/p}
\]
holds true for any $q\in (1, \infty]$ if and only if $w \in A_p(\rn)$. There do exist
doubling weights
which do not belong to the \emph{Muckenhoupt class}
$A_{\infty}(\rn)$ {\rm (see \cite{fm74})}.

ii)
There exists a function space
such that even the operator $M_{r,\lz}$ is difficult to control.
For example, if ${\mathcal L}(\rn):=L^{1+\chi_{{\mathbb R}^n_+}}(\rn)$,
which is the set of all measurable functions $f$ on $\rn$ such
that
\[
\|f\|_{L^{1+\chi_{{\mathbb R}^n_+}}}
:=\inf\left\{
\lambda>0\,:\,
\int_{{\mathbb R}^n_+}
\left[\frac{|f(x)|}{\lambda}\right]^2\,dx
+\int_{{\mathbb R}^n \setminus {\mathbb R}^n_+}
\frac{|f(x)|}{\lambda}\,dx
\le 1
\right\}<\fz,
\]
where ${\mathbb R}^n_+:=\{x=(x_1,\cdots,x_n)\in\rn:\ x_n\in(0,\fz)\}$,
then it is somehow well known that the maximal operator $M_{r,\lambda}$
is not bounded on $L^{1+\chi_{{\mathbb R}^n_+}}(\rn)$
(see Lemma \ref{l10.2} below).
\end{example}

Throughout the whole paper, we denote by $C$ a \emph{positive constant}
which is independent of
the main parameters, but it may vary from line to line, while
$C(\az, \bz, \cdots)$ denotes a \emph{positive constant} depending on the
parameters $\az$, $\bz$, $\cdots$. The \emph{symbols} $A\ls B$
and $A\ls_{\az,\bz,\cdots}B$ mean, respectively, that
$A\le CB$ and $A\le C(\az,\bz,\cdots)B$. If $A\ls B$ and $B\ls A$, then we write $A\sim B$. If $E$
is a subset of $\rn$, we denote by $\chi_E$ its \emph{characteristic function}.
In what follows, for all $a,\,b\in\rr$, let $a\vee b:= \max\{a,\,b\}$
and $a\wedge b:= \min\{a,\,b\}$.
Also, we let $\Z_+ :=\{0,1,2\cdots\}$.
The \emph{notation} $\lfloor x \rfloor$, for any $x\in\rr$,
means the maximal integer not larger than $x$.
The following is our convention for dyadic cubes:
For $j\in\zz$ and $k\in\zz^n$,
denote by $Q_{jk}$ the \textit{dyadic cube} $2^{-j}([0,1)^n+k)$.
Let $\mathcal{Q}(\rn):= \{Q_{jk}:\ j\in\zz,\ k\in\zz^n\}$,
$$\mathcal{Q}_j(\rn):=\{Q\in \mathcal{Q}(\rn):\ \ell(Q)=2^{-j}\}.$$
For any $Q\in\cq(\rn)$, we let $j_Q$ be $-\log_2 \ell(Q)$,
$\ell(Q)$ its \textit{side length}, $x_Q$ its \textit{lower
left-corner $2^{-j}k$} and $c_Q$ its \textit{center}.
When the dyadic cube $Q$ appears as an index, such as
$\sum_{Q\in\mathcal{Q}(\rn)}$ and
$\{\cdot\}_{Q\in\mathcal{Q}(\rn)}$,
it is understood that $Q$ runs over all dyadic cubes in $\rn$.
For any cube $Q$ and $\kappa\in (0,\fz)$, we denote by $\kappa Q$
the \emph{cube with the same center as $Q$ but $\kappa$ times the sidelength
of $Q$}. Also, we write
\begin{equation}\label{1.4}
\|\vec{\alpha}\|_1:=\sum_{j=1}^n\alpha^j
\end{equation}
for a multiindex $\vec{\alpha}:=(\alpha^1,\alpha^2,\cdots,\alpha^n)
\in\zz_+^n$.
For $\sigma:=(\sigma_1,\cdots,\sigma_n)\in\mathbb{Z}^n_+,$
$\partial^{\sigma}:=(\frac{\partial}{\partial x_1})^{\sigma_1}
\cdots(\frac{\partial}{\partial x_n})^{\sigma_n}.$

\section{Fundamental settings and inequalities}
\label{s2}

\subsection{Basic assumptions}
\label{s2.1}

First of all,
we assume that $\cl(\rn)$ is a quasi-normed space of functions on $\rn$.
Following \cite[p.\,3]{bs}, we denote by $M_0(\rn)$ the
\emph{topological vector space} of all measurable complex-valued
almost everywhere finite functions modulo null functions
(namely, any two functions coinciding almost everywhere
is identified), topologized by
$$\rho_E(f):=\int_E\min\{1,|f(x)|\}\,dx,$$
where $E$ is any subset of $\rn$ with finite Lebesgue measure.
It is easy to show that this topology of $M_0(\rn)$ is
equivalent to the topology of convergence in measure
on sets of finite measure, which makes $M_0(\rn)$
to become a metrizable topological vector space
(see \cite[p.\,30]{bs}).

First, we consider a mapping
$\|\cdot\|_{\cl(\rn)}:M_0(\rn) \to [0,\infty]$
satisfying the following fundamental conditions:
\begin{enumerate}
\item[$(\cl1)$]
An element $f \in M_0(\rn)$ satisfies
$\|f\|_{\cl(\rn)}=0$ if and only if $f=0$.
(Positivity)
\item[$(\cl2)$]
Let $f \in M_0(\rn)$ and $\alpha \in \C$.
Then
$\|\alpha f\|_{\cl(\rn)}=|\alpha|\|f\|_{\cl(\rn)}$.
(Homogeneity)
\item[$(\cl3)$]
The norm $\|\cdot\|_{\cl(\rn)}$ satisfies
the $\theta$-triangle inequality.
That is, there exists a positive constant
$\theta=\theta(\cl(\rn))\in (0,1]$ such that
\[
\|f+g\|_{\cl(\rn)}^\theta
\le
\|f\|_{\cl(\rn)}^\theta
+
\|g\|_{\cl(\rn)}^\theta
\]
for all $f,g \in M_0(\rn)$.
(The $\theta$-triangle inequality)

\item[$(\cl4)$]
If a pair $(f,g) \in M_0(\rn) \times M_0(\rn)$ satisfies $|g| \le |f|$,
then $\|g\|_{\cl(\rn)} \le \|f\|_{\cl(\rn)}$.
(The lattice property)

\item[$(\cl5)$]
Suppose that $\{f_j\}_{j=1}^\infty$ is a sequence of functions
satisfying
$$
\sup_{j \in \N}\|f_j\|_{\cl(\rn)}<\infty, \quad
0 \le f_1 \le f_2 \le \ldots \le f_j \le \ldots.
$$
Then the limit
$f:=\lim_{j \to \infty}f_j$
belongs to $\cl(\rn)$ and
$\|f\|_{\cl(\rn)}\le\sup_{j \in \N}\|f_j\|_{\cl(\rn)}$
holds true. (The Fatou property)
\end{enumerate}
Given the mapping $\|\cdot\|_{\cl(\rn)}$ satisfying
$(\cl1)$ through $(\cl5)$,
the \emph{space $\cl(\rn)$} is defined by
\[
\cl(\rn) :=\{f \in M_0(\rn)\,:\,\|f\|_{\cl(\rn)}<\infty\}.
\]

\begin{remark}\label{r2.x1}
We point out that the assumptions $(\cl1)$, $(\cl2)$ and $(\cl3)$
can be replaced by the following assumption:

$\cl(\rn)$ is a quasi-normed linear space of functions.
Indeed, if  $(\cl(\rn), \|\cdot\|_{\cl(\rn)})$ is a quasi-normed
linear space of function,
by the Aoki-Rolewicz theorem (see \cite{a42,r57}),
there exists an equivalent quasi-norm $\qn\cdot\qn$ and $\wz\tz\in(0,1]$ such that,
for all $f,g\in\cl(\rn)$,
\begin{equation}\label{2.x1}
\|\cdot\|_{\cl(\rn)}\sim\qn\cdot\qn,
\end{equation}
\begin{equation*}
\qn f+g\qn^{\wz\tz}\le\qn f\qn^{\wz\tz}+\qn g\qn^{\wz\tz}.
\end{equation*}
Thus, $(\cl(\rn), \qn\cdot\qn)$ satisfies $(\cl1)$, $(\cl2)$ and $(\cl3)$.
Since all results are invariant on equivalent quasi-norms,
by \eqref{2.x1}, we know that all results are still true for the quasi-norm
$\|\cdot\|_{\cl(\rn)}$.
\end{remark}

Motivated by \cite{ru11,u10},
we also assume that $\cl(\rn)$ enjoys the following property.

\begin{enumerate}
\item[$(\cl6)$]
The $(1+|\cdot|)^{-N_0}$ belongs to $\cl(\rn)$
for some $N_0\in (0,\fz)$
and the estimate
\begin{equation}\label{eq:L6}
\|\chi_{Q_{jk}}\|_{\cl(\rn)}=
\|\chi_{2^{-j}k+2^{-j}[0,1)^n}\|_{\cl(\rn)}
\gtrsim 2^{-j\gamma}(1+|k|)^{-\delta},
\quad j \in \Z_+, \, k \in \Z^n
\end{equation}
holds true for some $\gamma,\delta\in[0,\fz)$,
where the implicit positive constant is independent of $j$ and $k$.
(The non-degenerate condition)
\end{enumerate}

We point out that $(\cl6)$ is a key assumption, which
makes our definitions of quasi-normed spaces
a little different from that in \cite{bs}. This condition has been used
by Rauhut and Ullrich \cite[Definition 4.4]{ru11} in order to define coorbits
of Peetre type spaces in a reasonable way. Indeed, in \cite{bs},
it is necessary to assume that $\chi_E \in \cl(\rn)$
if $E$ is a measurable set of finite measure.

Moreover, from $(\cl4)$ and $(\cl5)$, we deduce the following
Fatou property of $\cl(\rn)$.

\begin{proposition}\label{p2.1}
If $\cl(\rn)$ satisfies $(\cl4)$ and $(\cl5)$, then, for all
sequences $\{f_m\}_{m\in\nn}$ of non-negative functions of $\cl(\rn)$,
$$\lf\|\liminf_{m\to\fz}f_m\r\|_{\cl(\rn)}\le \liminf_{m\to\fz}\|f_m\|_{\cl(\rn)}.$$
\end{proposition}

\begin{proof}
Without loss of generality, we may assume that $\liminf_{m\to\fz}\|f_m\|_{\cl(\rn)}<\fz.$
Recall that $\liminf_{m\to\fz}f_m=\sup_{m\in\nn}\inf_{k\ge m}\{f_k\}$. For all
$m\in\nn$, let $g_m:=\inf_{k\ge m}\{f_k\}$. Then $\{g_m\}_{m\in\nn}$
is a sequence of nonnegative functions satisfying that
$g_1\le g_2 \le \cdots \le g_m\le\cdots.$
Moreover, by $(\cl4)$, we conclude that
$$\sup_{m\in\nn}\|g_m\|_{\cl(\rn)}
\le\liminf_{m\to\fz}\|f_m\|_{\cl(\rn)}<\fz.$$
Then, from this and $(\cl5)$, we further deduce that
$\liminf_{m\to\fz}f_m=\sup_{m\in\nn}\{g_m\}\in \cl(\rn)$ and
$$\lf\|\liminf_{m\to\fz}f_m\r\|_{\cl(\rn)}
\le\sup_{m\in\nn}\|g_m\|_{\cl(\rn)}\le \liminf_{m\to\fz}\|f_m\|_{\cl(\rn)},$$
which completes the proof of Proposition \ref{p2.1}.
\end{proof}

We also remark that the completeness of $\cl(\rn)$
is not necessary.
It is of interest to have completeness automatically,
as Proposition \ref{p3.1} below shows.

Let us additionally recall the following class
$\mathcal{W}^{\alpha_3}_{\alpha_1,\alpha_2}$
of weights which was used recently in \cite{ru11}.
This class of weights has been introduced for the
definition of $2$-microlocal Besov-Triebel-Lizorkin spaces;
see \cite{k10-1,k10-2}. As in Example \ref{e1.1}(ii), let
$$\rr^{n+1}_+:=\{(x,x_{n+1}):\ x\in\rn,\ x_{n+1}\in(0,\fz)\}.$$ We also
let
$\R^{n+1}_{\Z_+}:=\{(x,t) \in \R^{n+1}_+:\ -\log_2 t \in \Z_+\}$.

\begin{definition}\label{d2.1}
Let $\alpha_1, \alpha_2, \alpha_3 \in [0,\infty)$.
The \emph{class}
${\mathcal W}_{\alpha_1,\alpha_2}^{\alpha_3}$
of weights is defined
as the set of all measurable functions
$w:\rr^{n+1}_{\zz_+} \to (0,\infty)$
satisfying the following conditions{\rm:}

(W1) There exists a positive constant $C$ such that, for all $x \in \rn$
and $j,\, \nu \in \Z_+$
with $j \ge \nu$,
\begin{equation}
\label{2.1}
C^{-1}2^{-(j-\nu)\alpha_1}
w(x,2^{-\nu})
\le
w(x,2^{-j})
\le
C2^{-(\nu-j)\alpha_2}
w(x,2^{-\nu}).
\end{equation}

(W2) There exists a positive constant $C$ such that, for all $x,\,y \in \rn$
and $j \in \Z_+$,
\begin{equation}\label{2.2}
w(x,2^{-j})
\le Cw(y,2^{-j})
\left(1+2^{j}|x-y|\right)^{\alpha_3}.
\end{equation}
\end{definition}

Given a weight $w$ and $j \in \zz_+$,
we often write
\begin{equation}\label{2.3}
w_j(x):=
w(x,2^{-j})
\quad (x \in \rn, \, j \in \Z_+),
\end{equation}
which is a convention used
until the end of Section \ref{s8}.
With the convention $(\ref{2.3})$,
$(\ref{2.1})$ and $(\ref{2.2})$ can be read as
\[
C^{-1}2^{-(j-\nu)\az_1}w_\nu(x)
\le
w_j(x)
\le
C2^{-(\nu-j)\az_2}w_\nu(x)
\]
and
\[
w_j(x) \le C w_j(y)(1+2^j|x-y|)^{\az_3},
\]
respectively.
In what follows, for all $a\in\rr$,
$a_+:=\max(a,0)$.

\begin{example}\label{e2.1}
(i)
The most familiar case,
the classical Besov spaces $B^s_{p,q}(\rn)$ and Triebel-Lizorkin spaces
$F^s_{p,q}(\rn)$,
are realized by letting
$
w_j\equiv 2^{js}
$
with $j \in \Z_+$ and $s \in \R$.

(ii)
In general
when $w_j(x)$ with $j\in\zz_+$ and $x\in\rn$ is independent of $x$,
then we see that $\alpha_3=0$.
For example, when $w_j(x) \equiv 2^{js}$ for some $s \in \R$ and all $x\in\rn$.
Then
$w_j \in \cw_{\max(0,-s),\max(0,s)}^0$.

(iii)
Let $w \in \cw_{\az_1,\az_2}^{\az_3}$ and $s \in \R$.
Then the weight given by
\[
\wz w_j(x)
:=2^{js}w_j(x)
\quad (x \in \rn, \, j \in \Z_+)
\]
belongs to the class $\cw_{(\az_1-s)_+,(\az_2+s)_+}^{\az_3}$.
\end{example}

In the present paper,
we consider six underlying function spaces,
two of which are special cases
of other four spaces.
At first glance the definitions
of
$\ell^q(\cl^w_\tau(\rn,\zz_+))$
and
$\ell^q(\cn\cl^w_\tau(\rn,\zz_+))$
might be identical.
However, in \cite{syy},
we showed that they are different in general cases.
In the present paper,
we generalize this fact
in Theorem \ref{t8.6}.

\begin{definition}\label{d2.2}
Let $q\in(0,\fz]$ and $\tau\in[0,\fz)$.
Suppose
$w \in {\mathcal W}^{\alpha_3}_{\alpha_1,\alpha_2}$
with $\alpha_1, \alpha_2, \alpha_3 \in [0,\infty)$.
Let $w_j$ for $j\in\zz_+$ be as in \eqref{2.3}.

{\rm (i)}
The \emph{space}
$\cl^w(\ell^q(\rn,\zz_+))$ is defined to be the
set of all sequences $G:=\{g_j\}_{j\in\zz_+}$ of
measurable functions on $\rn$ such that
\begin{equation}\label{2.4}
\|G\|_{\cl^w(\ell^q(\rn,\zz_+))}
:=
\left\|\lf(\sum_{j=0}^\fz
|w_jg_j|^q\r)^{1/q}\right\|_{\cl(\rn)}<\fz.
\end{equation}
In analogy,
the \emph{space}
$\cl^w(\ell^q(\rn,E))$
is defined for a subset $E \subset \zz$.

{\rm (ii)}
The \emph{space}
$\ell^q(\cl^w(\rn,\zz_+))$ is defined to be the set of all
sequences $G:=\{g_j\}_{j\in\zz_+}$ of measurable
functions on $\rn$ such that
\begin{equation}\label{2.5}
\|G\|_{\ell^q(\cl^w(\rn,\zz_+))}
:=
\lf\{\sum_{j=0}^\fz \|w_jg_j\|_{\cl(\rn)}^q\r\}^{1/q}<\fz.
\end{equation}
In analogy,
the \emph{space}
$\ell^q(\cl^w(\rn,E))$
is defined for a subset $E \subset \zz$.

{\rm (iii)}
The \emph{space}
$\cl^w_\tau(\ell^q(\rn,\zz_+))$
is defined to be the set of all
sequences $G:=\{g_j\}_{j\in\zz_+}$ of measurable
functions on $\rn$ such that
\begin{equation}\label{2.6}
\|G\|_{\cl^w_\tau(\ell^q(\rn,\zz_+))}
:= \sup_{P\in\mathcal{Q}(\rn)}\frac1{|P|^\tau}
\|\{\chi_P w_jg_j\}_{j=j_P \vee 0}^\infty\|_{\cl^w(\ell^q(\rn,\zz_+
\cap[j_P,\infty)))}
<\fz.
\end{equation}

{\rm (iv)}
The \emph{space}
$\ce\cl^w_\tau(\ell^q(\rn,\zz_+))$ is defined to be the
set of all sequences $G:=\{g_j\}_{j\in\zz_+}$ of
measurable functions on $\rn$ such that
\begin{equation}\label{2.9}
\|G\|_{\ce\cl^w_\tau(\ell^q(\rn,\zz_+))}
:= \sup_{P\in\mathcal{Q}(\rn)}\frac1{|P|^\tau}
\|\{\chi_P w_jg_j\}_{j=0}^\infty\|_{\cl^w(\ell^q(\rn,\zz_+))}<\fz.
\end{equation}

{\rm (v)}
The \emph{space}
$\ell^q(\cl^w_\tau(\rn,\zz_+))$ is defined to be the
set of all sequences $G:=\{g_j\}_{j\in\zz_+}$ of
measurable functions on $\rn$ such that
\begin{equation}\label{2.8}
\|G\|_{\ell^q(\cl^w_\tau(\rn,\zz_+))}
:= \sup_{P\in\mathcal{Q}(\rn)}\frac1{|P|^\tau}
\|\{\chi_P w_jg_j\}_{j=j_P \vee 0}^\infty\|_{\ell^q(\cl^w(\rn,\zz_+
\cap[j_P,\infty)))}<\fz.
\end{equation}

{\rm (vi)}
The \emph{space}
$\ell^q(\cn\cl^w_\tau(\rn,\zz_+))$
is defined to be the set of all
sequences $G:=\{g_j\}_{j\in\zz_+}$ of measurable
functions on $\rn$ such that
\begin{equation}\label{2.7}
\|G\|_{\ell^q(\cn\cl^w_\tau(\rn,\zz_+))}
:=
\lf\{\sum_{j=0}^\fz
\sup_{P\in\mathcal{Q}(\rn)}
\left(\frac{\|\chi_P w_jg_j\|_{\cl(\rn)}}{|P|^{\tau}}\right)^q\r\}^{1/q}<\fz.
\end{equation}

When $q=\infty$, a natural modification is made
in {\rm \eqref{2.4}} through {\rm \eqref{2.7}}.
\end{definition}

We also introduce the homogeneous counterparts of these spaces in Section \ref{s9}.
One of the reasons why we are led to introduce
${\mathcal W}^{\az_3}_{\az_1,\az_2}$
is the necessity of describing the smoothness
by using our new weighted function spaces
more precisely
than by using the classical Besov-Triebel-Lizorkin spaces.
For example, in \cite{yo10},
Yoneda considered the following norm. In what follows, ${\mathcal P}(\rn)$
denotes the \emph{set of all polynomials on $\rn$}.

\begin{example}[{\rm \cite{yo10}}]\label{e2.2}
The \emph{space }$\dot{B}_{\infty\infty}^{-1,\sqrt{\cdot}}(\rn)$
denotes the set of all $f \in {\mathcal S}'(\rn)/{\mathcal P}(\rn)$
for which the norm
\begin{equation*}
\|f\|_{\dot{B}_{\infty\infty}^{-1,\sqrt{\cdot}}(\rn)}
:=\sup_{j \in {\mathbb Z}}
2^{-j}\sqrt{|j|+1}
\|\varphi_j*f\|_{L^\infty(\rn)}<\infty.
\end{equation*}
\end{example}

If $\tau=0$, $a\in (0,\fz)$ and $w_j(x):=2^{-j}\sqrt{|j|+1}$
for all $x\in\rn$ and $j\in\zz$, then it can be shown that the space
$\dot{B}_{\infty\infty}^{-1,\sqrt{\cdot}}(\rn)$ and
the space $\dot B^{w,\tau}_{L^\fz,\fz,a}(\rn)$,
introduced in Definition \ref{d9.4}
below, coincide with equivalent norms. This can be proved
by an argument similar to that used in the proof
of \cite[Theorem 2.9]{u10} and we omit the details.
An inhomogeneous variant of this result is also true.
Moreover, we refer to Subsection \ref{s10.8}
for another example of non-trivial weights $w$.
This is a special case of generalized smoothness.
The weight $w$ also plays a role of variable smoothness.

In the present paper,
the spaces
$\ell^q(\cl^w_\tau(\rn,\zz_+))$,
$\ell^q(\cn\cl^w_\tau(\rn,\zz_+))$,
$\cl^w_\tau(\ell^q(\rn,\zz_+))$
and
$\ce\cl^w_\tau(\ell^q(\rn,\zz_+))$
play the central role,
while $\ell^q(\cl^w(\rn,\zz_+))$ and
$\cl^w(\ell^q(\rn,\zz_+))$
are auxiliary spaces.

By the monotonicity of $\ell^q$, we immediately obtain
the following useful conclusions. We omit the details.

\begin{lemma}\label{l2.1}
Let $0<q_1 \le q_2\le\infty$
and $\alpha_1, \alpha_2, \alpha_3, \tau \in [0,\infty)$
and
$w \in {\mathcal W}^{\alpha_3}_{\alpha_1,\alpha_2}$.
Then
\begin{eqnarray*}
&&\ell^{q_1}(\cl^w(\rn,\zz_+))
\hookrightarrow
\ell^{q_2}(\cl^w(\rn,\zz_+)),\\
&&\cl^w(\ell^{q_1}(\rn,\zz_+))
\hookrightarrow
\cl^w(\ell^{q_2}(\rn,\zz_+)),\\
&&\ell^{q_1}(\cl^w_\tau(\rn,\zz_+))
\hookrightarrow
\ell^{q_2}(\cl^w_\tau(\rn,\zz_+)),\\
&&\ell^{q_1}(\cn\cl^w_\tau(\rn,\zz_+))
\hookrightarrow
\ell^{q_2}(\cn\cl^w_\tau(\rn,\zz_+)),\\
&&\cl^w_\tau(\ell^{q_1}(\rn,\zz_+))
\hookrightarrow
\cl^w_\tau(\ell^{q_2}(\rn,\zz_+))
\end{eqnarray*}
and
\begin{eqnarray*}
\ce\cl^w_\tau(\ell^{q_1}(\rn,\zz_+))
\hookrightarrow
\ce\cl^w_\tau(\ell^{q_2}(\rn,\zz_+))
\end{eqnarray*}
in the sense of continuous embeddings.
\end{lemma}

\subsection{Inequalities}

Let us suppose that we are given a quasi-normed space
${\mathcal L}(\rn)$ satisfying $(\cl1)$ through $(\cl6)$.
The following lemma is immediately deduced
from $(\cl4)$ and $(\cl5)$. We omit the details.

\begin{lemma}\label{l2.2}
Let $q\in (0,\infty]$ and $w$ be as in Definition \ref{d2.2}.
If $\cl(\rn)$ is a quasi-normed space,
then

{\rm (i)}
the quasi-norms
$\|\cdot\|_{\ell^q(\cl^w_0(\rn,\zz_+))}$,
$\|\cdot\|_{\ell^q(\cn\cl^w_0(\rn,\zz_+))}$
and
$\|\cdot\|_{\ell^q(\cl^w(\rn,\zz_+))}$
are mutually equivalent;

{\rm (ii)}
the quasi-norms
$\|\cdot\|_{\cl^w_0(\ell^q(\rn,\zz_+))}$,
$\|\cdot\|_{\ce\cl^w_0(\ell^q(\rn,\zz_+))}$
and
$\|\cdot\|_{\cl^w(\ell^q(\rn,\zz_+))}$
are mutually equivalent.
\end{lemma}

Based on Lemma \ref{l2.2}, in what follows, we identify the spaces
appearing, respectively, in (i) and (ii) of Lemma \ref{l2.2}.

The following fundamental estimates
(\ref{2.14})-(\ref{2.17}) follow from the H\"older inequality
and the condition (W1) and (W2).
However, we need to keep in mind
that the condition \eqref{2.13} below is used
throughout the present paper.

\begin{lemma}\label{l2.3}
Let $D_1, D_2, \alpha_1, \alpha_2, \alpha_3, \tau \in [0,\infty)$
and $q\in(0,\fz]$ be fixed parameters satisfying that
\begin{equation}\label{2.13}
D_1\in(\alpha_1,\fz), \quad
D_2\in(n\tau+\alpha_2,\fz).
\end{equation}
Suppose that $\{g_\nu\}_{\nu\in\zz_+}$ is a given
family of measurable functions on $\rn$ and $w\in\cw_{\az_1,\az_2}^{\az_3}$.
For all $j\in\zz_+$ and $x\in\rn$, let
$$G_j(x):=\sum_{\nu=0}^{j}
2^{-(j-\nu)D_2}g_\nu(x)+
\sum_{\nu=j+1}^\fz
2^{-(\nu-j)D_1}g_\nu(x).$$
If $\cl(\rn)$ satisfies $(\cl1)$ through $(\cl4)$,
then the following estimates, with implicit positive constants
independent of $\{g_\nu\}_{\nu\in\zz_+}$,
hold true:
\begin{gather}
\label{2.14}
\|\{G_j\}_{j\in\zz_+}\|_{\ell^q(\cl^w_\tau(\rn,\zz_+))}
\ls
\|\{g_\nu\}_{\nu\in\zz_+}\|_{\ell^q(\cl^w_\tau(\rn,\zz_+))},\\
\label{2.15}
\|\{G_j\}_{j\in\zz_+}\|_{\ell^q(\cn\cl^w_\tau(\rn,\zz_+))}
\ls
\|\{g_\nu\}_{\nu\in\zz_+}\|_{\ell^q(\cn\cl^w_\tau(\rn,\zz_+))},\\
\label{2.16}
\|\{G_j\}_{j\in\zz_+}\|_{\cl^w_\tau(\ell^q(\rn,\zz_+))}
\ls
\|\{g_\nu\}_{\nu\in\zz_+}\|_{\cl^w_\tau(\ell^q(\rn,\zz_+))}
\end{gather}
and
\begin{gather}
\label{2.17}
\|\{G_j\}_{j\in\zz_+}\|_{\ce\cl^w_\tau(\ell^q(\rn,\zz_+))}
\ls
\|\{g_\nu\}_{\nu\in\zz_+}\|_{\ce\cl^w_\tau(\ell^q(\rn,\zz_+))}.
\end{gather}
\end{lemma}

\begin{proof}
Let us prove \eqref{2.16}.
The proofs of \eqref{2.14}, \eqref{2.15} and \eqref{2.17} are similar
and we omit the details.
Let us write
\begin{align*}
{\rm I}(P):=&
\frac1{|P|^\tau}
\lf\|\chi_P \lf[\sum_{j=j_P\vee0}^\fz
\lf|\sum_{\nu=0}^j w_j2^{(\nu-j)D_2}g_\nu\r|^q
\r]^{1/q}\r\|_{\cl(\rn)}\\
&+\frac1{|P|^\tau}
\lf\|\chi_P \lf[\sum_{j=j_P\vee0}^\fz
\lf|\sum_{\nu=j+1}^\fz w_j2^{(j-\nu)D_1}g_\nu\r|^q
\r]^{1/q}\r\|_{\cl(\rn)},
\end{align*}
where $P$ is a dyadic cube chosen arbitrarily.
If $j,\nu \in \Z_+$ and $\nu \ge j$,
then by \eqref{2.1}, we know that, for all $x\in\rn$,
\begin{equation}\label{2.18}
w_j(x) \ls 2^{-\alpha_1(j-\nu)}w_\nu(x).
\end{equation}
If $j,\nu \in \zz_+$ and $j \ge \nu$,
then by \eqref{2.1}, we see that,
for all $x\in\rn$,
\begin{equation}\label{2.19}
w_j(x) \ls 2^{\alpha_2(j-\nu)}w_\nu(x).
\end{equation}
If we combine \eqref{2.18} and \eqref{2.19},
then we conclude that, for all $x\in\rn$ and
$j,\nu \in \Z_+$,
\begin{equation}\label{2.20}
w_j(x) \ls
\begin{cases}
2^{-\alpha_1(j-\nu)}w_\nu(x), &\nu \ge j;\\
2^{\alpha_2(j-\nu)}w_\nu(x), &\nu \le j.
\end{cases}
\end{equation}
We need to show that
\begin{equation*}
{\rm I}(P)
\lesssim
\|\{g_\nu\}_{\nu\in\zz_+}\|_{\cl^w_\tau(\ell^q(\rn,\zz_+))}
\end{equation*}
with the implicit positive constant independent of $P$
and $\{g_\nu\}_{\nu\in \zz_+}$
in view of the definitions of $\{G_j\}_{j\in\zz_+}$ and
$\|\{G_j\}_{j\in\zz_+}\|_{\cl^w_\tau(\ell^q(\rn,\zz_+))}$.

Let us suppose $q\in(0,1]$ for the moment.
Then we deduce, from \eqref{2.20} and $(\cl4)$, that
\begin{align}\label{2.21}
{\rm I}(P)
\ls&\frac1{|P|^\tau}
\lf\|\chi_P \lf[\sum_{j=j_P\vee0}^\fz
\sum_{\nu=0}^j 2^{-(j-\nu)(D_2-\alpha_2) q}
\lf|w_\nu g_\nu\r|^q
\r]^{1/q}\r\|_{\cl(\rn)}\noz\\
&+\frac1{|P|^\tau}
\lf\|\chi_P \lf[\sum_{j=j_P\vee0}^\fz
\sum_{\nu=j+1}^\infty 2^{-(\nu-j)(D_1-\alpha_1) q}
\lf|w_\nu g_\nu\r|^q
\r]^{1/q}\r\|_{\cl(\rn)}
\end{align}
by the inequality that, for all $r\in(0,1]$ and
$\{a_j\}_j\st \cc$,
\begin{equation}\label{2.22}
\lf(\sum_{j}|a_j|\r)^r\le\sum_{j}|a_j|^r.
\end{equation}
In \eqref{2.21},
we change the order of the summations in its right-hand side to obtain
\begin{align}
{\rm I}(P)
\ls&\frac1{|P|^\tau}
\lf\|\chi_P \lf[\sum_{\nu=0}^\fz\sum_{j=\nu\vee j_P\vee0}^\fz
2^{-(j-\nu)(D_2-\alpha_2) q}
\lf|w_\nu g_\nu\r|^q
\r]^{1/q}\r\|_{\cl(\rn)}\noz\\
\noz
&+\frac1{|P|^\tau}
\lf\|\chi_P \lf[\sum_{\nu=j_P\vee0}^\infty\sum_{j=j_P\vee0}^\nu
2^{-(\nu-j)(D_1-\alpha_1) q}
\lf|w_\nu g_\nu\r|^q
\r]^{1/q}\r\|_{\cl(\rn)}.
\end{align}
Now we decompose the summations
with respect to $\nu$ according to
$\nu \ge j_P \vee 0$ or $\nu< j_P \vee 0$.
Since $D_2\in(\alpha_2+n\tau,\fz)$,
we can choose $\ez\in(0,\fz)$
such that $D_2\in(\alpha_2+n\tau+\ez,\fz)$.
From this, $D_1\in(\alpha_1,\fz)$,
the H\"older inequality, $(\cl2)$ and $(\cl4)$,
it follows that
\begin{eqnarray}\label{eq:120223-1}
{\rm I}(P)&&\lesssim
\|\{g_\nu\}_{\nu\in\zz_+}\|_{\cl^w_\tau(\ell^q(\rn,\zz_+))}\notag\\
&&\hs+\frac1{|P|^\tau}
\lf\|\chi_P \lf[\sum_{\nu=0}^{j_P \vee 0}\sum_{j=j_P\vee0}^\fz
2^{-(j-\nu)(D_2- \alpha_2) q}
\lf|w_\nu g_\nu\r|^q
\r]^{\frac 1q}\r\|_{\cl(\rn)}\notag\\
&&\ls\|\{g_\nu\}_{\nu\in\zz_+}\|_{\cl^w_\tau(\ell^q(\rn,\zz_+))}\notag\\
&&\hs+\frac{2^{-(j_P\vee0)(D_2-\az_2-\ez)}}{|P|^\tau}
\lf\|\chi_P \sum_{\nu=0}^{j_P \vee 0}
2^{\nu(D_2-\alpha_2-\ez)}
\lf|w_\nu g_\nu\r|
\r\|_{\cl(\rn)}.
\end{eqnarray}
We write $2^{j_P \vee 0-\nu}P$
for the $2^{j_P \vee 0-\nu}$ times expansion of $P$
as our conventions at the end of Section \ref{s1}.
If we use the assumption $(\cl3)$, we see that
\begin{eqnarray*}
{\rm I}(P)&&\lesssim
\|\{g_\nu\}_{\nu\in\zz_+}\|_{\cl^w_\tau(\ell^q(\rn,\zz_+))}\notag\\
&&\hs+\frac{2^{-(j_P\vee0)(D_2-\az_2-\ez)}}{|P|^\tau}
\left\{\sum_{\nu=0}^{j_P \vee 0}
\lf\|2^{\nu(D_2-\alpha_2-\ez)}
\chi_P w_\nu g_\nu
\r\|_{\cl(\rn)}^{\theta}\right\}^{1/\theta}\notag\\
&&\lesssim
\|\{g_\nu\}_{\nu\in\zz_+}\|_{\cl^w_\tau(\ell^q(\rn,\zz_+))}\notag\\
&&\hspace{0.2cm}+2^{-(j_P\vee0)(D_2-\az_2-\ez)}\left\{\sum_{\nu=0}^{j_P \vee 0}
\left[
\frac{2^{\nu(D_2-\alpha_2-n\tau-\ez)+n\tau(j_P\vee 0)}}
{|2^{(j_P \vee 0)-\nu}P|^\tau}\lf\|\chi_{2^{(j_P \vee 0)-\nu}P}
w_\nu g_\nu\r\|_{\cl(\rn)}\right]^\theta\right\}^{1/\theta}\notag\\
\notag
&&\lesssim
\|\{g_\nu\}_{\nu\in\zz_+}\|_{\cl^w_\tau(\ell^q(\rn,\zz_+))}.
\end{eqnarray*}
Since the dyadic cube $P$ is arbitrary,
by taking the supremum of all $P$,
the proof of the case that $q \in(0, 1]$ is now complete.

When $q \in (1,\fz]$,
choose $\kappa\in(0,\fz)$ such that $\kappa+\az_1<D_1$ and
$\kappa+n\tau+\alpha_2<D_2$.
Then, by virtue of the H\"{o}lder inequality, we are led to
\begin{eqnarray}
{\rm I}(P)
&&\le\frac1{|P|^\tau}\lf\{
\lf\|\chi_P \lf[\sum_{j=j_P\vee0}^\fz
\sum_{\nu=0}^j 2^{-(j-\nu)(D_2-\kappa-\alpha_2) q}
\lf|w_\nu g_\nu\r|^q
\r]^{1/q}\r\|_{\cl(\rn)}\r.\noz\\
&&\hs\lf.+\lf\|\chi_P \lf[\sum_{j=j_P\vee0}^\fz
\sum_{\nu=j+1}^\fz 2^{-(\nu-j)(D_2-\kappa-\alpha_2) q}
\lf|w_\nu g_\nu\r|^q
\r]^{1/q}\r\|_{\cl(\rn)}\r\},\noz
\end{eqnarray}
where the only difference from \eqref{2.21}
lies in the point that $D_1$ and $D_2$ are, respectively,
replaced by $D_1-\kappa$ and $D_2-\kappa$.
With $D_1$ and $D_2$, respectively, replaced by $D_1-\kappa$ and $D_2-\kappa$,
the same argument as above works.
This finishes the proof of Lemma \ref{l2.3}.
\end{proof}

The following lemma is frequently used
in the present paper, which previously appeared in \cite[Lemmas B.1 and B.2]{fj90}, \cite[p.\,466]{gr08},
\cite[Lemmas 1.2.8 and 1.2.9]{hn07}, \cite[Lemma 1]{r99}
 or \cite[Lemma A.3]{u10}. In the last reference
the result is stated in terms of the continuous wavelet transform.
Denote by $\omega_n$ the \emph{volume of the unit ball in $\rn$}
and by $C^L(\R^n)$ the \emph{space of all functions having continuous
derivatives up to order $L$}.

\begin{lemma}\label{l2.4}
Let $j,\nu\in\zz_+$, $M,N\in(0,\fz)$,
and $L\in\mathbb{N}\cup\{0\}$ satisfy $\nu\geq j$
and $N>M+L+n$.
Suppose that $\phi_j\in C^L(\R^n)$ satisfies that,
for all $\|{\vec{\alpha}}\|_1=L$,
\begin{equation*}
\lf|\partial^{\vec{\alpha}}\phi_j(x)\r|
\leq A_{\vec{\alpha}}\frac{2^{j(n+L)}}{(1+2^{j}|x-x_j|)^M},
\end{equation*}
where $A_{\vec{\az}}$ is a positive constant independent of $j$, $x$ and $x_j$.
Furthermore, suppose that another function
$\phi_\nu$ is a measurable function satisfying that, for all $\|\vec{\beta}\|_1\leq L-1$,
\begin{equation*}
\int_{\R^n}\phi_\nu(y) y^{\vec{\beta}}\,dy=0\
\text{and, for all}\ x\in\rn,\
|\phi_\nu(x)|\leq B\frac{2^{\nu n}}{(1+2^\nu|x-x_\nu|)^N},
\end{equation*}
where the former condition is supposed to be vacuous in the case when $L=0$. Then
\begin{equation*}
\left|\int_{\R^n}\phi_j(x)\phi_\nu(x)dx\right|
\leq
\left(\sum_{\|{\vec{\alpha}}\|_1=L}
\frac{A_{\vec{\alpha}}}{{\vec{\alpha}}!}\right)\frac{N-M-L}{N-M-L-n}B\omega_n
\,2^{j n-(\nu-j)L}(1+2^j|x_j-x_\nu|)^{-M}.
\end{equation*}
\end{lemma}

\section{Besov-type and Triebel-Lizorkin-type spaces}\label{s3}

\subsection{Definitions}\label{s3.1}

Through the spaces in Definition \ref{d2.2},
we introduce the following Besov-type and Triebel-Lizorkin-type spaces on $\rn$.

\begin{definition}\label{d3.1}
Let $a\in (0,\fz)$,
$\alpha_1, \alpha_2, \alpha_3, \tau\in[0,\fz)$, $q\in(0,\,\fz]$
and $w \in \cw_{\alpha_1,\alpha_2}^{\alpha_3}$.
Assume that $\Phi,\,\vz\in\cs(\rn)$ satisfy, respectively,
\eqref{1.1} and \eqref{1.2} and that $\cl(\rn)$ is a quasi-normed
space satisfying $(\cl1)$ through $(\cl4)$.
For any $f \in \cs'(\rn)$, let $\{(\vz_j^\ast f)_a\}_{j\in\zz_+}$
be as in \eqref{1.3}.

{\rm (i)} The {\it inhomogeneous generalized Besov-type space
$B^{w,\tau}_{\cl,q,a}(\rn)$}
is defined to be the set of
all $f\in\cs'(\rn)$ such that
\begin{equation*}
\|f\|_{B^{w,\tau}_{\cl,q,a}(\rn)}:=
\lf\|\lf\{(\vz_j^\ast f)_a\r\}_{j\in\zz_+}\r\|_{\ell^q(\cl_\tau^w(\rn,\zz_+))}<\fz.
\end{equation*}

{\rm (ii)} The {\it inhomogeneous generalized Besov-Morrey space
$\cn^{w,\tau}_{\cl,q,a}(\rn)$}
is defined to be the set of
all $f\in\cs'(\rn)$ such that
\begin{equation*}
\|f\|_{\cn^{w,\tau}_{\cl,q,a}(\rn)}:=
\lf\|\lf\{(\vz_j^\ast f)_a\r\}_{j\in\zz_+}\r\|_{\ell^q(\cn\cl_\tau^w(\rn,\zz_+))}<\fz.
\end{equation*}

{\rm (iii)} The {\it inhomogeneous generalized Triebel-Lizorkin-type space
$F^{w,\tau}_{\cl,q,a}(\rn)$}
is defined to be the set of all $f\in\cs'(\rn)$ such that
\begin{equation*}
\|f\|_{F^{w,\tau}_{\cl,q,a}(\rn)}:=
\lf\|\lf\{(\vz_j^\ast f)_a\r\}_{j\in\zz_+}\r\|_{\cl_\tau^w(\ell^q(\rn,\zz_+))}<\fz.
\end{equation*}%

{\rm (iv)} The {\it inhomogeneous generalized Triebel-Lizorkin-Morrey space
$\ce^{w,\tau}_{\cl,q,a}(\rn)$}
is defined to be the set of all $f\in\cs'(\rn)$ such that
\begin{equation*}
\|f\|_{\ce^{w,\tau}_{\cl,q,a}(\rn)}:=
\lf\|\lf\{(\vz_j^\ast f)_a\r\}_{j\in\zz_+}\r\|_{\ce\cl_\tau^w(\ell^q(\rn,\zz_+))}<\fz.
\end{equation*}
\end{definition}

The \emph{space $A^{w,\tau}_{\cl,q,a}(\rn)$}
stands for either one of $B^{w,\tau}_{\cl,q,a}(\rn)$,
$\cn^{w,\tau}_{\cl,q,a}(\rn)$, $F^{w,\tau}_{\cl,q,a}(\rn)$
or $\ce^{w,\tau}_{\cl,q,a}(\rn)$.
When ${\mathcal L}(\rn)=L^p(\rn)$ and $w_j(x):= 2^{js}$
for $x \in \rn$ and $j \in \Z_+$,
write
\begin{equation}\label{3.5}
A^{s,\tau}_{p,q,a}(\rn):= A^{w,\tau}_{{\mathcal L},q,a}(\rn).
\end{equation}

In what follows, if $\tau=0$, we omit $\tau$ in the notation of the spaces
introduced by Definition \ref{d3.1}.

\begin{remark}\label{r2.5}
Let us review what parameters function spaces
carry with.

i) The function space $\cl(\rn)$ is equipped with
$\theta, N_0, \gamma, \delta$ satisfying
\begin{equation}\label{2.10}
\theta \in (0, 1], \quad
N_0\in(0,\infty), \quad
\gamma \in [0,\infty), \quad
\delta \in [0,\infty).
\end{equation}

ii) The class $\cw_{\az_1,\az_2}^{\az_3}$ of weights is equipped with
$\alpha_1, \alpha_2, \alpha_3$ satisfying
\begin{equation}\label{2.11}
\alpha_1, \alpha_2, \alpha_3 \in [0,\infty).
\end{equation}

iii) In general function spaces
$A_{\cl,q,a}^{w,\tau}(\rn)$, the indices $\tau, q$ and $a$ satisfy
\begin{equation}\label{2.12}
\tau \in [0,\infty), \, q\in (0,\infty], \, a\in (N_0+\alpha_3,\infty),
\end{equation}
where in {\rm \eqref{3.35}} below we need to assume $a\in (N_0+\alpha_3,\infty)$
in order to guarantee that $\cs(\rn)$ is contained in the function spaces.
\end{remark}

In the following, we content ourselves with considering
the case when ${\mathcal L}(\rn)=L^p(\rn)$
as an example,
which still enables us to see why we introduce
these function spaces in this way.
Further examples are given in Section \ref{s10}.
\begin{example}\label{e3.1}
Let $q \in (0,\infty]$, $s \in \R$ and $\tau \in [0,\infty)$.
In {\rm \cite{yy1,yy2}}, the
\emph{Besov-type space} $B_{p,q}^{s,\tau}(\rn)$ with $p\in
(0,\infty]$ and the \emph{Triebel-Lizorkin-type space}
$F_{p,q}^{s,\tau}(\rn)$ with $p\in (0,\infty)$ were,
respectively, defined to be the set of all $f\in\cs'(\rn)$
such that
$$\| f \|_{B_{p,q}^{s,\tau}(\rn)}
:= \sup_{P \in {\mathcal Q}(\rn)} \frac{1}{|P|^\tau} \left\{
\sum_{j=j_P \vee 0}^\infty \left[\int_P|2^{js}\vz_j*f(x)|^p\,dx
\right]^{\frac{q}{p}} \right\}^{\frac{1}{q}}<\infty$$ and
$$\| f \|_{F_{p,q}^{s,\tau}(\rn)}
:= \sup_{P \in {\mathcal Q}(\rn)} \frac{1}{|P|^\tau} \left\{
\int_P \left[\sum_{j=j_P \vee 0}^\infty |2^{js}\vz_j*f(x)|^q
\right]^{\frac{p}{q}} \,dx \right\}^{\frac{1}{p}}<\infty$$ with the
usual modifications made when $p=\infty$ or $q=\infty$.
Here $\vz_0$ is understood as $\Phi$.
Then, we have shown
in {\rm \cite{LSUYY1}}
that $B^{s,\tau}_{p,q,a}(\rn)$ coincides with $B^{s,\tau}_{p,q}(\rn)$
as long as $a\in(\frac{n}{p},\fz)$.
Likewise $F^{s,\tau}_{p,q,a}(\rn)$ coincides with $F^{s,\tau}_{p,q}(\rn)$
as long as $a\in(\frac{n}{\min(p,q)},\fz)$.
Notice that $B^{s,0}_{p,q,a}(\rn)$ and $F^{s,0}_{p,q,a}(\rn)$
are isomorphic to $B^s_{p,q}(\rn)$ and $F^s_{p,q}(\rn)$
respectively
by virtue of the Plancherel-Polya-Nikolskij inequality
{\rm (}Lemma {\rm \ref{l1.1})} and the Fefferman-Stein vector-valued inequality
{\rm (}see {\rm \cite{fs,gr85,gr08,st93})}.
This fact is generalized to our current setting.
The atomic decomposition of these spaces can be found in {\rm \cite{syy,ysy}}.
Needless to say, in this setting, $\cl(\rn)=\lp$ satisfies $(\cl1)$ through $(\cl6)$.
\end{example}

Observe that the function spaces
$B^{w,\tau}_{\cl,q,a}(\rn)$,
$F^{w,\tau}_{\cl,q,a}(\rn)$,
$\cn^{w,\tau}_{\cl,q,a}(\rn)$ and
$\ce^{w,\tau}_{\cl,q,a}(\rn)$
depend upon $a\in(0,\fz)$, as the following example
shows.

\begin{example}\label{e3.2}
Let $m\in\nn$, $b\in(0,\fz)$, $f_m(t):=[\frac{2\sin(2^{-2mb}t)}{t}]^m$
for all $t \in {\mathbb R}$, and $\cl(\R)=L^p(\R)$ with $p\in (0,\fz]$.
If $\tau$, $a$, $q$ and $w$ are as in Definition \ref{d3.1}
with $w(x,1)$ independent of $x\in\rr$, then $f_m \in
B^{w,\tau}_{\cl,q,a}(\R)
\cup
F^{w,\tau}_{\cl,q,a}(\R)
\cup
\cn^{w,\tau}_{\cl,q,a}(\R)
\cup
\ce^{w,\tau}_{\cl,q,a}(\R)$
if and only if
\[
p\min(a,m)>1,
\]
and, in this case,
we have $f_m \in
B^{w,\tau}_{\cl,q,a}(\R)
\cap
F^{w,\tau}_{\cl,q,a}(\R)
\cap
\cn^{w,\tau}_{\cl,q,a}(\R)
\cap
\ce^{w,\tau}_{\cl,q,a}(\R).
$
To see this,
notice that, for all $t\in\rr$,
$$\wh{\chi}_{[-2^{-mb},2^{-mb}]}(t)
=\int_{-2^{-mb}}^{2^{-mb}}\cos (xt)\,dx=\frac{2\sin(2^{-mb}t)}t,$$
which implies that
$$\wh{f}_m:=
\overbrace{\chi_{[-2^{-mb},2^{-mb}]}*\cdots*
\chi_{[-2^{-mb},2^{-mb}]}}^{m\ {\rm times}}$$
and that $\supp \wh f_m\st[-m2^{-mb},m2^{-mb}]$.
Choose $b\in(0,\fz)$ large enough such that
$$[-m2^{-mb},m2^{-mb}]\st[-1/2,1/2].$$
Let $\Phi, \vz\in\cs(\rr)$ satisfy \eqref{1.1} and \eqref{1.2},
and
assume additionally that
$$\chi_{B(0,1)}\le\wh\Phi\le\chi_{B(0,2)}\ {\rm and}\
\supp \wh\vz\st\lf\{\xi\in\rr:\ \frac12\le|\xi|\le2\r\}.$$
Then, by the size of the frequency support,
we see that $\Phi\ast f_m=f_m$
and that $\varphi_j\ast f_m=0$ for all $j\in\nn$.
Therefore, for all $x\in\rr$,
\[
(\Phi^*f_m)_a(x)=
\sup_{z\in\rr}\frac{|2\sin(2^{-mb}(x+z))|^m}{|x+z|^m(1+|z|)^a}
 \sim_m (1+|x|)^{\max(-a,-m)} \
{\rm and}\ (\varphi_j^*f_m)_a(x)=0,
\]
which implies the claim.
Here, ``$\sim_m$'' denotes the implicit positive
equivalent constants depending on $m$.
\end{example}

For the time being,
we are oriented to justifying Definition \ref{d3.1}.
That is, we show that the spaces $A^{s,\tau}_{p,q,a}(\rn)$
are independent of the choices of $\Phi$ and $\vz$
by proving the following Theorem \ref{t3.1},
which covers the local means as well.
Notice that a special case $A^{s,\tau}_{p,q,a}(\rn)$
of these results was dealt with in \cite{yy3,ysy-textbook}.

\begin{theorem}\label{t3.1}
Let $a, \alpha_1, \alpha_2, \alpha_3, \tau, q$, $w$
and $\cl(\rn)$ be as in Definition \ref{d3.1}.
Let $L \in \Z_+$ be such that
\begin{equation}\label{3.6}
L+1>\alpha_1 \vee (a+n\tau+\alpha_2).
\end{equation}
Assume that $\Psi $,
$\psi \in {\mathcal S}(\rn)$
satisfies that, for all $\alpha$ with $\|\alpha\|_1 \le L$ and
some $\varepsilon\in(0,\fz)$,
\begin{equation}\label{3.7}
\widehat{\Psi}(\xi)\ne 0 \mbox{ if }
|\xi|<2\varepsilon, \
\partial^{\alpha}\widehat{\psi}(0) = 0, \
 and\ \widehat{\psi}(\xi)\ne 0 \mbox{ if }\
\frac{\varepsilon}{2}<|\xi|<2\varepsilon.
\end{equation}
Let $\psi_j(\cdot):= 2^{jn}\psi(2^j\cdot)$
for all $j \in {\mathbb N}$ and $\{(\psi_j^\ast f)_a\}_{j\in\zz_+}$ be as
in \eqref{1.3} with $\Phi$ and $\varphi$ replaced, respectively,
by $\Psi$ and $\psi$.
Then
\begin{gather}
\label{3.8}
\|f\|_{B^{w,\tau}_{\cl,q,a}(\rn)}
\sim
\lf\|\lf\{(\psi_j^\ast f)_a\r\}_{j\in\zz_+}\r\|_{\ell^q(\cl_\tau^w(\rn,\zz_+))},\\
\label{3.9}
\|f\|_{\cn^{w,\tau}_{\cl,q,a}(\rn)}
\sim
\lf\|\lf\{(\psi_j^\ast f)_a\r\}_{j\in\zz_+}\r\|_{\ell^q(\cn\cl_\tau^w(\rn,\zz_+))},\\
\label{3.10}
\|f\|_{F^{w,\tau}_{\cl,q,a}(\rn)}
\sim
\lf\|\lf\{(\psi_j^\ast f)_a\r\}_{j\in\zz_+}\r\|_{\cl_\tau^w(\ell^q(\rn,\zz_+))}
\end{gather}
and
\begin{gather}\label{3.11}
\|f\|_{\ce^{w,\tau}_{\cl,q,a}(\rn)}
\sim
\lf\|\lf\{(\psi_j^\ast f)_a\r\}_{j\in\zz_+}\r\|_{\ce\cl_\tau^w(\ell^q(\rn,\zz_+))}
\end{gather}
with equivalent positive constants independent of $f$.
\end{theorem}

\begin{proof}
To show Theorem \ref{t3.1}, we only need to prove that,
for all $f\in\cs'(\rn)$ and $x\in\rn$,
\begin{equation}\label{3.12}
(\Psi^*f)_a(x)\lesssim
(\Phi^*f)_a(x)
+
\sum_{\nu=1}^\infty 2^{-\nu(L+1-a)}\vz_\nu^*f(x)
\end{equation}
and that
\begin{equation}\label{3.13}
(\psi_j^*f)_a(x) \lesssim
2^{-j(L+1-a)}(\Phi^*f)_a(x)
+
\sum_{\nu=1}^\infty 2^{-|\nu-j|(L+1)+a[(j-\nu)\vee 0]}
(\vz_\nu^*f)_a(x).
\end{equation}
Once we prove \eqref{3.12} and \eqref{3.13},
then we are in the position of applying Lemma \ref{l2.3}
to conclude \eqref{3.8} through \eqref{3.11}.

We now establish \eqref{3.13}.
The proof of \eqref{3.12} is easier and we omit the details.
For a non-negative integer $L$ as in \eqref{3.6},
by \cite[Theorem 1.6]{r01},
we know that there exist
$\Psi^\dagger, \, \psi^\dagger \in {\mathcal S}(\rn)$
such that, for all $\beta$ with $|\beta| \le L$,
\begin{equation}\label{3.14}
\int_{\rn}\psi^\dagger(x)x^\beta\,dx=0
\end{equation}
and that
\begin{equation}\label{3.15}
\Psi^\dagger*\Phi+\sum_{\nu=1}^\infty \psi_\nu^\dagger*\vz_\nu
=\delta_0
\end{equation}
in $\cs'(\rn)$,
where $\psi^\dagger_\nu(\cdot):= 2^{\nu n}\psi^\dagger(2^\nu \cdot)$
for $\nu \in {\mathbb N}$ and $\delta_0$ is the \emph{dirac
distribution at origin}.
We decompose $\psi_j$ along \eqref{3.15} into
\begin{equation*}
\psi_j
=
\psi_j*\Psi^\dagger*\Phi
+\sum_{\nu=1}^\infty
\psi_j*\psi_\nu^\dagger*\vz_\nu.
\end{equation*}
From \eqref{3.7} and \eqref{3.14},
together with Lemma \ref{l2.4},
we infer that, for all $j\in\zz_+$ and $y\in\rn$,
\begin{equation}\label{3.17}
|\psi_j*\Psi^\dagger(y)|
\lesssim
\frac{2^{-j(L+1)}}{(1+|y|)^{n+1+a}}
\quad{\rm and \quad}
|\psi_j*\psi_\nu^\dagger(y)|
\lesssim
\frac{2^{n (j\wedge \nu)-|j-\nu|(L+1)}}{(1+2^{ j\wedge \nu}|y|)^{n+1+a}}.
\end{equation}
By (\ref{3.17}), we further see that, for all $j\in\zz_+$ and $x\in\rn$,
\begin{eqnarray*}
&&\sup_{z \in {\mathbb R}^n}
\frac{|\psi_j*f(x+z)|}{(1+2^j|z|)^a}\\
&&\hs\lesssim
2^{-j(L+1-a)}(\Phi^*f)_a(x)
+
\sum_{\nu=1}^\infty
2^{-|j-\nu|(L+1)}(\vz^*_\nu f)_a(x)
\int_{\rn}
\frac{2^{n (j\wedge \nu)}(1+2^\nu|y|)^a}{(1+2^{ j\wedge \nu}|y|)^{n+1+a}}\,dy\\
&&\hs\lesssim
2^{-j(L+1-a)}(\Phi^*f)_a(x)
+
\sum_{\nu=1}^\infty
2^{-|j-\nu|(L+1)+a[(j-\nu)\vee 0]}(\vz^*_\nu f)_a(x)
\int_{\rn}
\frac{2^{n (j\wedge \nu)}\,dy}{(1+2^{ j\wedge \nu}|y|)^{n+1}}\\
&&\hs\sim
2^{-j(L+1-a)}(\Phi^*f)_a(x)
+\sum_{\nu=1}^\infty
2^{-|j-\nu|(L+1)+a[(j-\nu)\vee 0]}(\vz^*_\nu f)_a(x),
\end{eqnarray*}
which completes the proof of \eqref{3.13} and hence Theorem \ref{t3.1}.
\end{proof}

Notice that the moment condition on $\Psi$ in Theorem \ref{t3.1}
is not necessary due to \eqref{3.7}.
Moreover, in view of the calculation presented in the proof of
Theorem \ref{t3.1},
we also have the following assertion.

\begin{corollary}\label{c3.1}
Under the notation of Theorem {\rm \ref{t3.1}},
for some $N\in\nn$ and all $x\in\rn$, let
\[
{\mathfrak M}f(x,2^{-j})
:=
\begin{cases}
\dis\sup_\psi|\psi_{j}*f(x)|,
& \, j \in \N;\\
\dis\sup_\Psi|\Psi*f(x)|,
& \, j=0,
\end{cases}
\]
where the supremum is taken over all $\psi$ and $\Psi$
in $\cs(\rn)$
satisfying
\[
\sum_{|\alpha| \le N}
\sup_{x \in \rn}
(1+|x|)^N |\partial^\alpha \psi(x)|
+
\sum_{|\alpha| \le N}
\sup_{x \in \rn}
(1+|x|)^N |\partial^\alpha \Psi(x)|
\le 1
\]
as well as {\rm \eqref{3.7}}.
Then, if $N$ is large enough, for all $f \in \cs'(\rn)$,
\begin{eqnarray*}
&&\|f\|_{B^{w,\tau}_{\cl,q,a}(\rn)}
\sim
\lf\|\lf\{{\mathfrak M} f(\cdot, 2^{-j})\r\}_{j\in\zz_+}\r\|_{\ell^q(\cl_\tau^w(\rn,\zz_+))},\\
&&\|f\|_{\cn^{w,\tau}_{\cl,q,a}(\rn)}
\sim
\lf\|\lf\{{\mathfrak M} f(\cdot, 2^{-j})\r\}_{j\in\zz_+}\r\|_{\ell^q(\cn\cl_\tau^w(\rn,\zz_+))},\\
&&\|f\|_{F^{w,\tau}_{\cl,q,a}(\rn)}
\sim
\lf\|\lf\{{\mathfrak M} f(\cdot, 2^{-j})\r\}_{j\in\zz_+}\r\|_{\cl_\tau^w(\ell^q(\rn,\zz_+))}
\end{eqnarray*}
and
\begin{eqnarray*}
&&\|f\|_{\ce^{w,\tau}_{\cl,q,a}(\rn)}
\sim
\lf\|\lf\{{\mathfrak M} f(\cdot, 2^{-j})\r\}_{j\in\zz_+}\r\|_{\ce\cl_\tau^w(\ell^q(\rn,\zz_+))}
\end{eqnarray*}
with implicit positive constants independent of $f$.
\end{corollary}

Another corollary is the characterization of these spaces
via local means.
Recall that
$\Delta:=\sum_{j=1}^n\frac{\partial^2}{\partial x_j^2}$
denotes the Laplacian.

\begin{corollary}\label{c3.2}
Let $a, \alpha_1, \alpha_2, \alpha_3, \tau, q$, $w$
and $\cl(\rn)$ be as in Definition \ref{d3.1}.
Assume that $\Psi \in C^\infty_{\rm c}({\mathbb R}^n)$
satisfies $\chi_{B(0,1)} \le \Psi \le \chi_{B(0,2)}$.
Assume, in addition, that
$\psi=\Delta^{\ell_0+1}\Psi$
for some $\ell_0 \in \N$ such that
\begin{equation*}
2\ell_0+1>\alpha_1 \vee (a+n\tau+\alpha_2).
\end{equation*}%
Let $\psi_j(\cdot):= 2^{jn}\psi(2^j\cdot)$
for all $j \in {\mathbb N}$ and $\{(\psi_j^\ast f)_a\}_{j\in\zz_+}$ be as
in \eqref{1.3} with $\Phi$ and $\varphi$ replaced, respectively,
by $\Psi$ and $\psi$.
Then, for all $f \in \cs'(\rn)$,
\begin{eqnarray*}
&&\|f\|_{B^{w,\tau}_{\cl,q,a}(\rn)}
\sim
\lf\|\lf\{(\psi_j^\ast f)_a\r\}_{j\in\zz_+}\r\|_{\ell^q(\cl_\tau^w(\rn,\zz_+))},\\
&&\|f\|_{\cn^{w,\tau}_{\cl,q,a}(\rn)}
\sim
\lf\|\lf\{(\psi_j^\ast f)_a\r\}_{j\in\zz_+}\r\|_{\ell^q(\cn\cl_\tau^w(\rn,\zz_+))},\\
&&\|f\|_{F^{w,\tau}_{\cl,q,a}(\rn)}
\sim
\lf\|\lf\{(\psi_j^\ast f)_a\r\}_{j\in\zz_+}
\r\|_{\cl_\tau^w(\ell^q(\rn,\zz_+))}
\end{eqnarray*}
and
\begin{eqnarray*}
&&\|f\|_{\ce^{w,\tau}_{\cl,q,a}(\rn)}
\sim
\lf\|\lf\{(\psi_j^\ast f)_a\r\}_{j\in\zz_+}\r\|_{\ce\cl_\tau^w(\ell^q(\rn,\zz_+))}
\end{eqnarray*}
with equivalent positive constants independent of $f$.
\end{corollary}

\subsection{Fundamental properties}

With the fundamental theorem
on our function spaces stated and proven as above,
we now take up some inclusion relations.
The following lemma is immediately deduced from
Lemma \ref{l2.1} and Definition \ref{d3.1}.

\begin{lemma}\label{l3.1}
Let $\alpha_1, \alpha_2, \alpha_3, \tau \in [0,\infty)$,
$q, q_1,q_2\in(0,\,\fz]$, $q_1\le q_2$
and $w\in{\mathcal W}^{\alpha_3}_{\alpha_1,\alpha_2}$.
Let $\cl(\rn)$ be a quasi-normed space satisfying
$(\cl1)$ through $(\cl4)$. Then
\begin{eqnarray*}%
B^{w,\tau}_{\cl,q_1,a}(\rn)\hookrightarrow B^{w,\tau}_{\cl,q_2,a}(\rn),\\
\cn^{w,\tau}_{\cl,q_1,a}(\rn)\hookrightarrow \cn^{w,\tau}_{\cl,q_2,a}(\rn),\\
F^{w,\tau}_{\cl,q_1,a}(\rn)\hookrightarrow F^{w,\tau}_{\cl,q_2,a}(\rn),\\
\ce^{w,\tau}_{\cl,q_1,a}(\rn)\hookrightarrow \ce^{w,\tau}_{\cl,q_2,a}(\rn)
\end{eqnarray*}
and
\begin{align}
\label{3.22}
B^{w,\tau}_{\cl,q,a}(\rn), \,
\cn^{w,\tau}_{\cl,q,a}(\rn), \,
F^{w,\tau}_{\cl,q,a}(\rn), \,
\ce^{w,\tau}_{\cl,q,a}(\rn)
\hookrightarrow \cn^{w,\tau}_{\cl,\infty,a}(\rn)
\end{align}
in the sense of continuous embedding.
\end{lemma}

\begin{remark}\label{r3.1}
(i) It is well known
that $F^s_{p,q}(\rn) \hookrightarrow B^s_{p,\max(p,q)}(\rn)
\hookrightarrow B^s_{p,\infty}(\rn)$ (see, for example, \cite{t83}).
However, as an example in {\rm \cite{s08}} shows,
with $q \in (0,\infty]$ fixed,
{\rm \eqref{3.22}} is optimal
in the sense that the continuous embedding
$F^{w,\tau}_{\cl,q,a}(\rn)\hookrightarrow \cn^{w,\tau}_{\cl,r,a}(\rn)$
holds true for all admissible $a,w,\tau$ and $\cl(\rn)$
if and only if $r=\infty$.

(ii) From the definitions of the spaces $A_{\cl,q,a}^{w,\tau}(\rn)$,
we deduce that
$$A_{\cl,q,a}^{w,\tau}(\rn)\hookrightarrow B_{\cl,\fz,a}^{w,\tau}(\rn).$$

Indeed,
for example,
the proof of
$\ce_{\cl,q,a}^{w,\tau}(\rn)\hookrightarrow B_{\cl,\fz,a}^{w,\tau}(\rn)$
is as follows:
\begin{eqnarray*}
\|f\|_{\ce^{w,\tau}_{\cl,q,a}(\rn)}
&&=\sup_{P\in\mathcal{Q}(\rn)}\frac1{|P|^\tau}
\|\{\chi_{[j_P,\infty)}(j)\chi_P w_j(\vz_j^\ast f)_a\}_{j=0}^\infty\|_{\cl^w(\ell^q(\rn,\zz_+))}\\
&&\ge\sup_{P\in\mathcal{Q}(\rn)}
\sup_{j\ge j_P}\frac1{|P|^\tau}
\|\chi_P w_j(\vz_j^\ast f)_a\|_{\cl(\rn)}
=\|f\|_{B_{\cl,\fz,a}^{w,\tau}(\rn)}.
\end{eqnarray*}
\end{remark}

Now we are going to discuss the lifting property
of the function spaces,
which also justifies our new framework
of function spaces. Recall that, for all $f\in\cs'(\rn)$
and $\xi\in\rn$, we let $((1-\Delta)^{s/2}f)\ \wh{ }\ (\xi)
:=(1+|\xi|^2)^{s/2}\wh{f}(\xi)$ for all $\xi\in\rn$.

\begin{theorem}\label{t3.2}
Let $a,\, \alpha_1,\, \alpha_2, \,\alpha_3, \,\tau,\, q$, $w$
and $\cl(\rn)$ be as in Definition \ref{d3.1}
and $s\in\rr$.
For all $x \in \rn$ and $j \in \Z_+$, let
\[
w^{(s)}(x,2^{-j}) :=2^{-js}w_j(x).
\]
Then the lift operator $(1-\Delta)^{s/2}$
is bounded from $A^{w,\tau}_{\cl,q,a}(\rn)$
to $A^{w^{(s)},\tau}_{\cl,q,a}(\rn)$.
\end{theorem}

For the proof of Theorem \ref{t3.2},
the following lemma is important.
Once we prove this lemma,
Theorem \ref{t3.2} is obtained
by virtue of Lemma \ref{l3.2} and (W1).

\begin{lemma}\label{l3.2}
Let $a\in(0,\fz)$, $s\in\rr$ and $\Phi$, $\varphi \in \cs(\rn)$ be such that
\[
\supp \widehat{\Phi} \subset \{\xi \in \rn\,:\,|\xi| \le 2\},
\
\supp\widehat{\varphi} \subset \{\xi \in \rn\,:\,1/2 \le |\xi| \le 2\}
{\rm\  and}\
\widehat{\Phi}+\sum_{j=1}^\infty \widehat{\varphi_j} \equiv 1,
\]
where
$\varphi_j(\cdot):= 2^{jn}\varphi(2^j\cdot)$
for each $j \in \nn$.
Then, there exists a positive constant $C$ such that,
for all  $f\in\cs'(\rn)$ and $x\in\rn$,
\begin{eqnarray}\label{3.23}
&&(\Phi^*((1-\Delta)^{s/2}f))_a(x)
\le C
\lf[(\Phi^*f)_a(x)+(\varphi_1^*f)_a(x)\r],\\
\label{3.24}
&&(\varphi_1^*((1-\Delta)^{s/2}f))_a(x)
\le C
\lf[(\Phi^*f)_a(x)+(\varphi_1^*f)_a(x)+(\varphi_2^*f)_a(x)\r],
\end{eqnarray}
and
\begin{eqnarray}
\label{3.25}
&&(\varphi_j^*((1-\Delta)^{s/2}f))_a(x)
\le C
2^{js}(\varphi_j^*f)_a(x)
\end{eqnarray}
for all $j \ge 2$.
\end{lemma}

\begin{proof}
The proofs of \eqref{3.23} and \eqref{3.24} being simpler,
let us prove \eqref{3.25}.
In view of the size of supports,
we see that, for all $j\ge 2$ and $x\in\rn$,
\begin{eqnarray*}
&&(\varphi_j^*((1-\Delta)^{s/2}f))_a(x)\\
&&\quad=
\sup_{z \in \rn}
\frac{|\varphi_j*[(1-\Delta)^{s/2}f](x+z)|}{(1+2^j|z|)^a}\\
&&\quad=
\sup_{z \in \rn}
\frac{|(1-\Delta)^{s/2}(\varphi_{j-1}+\varphi_j+\varphi_{j+1})
*\varphi_j*f(x+z)|}{(1+2^j|z|)^a}\\
&&\quad=
\sup_{z \in \rn}\frac{1}{(1+2^j|z|)^a}
\left|\int_{\rn}
(1-\Delta)^{s/2}(\varphi_{j-1}+\varphi_j+\varphi_{j+1})(y)
\varphi_j*f(x+z-y)dy\right|.
\end{eqnarray*}
Now let us show that, for all $j\ge 2$ and $y\in\rn$,
\begin{equation}\label{3.26}
|(1-\Delta)^{s/2}(\varphi_{j-1}+\varphi_j+\varphi_{j+1})(y)|
\lesssim
\frac{2^{j(s+n)}}{(1+2^j|y|)^{a+n+1}}.
\end{equation}
Once we prove \eqref{3.26},
by inserting \eqref{3.26} to the above equality
we conclude the proof of \eqref{3.25}.

To this end,
we observe that, for all $j\ge 2$ and $y\in\rn$,
$$
(1-\Delta)^{s/2}
\left(\sum_{l=-1}^1\varphi_{j+l}\right)(y)
=
\left((1+|\xi|^2)^{s/2}
[\widehat{\varphi}(2^{-j+1}\xi)
+\widehat{\varphi}(2^{-j}\xi)
+\widehat{\varphi}(2^{-j-1}\xi)
]\right
)^\vee(y).
$$
Since, for all multiindices $\vec{\alpha}$, $j\ge 2$ and $\xi\in\rn$,
a pointwise estimate
$$
\lf|\partial^{\vec{\alpha}}\lf((1+|\xi|^2)^{s/2}
\lf[\widehat{\varphi}(2^{-j+1}\xi)
+\widehat{\varphi}(2^{-j}\xi)
+\widehat{\varphi}(2^{-j-1}\xi)\r]\r)\r|
\lesssim
2^{(s-\|\vec{\alpha}\|_1)j}(1+2^{-j}|\xi|)^{-n-1}
$$
holds true,
\eqref{3.26} follows from the definition of the Fourier transform,
which completes the proof of Lemma \ref{l3.2}.
\end{proof}

The next Theorem \ref{t3.3} is mainly a consequence
of the assumptions $(\cl1)$ through $(\cl4)$ and $(\cl6)$.
To show it,
we need to introduce a new class of weights,
which are used later again.

\begin{definition}\label{d3.3}
Let $\alpha_1, \alpha_2, \alpha_3 \in [0,\infty)$.
The \emph{class
$\star-{\mathcal W}_{\alpha_1,\alpha_2}^{\alpha_3}$
of weights} is defined
as the set of all measurable functions
$w:\R^n_{\Z_+} \to (0,\infty)$
satisfying (W$1^\star$) and (W2),
where (W2) is defined as in Definition {\rm \ref{d2.1}}
and

(W$1^\star$)
there exists a positive constant $C$ such that,
for all $x \in \rn$ and $j, \nu \in \Z_+$ with $j \ge \nu$,
$C^{-1}2^{(j-\nu)\alpha_1}
w(x,2^{-\nu})
\le
w(x,2^{-j})
\le
C2^{-(\nu-j)\alpha_2}
w(x,2^{-\nu})$.
\end{definition}

It is easy to see that $\star-\cw_{\az_1,\az_2}^{\az_3}\subsetneqq\cw_{\az_1,\az_2}^{\az_3}$.

\begin{example}\label{e3.3}
If
$s \in [0,\infty)$ and
$w_j(x):=2^{js}$
for all $x \in \rn$ and $j \in \Z_+$,
then it is easy to see $w \in \star-\cw_{s,s}^{0}$.
\end{example}

With the terminology for the proof is fixed,
we state and prove the following theorem.

\begin{theorem}\label{t3.3}
Let $a, \alpha_1, \alpha_2, \alpha_3, \tau$ and $q$
be as in Definition \ref{d3.1}.
If $w\in\star-\cw_{\az_1,\az_2}^{\az_3}$ and
$\cl(\rn)$ satisfies $(\cl1)$ through $(\cl4)$
and $(\cl6)$, then $A^{w,\tau}_{\cl,q,a}(\rn)
\hookrightarrow \cs'(\rn)$ in the sense of continuous embedding.
\end{theorem}

\begin{proof}
Let $\Phi, \vz\in\cs(\rn)$ be as in Lemma \ref{l3.2}. Then
\begin{equation}\label{3.27}
\widehat{\Phi}+\sum_{j=1}^\infty \widehat{\vz_j} \equiv 1.
\end{equation}
We first assume that (W$1^\star$) holds true with
\begin{equation}\label{3.28}
\alpha_1-N+n-\gamma+n\tau>0 \ {\rm and}\
N>\delta+n
\end{equation}
for some $N\in(0,\fz)$.

For any $f\in A^{w,\tau}_{\cl,q,a}(\rn)$, by the definition,
we see that, for all $Q\in\cq(\rn)$ with $j_Q \in\nn$,
\begin{eqnarray*}
\frac1{|Q|^{\tau}}
\lf\|\chi_Q \cdot w(\cdot,2^{-j_Q})
(\varphi_{j_Q}^*f)_a
\r\|_{\cl(\rn)}
\ls\|f\|_{A^{w,\tau}_{\cl,q,a}(\rn)}.
\end{eqnarray*}
Consequently, from (W$1^\star$), we deduce that
\begin{equation}\label{3.29}
\lf\|\chi_Q \cdot w(\cdot,1)
(\varphi_{j_Q}^*f)_a
\r\|_{\cl(\rn)}
\lesssim 2^{-j_Q(\alpha_1+n\tau)}\|f\|_{A^{w,\tau}_{\cl,q,a}(\rn)}.
\end{equation}
Now let $\zeta \in \cs(\rn)$ be an arbitrary test function
and define
\begin{equation*}
p(\zeta):=\sup_{x \in \rn}(1+|x|)^{\alpha_3+N}\zeta(x).
\end{equation*}%
Then from \eqref{3.29} and the partition $\{Q_{jk}\}_{k \in \Z^n}$ of $\rn$,
we infer that
\[
\int_{\rn}|\zeta(x)\varphi_{j}*f(x)|\,dx
\ls p(\zeta)
\sum_{k \in \Z^n}
(1+|2^{-j}k|)^{-N-\az_3}
\int_{Q_{jk}}|\varphi_{j}*f(x)|\,dx.
\]
If we use the condition (W2) twice and the fact that $j \in [0,\infty)$,
then we have
\begin{eqnarray*}
&&
\int_{\rn}|\zeta(x)\varphi_{j}*f(x)|\,dx\\
&&\quad\lesssim
p(\zeta)
\sum_{k \in \Z^n}
(1+|2^{-j}k|)^{-N}\inf_{y \in Q_{jk}} w(y,1)
\int_{Q_{jk}}|\varphi_{j}*f(x)|\,dx\\
&&\quad\lesssim
p(\zeta)
\sum_{k \in \Z^n}
2^{jN}(1+|k|)^{-N}
|Q_{jk}|\inf_{y \in Q_{jk}}\lf\{w(y,1)
(\varphi_{j}^*f)_a(y)\r\}.
\end{eqnarray*}
Now we use \eqref{3.29} and the assumption $(\cl6)$ to conclude
\begin{eqnarray}\label{3.31}
&&\int_{\rn}|\zeta(x)\varphi_{j}*f(x)|\,dx\noz\\
&&\hs\lesssim
p(\zeta)
\sum_{k \in \Z^n}
2^{j(N-n+\gamma)}(1+|k|)^{-N+\delta}
\lf\|\chi_{Q_{jk}} w(\cdot,1)
(\varphi_{j}^*f)_a
\r\|_{\cl(\rn)}\noz\\
&&\hs\lesssim
p(\zeta)
\sum_{k \in \Z^n}
2^{-j(\alpha_1-N+n-\gamma+n\tau)}(1+|k|)^{-N+\delta}
\|f\|_{A^{w,\tau}_{\cl,q,a}(\rn)}\noz\\
&&\hs\sim
2^{-j(\alpha_1-N+n-\gamma+n\tau)}
p(\zeta)
\|f\|_{A^{w,\tau}_{\cl,q,a}(\rn)}.
\end{eqnarray}

By replacing $\vz_0$ with $\Phi$ in the above
argument, we see that
\begin{equation}\label{3.32}
\int_{\rn}|\zeta(x)\Phi*f(x)|\,dx
\ls p(\zeta)
\|f\|_{A^{w,\tau}_{\cl,q,a}(\rn)}.
\end{equation}
Combining \eqref{3.27},  \eqref{3.31} and \eqref{3.32},
we then conclude that, for all $\zeta \in \cs(\rn)$,
\begin{equation}\label{3.330}
|\langle f,\zeta\rangle|
\le
|\langle \Phi*f,\zeta\rangle|
+
\sum_{j=1}^\infty
|\langle \vz_j*f,\zeta\rangle|
\lesssim
p(\zeta)\|f\|_{A^{w,\tau}_{\cl,q,a}(\rn)},
\end{equation}
which implies that $f\in\cs'(\rn)$ and hence $A^{w,\tau}_{\cl,q,a}(\rn)
\hookrightarrow \cs'(\rn)$ in the sense of continuous embedding.

We still need to remove the restriction \eqref{3.28}.
Indeed, for any $\az_1\in[0,\fz)$ and $f\in A_{\cl,q,a}^{w,\tau}(\rn)$,
choose $s\in(-\fz,0)$
small enough such that $\az_1-s>\gz+\dz-n\tau$.
By Theorem \ref{t3.2}, we have
$(1-\bdz)^{s/2}f\in A_{\cl,q,a}^{w^{(s)},\tau}(\rn)$.
Then, defining a seminorm $\rho$ by $\rho(\zeta):=p((1-\bdz)^{s/2}\zeta)$
for all $\zeta\in\cs(\rn)$, by \eqref{3.330},
we have
\begin{eqnarray*}
|\la f, \zeta\ra|&&=|\la(1-\bdz)^{s/2}f,(1-\bdz)^{-s/2}\zeta\ra|\\
&&\ls\rho((1-\bdz)^{-s/2}\zeta)\|(1-\bdz)^{s/2}f\|_{A_{\cl,q,a}^{w^{(s)},\tau}(\rn)}
\ls p(\zeta)\|f\|_{A_{\cl,q,a}^{w,\tau}(\rn)},
\end{eqnarray*}
which completes the proof of Theorem \ref{t3.3}.
\end{proof}

\begin{remark}\label{r3.2}
In the course of the proof of Theorem \ref{t3.3},
the inequality
\begin{equation*}
\int_{\kappa Q_{jk}}|\varphi_{j}*f(x)|\,dx
\ls
\kappa^M2^{-j(\alpha_1+n+n\tau-\gamma)}(1+|k|)^{\delta}
\|f\|_{A^{w,\tau}_{\cl,q,a}(\rn)}
\end{equation*}%
is proved.
Here $\kappa \ge 1$, $M$ and the implicit positive constant are independent of
$j,k$ and $\kappa$.
\end{remark}

It follows from Theorem \ref{t3.3}
that we have the following conclusions, whose proof is similar
to that of \cite[pp.\,48-49,\ Theorem 2.3.3]{t83}. For the sake of convenience, we
give some details here.

\begin{proposition}\label{p3.1}
Let $a, \alpha_1, \alpha_2, \alpha_3, \tau$ and $q$
be as in Definition \ref{d3.1}.
If $w\in\star-\cw_{\az_1,\az_2}^{\az_3}$ and
$\cl(\rn)$ satisfies $(\cl1)$ through
$(\cl6)$, then the spaces $B^{w,\tau}_{\cl,q,a}(\rn)$,
$\cn^{w,\tau}_{\cl,q,a}(\rn)$, $F^{w,\tau}_{\cl,q,a}(\rn)$ and
$\ce^{w,\tau}_{\cl,q,a}(\rn)$ are complete.
\end{proposition}

\begin{proof}
By similarity, we only give the proof for the space $F^{w,\tau}_{\cl,q,a}(\rn)$.
Let $\{f_l\}_{l\in\nn}$ be a Cauchy sequence in $F^{w,\tau}_{\cl,q,a}(\rn)$.
Then from Theorem \ref{t3.3}, we infer that $\{f_l\}_{l\in\nn}$ is also
a Cauchy sequence in $\cs'(\rn)$. By the completeness of $\cs'(\rn)$,
there exists an $f\in\cs'(\rn)$ such that, for all Schwartz
functions $\vz$,
$\vz\ast f_l\to \vz\ast f$ pointwise as $l\to\fz$ and
hence
$$\vz\ast (f_l-f)=\lim_{m\to\fz}\vz\ast (f_l-f_m)$$
pointwise. Therefore, for all $j\in\zz_+$ and $x\in\rn$,
$$\sup_{z\in\rn}\frac{|\vz_j\ast (f_l-f)(x+z)|}{(1+2^j|z|)
^a}\le \liminf_{m\to\fz}
\sup_{z\in\rn}\frac{|\vz_j\ast (f_l-f_m)(x+z)|}{(1+2^j|z|)
^a},$$
which, together with $(\cl4)$, the Fatou property of $\cl(\rn)$
in Proposition \ref{p2.1}, and the Fatou property of $\ell^q$, implies that
$$\limsup_{l\to\fz}\|f_l-f\|_{F^{w,\tau}_{\cl,q,a}(\rn)}\le \limsup_{l\to\fz}
\left(\liminf_{m\to\fz} \|f_l-f_m\|_{F^{w,\tau}_{\cl,q,a}(\rn)}\right)=0.$$
Thus, $f=\lim_{m\to\fz}f_m$ in
$F^{w,\tau}_{\cl,q,a}(\rn)$, which shows that $F^{w,\tau}_{\cl,q,a}(\rn)$
is complete.
This finishes the proof of Proposition \ref{p3.1}.
\end{proof}

Assuming $(\cl6)$, we can prove
that $\cs(\rn)$ is embedded into ${A^{w,\tau}_{\cl,q,a}(\rn)}.$

\begin{theorem}\label{t3.4}
Let $a, \alpha_1, \alpha_2, \alpha_3, \tau, q$ and $w$
be as in Definition \ref{d3.1}.
Then if $\cl(\rn)$ satisfies $(\cl1)$ through $(\cl6)$ and
\begin{equation}\label{3.35}
a\in (N_0+\alpha_3,\infty),
\end{equation}
then
$\cs(\rn) \hookrightarrow A^{w,\tau}_{\cl,q,a}(\rn)$
in the sense of continuous embedding.
\end{theorem}

\begin{proof}
Let $f \in \cs(\rn)$.
Then, for all $x\in\rn$ and $j\in\nn$,
we have
\[
\sup_{z \in \rn}\frac{|\vz_j*f(x+z)|}{(1+2^j|z|)^a} \lesssim
\frac{1}{(1+|x|)^a}\sup_{y \in \rn}(1+|y|)^{a+n+1}|f(y)|.
\]
In view of (W2), $(\cl6)$ and \eqref{3.35},
we have $(1+|\cdot|)^{-a}w(\cdot,1) \in \cl(\rn)$.
Consequently
\begin{equation}\label{3.36}
\left\|
w_j
\sup_{z \in \rn}\frac{|\vz_j*f(\cdot+z)|}{(1+2^j|z|)^a}
\right\|_{\cl(\rn)}
\lesssim
2^{j\alpha_2}\sup_{y \in \rn}(1+|y|)^{a+n+1}|f(y)|.
\end{equation}
Let $\ez$ be a positive constant. Set
$w^*_j(x):= 2^{-j(\alpha_2+n\tau+\ez)}w_j(x)$
for all $x \in \R^n$ and $j \in \Z_+$.
The estimate \eqref{3.36} and its counterpart for $j=0$ show that
$\cs(\rn) \hookrightarrow A^{w^*,\tau}_{\cl,q,a}(\rn)$
and hence Theorem \ref{t3.2} shows that
$\cs(\rn) \hookrightarrow A^{w,\tau}_{\cl,q,a}(\rn)$,
which completes the proof of Theorem \ref{t3.4}.
\end{proof}

Motivated by Theorem \ref{t3.4},
we  postulate \eqref{3.35} on the parameter $a$ here and below.

In analogy with Theorem \ref{t3.2},
we have the following result of boundedness of pseudo-differential
operators of H\"ormander-Mikhlin type.

\begin{proposition}\label{p3.2}
Let $a, \alpha_1, \alpha_2, \alpha_3, \tau, q, w$ and $\cl(\rn)$
be as in Definition \ref{d3.1}.
Assume that $m \in C^\infty_{\rm c}(\rn)$
satisfies that, for all multiindices $\vec{\alpha}$,
$$
M_{\vec{\alpha}}:=\sup_{\xi \in \rn}
(1+|\xi|)^{|\vec{\alpha}|}|\partial^{\vec{\alpha}} m(\xi)|<\infty.
$$
Define $I_m f:=(m \hat{f})^\vee$.
Then the operator $I_m$ is bounded
on $A^{w,\tau}_{\cl,q,a}({\mathbb R}^n)$
and there exists $K \in \N$
such that the operator norm is bounded by a positive constant multiple
of $\sum_{\|\vec{\alpha}\|_1 \le K}M_{\vec{\alpha}}$.
\end{proposition}

\begin{proof}
Going through an argument similar
to the proof of Lemma \ref{l3.2},
we are led to \eqref{3.26}
with $s=0$ and $(1-\Delta)^{s/2}$ replaced
by $I_m$.
Except this change, the same argument as therein works.
We omit the details. This finishes the proof of Proposition
\ref{p3.2}.
\end{proof}

In Chapter \ref{s4.5} below, we will give some further results of
pseudo-differential operators.

To conclude this section,
we investigate an embedding of Sobolev type.
\begin{proposition}\label{p3.3}
Let $a, \alpha_1, \alpha_2, \alpha_3, \tau, q, w$ and $\cl(\rn)$
be as in Definition \ref{d3.1}.
Define
\begin{equation}\label{wjstar}
w_j^*(x)
:=
2^{j(\tau-\gamma)}(1+|x|)^{\delta}w_j(x)
\end{equation}
for all $x \in \rn$ and $j \in \Z_+$.
Then
$A^{w,\tau}_{\cl,q,a}({\mathbb R}^n)$
is embedded into
$B^{w^*}_{\infty,\infty,a}({\mathbb R}^n)$.
\end{proposition}

Observe that if $w \in \cw^{\az_3}_{\az_1,\az_2}$,
then
$w^* \in \cw^{\az_3+\delta}_{(\az_1+\gamma-\tau)_+,(\az_2+\tau-\gamma)_+}$
and hence
\begin{equation*}
(w^*)^{-1} \in
\cw^{\az_3+\delta}_{(\az_2+\tau-\gamma)_+,(\az_1+\gamma-\tau)_+}.
\end{equation*}%

\begin{proof}[Proof of Proposition \ref{p3.3}]
Let $P \in \cq_j(\rn)$ be fixed for $j\in\zz_+$.
Then we see that, for all $x,z \in P$,
\[
\frac{|\vz_j*f(x+y)|}{(1+2^j|y|)^a} \lesssim
\frac{|\vz_j*f(z+(y+x-z))|}{(1+2^j|y+x-z|)^a},
\]
where, when $j=0$, $\vz_0$ is replaced by $\Phi$.
Consequently,
by (W2),
we conclude that, for all $x\in P$,
\begin{eqnarray*}
w_j(x)(\vz_j^*f)_a(x)
&&=\sup_{u \in P}\sup_{y \in \rn}
w_j(u)\frac{|\vz_j*f(u+y)|}{(1+2^j|y|)^a}\\
&&\lesssim
\inf_{z \in P}\sup_{u \in P}\sup_{y \in \rn}
w_j(u)\frac{|\vz_j*f(z+(y+u-z))|}{(1+2^j|y+u-z|)^a}\\
&&\lesssim
\inf_{z \in P}\sup_{u \in P}\sup_{w \in \rn}
w_j(u)\frac{|\vz_j*f(z+w)|}{(1+2^j|w|)^a}\\
&&\lesssim
\inf_{z \in P}\sup_{y \in \rn}
w(z,2^{-j})\frac{|\vz_j*f(z+y)|}{(1+2^j|y|)^a}
\lesssim
\inf_{z \in P}w(z,2^{-j})(\vz_j^*f)_a(z).
\end{eqnarray*}
Thus,
\[
\sup_{x \in P}w_j(x)(\vz_j^*f)_a(x)
\lesssim
\frac{1}{\|\chi_P\|_{\cl(\rn)}}
\|\chi_Pw_j\vz_j^**f \|_{\cl(\rn)}
\le
\frac{|P|^\tau}{\|\chi_P\|_{\cl(\rn)}}
\|f\|_{A^{w,\tau}_{\cl,q,a}(\rn)},
\]
which implies the desired result
and hence completes the proof of Proposition \ref{p3.3}.
\end{proof}

It is also of essential importance to provide a duality result of the following type,
when we consider the wavelet decomposition in
Section \ref{s4}.

In what follows, for $p,q\in(0,\fz]$, $w\in\cw_{\az_1,\az_2}^{\az_3}$ with
$\az_1, \alpha_2, \alpha_3 \in [0,\infty)$,
$w_j$ for $j\in\zz_+$ as in \eqref{2.3},
the \emph{space} $B^{w}_{p,q}(\rn)$
is defined to be the set of all
$f\in \cs'(\rn)$ such that
$$\|f\|_{B^{w}_{p,q}(\rn)}
:=
\|\{w_j\vz_j*f\}_{j \in \Z_+}\|_{\ell^q(L^p(\rn,\Z_+))}<\fz,
$$
where $\Phi, \vz\in\cs(\rn)$, satisfy \eqref{1.1} and \eqref{1.2},
$\vz_0:=\Phi$ and $\vz_j(\cdot):=2^{jn}\vz(2^j\cdot)$ for all $j\in\nn$.

\begin{proposition}\label{p3.4}
Let $\alpha_1, \alpha_2, \alpha_3 \in [0,\infty)$ and $w \in
\mathcal{W}^{\alpha_3}_{\alpha_1,\alpha_2}$.
Assume, in addition, that
there exist $\Phi, \vz\in\cs(\rn)$, satisfying \eqref{1.1} and \eqref{1.2},
such that
\[
\Phi*\Phi+\sum_{j=1}^\infty \vz_j*\vz_j=\delta \quad in \quad \cs'(\rn).
\]
Any $g \in B^w_{\infty,\infty}(\rn)$
defines a continuous functional, $L_g$,
on $B^{w^{-1}}_{1,1}(\rn)$ such that
\[
L_g:\ f \in B^{w^{-1}}_{1,1}(\rn)
\mapsto
\langle \Phi*g,\Phi*f \rangle
+
\sum_{j=1}^\infty \langle \vz_j*g,\vz_j*f \rangle
\in \C.
\]
\end{proposition}

\begin{proof}
The proof is straightforward.
Indeed, for all $g \in B^w_{\infty,\infty}(\rn)$
and $f \in B^{w^{-1}}_{1,1}(\rn)$,
we have
\begin{align*}
|\langle \Phi*g,\Phi*f \rangle|
+
\sum_{j=1}^\infty
|\langle \vz_j*g,\vz_j*f \rangle|
\lesssim
\|g\|_{B^{w}_{\infty,\infty}(\rn)}
\|f\|_{B^{w^{-1}}_{1,1}(\rn)},
\end{align*}
which completes the proof of Proposition \ref{p3.4}.
\end{proof}

We remark that the spaces $B^w_{p,q}(\rn)$
were intensively studied by Kempka \cite{k10-1}
and it was proved in \cite[p. 134]{k10-1} that
they are independent of the choices of $\Phi$ and $\vz$.

\section{Atomic decompositions and wavelets}\label{s4}

Now we place ourselves once again
in the setting of a quasi-normed space $\cl(\rn)$
satisfying only $(\cl1)$ through $(\cl6)$;
recall that we do not need to use
the Hardy-Littlewood maximal operator.

In what follows, for a function $F$ on
$\rr_{\zz_+}^{n+1}:=\rn\times \{2^{-j}:\ j\in\zz_+\}$,
we define
$$
\|F\|_{L_{\cl,q,a}^{w,\tau}(\rr_{\zz_+}^{n+1})}:
=\lf\|\lf\{\sup_{y\in\rn} \frac{|F(y,2^{-j})|}{(1+2^j|\cdot-y|)^a}
\r\}_{j\in\zz_+}\r\|_{\ell^q(\cl_\tau^w(\rn,\zz_+))},
$$
$$
\|F\|_{\cn_{\cl,q,a}^{w,\tau}(\rr_{\zz_+}^{n+1})}:
=\lf\|\lf\{\sup_{y\in\rn} \frac{|F(y,2^{-j})|}{(1+2^j|\cdot-y|)^a}
\r\}_{j\in\zz_+}\r\|_{\ell^q(\cn\cl_\tau^w(\rn,\zz_+))}
$$
$$
\|F\|_{F_{\cl,q,a}^{w,\tau}(\rr_{\zz_+}^{n+1})}:
=\lf\|\lf\{\sup_{y\in\rn} \frac{|F(y,2^{-j})|}{(1+2^j|\cdot-y|)^a}
\r\}_{j\in\zz_+}\r\|_{\cl_\tau^w(\ell^q(\rn,\zz_+))}
$$
and
$$
\|F\|_{\ce_{\cl,q,a}^{w,\tau}(\rr_{\zz_+}^{n+1})}:
=\lf\|\lf\{\sup_{y\in\rn} \frac{|F(y,2^{-j})|}{(1+2^j|\cdot-y|)^a}
\r\}_{j\in\zz_+}\r\|_{\ce\cl_\tau^w(\ell^q(\rn,\zz_+))}.
$$

\subsection{Atoms and molecules}

Now we are going to consider the atomic decompositions,
where we use \eqref{1.4} to denote
the length of multi-indices.

\begin{definition}\label{d4.1}
Let $K \in \Z_+$ and $L \in \Z_+ \cup\{-1\}$.

{\rm (i)}
Let $Q \in \cq(\rn)$.
A \emph{$(K,L)$-atom {\rm(}for $A^{s,\tau}_{\cl,q,a}(\rn)${\rm)}
supported near $Q$} is a $C^K(\rn)$-function ${\mathfrak A}$ satisfying
\begin{eqnarray*}
\mbox{\rm (support condition)}&\quad&
\supp({\mathfrak A}) \subset 3Q,\\
\mbox{\rm (size condition)}&&
\|\partial^{\vec{\alpha}} {\mathfrak A}\|_{L^\infty}
\le |Q|^{-\|{\vec{\alpha}}\|_1/n},\\
\mbox{\rm (moment condition if $\ell(Q)<1$)}&&
\int_{\rn}x^{\vec{\beta}} {\mathfrak A}(x)\,dx=0
\end{eqnarray*}
for all multiindices ${\vec{\alpha}}$ and ${\vec{\beta}}$
satisfying
$\|{\vec{\alpha}}\|_1 \le K$ and $\|{\vec{\beta}}\|_1 \le L$.
Here the moment condition with $L=-1$
is understood as vacant condition.

{\rm (ii)}
A set $\{{\mathfrak A}_{jk}\}_{j \in \zz_+, \, k \in \Z^n}$
of $C^K(\rn)$-functions is called a \emph{collection of
$(K,L)$-atoms {\rm(}for $A^{s,\tau}_{\cl,q,a}(\rn)${\rm)}} if
each ${\mathfrak A}_{jk}$ is a $(K,L)$-atom
supported near $Q_{jk}$.
\end{definition}

\begin{definition}\label{d4.2}
Let $K \in \zz_+$, $L \in \Z_+ \cup\{-1\}$ and $N \in \R$ satisfy
\begin{equation*}
N> L+n.
\end{equation*}%

{\rm (i)}
Let $Q \in \cq(\rn)$.
A \emph{$(K,L)$-molecule {\rm(}for $A^{s,\tau}_{\cl,q,a}(\rn)${\rm)}
associated with a cube $Q$}
is a $C^K(\rn)$-function ${\mathfrak M}$ satisfying
\begin{gather*}
\mbox{\rm (the decay condition)}\quad
|\partial^{\vec{\alpha}} {\mathfrak M}(x)| \le
\left(1+\frac{|x-c_Q|}{\ell(Q)}\right)^{-N}\
{\rm for\ all\ } x\in\rn,\\
\mbox{\rm (the moment condition if $\ell(Q)<1$)}\quad
\int_{\rn}y^{\vec{\beta}} {\mathfrak M}(y)\,dy=0
\end{gather*}
for all multiindices ${\vec{\alpha}}$ and ${\vec{\beta}}$
satisfying
$\|{\vec{\alpha}}\|_1 \le K$ and $\|{\vec{\beta}}\|_1 \le L$.
Here $c_Q$ and $\ell(Q)$ denote, respectively, the center
and the side length of $Q$, and the moment condition with $L=-1$
is understood as vacant condition.

{\rm (ii)}
A set $\{{\mathfrak M}_{jk}\}_{j \in \zz_+, \, k \in \Z^n}$
of $C^K(\rn)$-functions is called a \emph{collection of
$(K,L)$-molecules {\rm(}for $A^{s,\tau}_{\cl,q,a}(\rn)${\rm)}} if
each ${\mathfrak M}_{jk}$ is a $(K,L)$-molecule
associated with $Q_{jk}$.
\end{definition}

\begin{definition}\label{d4.3}
Let $\alpha_1, \alpha_2, \alpha_3, \tau \in [0,\infty)$, $a\in (N_0+\alpha_3,\infty)$
and $q\in(0,\,\fz]$,
where $N_0$ is from $(\cl6)$.
Suppose that $w\in{\mathcal W}^{\alpha_3}_{\alpha_1,\alpha_2}$.
Let $\lambda:=\{\lambda_{jk}\}_{j \in \zz_+, \, k \in \Z^n}$
be a doubly indexed complex sequence.
For $(x,2^{-j}) \in \R^n_{\Z_+}$, let
\begin{equation*}
\Lambda(x,2^{-j})
:=
\sum_{k \in \Z^n}\lambda_{jk}\chi_{Q_{jk}}(x).
\end{equation*}%

{\rm (i)}
The {\it inhomogeneous  sequence space
$b^{w,\tau}_{\cl,q,a}(\rn)$}
is defined to be the set of
all $\lambda$ such that
$
\|\lambda\|_{b^{w,\tau}_{\cl,q,a}(\rn)}:=
\|\Lambda\|_{L_{\cl,q,a}^{w,\tau}(\rr_{\zz_+}^{n+1})}<\fz.
$

{\rm (ii)}
The {\it inhomogeneous  sequence space
$n^{w,\tau}_{\cl,q,a}(\rn)$}
is defined to be the set of
all $\lambda$ such that
$
\|\lambda\|_{n^{w,\tau}_{\cl,q,a}(\rn)}:=
\|\Lambda\|_{\cn_{\cl,q,a}^{w,\tau}(\rr_{\zz_+}^{n+1})}<\fz.
$

{\rm (iii)}
The {\it inhomogeneous sequence space
$f^{w,\tau}_{\cl,q,a}(\rn)$}
is defined to be the set of all $\lambda$ such that
$
\|\lambda\|_{f^{w,\tau}_{\cl,q,a}(\rn)}:=
\|\Lambda\|_{F_{\cl,q,a}^{w,\tau}(\rr_{\zz_+}^{n+1})}<\fz.
$

{\rm (iv)}
The {\it inhomogeneous sequence space
$e^{w,\tau}_{\cl,q,a}(\rn)$}
is defined to be the set of
all $\lambda$ such that
$
\|\lambda\|_{e^{w,\tau}_{\cl,q,a}(\rn)}:=
\|\Lambda\|_{\ce_{\cl,q,a}^{w,\tau}(\rr_{\zz_+}^{n+1})}<\fz.
$

When $\tau=0$, then $\tau$ is omitted from the above notation.
\end{definition}

In the present paper we take up many types of atomic decompositions.
To formulate them, it may be of use
to present the following definition.

\begin{definition}\label{d4.4}
Let $X$ be a function space embedded into $\cs'(\rn)$
and ${\mathcal X}$ a quasi-normed space of sequences.
The pair $(X,{\mathcal X})$ is called to \emph{admit the atomic decomposition},
if it satisfies the following two conditions:

{\rm (i)}
{\rm (Analysis condition)}
For any $f \in\! X$,
there exist a collection of atoms,
$\{{\mathfrak A}_{jk}\}_{j \in \zz_+, \, k \in \Z^n}$,
and a complex sequence
$\{\lambda_{jk}\}_{j \in \zz_+, \, k \in \Z^n}$
such that
$$f=\sum_{j=0}^\infty
\sum_{k \in \Z^n}\lambda_{jk}{\mathfrak A}_{jk}$$
in ${\mathcal S}'(\rn)$
and that
$\|\{\lambda_{jk}\}_{j \in \zz_+, \, k \in \Z^n}\|_{{\mathcal X}}
\lesssim
\|f\|_X$
with the implicit positive constant independent of $f$.

{\rm (ii)}
{\rm (Synthesis condition)}
Given a collection of atoms,
$\{{\mathfrak A}_{jk}\}_{j \in \zz_+, \, k \in \Z^n}$,
and a complex sequence
$\{\lambda_{jk}\}_{j \in \zz_+, \, k \in \Z^n}$
satisfying
$
\|\{\lambda_{jk}\}_{j \in \zz_+, \, k \in \Z^n}\|_{{\mathcal X}}
<\infty
$,
then
$f := \sum_{j=0}^\infty
\sum_{k \in \Z^n}\lambda_{jk}{\mathfrak A}_{jk}$
converges in $\cs'(\rn)$
and satisfies that
$\|f\|_X
\lesssim
\|\{\lambda_{jk}\}_{j \in \zz_+, \, k \in \Z^n}\|_{{\mathcal X}}
$ with the implicit positive constant independent of
$\{\lambda_{jk}\}_{j \in \zz_+, \, k \in \Z^n}$.

In analogy,
a pair $(X,{\mathcal X})$ is said to admit the \emph{molecular decomposition}
or the \emph{wavelet decomposition},
where the definition of wavelets appears in
Subsection {\rm \ref{s4.4}} below.
\end{definition}

In this section, we aim to prove the following conclusion.

\begin{theorem}\label{t4.1}
Let $K \in \Z_+$, $L \in \Z_+$,
$\alpha_1, \alpha_2, \alpha_3, \tau \in [0,\infty)$
and $q\in(0,\,\fz]$.
Suppose that $w\in{\mathcal W}^{\alpha_3}_{\alpha_1,\alpha_2}$
and that $\eqref{3.35}$ holds true, namely, $a\in (N_0+\alpha_3,\infty)$.
Let $\delta$ be as in $(\cl6)$.
Assume, in addition, that
\begin{equation}\label{4.3}
L>\alpha_3+\delta+n-1+\gamma-n\tau+\alpha_1,
\end{equation}
\begin{equation}\label{4.4}
N>L+\az_3+\dz+2n
\end{equation}
and that
\begin{equation}\label{4.5}
K+1>\alpha_2+n\tau, \,
L+1>\alpha_1.
\end{equation}
Then the pair
$(A^{w,\tau}_{\cl,q,a}(\rn),a^{w,\tau}_{\cl,q,a}(\rn))$
admits the atomic / molecular decompositions.
\end{theorem}

\subsection{Proof of Theorem \ref{t4.1}}

The proof of Theorem \ref{t4.1} is made up
of several lemmas.
Our primary concern for the proof of Theorem \ref{t4.1}
is the following question:
\begin{quote}
Do the summations
$\sum_{j=0}^\infty
\sum_{k \in {\mathbb Z}^n}\lambda_{jk}{\mathfrak A}_{jk}$
and
$\sum_{j=0}^\infty
\sum_{k \in {\mathbb Z}^n}\lambda_{jk}{\mathfrak M}_{jk}$
converge in $\cs'(\rn)$?
\end{quote}

Recall again that we are assuming only $(\cl1)$ through $(\cl6)$.

\begin{lemma}\label{l4.1}
Let $\alpha_1, \alpha_2, \alpha_3 \in [0,\infty)$
and $w \in {\mathcal W}^{\alpha_3}_{\alpha_1,\alpha_2}$.
Assume, in addition, that the parameters
$K \in \Z_+$,
$L \in \Z_+$
and
$N \in (0,\infty)$ in Definition {\rm \ref{d4.2}}
satisfy \eqref{4.3}, \eqref{4.4} and \eqref{4.5}.
Assume that $\lambda:=\{\lambda_{jk}\}_{j \in \zz_+, \, k \in {\mathbb Z}^n}
\in b^{w,\tau}_{\cl,\infty,a}(\rn)$
and $\{{\mathfrak M}_{jk}\}_{j \in \Z_+, \, k \in {\mathbb Z}^n}$
is a family of ${\rm (}K,L{\rm )}$-molecules.
Then the series
\begin{equation}\label{4.6}
f=\sum_{j=0}^\infty
\sum_{k \in {\mathbb Z}^n}\lambda_{jk}{\mathfrak M}_{jk}
\end{equation}
converges in ${\mathcal S}'(\rn)$.
\end{lemma}

\begin{proof}
Let $\varphi \in \cs(\rn)$.
Recall again that $\gamma$ and $\delta$ are constants appearing
in the assumption $(\cl6)$.
By \eqref{4.3} and \eqref{4.4}, we can choose $M\in(\az_3+\dz+n,\fz)$
such that
\begin{equation}\label{4.7}
L+1-\gamma-\alpha_1-M+n\tau>0
\quad {\rm and\quad } N>L+M+n.
\end{equation}
It follows, from the definition of molecules and Lemma \ref{l2.4},
that
$$
\left|\int_{\rn}
{\mathfrak M}_{jk}(x)\varphi(x)\,dx\right|
\lesssim
2^{-j(L+1)}(1+2^{-j}|k|)^{-M}.
$$
By the assumption $(\cl6)$, we conclude that
\begin{equation}\label{4.8}
\left|\int_{\rn}
{\mathfrak M}_{jk}(x)\varphi(x)\,dx\right|
\lesssim
2^{-j(L+1-\gamma)}(1+2^{-j}|k|)^{-M}(1+|k|)^{\delta}
\|\chi_{Q_{jk}}\|_{\cl(\rn)}.
\end{equation}
From the condition (W1),
we deduce that,
for all $j\in\zz_+$ and $x\in\rn$,
$2^{-j\alpha_1}w(x,1) \lesssim w_j(x)$
and, from (W2), that, for all $x\in\rn$,
$w(0,1) \lesssim w(x,1)(1+|x|)^{\alpha_3}$.
Combining them, we conclude that,
for all $j\in\zz_+$ and $x\in\rn$,
\begin{equation}\label{eq:120223-3}
w(0,1) \lesssim (1+|x|)^{\alpha_3} 2^{j\alpha_1}w_j(x).
\end{equation}
Consequently, we have
\begin{equation}\label{4.9}
1 \lesssim (1+|k|)^{\alpha_3} 2^{j\alpha_1}w_j(x)
\end{equation}
for all $x \in Q_{jk}$ with $j\in\zz_+$ and $k\in\zz^n$.
By \eqref{4.8} and \eqref{4.9}, we further see that,
for all $j\in\zz_+$ and $k\in\zz^n$,
\begin{equation}\label{eq:120223-4}
\left|\lambda_{jk}\int_{\rn}
{\mathfrak M}_{jk}(x)\varphi(x)\,dx\right|
\lesssim
2^{-j(L+1-\gamma-\alpha_1-M+n\tau)}
(1+|k|)^{-M+\alpha_3+\delta}
\|\lambda\|_{b^{w,\tau}_{\cl,\infty,a}(\rn)}.
\end{equation}
So by \eqref{4.7},
this inequality is summable over $j \in \Z_+$
and $k \in \Z^n$,
which completes the proof of Lemma \ref{l4.1}.
\end{proof}

In view of Lemma \ref{l3.1},
Lemma \ref{l4.1} is sufficient
to ensure that, for any $f\in A_{\cl, q,a}^{w,\tau}(\rn)$,
the convergence
in \eqref{4.6} takes place in ${\mathcal S}'(\rn)$.
Indeed, in view of Remark \ref{r3.1},
without loss of generality,
we may assume that $f \in B^{w,\tau}_{\cl,\infty,a}(\R^n)$.
Then, by Lemma \ref{l4.1},
we see that the convergence
in \eqref{4.6} takes place in ${\mathcal S}'(\rn)$.

Next, we consider the synthesis part of Theorem \ref{t4.1}.

\begin{lemma}\label{l4.2}
Let $s\in (0,\fz), \, \alpha_1, \alpha_2, \alpha_3, \tau \in [0,\infty)$,
$a\in (N_0+\alpha_3,\infty)$ and $q\in(0,\,\fz]$.
Suppose that $w\in{\mathcal W}^{\alpha_3}_{\alpha_1,\alpha_2}$.
Assume, in addition, that
$K \in \Z_+$ and $L \in \Z_+ $ satisfy
\eqref{4.3}, \eqref{4.4} and \eqref{4.5}.
Let $\lambda:=\{\lambda_{jk}\}_{j \in \zz_+, \, k \in {\mathbb Z}^n}
\in a^{w,\tau}_{\cl,q,a}(\rn)$
and ${\mathfrak M}:=\{{\mathfrak M}_{jk}\}_{j \in \Z_+, \,
k \in {\mathbb Z}^n}$
be a collection of ${\rm (}K,L{\rm )}$-molecules.
Then the series
\[
f=\sum_{j=0}^\infty
\sum_{k \in \Z^n}\lambda_{jk}{\mathfrak M}_{jk}
\]
converges in ${\mathcal S}'(\rn)$
and defines an element in $A^{w,\tau}_{{\cl},q,a}(\rn)$.
Furthermore,
\[
\|f\|_{A^{w,\tau}_{{\cl},q,a}(\rn)}
\lesssim
\|\lambda\|_{a^{w,\tau}_{{\cl},q,a}(\rn)},
\]
with the implicit positive constant independent of $f$.
\end{lemma}

\begin{remark}\label{r4.1}
One of the differences from the classical theory of molecules
is that there is no need to distinguish
Besov-type spaces and Triebel-Lizorkin-type spaces.
Set
$  \sigma_p:=\max\{0,n/p-n\}$.
For example,
recall that in {\rm \cite[Theorem 13.8]{t97}}
we need to assume
\[
L\ge \max(-1,\lfloor \sigma_p-s\rfloor)
\mbox{ or }
L\ge \max(-1,\lfloor \max(\sigma_p,\sigma_q)-s\rfloor)
\]
according as we consider Besov spaces
or Triebel-Lizorkin spaces.
However, our approach does not require
such a distinction.
This seems due to the fact that
we are using the Peetre maximal operator.
\end{remark}

\begin{proof}[Proof of Lemma \ref{l4.2}]
The convergence of $f$ in ${\mathcal S}'(\rn)$
is a consequence of Lemma \ref{l4.1}.

Let us prove $\|f\|_{A^{w,\tau}_{{\cl},q,a}(\rn)}
\lesssim
\|\lambda\|_{a^{w,\tau}_{{\cl},q,a}(\rn)}$.
To this end,
we fix $z \in \rn$ and $j,l \in \Z_+$.
Let us abbreviate
$\sum_{k \in \Z^n}\lambda_{lk}{\mathfrak M}_{lk}$
to $f_l$.
Then we have
\begin{eqnarray*}
\sup_{z \in \R^n}\frac{|\varphi_j*f_l(x+z)|}{(1+2^j|z|)^a}
\lesssim
\begin{cases}
\dis
\sup_{z \in \R^n}
\left\{
\sum_{k \in \Z^n}
\frac{2^{ln-(j-l)(L+1)}|\lambda_{lk}|}
{(1+2^l|z|)^a(1+2^{l}|x+z-2^{-l}k|)^M}
\right\},&j \ge l;\\
\dis
\sup_{z \in \R^n}
\left\{
\sum_{k \in \Z^n}
\frac{2^{jn-(l-j)(K+1)}|\lambda_{lk}|}
{(1+2^j|z|)^a(1+2^{j}|x+z-2^{-l}k|)^M}
\right\},&j<l
\end{cases}
\end{eqnarray*}
by Lemma \ref{l2.4}, where $M$ is as in \eqref{4.7}.
Consequently,
by virtue of the inequalities
$1+2^j|z| \le 1+2^{\max(j,l)}|z|$
for all $z\in\rn$ and $j,l\in\zz_+$,
we have
\begin{eqnarray*}
&&
\sup_{z \in \R^n}\frac{|\varphi_j*f_k(x+z)|}{(1+2^j|z|)^a}\\
&&\quad\lesssim
\begin{cases}
\dis\sup_{z,w \in \R^n}
\left\{
\sum_{m \in \Z^n}\sum_{k \in \Z^n}
\frac{2^{ln-(j-l)(L+1)}(1+2^l|w|)^{-a}
|\lambda_{lm}|\chi_{Q_{lm}}(x+w)}
{(1+2^{l}|x+z-2^{-l}k|)^M}
\right\},& j \ge l;\\
\dis\sup_{z,w \in \R^n}
\left\{
\sum_{m \in \Z^n}
\sum_{k \in \Z^n}
\frac{2^{jn+(j-l)(K+1)}(1+2^l|w|)^{-a}
|\lambda_{lm}|\chi_{Q_{lm}}(x+w)}
{(1+2^{j}|x+z-2^{-l}k|)^M}
\right\},& j<l.
\end{cases}
\end{eqnarray*}
By
$$\sum_{k\in\zz^n}\frac{2^{ln}}{(1+2^{l}|x+z-2^{-l}k|)^M}
+\sum_{k\in\zz^n}\frac{2^{jn}}{(1+2^{j}|x+z-2^{-l}k|)^M}
\ls\int_\rn\frac1{(1+|y|)^M}\,dy$$
and $M\in(\az_3+\dz+n,\fz)$, we conclude that
\begin{eqnarray}
\label{eq:120223-2}
\sup_{z \in \R^n}\frac{|\varphi_j*f_l(x+z)|}{(1+2^j|z|)^a}
\lesssim
\begin{cases}
\dis2^{-(j-l)(L+1)}
\sum_{m \in \Z^n}
\left[
\sup_{w \in \rn}
\frac{|\lambda_{lm}|\chi_{Q_{lm}}(x+w)}{(1+2^l|w|)^a}
\right],& j \ge l;\\
\dis2^{(j-l)(K+1)}
\sum_{m \in \Z^n}
\left[\sup_{w \in \rn}
\frac{|\lambda_{lm}|\chi_{Q_{lm}}(x+w)}{(1+2^l|w|)^a}
\right],& j<l.
\end{cases}
\end{eqnarray}
If we use \eqref{4.5} and
Lemma \ref{l2.3},
then we obtain the desired result,
which completes the proof of Lemma \ref{l4.2}.
\end{proof}

With these preparations in mind,
let us prove Theorem \ref{t4.1}.
We investigate the case
of $F_{\cl,q,a}^{w,\tau}(\rn)$,
other cases being similar.

\begin{proof}[Proof of Theorem \ref{t4.1} {\rm (Analysis part)}]
Let $L \in \Z_+$ satisfying \eqref{4.3} be fixed.
Let us choose $\Psi,\psi \in C^\infty_{\rm c}(\rn)$
such that
\begin{equation}\label{4.10}
\supp\Psi,\ \supp \psi \subset
\{x=(x_1,x_2,\cdots,x_n)\,:\
\max(|x_1|,|x_2|,\cdots,|x_n|) \le 1\}
\end{equation}
and that
\begin{equation}\label{4.11}
\int_{\rn}\psi(x)x^{\vec{\beta}} \,dx=0
\end{equation}
for all multiindices ${\vec{\beta}}$
with $\|{\vec{\beta}}\|_1 \le L$
and that
$\Psi*\Psi+\sum_{j=1}^\infty \psi_j*\psi_j=\delta_0$
in $\cs'(\rn)$,
where $\psi_j := 2^{jn}\psi(2^j \cdot)$
for all $j \in {\mathbb N}$.
Then, for all $f\in F_{\cl,q,a}^{w,\tau}(\rn)$,
\begin{equation}\label{4.12}
f=\Psi*\Psi*f+\sum_{j=1}^\infty \psi_j*\psi_j*f
\end{equation}
in $\cs'(\rn)$.
With this in mind,  let us set, for all $j\in\nn$ and $k\in\zz^n$,
\begin{equation}\label{4.13}
\lambda_{0k}:=
\int_{Q_{0k}}|\Psi*f(y)|\,dy, \quad
\lambda_{jk}:=
2^{jn}\int_{Q_{jk}}|\psi_j*f(y)|\,dy
\end{equation}
and, for all $x\in\rn$,
\begin{equation}\label{4.14}
{\mathfrak A}_{0k}(x)
:=
\frac{1}{\lambda_{0k}}\int_{Q_{0k}}\Psi(x-y)\Psi*f(y)\,dy,
\
{\mathfrak A}_{jk}(x)
:=
\frac{1}{\lambda_{jk}}\int_{Q_{jk}}\psi_j(x-y)\psi_j*f(y)\,dy.
\end{equation}
Here in the definition \eqref{4.14} of ${\mathfrak A}_{jk}$
for $j \in \zz_+$ and $k \in \Z^n$,
if $\lambda_{jk}=0$, then accordingly we redefine
${\mathfrak A}_{jk} := 0$.

Observe that
$
f := \sum_{j=0}^\infty \sum_{k \in \Z^n}\lambda_{jk}{\mathfrak A}_{jk}
$
in $\cs'(\rn)$
by virtue of \eqref{4.12} and \eqref{4.14}.
Let us prove that ${\mathfrak A}_{jk}$, given by \eqref{4.14},
is an atom supported near $Q_{jk}$
modulo a multiplicative constant
and that $\lambda:=\{\lambda_{jk}\}_{j \in \N, \, k \in \Z^n}$,
given by \eqref{4.13}, satisfies that
\begin{equation}\label{4.15}
\|\lambda\|_{f^{w,\tau}_{\cl,q,a}(\rn)} \lesssim
\|f\|_{F^{w,\tau}_{\cl,q,a}(\rn)}.
\end{equation}

Observe that, when $x+z \in Q_{jk}$,
then by the Peetre inequality we have
\begin{eqnarray*}
\frac{2^{jn}}{(1+2^j|z|)^a}\int_{Q_{jk}}|\psi_j*f(y)|\,dy
&&=
\frac{2^{jn}}{(1+2^j|z|)^a}\int_{x+z-Q_{jk}}|\psi_j*f(x+z-y)|\,dy\\
&&\lesssim
\int_{x+z-Q_{jk}}
\frac{2^{jn}}{(1+2^j|z|)^a(1+2^j|y|)^a}|\psi_j*f(x+z-y)|\,dy\\
&&\le
\int_{x+z-Q_{jk}}\frac{2^{jn}}{(1+2^j|z-y|)^a}|\psi_j*f(x+z-y)|\,dy\\
&&\lesssim
\sup_{w \in \rn}
\frac{|\psi_j*f(x-w)|}{(1+2^j|w|)^a}.
\end{eqnarray*}
Consequently, we see that
\begin{equation}\label{4.16}
\sup_{w \in Q_{jk}}
\left\{
\frac{2^{jn}}{(1+2^j|x-w|)^a}\int_{Q_{jk}}|\psi_j*f(y)|\,dy
\right\}
\lesssim
\sup_{z \in \rn}
\frac{|\psi_j*f(x-z)|}{(1+2^j|z|)^a}.
\end{equation}
In view of the fact that
$\{Q_{jk}\}_{k \in \Z^n}$
is a disjoint family for each fixed $j \in \zz_+$,
\eqref{4.16} reads as
\begin{equation}\label{4.17}
\sup_{z \in \rn}\frac{1}{(1+2^j|z|)^a}
\left|\sum_{k \in \Z^n}\lambda_{jk}\chi_{Q_{jk}}(x+z)\right|
\lesssim
\sup_{z \in \rn}
\frac{|\psi_j*f(x-z)|}{(1+2^j|z|)^a}.
\end{equation}
In particular, when $j=0$, we see that
\begin{equation}\label{4.18}
\sup_{z \in \rn}\frac{1}{(1+|z|)^a}
\left|\sum_{k \in \Z^n}\lambda_{0k}\chi_{Q_{0k}}(x+z)\right|
\lesssim
\sup_{z \in \rn}
\frac{|\Psi*f(x-z)|}{(1+|z|)^a}.
\end{equation}
Consequently, from \eqref{4.17} and \eqref{4.18},
we deduce the estimate \eqref{4.15}.

Meanwhile, via \eqref{4.10}, a direct calculation about the size of supports
yields
\begin{equation}\label{4.19}
{\rm supp}({\mathfrak A}_{jk})
\subset
Q_{jk}+{\rm supp}(\psi_j)
\subset
3Q_{jk}
\end{equation}
and that
there exists a positive constant $C_{\vec{\alpha}}$ such that
\begin{equation}\label{4.20}
|\partial^{\vec{\alpha}} {\mathfrak A}_{jk}(x)|
=\frac{2^{j(\|{\vec{\alpha}}\|_1+n)}}{\lambda_{jk}}
\left|\int_{Q_{jk}}\partial^{\vec{\alpha}}\psi(2^j(x-y))\psi_j*f(y)\,dy\right|
\le{}C_{{\vec{\alpha}}} 2^{j\|{\vec{\alpha}}\|_1}
\end{equation}
for all multiindices ${\vec{\alpha}}$ as long as $\lambda_{jk} \ne 0$.

Keeping \eqref{4.19} and \eqref{4.20}
in mind, let us show that each ${\mathfrak A}_{jk}$ is an atom
modulo a positive multiplicative constant
$\sum_{\|{\vec{\alpha}}\|_1 \le K}C_{\vec{\alpha}}$.
The support condition follows from \eqref{4.19}.
The size condition follows from \eqref{4.20}.
Finally, the moment condition follows from \eqref{4.11},
which completes the proof of Theorem \ref{t4.1}.
\end{proof}

\subsection{The regular case}

Motivated by Remark \ref{r4.1},
we are now going to consider
the regular case of Theorem \ref{t4.1}.
That is, we are going to discuss the possibility
of the case when $L=-1$ of Theorem \ref{t4.1}.
This is achieved by polishing
a crude estimate \eqref{2.18}.
Our result is the following.

\begin{theorem}\label{t4.2}
Let $K \in \N \cup\{0\}$,
$L=-1$,
$\alpha_1, \alpha_2, \alpha_3, \tau \in [0,\infty)$ and $q\in(0,\,\fz]$.
Suppose that $w\in\star-{\mathcal W}^{\alpha_3}_{\alpha_1,\alpha_2}$.
Assume, in addition, that {\rm \eqref{3.35}} and
$(\ref{4.4})$ hold true,
and that
\begin{equation}\label{4.21}
0>\alpha_3+\delta+n+\gamma-n\tau-\alpha_1
\end{equation}
and
\begin{equation}\label{4.22}
\alpha_1>n\tau, \, K+1>\alpha_2+n\tau.
\end{equation}
Then the pair
$(A^{w,\tau}_{\cl,q,a}(\rn),a^{w,\tau}_{\cl,q,a}(\rn))$
admits the atomic / molecular decompositions.
\end{theorem}

To prove Theorem \ref{t4.2}
we need to modify Lemma \ref{l2.3}.

\begin{lemma}\label{l4.3}
Let $D_1, D_2$,
$\alpha_1, \alpha_2, \alpha_3, \tau \in [0,\infty)$
and $q\in(0,\fz]$ be parameters satisfying
\begin{equation*}
D_1+\alpha_1>0 \quad and \quad
D_2-\alpha_2>n \tau.
\end{equation*}%
Let $\{g_\nu\}_{\nu\in\zz_+}$
be a family of measurable functions on $\rn$
and $w \in\star-{\mathcal W}^{\alpha_3}_{\alpha_1,\alpha_2}$.
For all $j\in\zz_+$ and $x\in\rn$, let
$$G_j(x):=
\sum_{\nu=j+1}^\fz
2^{-(\nu-j)D_1}g_\nu(x)
+
\sum_{\nu=0}^{j}
2^{-(j-\nu)D_2}g_\nu(x).$$
Then \eqref{2.14} through \eqref{2.17}
hold true.
\end{lemma}

\begin{proof}
The proof is based upon
the modification of \eqref{2.20}.

If, in Definition \ref{d3.3}, we let $t:=2^{-\nu}$ and $s:=2^{-j}$
for $j,\nu \in \Z_+$ with $\nu \ge j$,
then we have
\begin{equation}\label{4.24}
w_j(x) \ls 2^{\alpha_1(j-\nu)}w_\nu(x).
\end{equation}
If, in Definition \ref{d3.3}, we let $t=2^{-j}$ and $s=2^{-\nu}$
for $j,\nu \in \N$ with $j \ge \nu$,
then we have
\begin{equation}\label{4.25}
w_j(x) \ls 2^{\alpha_2(j-\nu)}w_\nu(x).
\end{equation}
If we combine \eqref{4.24} and \eqref{4.25},
then we see that
\begin{equation}\label{4.26}
w_j(x) \ls
\begin{cases}
2^{\alpha_1(j-\nu)}w_\nu(x),&\nu \ge j;\\
2^{\alpha_2(j-\nu)}w_\nu(x),&\nu \le j
\end{cases}
\end{equation}
for all $j,\nu \in \Z_+$.
Let us write
\begin{align*}
{\rm I}(P):=&
\frac1{|P|^\tau}
\lf\|\chi_P \lf[\sum_{j=j_P\vee0}^\fz
\lf|\sum_{\nu=0}^j w_j2^{(\nu-j)D_2}g_\nu\r|^q
\r]^{1/q}\r\|_{\cl(\rn)}\\
&+\frac1{|P|^\tau}
\lf\|\chi_P \lf[\sum_{j=j_P\vee0}^\fz
\lf|\sum_{\nu=j+1}^\fz w_j2^{(j-\nu)D_1}g_\nu\r|^q
\r]^{1/q}\r\|_{\cl(\rn)},
\end{align*}
where $P$ is a dyadic cube chosen arbitrarily.

Let us suppose $q\in(0,1]$,
since when $q\in(1,\fz]$, an argument similar to Lemma \ref{l2.3} works.
Then we deduce, from \eqref{4.26} and $(\cl4)$, that
\begin{align*}
{\rm I}(P)
\ls&\frac1{|P|^\tau}
\lf\|\chi_P \lf[\sum_{j=j_P\vee0}^\fz
\sum_{\nu=0}^j 2^{-(j-\nu)(D_2-\alpha_2) q}
\lf|w_\nu g_\nu\r|^q
\r]^{1/q}\r\|_{\cl(\rn)}\noz\\
&+\frac1{|P|^\tau}
\lf\|\chi_P \lf[\sum_{j=j_P\vee0}^\fz
\sum_{\nu=j+1}^\infty 2^{-(\nu-j)(D_1+\alpha_1) q}
\lf|w_\nu g_\nu\r|^q
\r]^{1/q}\r\|_{\cl(\rn)}
\end{align*}
by virtue of (W1) and (\ref{2.22}).
We change the order of summations in the right-hand side
of the above inequality to obtain
\begin{align}
{\rm I}(P)
\ls&\frac1{|P|^\tau}
\lf\|\chi_P \lf[\sum_{\nu=0}^\fz\sum_{j=\nu\vee j_P\vee0}^\fz
2^{-(j-\nu)(D_2-\alpha_2) q}
\lf|w_\nu g_\nu\r|^q
\r]^{1/q}\r\|_{\cl(\rn)}\noz\\
\noz
&+\frac1{|P|^\tau}
\lf\|\chi_P \lf[\sum_{\nu=j_P\vee0}^\infty\sum_{j=j_P\vee0}^\nu
2^{-(\nu-j)(D_1+\alpha_1) q}
\lf|w_\nu g_\nu\r|^q
\r]^{1/q}\r\|_{\cl(\rn)}.
\end{align}
Now we decompose the summand
with respect to $\nu$ according to
$j \ge j_P \vee 0$ or $j< j_P \vee 0$.
Since $D_2\in(\alpha_2+n\tau,\fz)$,
we can choose $\ez\in(0,\fz)$
such that $D_2\in(\alpha_2+n\tau+\ez,\fz)$.
From this, $D_1\in(-\alpha_1,\fz)$,
the H\"older inequality, $(\cl2)$ and $(\cl4)$,
it follows that
\begin{eqnarray*}
{\rm I}(P)&&\lesssim
\|\{g_\nu\}_{\nu\in\zz_+}\|_{\cl^w_\tau(\ell^q(\rn,\zz_+))}\\
&&\hs+\frac1{|P|^\tau}
\lf\|\chi_P \lf[\sum_{\nu=0}^{j_P \vee 0}\sum_{j=j_P\vee0}^\fz
2^{-(j-\nu)(D_2- \alpha_2) q}
\lf|w_\nu g_\nu\r|^q
\r]^{\frac 1q}\r\|_{\cl(\rn)}\\
&&\ls
\|\{g_\nu\}_{\nu\in\zz_+}\|_{\cl^w_\tau(\ell^q(\rn,\zz_+))}
+\frac{2^{-(j_P\vee0)(D_2-\az_2-\ez)}}{|P|^\tau}
\lf\|\chi_P \sum_{\nu=0}^{j_P \vee 0}
2^{\nu(D_2-\alpha_2-\ez)}
\lf|w_\nu g_\nu\r|
\r\|_{\cl(\rn)},
\end{eqnarray*}
which is just (\ref{eq:120223-1}).
Therefore, we can go through the argument same as the proof of Lemma \ref{l2.3},
which completes the proof of Lemma \ref{l4.3}.
\end{proof}

\begin{proof}[Proof of Theorem \ref{t4.2}]
The proof of this theorem
is based upon reexamining that of Theorem \ref{t4.1}.
Recall that the proof of Theorem \ref{t4.1}
is made up of three parts:
Lemma \ref{l4.1}, Lemma \ref{l4.2} and the analysis condition.
Let us start with modifying Lemma \ref{l4.1}.
By \eqref{4.21}, we choose $M \in (\alpha_3+\delta+n,\infty)$ so that
\begin{equation}\label{eq:120223-5}
-\gamma+\alpha_1-M+n\tau>0
\quad {\rm and\quad } N>L+2n+\az_3+\delta.
\end{equation}
Assuming that
$w\in\star-{\mathcal W}^{\alpha_3}_{\alpha_1,\alpha_2}$,
we see that $\alpha_1$ in the proof of Theorem \ref{t4.1}
and the related statements
can be all replaced with $-\alpha_1$.
More precisely,
(\ref{eq:120223-3}) undergoes the following change:
\[
w(0,1) \lesssim (1+|x|)^{\alpha_3} 2^{-j\alpha_1}w_j(x)
\mbox{ for all } x\in\rn \mbox{ and } j\in\zz_+.
\]
Assuming $L=-1$,
we can replace (\ref{eq:120223-4}) with the following estimate:
for all $j\in\zz_+$ and $k\in\zz^n$,
\[
\left|\lambda_{jk}\int_{\rn}
{\mathfrak M}_{jk}(x)\varphi(x)\,dx\right|
\lesssim
2^{-j(-\gamma+\alpha_1-M+n\tau)}
(1+|k|)^{-M+\alpha_3+\delta}
\|\lambda\|_{b^{w,\tau}_{\cl,\infty,a}(\rn)}.
\]
Since we are assuming $(\ref{eq:120223-5})$,
we have a counterpart for Lemma \ref{l4.1},
that is, the series
$f=\sum_{j=0}^\infty
\sum_{k \in {\mathbb Z}^n}\lambda_{jk}{\mathfrak M}_{jk}
$
converges in ${\mathcal S}'(\rn)$.

Next, we reconsider Lemma \ref{l4.2}.
Lemma \ref{l4.2} remains unchanged except
that we substitute $L=-1$.
Thus, the concluding estimate (\ref{eq:120223-2})
undergoes the following change:
\[
\sup_{z \in \R^n}\frac{|\varphi_j*f_l(x+z)|}{(1+2^j|z|)^a}
\lesssim
\begin{cases}
\dis
\sum_{m \in \Z^n}
\left[
\sup_{w \in \rn}
\frac{|\lambda_{lm}|\chi_{Q_{lm}}(x+w)}{(1+2^l|w|)^a}
\right],& j \ge l;\\
\dis2^{(j-l)(K+1)}
\sum_{m \in \Z^n}
\left[\sup_{w \in \rn}
\frac{|\lambda_{lm}|\chi_{Q_{lm}}(x+w)}{(1+2^l|w|)^a}
\right],& j<l.
\end{cases}
\]
Assuming (\ref{4.22}),
we can use Lemma \ref{l4.3}
with $D_1=0$ and $D_2=K+1$.

Finally, the analysis part of the proof of Theorem \ref{t4.1}
remains unchanged in the proof of Theorem \ref{t4.2}.
Indeed, we did not use the condition for weights
or the moment condition here.

Therefore, with these modifications,
the proof of Theorem \ref{t4.2} is complete.
\end{proof}

\subsection{Biorthogonal wavelet decompositions}
\label{s4.4}
We use biorthogonal wavelet bases on $\R$, namely, a system satisfying
\begin{eqnarray*}
&&\langle \psi^0(\cdot- k), \widetilde{\psi}^0(\cdot -m) \rangle_{L^2(\R)}
=
\delta_{k,m} \quad (k,m \in \Z),\\
&&\langle 2^{jn/2}\psi^1(2^j\cdot-k),
2^{\nu n/2}\widetilde{\psi}^1(2^{\nu} \cdot -m)\rangle_{L^2(\R)}
=
\delta_{(j,k),(\nu,m)} \quad (j,k,\nu,m \in \Z)
\end{eqnarray*}
of scaling functions
$(\psi^0,\widetilde{\psi}^0)$
and associated wavelets
$(\psi^1,\widetilde{\psi}^1)$.
Notice that the latter includes that,
for all $f\in L^2(\rn)$,
\begin{align*}
f
&=
\sum_{j,k\in {\mathbb Z}}
2^{jn}
\langle f, \psi^1(2^j \cdot -k)\rangle_{L^2(\R)} \widetilde{\psi}^1(2^j \cdot -k)\\
&=
\sum_{j,k\in {\mathbb Z}}
2^{jn}
\langle f, \widetilde{\psi}^1(2^j \cdot -k) \rangle_{L^2(\R)}\psi^1(2^j \cdot -k)\\
&=
\sum_{k \in {\mathbb Z}}
\langle f, \psi^0(\cdot -k)\rangle_{L^2(\R)} \widetilde{\psi}^0(\cdot -k)
+
\sum_{(j,k)\in {\mathbb Z}_+ \times {\mathbb Z}}
2^{jn}
\langle f, \psi^1(2^j \cdot -k)\rangle_{L^2(\R)} \widetilde{\psi}^1(2^j \cdot -k)\\
&=
\sum_{k \in {\mathbb Z}}
\langle f, \widetilde{\psi}^0(\cdot -k) \rangle_{L^2(\R)}\psi^0(\cdot -k)
+
\sum_{(j,k)\in {\mathbb Z}_+ \times {\mathbb Z}}
2^{jn}
\langle f, \widetilde{\psi}^1(2^j \cdot -k) \rangle_{L^2(\R)}\psi^1(2^j \cdot -k)
\end{align*}
holds true in $L^2({\mathbb R})$.
We construct a basis in $L^2(\rn)$ by using the well-known tensor product
procedure. Set $E := \{0, 1\}^n \setminus \{(0,\ldots, 0)\}$.
We need to consider the tensor products
$$
\Psi^{\bf c} := \otimes_{j=1}^n \psi^{c_j}\quad\mbox{and}\quad
\widetilde{\Psi}^{\bf c} := \otimes_{j=1}^n \widetilde{\psi}^{c_j}
$$
for ${\bf c}:=(c_1,\cdots, c_n) \in \{0,1\}^n$.
The following result is well known for orthonormal wavelets;
see, for example, \cite{codafe92} and \cite[Section 5.1]{wo97}.
However, it is straightforward to prove it for biorthogonal wavelets.
Moreover, it can be arranged so that the functions
$\psi^0, \psi^1, \widetilde{\psi}^0, \widetilde{\psi}^1$
have compact supports.

\begin{lemma}\label{l4.4}
Suppose that a biorthogonal system $\{\Psi^{\bf c},\widetilde\Psi^{\bf c}\}_{{\bf c}\in E}$
is given as above.
Then for every $f \in L^2(\rn)$,
\begin{eqnarray*}
&&f=
\sum\limits_{{\bf c}\in \{0,1\}^n}\sum_{k\in \Z^n}
\langle f,\widetilde{\Psi}^{\bf c}(\cdot -k)\rangle_{L^2(\rn)}
\Psi^{\bf c}(\cdot -k)\\
&&\hs\hs
+
\sum_{{\bf c}\in E}
\sum_{j = 0}^{\infty}
\sum_{k\in \Z^n}
\langle f,2^{jn/2}\widetilde{\Psi}^{\bf c}(2^j \cdot -k)\rangle_{L^2(\rn)}
2^{jn/2}\Psi^{\bf c}(2^j \cdot -k)
\end{eqnarray*}
with convergence in $L^2({\mathbb R}^n)$.
\end{lemma}
Notice that the above lemma covers the theory of wavelets
(see, for example, \cite{dau,hw,m,wo97} for the elementary facts)
in that this reduces to a theory of wavelets
when $\psi^0=\widetilde{\psi}^0$ and $\psi^1=\widetilde{\psi}^1$.
In
what follows we state conditions on the smoothness, the decay, and
the number of vanishing moments for the wavelets $\psi^1, \widetilde{\psi}^1$ and
the respective scaling functions $\psi^0, \widetilde{\psi}^0$ in order to make them
suitable for our function spaces.

Recall first that
$\az_1,\az_2,\az_3,\delta,\gamma,\tau$
are given in Definition \ref{d3.1}.
Suppose that the integers $K,L,N$ satisfy
\begin{equation}\label{4.103}
L>\alpha_3+\delta+n-1+\gamma-n\tau+\alpha_1,
\end{equation}
\begin{equation}\label{4.104}
N>L+\az_3+\dz+2n
\end{equation}
and
\begin{equation}\label{4.105}
K+1>\alpha_2+n\tau, \,
L+1>\alpha_1.
\end{equation}
Assume that the $C^K(\R)$-functions
$\psi^0,\psi^1$ satisfy that, for all $\alpha \in \Z_+$ with $\alpha \le K$,
\begin{equation}\label{4.106}
|\partial^{\alpha}\psi^0(t)|
+
|\partial^{\alpha}\psi^1(t)|
\lesssim
(1+|t|)^{-N}, \quad\quad t \in {\mathbb R},
\end{equation}
and that
\begin{equation}\label{4.107}
\int_{\rn} t^{\bz}\psi^1(t)\,dt
=0
\end{equation}
for all $\bz \in \Z_+$ with $\bz \le L$. Similarly, the integers
$\widetilde{K}, \widetilde{L}, \widetilde{N}$ are supposed to satisfy
\begin{equation}\label{4.108}
\widetilde{L}>\alpha_3+2\delta+n-1+\gamma+\max(n/2,(\az_2-\gamma)_+),
\end{equation}
\begin{equation}\label{4.109}
\widetilde{N}>\widetilde{L}+\az_3+2\dz+2n
\end{equation}
and
\begin{equation}\label{4.110}
\widetilde{K}+1>\az_1+\gamma, \,
\widetilde{L}+1>\max(n/2,(\az_2-\gamma)_+).
\end{equation}

Let now the $C^{\widetilde{K}}(\R)$-functions
$\widetilde{\psi}^0$ and $\widetilde{\psi}^1$
satisfy that, for all $\alpha \in \Z_+$ with $\alpha \le \widetilde{K}$,
\begin{equation}\label{4.111}
|\partial^{\alpha}\widetilde{\psi}^0(t)|+
|\partial^{\alpha}\widetilde{\psi}^1(t)|
\lesssim
(1+|t|)^{-\widetilde{N}},
\quad\quad t \in {\mathbb R}
\end{equation}
and that, for all $\bz \in \Z_+$ with $\bz \le \widetilde{L}$,
\begin{equation}\label{4.112}
\int_{\mathbb R} t^{\bz}\widetilde{\psi}^1(t)\,dt
=0.
\end{equation}

Assume, in addition, that
\begin{equation}\label{4.1130}
\widetilde{K}+1 \ge
\widetilde{L}>2a+n\tau, \,
\widetilde{N}>a+n.
\end{equation}

Observe that (\ref{4.106}) and (\ref{4.107})
correspond to the decay condition and the moment condition
of $\psi^0$ and $\psi^1$
in Definition \ref{d4.2}, respectively. Let us now define
the weight sequence
\begin{equation}\label{Wj}
  W_j(x)
:=
[w^*_j(x)]^{-1} \wedge 2^{jn/2}
\in
\cw^{\az_3+\delta}_{\max(n/2,(\az_2+\tau-\gamma)_+),(\az_1+\gamma-\tau)_+}\,,
\end{equation}
where $x\in\rn$ and $w_j^{\ast}$ is defined as in \eqref{wjstar}.

If $a\in(n+\alpha_3,\fz)$,
using Proposition \ref{p8.1} below,
which can be proved independently,
together with the translation invariance
of $L^\infty(\rn)$ and $L^1(\rn)$,
we have
\begin{equation}\label{equivalence}
\|f\|_{B^\rho_{\infty,\infty,a}(\rn)}
\sim
\sup_{j \in \Z_+}\|\rho_j (\varphi_j*f)\|_{L^\infty(\rn)}, \,
\|f\|_{B^\rho_{1,1,a}(\rn)}
\sim
\sum_{j=0}^\infty \|\rho_j (\varphi_j*f)\|_{L^1(\rn)}
\end{equation}
for all $f \in \cs'(\rn)$ and $\rho \in \cw_{\az_1,\az_2}^{\az_3}$.
See also \cite[Theorem 3.6]{LSUYY1} for a similar conclusion,
where the case when $\rho$ is independent of $j$
is treated.
Thus, if we assume that
\begin{equation}\label{4.11111}
a>n+\alpha_3+\delta,
\end{equation}
we then see that
$$
\|f\|_{B^{W^{-1}}_{\infty,\infty,a}(\rn)}
\sim
\sup_{j \in \Z_+}\|W_j{}^{-1} (\varphi_j*f)\|_{L^\infty(\rn)}, \,
\|f\|_{B^W_{1,1,a}(\rn)}
\sim
\sum_{j=0}^\infty \|W_j (\varphi_j*f)\|_{L^1(\rn)}.
$$
Observe that
(\ref{4.108}), \eqref{4.109} and (\ref{4.110})
guarantee that $B_{1,1,a}^{W}(\rn)$ has the atomic/molecular
characterizations; see Theorem \ref{t4.1} and the assumptions
\eqref{4.103}, \eqref{4.104} and \eqref{4.105}.
Indeed, in $A^{w,\tau}_{\cl,q,a}(\rn)$,
we need to choose
\[
A=B, \,
\cl(\rn)=L^1(\rn), \,
q=1, \,
w=W,
\tau=0,
\]
and hence,
we have to replace
$(\alpha_1,\alpha_2,\alpha_3)$
with
\[
(\max(n/2,(\az_2-\gamma)_+),\ \az_1+\gamma,\az_3+\delta)
\]
and $N_0$ should be bigger than $n$.
Therefore, \eqref{4.103}, \eqref{4.104} and \eqref{4.105}
become (\ref{4.108}), \eqref{4.109} and (\ref{4.110}), respectively.

In view of Propositions \ref{p3.3} and \ref{p3.4}, we
define,
for every ${\bf c}\in \{0,1\}^n$,
a sequence $\{\lambda^{\bf c}_{j,k}\}_{j\in \Z_+,k\in \Z^n}$
by
\begin{equation}\label{sequ}
 \lambda^{\bf c}_{j,k}:=\lambda^{\bf c}_{j,k}(f):= \langle f,2^{jn/2}
\widetilde{\Psi}^{\bf c}(2^j \cdot
-k)\rangle,\quad  j\in \Z_+,\ k\in \Z^n\,,
\end{equation}
for a fixed $f\in B^{W^{-1}}_{\infty,\infty}(\R^n)$.
In particular, when ${\bf c} = 0$, we let $\lambda^{\bf c}_{j,k}=0$ whenever $j\in\nn$.

It should be noticed
that $K$ and $\widetilde{K}$ can differ
as was the case with \cite{ru11}.

As can be seen from the textbook \cite{codafe92},
the existences of $\psi^0,\psi^1,\widetilde{\psi}^0,\widetilde{\psi}^1$
are guaranteed.
Indeed, we just construct
$\psi^0,\psi^1$
which are sufficiently smooth.
Accordingly,
we obtain
$\widetilde{\psi}^0,\widetilde{\psi}^1$
which are almost as smooth as $\psi^0,\psi^1$.
Finally,
we obtain $\{\Phi^{\bf c},\widetilde{\Phi}^{\bf c}\}_{{\bf c} \in E}$.

\begin{theorem}\label{t4.3} Let $\alpha_1, \alpha_2, \alpha_3,
\tau \in [0,\infty)$ and $q\in(0,\,\fz]$. Suppose that
$\mathcal{L}(\R^n)$ satisfies
$(\mathcal{L}1)$ through $(\mathcal{L}6)$, $w\in{\mathcal
W}^{\alpha_3}_{\alpha_1,\alpha_2}$ and $a\in
(N_0+\alpha_3,\infty)$, where $N_0$ is as in $(\mathcal{L}6)$. Choose
scaling functions $(\psi^0,\widetilde{\psi}^0) \in C^K(\R) \times
C^{\widetilde{K}}(\R)$ and associated wavelets $(\psi^1, \widetilde{\psi}^1)
\in C^{K}(\R) \times C^{\widetilde{K}}(\R)$ satisfying \eqref{4.106}, \eqref{4.107}
\eqref{4.111}, \eqref{4.112}, where $L,\widetilde{L}, N,\widetilde{N},K,\widetilde{K}
 \in \Z_+$ are chosen according to \eqref{4.103},
\eqref{4.104}, \eqref{4.105}, \eqref{4.108},
\eqref{4.109}, \eqref{4.110}, \eqref{4.1130} and \eqref{4.11111}. For
every $f \in B^{W^{-1}}_{\infty,\infty}(\R^n)$ and every ${\bf c} \in
\{0,1\}^n$, the sequences $\{\lambda^{\bf c}_{j,k}\}_{j\in \Z_+,k\in \Z^n}$ in \eqref{sequ} are well defined.

{\rm (i)} The sequences $\{\lambda^{\bf c}_{j,k}\}_{j\in
\Z_+,k\in \Z^n}$ belong to $a^{w,\tau}_{\mathcal{L},q,a}(\rn)$ for all ${\bf c} \in
\{0,1\}^n$ if and only if $f\in
A^{w,\tau}_{\mathcal{L},q,a}(\rn)$. Indeed, for all $f\in
B^{W^{-1}}_{\infty,\infty}(\rn)$, the following holds true:
\begin{eqnarray*}
&&\sum\limits_{{\bf c} \in \{0,1\}^n}
\|\{\delta_{j,0}\langle f,\widetilde{\Psi}^{\bf c}(\cdot -k)\rangle
\}_{j \in \Z_+, \,k \in \Z^n}\|_{a^{w,\tau}_{\cl,q,a}(\rn)}\\
&&\quad+
\sum_{{\bf c} \in E}
\|\{
\langle f,2^{jn/2}
\widetilde{\Psi}^{\bf c}(2^j \cdot -k)\rangle
\}_{j \in \Z_+, \, k \in \Z^n}\|_{a^{w,\tau}_{\cl,q,a}(\rn)}
\sim \|f\|_{A^{w,\tau}_{\cl,q,a}(\rn)}\,,
\end{eqnarray*}
where ``$\infty$'' is admitted in both sides.

{\rm (ii)}\,
If $f \in A^{w,\tau}_{\cl,q,a}(\rn)$, then
\begin{equation}\label{exp}
f(\cdot) =
\sum\limits_{{\bf c} \in \{0,1\}^n}\sum_{k\in \Z^n}
\lambda^{\bf c}_{0,k}
\Psi^{\bf c}(\cdot -k)
+
\sum_{{\bf c}\in E}
\sum_{j=0}^{\infty}
\sum_{k\in \Z^n}
\lambda^{\bf c}_{j,k}
2^{jn/2}\Psi^{\bf c}(2^j \cdot -k)
\end{equation}
in $\cs'(\rn)$. The equality \eqref{exp} holds true in
$A^{w,\tau}_{\cl,q,a}(\rn)$ if and only if the finite sequences are dense in
$a^{w,\tau}_{\cl,q,a}(\rn)$.
\end{theorem}

\begin{proof}
First, we show that if $f \in A^{w,\tau}_{\cl,q,a}(\rn)$,
then \eqref{exp} holds true in $\cs'(\rn)$.
By (\ref{equivalence}) and (\ref{Wj}),
together with Proposition \ref{p3.3},
the space $A^{w,\tau}_{\cl,q,a}(\rn)$
can be embedded into
$B^{W^{-1}}_{\infty,\infty}(\rn)$
which coincides with
$B^{W^{-1}}_{\infty,\infty,a}(\rn)$, when $a$ satisfies \eqref{4.11111}.
Fixing ${\bf c} \in \{0,1\}^n$ and letting
$\{\lambda^{\bf c}_{j,k}\}_{j\in \Z_+,k\in \Z^n}$ be as
in \eqref{sequ}, we define
\begin{equation}\label{exp2}
 f^{\bf c}(\cdot):= \sum_{k\in \Z^n}
\lambda^{\bf c}_{0,k}
\Psi^{\bf c}(\cdot -k)
+
\sum_{j=1}^{\infty}
\sum_{k\in \Z^n}
\lambda^{\bf c}_{j,k}
2^{jn/2}\Psi^{\bf c}(2^j \cdot -k)\,.
\end{equation}
Noticing that $\Psi^{\bf c}(2^j\cdot-k)$ is a molecule module
a multiplicative constant, by Lemma \ref{l4.2},
we know that $f^{\bf c}\in A^{w,\tau}_{\mathcal{L},q,a}(\rn)$ and
\begin{eqnarray*}
\|f^{\bf c}\|_{A^{w,\tau}_{\cl,q,a}(\rn)}
&&\ls\|\{\delta_{j,0}\lz^{\bf c}_{0,k}
\}_{j \in \Z_+, \,k \in \Z^n}\|_{a^{w,\tau}_{\cl,q,a}(\rn)}+
\|\{
\lz^{\bf c}_{j,k}
\}_{j \in \Z_+, \, k \in \Z^n}\|_{a^{w,\tau}_{\cl,q,a}(\rn)}\\
&&\sim\|\{\delta_{j,0}\langle f^{\bf c},\widetilde{\Psi}^{\bf c}(\cdot -k)\rangle
\}_{j \in \Z_+, \,k \in \Z^n}\|_{a^{w,\tau}_{\cl,q,a}(\rn)}\\
&&\hs+
\|\{
\langle f^{\bf c},2^{jn/2}
\widetilde{\Psi}^{\bf c}(2^j \cdot -k)\rangle
\}_{j \in \Z_+, \, k \in \Z^n}\|_{a^{w,\tau}_{\cl,q,a}(\rn)}
\,.
\end{eqnarray*}
Then we further see that $f^{\bf c} \in
B^{W^{-1}}_{\infty,\infty}(\rn)$.

We now show
that $f = \sum_{{\bf c} \in \{0,1\}^n} f^{\bf c} $.
Indeed,
for any
\begin{equation}\label{4.113}
F \in B^W_{1,1}({\mathbb R}^n)\
(\hookrightarrow
B^{n/2}_{1,1}({\mathbb R}^n)
\hookrightarrow
L^2({\mathbb R}^n)),
\end{equation}
if letting
$
\lambda^{\bf c}_{0,k}(F)
=
\langle F,\Psi^{\bf c}(\cdot-k) \rangle$
for all $k\in\zz^n$, and
$\lambda^{\bf c}_{j,k}(F)
=2^{jn/2}
\langle F,\Psi^{\bf c}(2^j\cdot-k) \rangle$
for all $j\in\zz_+$ and $k\in\zz^n$,
then by Theorem \ref{t4.1}, we conclude that
\begin{eqnarray}\label{eq:120209-1}
\sum_{{\bf c} \in E}
\|\{
\lambda^{\bf c}_{j,k}(F)
\}_{j \in \Z_+, \, k \in \Z^n}\|_{b^{W}_{1,1}(\rn)}
\lesssim \|F\|_{B^{W}_{1,1}(\rn)}.
\end{eqnarray}
From Lemma \ref{l4.4} and \eqref{4.113}, we deduce that the identity
\begin{equation}\label{exp3}
F(\cdot) =
\sum\limits_{{\bf c} \in \{0,1\}^n}\sum_{k\in \Z^n}
\lambda^{\bf c}_{0,k}(F)
\widetilde{\Psi}^{\bf c}(\cdot -k)
+
\sum_{{\bf c}\in E}
\sum_{j=1}^{\infty}
\sum_{k\in \Z^n}
\lambda^{\bf c}_{j,k}(F)
2^{jn/2}\widetilde{\Psi}^{\bf c}(2^j \cdot -k)
\end{equation}
holds true in $L^2(\R^n)$; moreover,
by virtue of (\ref{eq:120209-1}),
we also see that
(\ref{exp3})
holds true in the space $B^W_{1,1}(\rn)$.

Let $g := \sum_{{\bf c} \in
 \{0,1\}^n} f^{\bf c}$.
Then we see that $g \in
B^{W^{-1}}_{\infty,\infty}(\rn)$.
By Propositions \ref{p3.4}, together with \eqref{exp2} and \eqref{exp3}, we see
that $g(F) = f(F)$ for all $F\in B^W_{1,1}(\rn)$, which gives $g = f$ immediately.
Thus, \eqref{exp} holds true in $\cs'(\rn)$.

Thus, by Lemma \ref{l4.2} again, we obtain the ``$\gs$'' relation in (i).
Once we prove the ``$\ls$'' relation in (i),
then we immediately obtain the second conclusion in (ii),
that is, \eqref{exp} holds true in $A_{\cl, q,a}^{w,\tau}(\rn)$
if and only if the finite sequences are dense in $a_{\cl, q,a}^{w,\tau}(\rn)$.

So it remains to prove the ``$\lesssim$'' relation in (i) which
concludes the proof.
Returning to the definition of
the coupling
$\langle f,2^{jn/2}
\widetilde{\Psi}^{\bf c}(2^j \cdot -k)\rangle$
(see Proposition \ref{p3.4}),
we have
$$
\langle f,2^{jn/2}
\widetilde{\Psi}^{\bf c}(2^j \cdot -k)\rangle
=
2^{jn/2} \langle \Phi*f,\Phi*
\widetilde{\Psi}^{\bf c}(2^j \cdot -k)\rangle
+
\sum_{\ell=1}^\infty
2^{jn/2}\langle \varphi_{\ell}*f,\varphi_{\ell}*
\widetilde{\Psi}^{\bf c}(2^j \cdot -k) \rangle.
$$
In view of Lemma \ref{l2.4},
we see that, for all $j,\ \ell\in\zz_+$, $k\in\zz^n$ and $x\in\rn$,
\begin{eqnarray*}
|2^{jn}\varphi_{\ell}*\widetilde{\Psi}^{\bf c}(2^j x -k)|
\ls
2^{\min(j,\ell)n-|\ell-j|\widetilde{L}}
(1+2^{\min(j,\ell)}|x-2^{-j}k|)^{-\widetilde{N}}
\end{eqnarray*}
and hence, if $\widetilde{N}>a+n$ (see (\ref{4.1130})), by
the fact that $
2^l \le 2^{\min(j,l)+|j-l|}$,
we derive
\begin{eqnarray*}
&&2^{jn}|\langle \varphi_{\ell}*f,\varphi_{\ell}*
\widetilde{\Psi}^{\bf c}(2^j \cdot -k) \rangle|\\
&&\hs\ls 2^{\min(j,\ell)n-|\ell-j|\widetilde{L}}
\int_{\rn}\frac{|\varphi_{\ell}*f(x)|}
{(1+2^{\min(j,\ell)}|x-2^{-j}k|)^{\widetilde{N}}}\,dx\\
&&\hs\ls 2^{\min(j,\ell)n-|\ell-j|\widetilde{L}}
\sup_{y \in \rn}\frac{|\varphi_{\ell}*f(y)|}{(1+2^\ell|y-2^{-j}k|)^a}
\int_{\rn}\frac{(1+2^\ell|x-2^{-j}k|)^a}
{(1+2^{\min(j,\ell)}|x-2^{-j}k|)^{\widetilde{N}}}\,dx\\
&&\hs\ls 2^{-|\ell-j|(\widetilde{L}-a)}
\sup_{y \in \rn}\frac{|\varphi_{\ell}*f(y)|}{(1+2^\ell|y-2^{-j}k|)^a}
\end{eqnarray*}
with the implicit positive constant independent of $j,\ \ell,\ k$ and $f$.
A similar estimate holds true for $2^{jn/2} \langle \Phi*f,\Phi*
\widetilde{\Psi}^{\bf c}(2^j \cdot -k)\rangle$.
Consequently, by $(1+|y|)(1+|z|) \le(1+|y+z|)$ for all $y,\ z\in\rn$,
we see that, for all $x\in\rn$,
\begin{eqnarray*}
&&\sum_{k \in \Z^n}
\sum_{\ell=1}^\infty
2^{jn}|\langle \varphi_{\ell}*f,\varphi_{\ell}*
\widetilde{\Psi}^{\bf c}(2^j \cdot -k) \rangle|
\chi_{Q_{jk}}(x)\\
&&\hs\ls\sum_{k \in \Z^n}
\sum_{\ell=1}^\infty 2^{-|\ell-j|(\widetilde{L}-a)}
\sup_{y \in \rn}\frac{|\varphi_{\ell}*f(y)|}{(1+2^\ell|y-2^{-j}k|)^a}
\chi_{Q_{jk}}(x)\\
&&\hs\ls\sum_{k \in \Z^n}
\sum_{\ell=1}^\infty 2^{-|\ell-j|(\widetilde{L}-2a)}
\sup_{y \in \rn}\frac{|\varphi_{\ell}*f(y)|}{(1+2^\ell|y-x|)^a}
\chi_{Q_{jk}}(x)\\
&&\hs\ls
\sum_{\ell=1}^\infty 2^{-|\ell-j|(\widetilde{L}-2a)}
\sup_{y \in \rn}\frac{|\varphi_{\ell}*f(y)|}{(1+2^\ell|y-x|)^a},
\end{eqnarray*}
which, together with Lemma \ref{l2.3},
implies the ``$\lesssim$''-inequality in (i).
This finishes the proof of Theorem \ref{t4.3}.
\end{proof}

\begin{remark}\label{r4.13}
(i) As is the case with \cite{ru11},
bi-orthogonal systems in Theorem \ref{t4.3} can be replaced
by frames.

(ii) The wavelet characterizations for some special cases of the function spaces
in Theorem \ref{t4.3} are known; see, for example,
\cite{i11,is09,it,wo97}.
\end{remark}

\section{Pointwise multipliers and function spaces on domains}
\label{s4.5}

\subsection{Pointwise multipliers}

Let us recall that
$\cb^m(\rn):=
\cap_{\|\az\|_{1} \le m}
\left\{f \in C^m(\rn)\,:\
\partial^{\alpha}f \in L^\infty(\rn)\right\}$ for all $m\in\zz_+$.
As an application of the atomic decomposition
in the regular case,
we can establish the following result.

\begin{theorem}\label{t4.4}
Let $\alpha_1, \alpha_2, \alpha_3, \tau \in [0,\infty)$ and $q\in(0,\,\fz]$.
Suppose that $w\in\star-{\mathcal W}^{\alpha_3}_{\alpha_1,\alpha_2}$.
Assume, in addition, that {\rm \eqref{3.35}} holds true.
Then there exists $m_0 \in {\mathbb N}$
such that,
for all $m \in \cb^{m_0}(\rn)$,
the mapping $f \in \cs(\rn) \mapsto m f \in \cb^{m_0}(\rn)$
extends naturally to $A^{w,\tau}_{\cl,q,a}(\rn)$ so that it satisfies that
\begin{gather*}
\|m f\|_{B^{w,\tau}_{\cl,q,a}(\rn)}
\lesssim_m
\|f\|_{B^{w,\tau}_{\cl,q,a}(\rn)}
\quad (f \in B^{w,\tau}_{\cl,q,a}(\rn)),\\
\|m f\|_{F^{w,\tau}_{\cl,q,a}(\rn)}
\lesssim_m
\|f\|_{F^{w,\tau}_{\cl,q,a}(\rn)}
\quad (f \in F^{w,\tau}_{\cl,q,a}(\rn)),\\
\|m f\|_{\cn^{w,\tau}_{\cl,q,a}(\rn)}
\lesssim_m
\|f\|_{\cn^{w,\tau}_{\cl,q,a}(\rn)}
\quad (f \in \cn^{w,\tau}_{\cl,q,a}(\rn))
\end{gather*}
and
\begin{gather*}
\|m f\|_{\ce^{w,\tau}_{\cl,q,a}(\rn)}
\lesssim_m
\|f\|_{\ce^{w,\tau}_{\cl,q,a}(\rn)}
\quad (f \in \ce^{w,\tau}_{\cl,q,a}(\rn)).
\end{gather*}
\end{theorem}

\begin{proof}
Due to similarity,
we only deal with the case for $B^{w,\tau}_{\cl,q,a}(\rn)$.

Let $\az_1,\az_2$ and $\az_3$
fulfill (\ref{4.21}) and (\ref{4.22}).
We show the desired conclusion
by induction.
Let $m_0(w)$ be the smallest number such that
$w^* \in {\mathcal W}^{\az_3}_{\az_1\az_2}$,
where $w^*_\nu(x):=2^{m_0(w)\nu}w_\nu(x)$ for all $\nu\in\zz_+$
and $x\in\rn$.
If $m_0(w)$ can be taken $0$,
then we use Theorem \ref{t4.3} to find
that it suffices to define
\[
(m f)(\cdot):=
\sum\limits_{{\bf c} \in \{0,1\}^n}\sum_{k\in \Z^n}
\lambda^{\bf c}_{0,k}
m(\cdot)\Psi^{\bf c}(\cdot -k)
+
\sum_{{\bf c}\in E}
\sum_{j=0}^{\infty}
\sum_{k\in \Z^n}
\lambda^{\bf c}_{j,k}
m(\cdot)2^{jn/2}\Psi^{\bf c}(2^j \cdot -k),
\]
which, together with Theorem \ref{t4.2} and the fact
that $m(\cdot)2^{jn/2}\Psi^{\bf c}(2^j \cdot -k)$
is a molecule modulo a multiplicative constant, implies the desired
conclusion in this case.
Assume now that our theorem is true
for the class of weights $m_0(w) \in \{0,1,\cdots, N\}$, where $N\in\zz_+$.
For $m_0(w)=N+1$,
let us write
$f=(1-\Delta)^{-1}f-\sum_{j=1}^n \partial_j{}^2(1-\Delta)^{-1}f.$
Then we have
\begin{align*}
m f
&=
m (1-\Delta)^{-1}f-\sum_{j=1}^n
m \partial_j{}^2(1-\Delta)^{-1}f\\
&=
m (1-\Delta)^{-1}f
-\sum_{j=1}^n \partial_j\lf(m \partial_j(1-\Delta)^{-1}f\r)
+\sum_{j=1}^n (\partial_j m) \partial_j\lf((1-\Delta)^{-1}f\r).
\end{align*}
Notice that
$(1-\Delta)^{-1}f$
and
$\partial_j((1-\Delta)^{-1}f)$
belong to the space $B^{w^{**},\tau}_{\cl,q,a}(\rn)$,
where we write $w^{**}_\nu(x):=2^{\nu}w_\nu(x)$ for all $\nu\in\zz_+$
and $x\in\rn$. Notice that $m_0(w^{**})=m_0(w)-1$.
Consequently, by the induction assumption,
we have
\begin{eqnarray*}
\|m (1-\Delta)^{-1}f\|_{B^{w,\tau}_{\cl,q,a}(\rn)}
&&\le
\|m (1-\Delta)^{-1}f\|_{B^{w^{**},\tau}_{\cl,q,a}(\rn)}\\
&&\lesssim_m
\|(1-\Delta)^{-1}f\|_{B^{w^{**},\tau}_{\cl,q,a}(\rn)}
\lesssim_m
\|f\|_{B^{w,\tau}_{\cl,q,a}(\rn)}.
\end{eqnarray*}
Analogously, by Proposition \ref{p3.2} and Theorem \ref{t3.2}, we have
\begin{eqnarray*}
\|\partial_j\lf(m \partial_j(1-\Delta)^{-1}f\r)\|_{B^{w,\tau}_{\cl,q,a}(\rn)}
&&\lesssim
\|m \partial_j(1-\Delta)^{-1}f\|_{B^{w^{**},\tau}_{\cl,q,a}(\rn)}\\
&&\lesssim_m
\|\partial_j(1-\Delta)^{-1}f\|_{B^{w^{**},\tau}_{\cl,q,a}(\rn)}
\lesssim_m
\|f\|_{B^{w,\tau}_{\cl,q,a}(\rn)}
\end{eqnarray*}
and
\begin{eqnarray*}
\|(\partial_j m) \partial_j\lf((1-\Delta)^{-1}f\r)\|_{B^{w,\tau}_{\cl,q,a}(\rn)}
&&\le
\|(\partial_j m) \partial_j\lf((1-\Delta)^{-1}f\r)\|_{B^{w^{**},\tau}_{\cl,q,a}(\rn)}\\
&&\lesssim_m
\|\partial_j\lf((1-\Delta)^{-1}f\r)\|_{B^{w,\tau}_{\cl,q,a}(\rn)}
\lesssim_m\|f\|_{B^{w,\tau}_{\cl,q,a}(\rn)},
\end{eqnarray*}
which completes the proof of Theorem \ref{t4.4}.
\end{proof}

\subsection{Function spaces on domains}

In what follows, let  $\Omega$ be an open set of $\rn$,
$\cd(\boz)$ denote the {\it space of all
infinitely differentiable functions with compact support in $\boz$}
endowed with the inductive topology, and $\cd'(\boz)$ its {\it
topological dual} with the weak-$\ast$ topology which is called the
{\it space of distributions on $\boz$}.

Now we are oriented to defining
the spaces on $\Omega$.
Recall that a natural mapping
\[
f \in \cs'(\rn) \mapsto f|\Omega \in \cd'(\Omega)
\]
is well defined.

\begin{definition}\label{d4.5}
Let $s\in\rr$, $a\in (0,\fz)$,
$\alpha_1, \alpha_2, \alpha_3, \tau\in[0,\fz)$ and $q\in(0,\,\fz]$.
Let $w \in \cw_{\alpha_1,\alpha_2}^{\alpha_3}$.

{\rm (i)}
The \emph{space} $B^{w,\tau}_{\cl,q,a}(\Omega)$
is defined to be the set of all $f \in \cd'(\Omega)$
such that $f=g|\Omega$ for some $g \in B^{w,\tau}_{\cl,q,a}(\rn)$.
The \emph{norm} is given by
\begin{equation*}
\|f\|_{B^{w,\tau}_{\cl,q,a}(\Omega)}
:=
\inf\{\|g\|_{B^{w,\tau}_{\cl,q,a}(\rn)}\,:\,
g \in B^{w,\tau}_{\cl,q,a}(\rn), \, f=g|\Omega\}.
\end{equation*}%

{\rm (ii)}
The \emph{space} $F^{w,\tau}_{\cl,q,a}(\Omega)$
is defined to be the set of all $f \in \cd'(\Omega)$
such that $f=g|\Omega$ for some $g \in F^{w,\tau}_{\cl,q,a}(\rn)$.
The \emph{norm} is given by
\begin{equation*}
\|f\|_{F^{w,\tau}_{\cl,q,a}(\Omega)}
:=
\inf\{\|g\|_{F^{w,\tau}_{\cl,q,a}(\rn)}\,:\,
g \in F^{w,\tau}_{\cl,q,a}(\rn), \, f=g|\Omega\}.
\end{equation*}%

{\rm (iii)}
The \emph{space} $\cn^{w,\tau}_{\cl,q,a}(\Omega)$
is defined to be the set of all $f \in \cd'(\Omega)$
such that $f=g|\Omega$ for some $g \in \cn^{w,\tau}_{\cl,q,a}(\rn)$.
The \emph{norm} is given by
\begin{equation*}
\|f\|_{\cn^{w,\tau}_{\cl,q,a}(\Omega)}
:=
\inf\{\|g\|_{\cn^{w,\tau}_{\cl,q,a}(\rn)}\,:\,
g \in \cn^{w,\tau}_{\cl,q,a}(\rn), \, f=g|\Omega\}.
\end{equation*}%

{\rm (iv)}
The \emph{space} $\ce^{w,\tau}_{\cl,q,a}(\Omega)$
is defined to be the set of all $f \in \cd'(\Omega)$
such that $f=g|\Omega$ for some $g \in \ce^{w,\tau}_{\cl,q,a}(\rn)$.
The \emph{norm} is given by
\begin{equation*}
\|f\|_{\ce^{w,\tau}_{\cl,q,a}(\Omega)}
:=
\inf\{\|g\|_{\ce^{w,\tau}_{\cl,q,a}(\rn)}\,:\,
g \in \ce^{w,\tau}_{\cl,q,a}(\rn), \, f=g|\Omega\}.
\end{equation*}%
\end{definition}

A routine argument shows that
$B^{w,\tau}_{\cl,q,a}(\Omega)$,
$F^{w,\tau}_{\cl,q,a}(\Omega)$,
$\ce^{w,\tau}_{\cl,q,a}(\Omega)$ and
$\cn^{w,\tau}_{\cl,q,a}(\Omega)$
are all quasi-Banach spaces.

Here we are interested in bounded Lipschitz domains.
Let $\kz:\rr^{n-1} \to \R$ be a Lipschitz function.
Then define
$$\Omega_{\kz,+}:=
\{(x',x_{n})\in \R^{n}\,:\,x_{n}>\kz(x')\}$$
and
$$\Omega_{\kz,-}:=
\{(x',x_{n})\in \R^{n}\,:\,x_{n}<\kz(x')\}.$$
Let $\sigma \in S_{n}$ be a \emph{permutation}.
Then define
\begin{align*}
\Omega_{\kz,\pm;\sigma}&:=
\{(x',x_{n})\in \R^{n}\,:\,\sigma(x',x_{n})
\in \Omega_{\kz,\pm}\}.
\end{align*}
By a \emph{Lipschitz domain},
we mean an open set of the form
\[
\bigcup_{j=1}^J \sigma_j(\Omega_{f_j,+})
\cap
\bigcup_{i=1}^I \tau_i(\Omega_{g_i,-}),
\]
where
the functions $f_1,f_2,\cdots,f_J$ and $g_1,g_2,\cdots,g_I$
are all Lipschitz functions
and
the mappings $\sigma_1,\sigma_2,\cdots,\sigma_J$
and $\tau_1,\tau_2,\cdots,\tau_K$
belong to $S_{n}$.
With Theorem \ref{t4.4},
and a partition of unity,
without loss of generality, we may assume that
$\Omega:=\Omega_{\kz,\pm}$ for some \emph{Lipschitz function $\kz:\rn \to \R$}.
Furthermore, by symmetry,
we only need to deal with the case when $\Omega:=\Omega_{\kz,+}$.

To specify, we let $L$ be the \emph{positive Lipschitz constant} of $\kz$, namely,
the smallest number $L$ such that
for all $x',\ y'\in\rr^{n-1}$,
$|\kz(x')-\kz(y')| \le L|x'-y'|$.
Also, we let $K$ be the \emph{cone} given by
$$K:=\{(x',x_{n}) \in \R^{n}\,:\,L|x'|>-x_{n}\}.$$

We choose $\Psi \in \cd(\rn)$ so that
${\rm supp}\Psi \subset K$ and $\int_\rn \Psi(x)\,dx \ne 0$. Let
$$\Phi(x):=\Psi(x)-\Psi_{-1}(x)=\Psi(x)-2^{-n}\Psi(2^{-1}x)$$
for all $x\in\rn$.
Let $L \gg 1$ and choose
$\eta,\ \psi \in C^\infty_{\rm c}(K)$
so that
$\varphi:=\eta-\eta_{-1}$ satisfies
the moment condition of order $L$ and that
$
\psi*\Psi+\sum_{j=1}^\infty \varphi^j*\Phi^j=\delta
$
in $\cs'(\rn)$.
Define ${\mathcal M}^\Omega_{2^{-j},a}f(x)$, for all $j\in\zz_+$,
$f \in \cd'(\Omega)$ and $x\in\rn$, by
\begin{align*}
{\mathcal M}^\Omega_{2^{-j},a}f(x)
:=
\begin{cases}
\displaystyle
\sup_{y \in \Omega}\dfrac{|\Psi*f(y)|}{(1+|x-y|)^a},
&j=0;\\
\displaystyle
\sup_{y \in \Omega}\dfrac {|\Phi^j*f(y)|}{(1+2^j|x-y|)^a},
&j \in\nn
\end{cases}\\
=
\begin{cases}
\displaystyle
\sup_{y \in \Omega}
\dfrac{|\langle f,\Psi(y-\cdot) \rangle|}{(1+|x-y|)^a},
&j=0;\\
\displaystyle
\sup_{y \in \Omega}
\dfrac{|\langle f,\Phi^j(y-\cdot)\rangle|}{(1+2^j|x-y|)^a},
&j \in\nn.
\end{cases}
\end{align*}
Observe that this definition makes sense.
More precisely, the couplings
$\langle f,\Psi(y-\cdot) \rangle$
and
$\langle f,\Phi^j(y-\cdot) \rangle$
are well defined, because
$\Psi(y-\cdot)$
and
$\Phi^j(y-\cdot)$
have compact support and, moreover, are supported on $\Omega$
as the following calculation shows:
\[
{\rm supp}(\Psi(y-\cdot)), \,
{\rm supp}(\Phi^j(y-\cdot))
\subset
y-K
\subset
\{y+z\,:\,|z_n|>K|z'|\}
\subset
\Omega.
\]
Here we used the fact that $\Omega$
takes the form of $\Omega:=\Omega_{\kappa,+}$
to obtain the last inclusion.

In what follows, the mapping $(x',x_{n}) \mapsto (x',2\kz(x')-x_{n})
=:(y',y_n)$ is called to induce an \emph{isomorphism} of $\cl(\rn)$
with equivalent norms, if $f\in\cl(\rn)$ if and only if
$g_f(y',y_n):=f(x',2\kz(x')-x_{n})\in\cl(\rn)$
and, moreover, $\|f\|_{\cl(\rn)}\sim\|g_f\|_{\cl(\rn)}$.

Now we aim here to prove the following theorem.

\begin{theorem}
\label{t4.5}
Let $\Omega:=\Omega_{\kz,+}$ be as above
and assume that the reflection
$$
\iota:(x',x_{n}) \mapsto (x',2\kz(x')-x_{n})
$$
induces an isomorphism of $\cl(\rn)$ with equivalent norms.
Then

{\rm (i)} $f \in B^{w,\tau}_{\cl,q,a}(\Omega)$ if and only if
$f\in D'(\Omega)$ and
\begin{align*}
\lf\|\lf\{\chi_\Omega {\mathcal M}^\Omega_{2^{-j},a}f\r\}_{j\in\zz_+}
\r\|_{\ell^q(\cl^w_\tau(\rn,\zz_+))}<\fz;
\end{align*}
and there exists a positive constant $C$,
independent of $f$, such that
\begin{align}
\label{4.31}
C^{-1}\| f \|_{B^{w,\tau}_{\cl,q,a}(\Omega)}\le
\lf\|\lf\{\chi_\Omega {\mathcal M}^\Omega_{2^{-j},a}f\r\}_{j\in\zz_+}
\r\|_{\ell^q(\cl^w_\tau(\rn,\zz_+))}\le C\| f \|_{B^{w,\tau}_{\cl,q,a}(\Omega)};
\end{align}

{\rm (ii)}
$f \in F^{w,\tau}_{\cl,q,a}(\Omega)$ if and only if
$f\in D'(\Omega)$ and
\begin{align*}
\lf\|\lf\{\chi_\Omega {\mathcal M}^\Omega_{2^{-j},a}f\r\}_{j\in\zz_+}
\r\|_{\cl^w_\tau(\ell^q(\rn,\zz_+))}<\fz
\end{align*}
and there exists a positive constant $C$,
independent of $f$, such that
\begin{align*}
C^{-1}\| f \|_{F^{w,\tau}_{\cl,q,a}(\Omega)}
\le
\lf\|\lf\{\chi_\Omega {\mathcal M}^\Omega_{2^{-j},a}f\r\}_{j\in\zz_+}
\r\|_{\cl^w_\tau(\ell^q(\rn,\zz_+))}\le C^{-1}\| f \|_{F^{w,\tau}_{\cl,q,a}(\Omega)};
\end{align*}

{\rm (iii)} $f \in \cn^{w,\tau}_{\cl,q,a}(\Omega)$ if and only if
$f\in D'(\Omega)$ and
\begin{align*}
\lf\|\lf\{\chi_\Omega {\mathcal M}^\Omega_{2^{-j},a}f\r\}_{j\in\zz_+}
\r\|_{\ell^q(\cn\cl^w_\tau(\rn,\zz_+))}<\fz
\end{align*}
and there exists a positive constant $C$,
independent of $f$, such that
\begin{align*}
C^{-1}\| f \|_{\cn^{w,\tau}_{\cl,q,a}(\Omega)}
\le
\lf\|\lf\{\chi_\Omega {\mathcal M}^\Omega_{2^{-j},a}f\r\}_{j\in\zz_+}
\r\|_{\ell^q(\cn\cl^w_\tau(\rn,\zz_+))}\le
C\| f \|_{\cn^{w,\tau}_{\cl,q,a}(\Omega)};
\end{align*}

{\rm (iv)}  $f \in \ce^{w,\tau}_{\cl,q,a}(\Omega)$ if and only if
$f\in D'(\Omega)$ and
\begin{align*}
\lf\|\lf\{\chi_\Omega {\mathcal M}^\Omega_{2^{-j},a}f\r\}_{j\in\zz_+}
\r\|_{\ce\cl^w_\tau(\ell^q(\rn,\zz_+))}<\fz
\end{align*}
and there exists a positive constant $C$,
independent of $f$, such that
\begin{align*}
C^{-1}\| f \|_{\ce^{w,\tau}_{\cl,q,a}(\Omega)}\le
\lf\|\lf\{\chi_\Omega {\mathcal M}^\Omega_{2^{-j},a}f\r\}_{j\in\zz_+}
\r\|_{\ce\cl^w_\tau(\ell^q(\rn,\zz_+))}\le
C\| f \|_{\ce^{w,\tau}_{\cl,q,a}(\Omega)}.
\end{align*}
\end{theorem}

\begin{proof}
By similarity, we only give the proof of (i).
In any case the second inequality of \eqref{4.31}
follows from Corollary \ref{c3.1}.
Let us prove the first inequality of \eqref{4.31}.
Let $f \in B^{w,\tau}_{\cl,q,a}(\Omega)$.
Choose $G \in B^{w,\tau}_{\cl,q,a}(\rn)$
so that
\[
G_\Omega=f, \quad
\|f\|_{B^{w,\tau}_{\cl,q,a}(\Omega)}
\le
\|G\|_{B^{w,\tau}_{\cl,q,a}(\rn)}
\le
2
\|f\|_{B^{w,\tau}_{\cl,q,a}(\Omega)}.
\]
Define $F$ by
\begin{equation*}
F:=\psi*\Psi*G+
\sum_{j=1}^\infty \varphi^j*\Phi^j*G.
\end{equation*}

It is easy to see that $F|\Omega=f$ and $F \in \cs'(\rn)$,
since
$\psi*\Psi+\sum_{j=1}^\infty \varphi^j*\Phi^j=\delta$ in $\cs'(\rn)$.
Then
$\| f \|_{B^{w,\tau}_{\cl,q,a}(\Omega)}\ls\| F \|_{B^{w,\tau}_{\cl,q,a}(\R^n)}$.
To show the first inequality of \eqref{4.31},
it suffices to show that
$$\| F \|_{B^{w,\tau}_{\cl,q,a}(\R^n)}\ls\lf\|\lf\{\chi_\Omega {\mathcal M}^\Omega_{2^{-j},a}f\r\}_{j\in\zz_+}
\r\|_{\ell^q(\cl^w_\tau(\rn,\zz_+))}.$$
Since
$$\| F \|_{B^{w,\tau}_{\cl,q,a}(\R^n)}\ls \lf\|\lf\{{\mathcal M}^\Omega_{2^{-j},a}f\r\}_{j\in\zz_+}
\r\|_{\ell^q(\cl^w_\tau(\rn,\zz_+))},$$
we only need to prove that
$$ \lf\|\lf\{{\mathcal M}^\Omega_{2^{-j},a}f\r\}_{j\in\zz_+}
\r\|_{\ell^q(\cl^w_\tau(\rn,\zz_+))}\ls\lf\|\lf\{\chi_{\Omega}
{\mathcal M}^\Omega_{2^{-j},a}f\r\}_{j\in\zz_+}
\r\|_{\ell^q(\cl^w_\tau(\rn,\zz_+))}.$$
To see this, notice that if
$(x',x_n) \in \Omega$ and
$(y',y_n) \in \Omega$, since $\kz$ is a Lipschitz mapping,
we then conclude that
\begin{align*}
\lefteqn{
|x'-y'|^2+|y_n+x_n-2\kz(x')|^2
}\\
&\hs\sim
|x'-y'|^2
+
|y_n-\kz(y')+x_n-\kz(x')|^2\\
&\hs\sim
|x'-y'|^2+
|y_n-\kz(y')-x_n+\kz(x')|^2
+
|\kz(y')-\kz(x')|^2\\
&\hs\gtrsim
|x'-y'|^2+|y_n-x_n|^2\sim|x-y|^2.
\end{align*}
From this, together with the isomorphism property with equivalent norms
of the transform $(x',x_n) \in \R^n \setminus \Omega
\mapsto (x',2\kz(x')-x_n) \in \Omega$,
we deduce that
$$
\lf\|\lf\{\chi_{\rn\setminus\Omega}
{\mathcal M}^\Omega_{2^{-j},a}f\r\}_{j\in\zz_+}
\r\|_{\ell^q(\cl^w_\tau(\rn,\zz_+))}
\ls\lf\|\lf\{\chi_{\Omega}
{\mathcal M}^\Omega_{2^{-j},a}f\r\}_{j\in\zz_+}
\r\|_{\ell^q(\cl^w_\tau(\rn,\zz_+))},
$$
which further implies
the first inequality of \eqref{4.31}.
This finishes the proof of Theorem \ref{t4.5}.
\end{proof}

To conclude Section \ref{s4.5},
we present two examples concerning Theorem \ref{t4.5}.
\begin{example}\label{e4.1}
It is absolutely necessary to assume that
$(x',x_{n}) \mapsto (x',2\kz(x')-x_{n})$
induces an isomorphism of $\cl(\rn)$ with equivalent norms.
Here is a counterexample which shows this.

Let $n=1$, $\cl(\R):=L^1((1+t\chi_{(0,\infty)}(t))^{-N}\,dt)$
and $w_j(x):=1$ for all $x\in\rr$ and $j\in\zz_+$.
Consider the space $B^{0,0}_{\cl,\infty,2}((0,\infty))$,
whose notation is based on the convention \eqref{3.5}.
A passage to the higher dimensional case is readily done.
In this case the isomorphism is
$t \in \R \mapsto -t \in \R$.
Consider the corresponding maximal operators that for all
$f\in\cd'(0,\fz)$ and $t\in\rr$,
\[
{\mathcal M}^{(0,\infty)}_{1,2}f(t)
=
\sup_{s\in(0,\fz)}
\frac{|\psi*f(s)|}{(1+|t-s|)^2}
\]
and, for $j \in\nn$,
\[
{\mathcal M}^{(0,\infty)}_{2^{-j},2}f(t)
=
\sup_{s\in(0,\fz)}
\frac{|\vz_j*f(s)|}{(1+2^j|t-s|)^2},
\]
where $\psi$ and $\vz$ belong to $C^\infty_c((-2,-1))$
satisfying $\vz=\Delta^L \psi$,
and $\vz_j(t)=2^{j}\vz(2^jt)$ for all $t\in\rr$.
Let $f_0 \in C^\infty_{\rm c}((2,5))$
be a function such that
$\chi_{(3,4)} \le f_0 \le \chi_{(2,5)}$.
Set $f_a(t):=f_0(t-a)$ for all $t\in\rr$ and some $a \gg 1$.
Then, for all $t\in\rr$, we have
\[
{\mathcal M}^{(0,\infty)}_{1,2}f_a(t)
\sim
\frac{1}{(1+|t-a|)^2} \ {\rm and}\
{\mathcal M}^{(0,\infty)}_{2^{-j},2}f_a(t)
\sim
\frac{2^{-2jL}}{(1+|t-a|)^2}.
\]
Consequently, we see that
\[
\left\|\lf\{
\chi_{(0,\fz)}
{\mathcal M}^{(0,\fz)}_{2^{-j},2}f_a\r\}_{j\in\zz^+}
\right\|_{\ell^\fz (\cl(\rr,\zz_+))} \sim
\int_0^\infty \frac{1}{(1+t)^N(1+|t-a|)^2}\,dt.
\]
Let $\rho:\R \to \R$ be a smooth function
such that $\chi_{(8/5,\infty)} \le \rho \le \chi_{(3/2,\infty)}$.
If $f \in B^{0,0}_{\cl,\infty,2}(\rn)$ such that
$f|(0,\infty)=f_a$,
then
$
\|f\|_{B^{0,0}_{\cl,\infty,2}(\rn)}=\|\rho f\|_{B^{0,0}_{\cl,\infty,2}(\rn)}
\lesssim \|f\|_{B^{0,0}_{\cl,\infty,2}(\rn)}
$
by Theorem \ref{t4.4}.
Consequently we have
\begin{equation}\label{4.36}
\|f_a\|_{B^{0,0}_{\cl,\infty,2}(\Omega)}
\sim
\|f_0\|_{B^{0,0}_{\cl,\infty,2}(\rn)}
\sim
\frac{1}{a}.
\end{equation}
Meanwhile,
\begin{eqnarray*}
&&\left\|\lf\{
\chi_{(0,\fz)}
{\mathcal M}^{(0,\fz)}_{2^{-j},2}f_a\r\}_{j\in\zz^+}
\right\|_{\ell^\fz (\cl(\rr,\zz_+))}\\
&&\hs\sim
\int_0^\infty \frac{dt}{(1+t)^N(1+|t-a|)^2}\\
&&\hs\sim
\int_0^{a/2}\frac{dt}{(1+t)^N(1+|t-a|)^2}
+\int_{a/2}^\infty \frac{dt}{(1+t)^N(1+|t-a|)^2}\\
&&\hs\ls
\int_0^{a/2}\frac{dt}{(1+t)^N(1+|a|)^2}
+\int_{a/2}^\infty \frac{dt}{(1+a)^N(1+|t-a|)^2}
\ls \frac{1}{a^2}.
\end{eqnarray*}
In view of the above calculation and (\ref{4.36}),
we see that the conclusion \eqref{4.31} of Theorem \ref{t4.5}
fails unless we assume that
$(x',x_{n+1}) \mapsto (x',2\kz(x')-x_{n+1})$
induces an isomorphism of $\cl(\rn)$.
\end{example}

\begin{example}\label{e5.11}
As examples satisfying the assumption
of Theorem \ref{t4.5},
we can list weak-$L^p$ spaces,
Orlicz spaces and Morrey spaces.
For the detailed discussion of Orlicz spaces and Morrey spaces,
see Section \ref{s9}.
Here we content ourselves with giving the definition
of the norm and checking the assumption
of Theorem \ref{t4.5} for Orlicz spaces and Morrey spaces.

i)
By a \emph{Young function}
we mean a convex homeomorphism
$\Phi:[0,\infty) \to [0,\infty)$.

Given a Young function $\Phi$,
we define the \emph{Orlicz space }$L^{\Phi}(\rn)$
as the set of all measurable functions $f:\rn \to \C$
such that
$$\|f\|_{L^\Phi(\rn)}
:=\inf\left\{\lambda\in(0,\fz)\,:\
\int_{\rn}\Phi\!\left(\frac{|f(x)|}{\lambda}\right)\,dx\le1\right\}
<\infty.$$
Indeed, to check the assumption
of Theorem \ref{t4.5} for weak-$L^p$ spaces
and Orlicz spaces,
we just have to pay attention to the fact
that the Jacobian of the involution $\iota$ is $1$
and hence we can use the formula on the change of variables.

ii)
The \emph{Morrey norm} $\|\cdot\|_{{\mathcal M}^p_u(\rn)}$ is given by
\[
\| f \|_{{\mathcal M}^p_u(\rn)}
:= \sup_{x \in {\mathbb R}^n, \ r\in(0,\fz)}r^{\frac{n}{p}-\frac{n}{u}}
\left[\int_{B(x,r)}|f(y)|^u\,dy\right]^\frac{1}{u},
\]
where $B(x,r)$ denotes a ball centered at $x$ of radius $r\in(0,\fz)$
and $f$ is a measurable function.
Unlike Orlicz spaces,
for Morrey spaces,
we need more observation.
Since $\iota \circ \iota={\rm id}_{\rn}$,
we have only to prove that $\iota$ induces
a bounded mapping on Morrey spaces.
This can be showed as follows:
Let $B(x,r)$ be a ball.
Observe that $|x-y|<r$ implies
$|\iota(x)-\iota(y)|< Dr$,
since $\iota(x)=(x',2\kappa(x')-x_n)$
is a Lipschitz mapping with Lipschitz constant, say, $D$.
Therefore,
$\iota(B(x,r)) \subset B(\iota(x),Dr)$.
Hence we have
\begin{align*}
r^{\frac{n}{p}-\frac{n}{u}}
\left[\int_{B(x,r)}|f(\iota(y))|^u\,dy\right]^\frac{1}{u}
&=
r^{\frac{n}{p}-\frac{n}{u}}
\left[\int_{\iota(B(x,r))}|f(y)|^u\,dy\right]^\frac{1}{u}\\
&\le
r^{\frac{n}{p}-\frac{n}{u}}
\left[\int_{B(\iota(x),Dr)}|f(y)|^u\,dy\right]^\frac{1}{u}\\
&\le
D^{\frac{n}{u}-\frac{n}{p}}
\|f\|_{{\mathcal M}^p_u(\rn)},
\end{align*}
which implies  that $\iota$ induces a bounded mapping
on the Morrey space ${\mathcal M}^p_u(\rn)$
with norm less than or equal to $D^{\frac{n}{u}-\frac{n}{p}}$.
As a result, we see
that Morrey spaces satisfy the assumption
of Theorem \ref{t4.5}.
\end{example}

\section{Boundedness of operators}\label{s5}

Here,
as we announced in Section \ref{s1},
we discuss the boundedness of pseudo-differential operators.

\subsection{Boundedness of Fourier multipliers}
\label{s5.1}

We now refine Proposition \ref{p3.2}.
Throughout Section \ref{s5.1},
we
use a system $(\Phi,\varphi)$ of Schwartz functions satisfying
(\ref{1.1}) and (\ref{1.2}).

For $\ell\in\mathbb{N}$ and $\alpha\in\rr$,
$m\in C^{\ell}(\rr^n\backslash\{0\})$ is assumed to satisfy that,
for all $|\sigma|\leq \ell$,
\begin{equation}\label{fm1}
\sup_{R\in(1,\infty)}\left[R^{-n+2\alpha+2|\sigma|}\int_{R\leq|\xi|<2R}|
\partial_{\xi}^{\sigma}m(\xi)|^2\,d\xi\right]\leq A_{\sigma,1}<\infty
\end{equation}
and
\begin{equation}\label{fm2}
\int_{|\xi|<1}|
\partial_{\xi}^{\sigma}m(\xi)|^2\,d\xi\leq A_{\sigma,2}<\infty.
\end{equation}
The \emph{Fourier multiplier} $T_m$ is defined by setting,
for all $f\in\mathcal{S}(\rr^n)$,
$\widehat{(T_mf)}:=m\,\widehat{f}$.

\begin{lemma}\label{fm-l3.1}
Let $m$ be as in \eqref{fm1} and \eqref{fm2} and $K$ its inverse
Fourier transform. Then $K\in\cs'(\rn).$
\end{lemma}

\begin{proof}
Let $\vz\in\cs(\rn)$. Then
\begin{eqnarray*}
\langle K,\vz\rangle
&&=\int_\rn m(\xi)\wh\vz(\xi)\,d\xi
=\int_{|\xi|\geq1}m(\xi)\wh\vz(\xi)\,d\xi +
\int_{|\xi|<1}m(\xi)\wh\vz(\xi)=:{\rm I}_1+{\rm I}_2.
\end{eqnarray*}

Let $M=n-\az+1$. For ${\rm I}_1$, by the H\"{o}lder inequality and \eqref{fm1}, we see that
\begin{eqnarray*}
&&|{\rm I}_1|
\ls \sum_{k=0}^{\infty}\int_{2^k\leq|\xi|<2^{k+1}}|m(\xi)||\wh\vz(\xi)|\,d\xi\\
&&\hs\hs\ls\sum_{k=0}^{\infty}
\frac{\|(1+|x|)^{M}\wh\vz\|_{L^\fz(\rn)}}{(1+2^k)^M}
\int_{2^k\leq|\xi|<2^{k+1}}|m(\xi)|\,d\xi\\
&&\hs\hs\ls\sum_{k=0}^{\infty}
\frac{2^{nk/2}{\|(1+|x|)^{M}\wh\vz\|_{L^\fz(\rn)}}}{(1+2^k)^M}
\lf[\int_{2^k\leq|\xi|<2^{k+1}}|m(\xi)|^2\,d\xi\r]^{1/2}\\
&&\hs\hs\ls\sum_{k=0}^{\infty}
\frac{2^{k(n-\alpha)}{\|(1+|x|)^{M}\wh\vz\|_{L^\fz(\rn)}}}{(1+2^k)^M}
\ls{\|(1+|x|)^{M}\wh\vz\|_{L^\fz(\rn)}}.
\end{eqnarray*}

For ${\rm I}_2$,
by the H\"{o}lder inequality and \eqref{fm2}, we conclude that
\begin{eqnarray*}
|{\rm I}_2|
&&\ls\|\wh\vz\|_{L^\fz(\rn)}\lf[\int_{|\xi|<1}|m(\xi)|^2\,d\xi\r]^{1/2}
\ls\|\wh\vz\|_{L^\fz(\rn)}.
\end{eqnarray*}
This finishes
the proof of Lemma \ref{fm-l3.1}.
\end{proof}

The next lemma concerns
a piece of information
adapted to our new setting.

\begin{lemma}\label{fm-l3.2}
Let $\Psi,\,\psi$ be Schwartz functions on $\rr^n$ satisfying, respectively,
\eqref{1.1} and \eqref{1.2}.
Assume, in addition, that $m$ satisfies \eqref{fm1} and \eqref{fm2}.
If $a\in(0,\fz)$ and $\ell>a+n/2$, then there exists a
positive constant $C$ such that  for all $j\in\zz_+$,
$$\int_{\rr^n}\left(1+{2^j|z|}\right)^a|(K\ast\psi_j)(z)|
\, dz\leq C 2^{-j\alpha},$$
where $\psi_0=\Psi$ and $\psi_j(\cdot)=2^{-jn}\psi(2^j\cdot)$.
\end{lemma}

\begin{proof}
The proof for $j\in\nn$ is just \cite[Lemma 3.2(i)]{yyz12} with $t=2^{-j}$.
So we still need to prove the case when $j=0$.
Its proof is simple but for the sake of convenience for readers,
we supply the details.
When $j=0$, choose $\mu$ such that $\mu>n/2$ and $a+\mu\le \ell$.
From the H\"older inequality, the Plancherel theorem and \eqref{fm2}, we deduce that
\begin{eqnarray*}
&&\lf[\int_{\rr^n}\left(1+{|z|}\right)^a|(K\ast\Psi)(z)|\, dz\r]^2\\
&&\hs\ls\int_{\rr^n}\left(1+{|z|}\right)^{-2\mu}\, dz
\int_{\rr^n}\left(1+{|z|}\right)^{2(a+\mu)}|(K\ast\Psi)(z)|^2\, dz\\
&&\hs\ls\int_{\rr^n}\left(1+{|z|}\right)^{2\ell}|(K\ast\Psi)(z)|^2\, dz\\
&&\hs\ls\sum_{|\sigma|\le\ell}\int_{\rr^n}|z^{\sz}(K\ast\Psi)(z)|^2\, dz
\ls\sum_{|\sigma|\le\ell}\int_{|\xi|<2}|\pa_\xi^\sz [m(\xi)]|^2\, dz\ls1,
\end{eqnarray*}
which completes the proof of Lemma \ref{fm-l3.2}.
\end{proof}

Next we show that, via a suitable way, $T_m$ can also be defined
on the whole spaces ${F}^{w,\tau}_{\cl,q,a}(\rn)$
and ${B}^{w,\tau}_{\cl,q,a}(\rn)$.
Let $\Phi,\,\vz$ be Schwartz functions on $\rr^n$ satisfy, respectively,
\eqref{1.1} and \eqref{1.2}. Then there exist $\Phi^\dagger\in\cs(\rn)$,
satisfying \eqref{1.1}, and $\vz^\dagger\in\cs(\rn)$, satisfying \eqref{1.2},
such that
\begin{equation}\label{fm-3.1}
\Phi^\dagger*\Phi+\sum_{i=1}^\infty \vz^\dagger_i*\vz_i=\dz_0
\end{equation}
in $\cs'(\rn)$.
For any $f\in{F}^{w,\tau}_{\cl,q,a}(\rn)$ or ${B}^{w,\tau}_{\cl,q,a}(\rn)$,
we define a linear functional $T_mf$ on $\cs(\rn)$
by setting, for all $\phi\in\cs(\rn)$,
\begin{equation}\label{fm-3.2}
\langle T_mf,\phi\rangle:=
f*\Phi^\dagger*\Phi*\phi*K(0)
+\sum_{i\in\nn}f\ast\vz_i^\dagger\ast\vz_i\ast\phi\ast K(0)
\end{equation}
as long as the right-hand side converges.
In this sense, we say $T_mf\in\cs'(\rn)$.
The following result shows that the right-hand side of \eqref{fm-3.2}
converges and $T_mf$ in \eqref{fm-3.2}
is well defined.
Actually the right-hand side of (\ref{fm-3.2})
converges.

\begin{lemma}\label{fm-l3.3}
Let $\ell\in(n/2,\fz)$, $\az\in\rr$, $a\in (0,\fz)$,
$\alpha_1, \alpha_2, \alpha_3, \tau\in[0,\fz)$, $q\in(0,\,\fz]$,
$w \in \cw_{\alpha_1,\alpha_2}^{\alpha_3}$ and $f\in{F}^{w,\tau}_{\cl,q,a}(\rn)$
or ${B}^{w,\tau}_{\cl,q,a}(\rn)$. Then
the definition of $T_m f$ in \eqref{fm-3.2} is convergent
and
independent of the choices of the pair
$(\Phi^\dagger, \Phi, \vz^\dagger,\vz)$. Moreover, $T_mf\in \cs'(\rn)$.
\end{lemma}

\begin{proof} Due to similarity, we skip the proof for Besov spaces
$B^{w,\tau}_{\cl,q,a}(\rn)$.
Assume first that $f\in {F}^{w,\tau}_{\cl,q,a}(\rn)$.
Let $(\Psi^\dagger, \Psi, \psi^\dagger,\psi)$ be another pair of functions satisfying
\eqref{fm-3.1}.
Since $\phi\in\cs(\rn)$, by the Calder\'on reproducing formula,
we know that
$$\phi=\Psi^\dagger*\Psi*\phi+\sum_{j\in\nn}\psi_j^\dagger\ast\psi_j\ast\phi$$
in $\cs(\rn)$. Thus,
\begin{eqnarray*}
&&f*\Phi^\dagger*\Phi*\phi*K(0)+\sum_{i\in\nn}f\ast\vz_i^\dagger\ast\vz_i\ast\phi\ast K(0)\\
&&\hs=f*\Phi^\dagger*\Phi*\lf(\Psi^\dagger*\Psi*\phi+\sum_{j\in\nn}\psi_j^\dagger\ast\psi_j\ast\phi\r)*K(0)\\
&&\hs\hs+\sum_{i\in\nn}f\ast\vz_i^\dagger\ast\vz_i\ast
\lf(\Psi^\dagger*\Psi*\phi+\sum_{j\in\nn}\psi_j^\dagger\ast\psi_j\ast\phi\r)\ast K(0)\\
&&\hs=f*\Phi^\dagger*\Phi*\Psi^\dagger*\Psi*\phi*K(0)
+f*\Phi^\dagger*\Phi*\psi_1^\dagger\ast\psi_1\ast\phi*K(0)\\
&&\hs\hs+f\ast\vz_1^\dagger\ast\vz_1\ast\Psi^\dagger*\Psi*\phi\ast K(0)
+\sum_{i\in\nn}\sum_{j=i-1}^{i+1}
f\ast\vz_i^\dagger\ast\vz_i\ast\psi_j^\dagger\ast\psi_j\ast\phi\ast K(0),
\end{eqnarray*}
where the last equality follows from the fact that
$\vz_i\ast\psi_j=0$ if $|i-j|\geq2$.

Notice that
\begin{eqnarray*}
\left|\int_\rn f\ast \vz_i(y-z)\vz_i(-y)\,dy\right|\ls \sum_{k\in\zz^n}
\frac{2^{in}}{(1+|k|)^M}\int_{Q_{ik}}|\vz_i\ast f(y-z)|\,dy,
\end{eqnarray*}
where $M$ can be sufficiently large.
By $w \in \cw_{\alpha_1,\alpha_2}^{\alpha_3}$, we see that
\begin{eqnarray*}
\int_{Q_{ik}}|\vz_i\ast f(y-z)|\,dy\ls 2^{i(n-\az_1)}(1+2^i|z|)^{\az_3}
2^{-in\tau}\|f\|_{A^{w,\tau}_{\cl,q,a}(\rn)}.
\end{eqnarray*}
Thus, by Lemma \ref{fm-l3.2}, we conclude that
\begin{eqnarray*}
&&\sum_{i\in\nn}|f\ast\vz_i\ast\vz_i^\dagger\ast\psi_i\ast
\psi_i^\dagger\ast\phi\ast K(0)|\\
&&\hs=\sum_{i\in\nn}\int_\rn|f\ast\vz_i\ast\vz_i^\dagger(-z)
\psi_i\ast\psi_i^\dagger\ast\phi\ast K(z)|\,dz\\
&&\hs\ls \sum_{i\in\nn} 2^{i(n-\az_1-n\tau)}\|f\|_{A^{w,\tau}_{\cl,q,a}(\rn)}
\int_\rn\sum_{k\in\zz^n}\frac{2^{in}(1+2^i|z|)^{\az_3}}{(1+|k|)^M}|\psi_i\ast \psi_i\ast f(z)|\,dz\\
&&\hs\ls\sum_{i\in\nn} 2^{i(n-\az_1-n\tau)}2^{in}\|f\|_{A^{w,\tau}_{\cl,q,a}(\rn)}
\int_\rn(1+2^i|z|)^{\az_3}\int_\rn \frac{2^{-iM}}{(1+|y-z|)^M}|\psi_i\ast f(y)|\,dy\,dz\\
&&\hs\ls \sum_{i\in\nn} 2^{i(2n-\az_1-n\tau+\az_3-M)}\|f\|_{A^{w,\tau}_{\cl,q,a}(\rn)}
\int_\rn(1+2^i|y|)^{\az_3}|\psi_i\ast f(y)|\,dy\\
&&\hs\ls\sum_{i\in\nn} 2^{i(2n-\az_1-n\tau-M)}\|f\|_{A^{w,\tau}_{\cl,q,a}(\rn)}
\ls \|f\|_{A^{w,\tau}_{\cl,q,a}(\rn)},
\end{eqnarray*}
where $a$ is an arbitrary positive number.

By an argument similar to the above, we conclude that
\begin{eqnarray*}
  &&\left|f*\Phi^\dagger*\Phi*\lf(\Psi^\dagger*\Psi*\phi+
  \sum_{j\in\nn}\psi_j^\dagger\ast\psi_j\ast\phi\r)*K(0)\right|\\
&&\hs\hs+\left|\sum_{i\in\nn}f\ast\vz_i^\dagger\ast\vz_i\ast
\lf(\Psi^\dagger*\Psi*\phi+\sum_{j\in\nn}\psi_j^\dagger\ast\psi_j
\ast\phi\r)\ast K(0)\right|<\fz,
\end{eqnarray*}
which, together with the Calder\'on reproducing formula, further induces that
\begin{eqnarray*}
&&f*\Phi^\dagger*\Phi*\phi*K(0)+\sum_{i\in\nn}f\ast\vz_i^\dagger\ast\vz_i\ast\phi\ast K(0)\\
&&\hs=f*\Phi^\dagger*\Phi*\lf(\Psi^\dagger*\Psi*\phi+\sum_{j\in\nn}\psi_j^\dagger\ast\psi_j\ast\phi\r)*K(0)\\
&&\hs\hs+\sum_{i\in\nn}f\ast\vz_i^\dagger\ast\vz_i\ast
\lf(\Psi^\dagger*\Psi*\phi+\sum_{j\in\nn}\psi_j^\dagger\ast\psi_j\ast\phi\r)\ast K(0)\\
&&\hs=f*\Psi^\dagger*\Psi*\Psi*K(0)+\sum_{i\in\nn}f\ast\vz_i^\dagger\ast\vz_i\ast\Psi\ast K(0).
\end{eqnarray*}
Thus, $T_mf$ in \eqref{fm-3.2} is independent of the choices of the pair $(\Phi^\dagger,\Phi,\vz^\dagger,\vz)$.
Moreover, the previous argument also implies that $T_mf\in \cs'(\rn)$,
which completes the proof of Lemma \ref{fm-l3.3}.
\end{proof}

Then, by Lemma \ref{fm-l3.2}, we immediately have the following conclusion
and we omit the details here.

\begin{lemma}\label{fm-l3.4}
Let $\alpha\in\rr$, $a\in(0,\fz)$, $\ell\in\nn$,
$\Phi,\,\Psi\in\cs(\rn)$ satisfying \eqref{1.1} and $\vz,\psi\in\cs(\rn)$ satisfying \eqref{1.2}.
Assume that $m$
satisfies \eqref{fm1} and \eqref{fm2} and $f\in \cs'(\rn)$ such that $T_mf\in \cs'(\rn)$.
If $\ell>a+n/2$,
then there exists a positive constants $C$ such that,
for all $x,\,y\in\rr^n$ and $j\in\zz_+$,
$$|(T_mf\ast\psi_j)(y)|\leq C 2^{-j\alpha}\left(1+{2^j|x-y|}\right)^a(\vz_j^*f)_a(x).$$
\end{lemma}

Now we are ready to prove the following conclusion.

\begin{theorem}\label{f-m}
Let $\az\in\rr$, $a\in (0,\fz)$,
$\alpha_1, \alpha_2, \alpha_3, \tau\in[0,\fz)$, $q\in(0,\,\fz]$,
$w \in \cw_{\alpha_1,\alpha_2}^{\alpha_3}$ and $\wz w(x,2^{-j})=2^{j\az} w(x,2^{-j})$
for all $x\in\rn$ and $j\in\zz_+$.
Suppose that $m$ satisfies \eqref{fm1} and \eqref{fm2} with $\ell\in\mathbb{N}$ and $\ell>a+n/2$,
then there exists a positive constant $C_1$ such that, for all
$ f\in {F}^{w,\tau}_{\cl,q,a}(\rr^n)$,
$\|T_mf\|_{{F}^{\wz w,\tau}_{\cl, q,a}(\rr^n)}\leq
C_1\|f\|_{{F}^{w,\tau}_{\cl, q,a}(\rr^n)}$
and a positive constant $C_2$ such that, for all
$f\in {B}^{\wz w,\tau}_{\cl, q,a}(\rr^n)$,
$\|T_mf\|_{{B}^{\wz w,\tau}_{\cl, q,a}(\rr^n)}\leq
C_2\|f\|_{{B}^{w,\tau}_{\cl, q,a}(\rr^n)}.$
Similar assertions hold true
for $\ce^{w,\tau}_{\cl, q,a}(\rr^n)$
and $\cn^{w,\tau}_{\cl, q,a}(\rr^n)$.
\end{theorem}

\begin{proof}
By Lemma \ref{fm-l3.4} we conclude that,
if $\ell>a+n/2$, then for all $x\in\rn$ and $j\in\zz_+$,
$$2^{j\alpha}\left(\psi_j^{\ast}(T_mf)\right)_{a}(x)\lesssim
(\vz_j^*f)_a(x).$$
Then by the definitions of ${F}^{w,\tau}_{\cl,q,a}(\rn)$
and ${B}^{w,\tau}_{\cl,q,a}(\rn)$, we immediately conclude the desired
conclusions, which completes the proof of Theorem \ref{f-m}.
\end{proof}

\subsection{Boundedness of pseudo-differential operators}
\label{s5.2}

We consider the \emph{class} $S^0_{1,\mu}(\rn)$ with $\mu\in [0,1)$.
Recall that a function $a$ is said to belong
to a \emph{class $S^m_{1,\mu}(\rn)$}
of $C^\infty(\R^n_x\times\R^n_\xi)$-functions if
$$\sup_{x,\xi \in \rn}
(1+|\xi|)^{-m-\|{\vec{\alpha}}\|_1-\mu\|{\vec{\beta}}\|_1}
|\partial^{\vec{\beta}}_x\partial^{\vec{\alpha}}_\xi a(x,\xi)|
\lesssim_{\vec{\alpha},\vec{\beta}}
1
$$
for all multiindices ${\vec{\alpha}}$ and ${\vec{\beta}}$.
One defines, for all $x\in\rn$,
$$
a(X,D)(f)(x):= \int_\rn a(x,\xi)\hat{f}(\xi)
e^{ix\cdot\xi}\,d\xi$$
originally on $\cs(\rn)$, and further on $\cs'(\rn)$ via dual.

We aim here to establish the following in this subsection.
\begin{theorem}\label{t5.1}
Let
$w \in {\mathcal W}^{\alpha_3}_{\alpha_1,\alpha_2}$
with
$\alpha_1, \alpha_2, \alpha_3 \in [0,\infty)$
and a quasi-normed function space $\cl(\rn)$
satisfy $(\cl1)$ through $(\cl6)$.
Let $\mu\in[0,1)$,
$\tau\in(0,\fz)$ and $q\in (0,\infty]$.
Assume, in addition, that {\rm \eqref{3.35}} holds true, that is,
$a\in (N_0+\alpha_3,\infty)$,
where $N_0$ is as in $(\cl6)$.
Then the pseudo-differential operators
with symbol $S_{1,\mu}^0(\rn)$
are bounded on $A^{w,\tau}_{\cl,q,a}(\rn)$.
\end{theorem}

With the following decomposition,
we have only to consider the boundedness
of $a(\cdot,\cdot) \in S^{-M_0}_{1,\mu}(\rn)$ with an integer $M_0$ sufficiently large.
\begin{lemma}[\cite{st93}]\label{l5.1}
Let $\mu\in[0,1), \, a \in S_{1,\mu}^m(\rn)$ and $N \in {\mathbb N}$.
Then there exists a symbol
$b \in S^m_{1,\mu}(\rn)$ such that
\[
a(X,D)=(1+\Delta^{2N}) \circ b(X,D) \circ (1+\Delta^{2N})^{-1}.
\]
\end{lemma}

Based upon Lemma \ref{l5.1},
we plan to treat
\[
A(X,D) := b(X,D) \circ (1+\Delta^{2N})^{-1} \in S^{-2N}_{1,\mu}(\rn)
\]
and
\[
B(X,D) :=
\Delta^{2N} \circ b(X,D) \circ (1+\Delta^{2N})^{-1}
\in S^0_{1,\mu}(\rn),
\]
respectively.

The following is one of the key observations
in this subsection.
\begin{lemma}\label{l5.2}
Let $\mu\in[0,1)$, $w$, $q$, $\tau$, $a$ and $\cl$ be as in Theorem \ref{t5.1}.
Assume that $a \in S_{1,\mu}^0(\rn)$
satisfies that $a(\cdot,\xi)=0$ if $|\xi| \ge 1/2$.
Then
$a(X,D)$ is bounded on $A^{w,\tau}_{{\mathcal L},q,a}(\rn)$.
\end{lemma}

\begin{proof}
We fix $\Phi\in\cs(\rn)$ so that
$\hat{\Phi}(\xi)=1$ whenever $|\xi| \le 1$
and that
$\hat{\Phi}(\xi)=0$ whenever $|\xi| \ge 2$.
Then, by the fact that
$a(\cdot,\xi)=0$ if $|\xi| \ge 1/2$, we know that,
for all $f \in A^{w,\tau}_{{\mathcal L},q,a}(\rn)$,
$a(X,D)f=a(X,D)(\Phi\ast f).$ By this and that
the mapping $f \in A^{w,\tau}_{{\mathcal L},q,a}(\rn)
\mapsto \Phi*f \in A^{w,\tau}_{{\mathcal L},q,a}(\rn)$
is continuous, without loss of generality,
we may assume that the frequency support
of $f$ is contained in $\{|\xi| \le 2\}$.
Let $j \in \zz_+$ and $z \in \rn$ be fixed.
Then we have, for all $x\in\rn$,
\begin{eqnarray*}
\varphi_j*[a(X,D)f](x)&&=\int_{\rn}
\varphi_j(x-y)
\left[\int_{\rn}a(y,\xi)\widehat{f}(\xi)e^{i\xi y}\,d\xi
\right]\,dy\\
&&=\int_{\rn}\left[
\int_{\rn}\varphi_j(x-y)a(y,\xi)e^{i\xi y}\,dy
\right]\widehat{f}(\xi)\,d\xi\\
&&=\int_{\rn}\left(
\int_{\rn}\varphi_j(x-y)a(y,\cdot)e^{i\cdot y}\,dy
\right)^\land(z)f(z)\,dz
\end{eqnarray*}
by the Fubini theorem.
Notice that, again by virtue of the Fubini theorem,
we see that
\begin{eqnarray*}
\left(
\int_{\rn}\varphi_j(x-y)a(y,\cdot)e^{i\cdot y}\,dy
\right)^\land(z)
&&=
\int_{\rn}e^{-iz\xi}\left[
\int_{\rn}\varphi_j(x-y)a(y,\xi)e^{i\xi y}\,dy
\right]\,d\xi\\
&&=
\int_{\rn}\varphi_j(x-y)
\left[
\int_{\rn}a(y,\xi)e^{i\xi(y-z)}\,d\xi
\right]\,dy.
\end{eqnarray*}
Let us set
$\tau_j := (4^{-j}\Delta)^{-L}\varphi_j$
with $L \in \N$ large enough, say
\begin{equation*}
L=\lfloor a+n+\alpha_1+\alpha_2+1\rfloor.
\end{equation*}%
Then we have $\tau_j \in \cs(\rn)$ and
$\varphi_j(x)=2^{-2jL}\Delta^L\tau_j(x)$
for all $j\in\zz_+$ and $x\in\rn$.
Consequently, we have
\begin{eqnarray*}
&&\left[
\int_{\rn}\varphi_j(x-y)a(y,\cdot)e^{i\cdot y}\,dy
\right]^\land(z)\\
&&\quad=2^{-2jL}
\int_{\rn}\tau_j(x-y)
\Delta^L_y\left[
\int_{\rn}a(y,\xi)e^{i\xi(y-z)}\,d\xi
\right]\,dy
\end{eqnarray*}
by integration by parts.

Notice that by the integration by parts, we conclude that
\begin{eqnarray*}
&&\Delta^L_y\left(
\int_{\rn}a(y,\xi)e^{i\xi(y-z)}\,d\xi
\right)\noz\\
&&\hs=\sum_{\|\vec{\alpha}_1\|_1+\|\vec{\alpha}_2\|_1=2L}
\int_\rn
\left[\xi^{\vec\az_2}\partial^{\vec{\alpha}_1}_ya(y,\xi)
\right]e^{i\xi(y-z)}
\,d\xi\noz\\
&&\hs=\frac1{(1+|y-z|^2)^L}
\sum_{\|\vec{\alpha}_1\|_1+\|\vec{\alpha}_2\|_1=2L}
\int_\rn
(1-\Delta_\xi)^L\left[\xi^{\vec\az_2}\partial^{\vec{\alpha}_1}_ya(y,\xi)
\right]e^{i\xi(y-z)}
\,d\xi.
\end{eqnarray*}
Then, by the fact that $a\in S_{1,\mu}^0$ and $a(\cdot,\xi)=0$ if
$|\xi|\ge1/2$, we see that, for all $\xi,y\in\rn$,
\begin{gather}
\label{5.3}
\lf|(1-\Delta_\xi)^L\left(\xi^{\vec\az_2}\partial^{\vec{\alpha}_1}_ya(y,\xi)
\right)\r|\ls\chi_{B(0,2)}(\xi),
\end{gather}
and hence, for all $y,z\in\rn$,
\begin{eqnarray*}
&&\lf|\Delta^L_y\left(
\int_{\rn}a(y,\xi)e^{i\xi(y-z)}\,d\xi
\right)\r|\ls
\frac1{(1+|y-z|^2)^L}.
\end{eqnarray*}
Consequently, for all $j\in\zz_+$ and $x,y,z\in\rn$,
we have
\begin{equation}\label{5.5}
\left|
\left[
\int_{\rn}\varphi_j(x-y)a(y,\cdot)e^{i\cdot y}\,dy
\right]^\land(z)\right|
\lesssim 2^{-2jL}
\int_{\rn}\frac{|\tau_j(x-y)|}{(1+|y-z|^2)^L}\,dy
\end{equation}
and hence
\begin{eqnarray*}
\frac{|\varphi_j*(a(X,D)f)(x+z)|}{(1+2^j|z|)^a}
&&\lesssim
\int_{\rn}\int_{\rn}\frac{2^{-2jL}|\tau_j(x+z-y)|}
{(1+2^j|z|)^a(1+|y-w|^2)^L}
|f(w)|\,dy\,dw\\
&&\lesssim
\int_{\rn}\int_{\rn}\frac{2^{-2jL}|\tau_j(x+z-y)|}
{(1+|z|)^a(1+|y-w|^2)^L}
|f(w)|\,dy\,dw\\
&&\lesssim
2^{-2jL}\sup_{w \in \rn}\frac{|f(x+w)|}{(1+|w|)^a}.
\end{eqnarray*}
A similar argument also works for $\Phi*(a(X,D)f)$
(without using integration by parts)
and we obtain
\begin{eqnarray*}
\frac{|\Phi*(a(X,D)f)(x+z)|}{(1+|z|)^a}
\lesssim
\sup_{w \in \rn}\frac{|f(x+w)|}{(1+|w|)^a}.
\end{eqnarray*}
With this pointwise estimate, the condition of $L$
and the assumption that $\mu<1$,
we obtain the desired result,
which completes the proof of Lemma \ref{l5.2}.
\end{proof}

If we reexamine the above calculation,
then we obtain the following:

\begin{lemma}\label{l5.3}
Assume that $\mu\in[0,1)$
and that $a \in S_{1,\mu}^{-2M_0}(\rn)$
satisfies $a(\cdot,\xi)=0$ if $2^{k-2} \le |\xi| \le 2^{k+2}$.
Then
$a(X,D)$ is bounded on $A^{w,\tau}_{{\mathcal L},q,a}(\rn)$.
Moreover, there exist a positive constant $E$ and a positive
constant $C(E)$, depending on $E$, such that
the operator norm satisfies that
\[
\|a(X,D)\|_{A^{w,\tau}_{{\mathcal L},q,a}(\rn) \to A^{w,\tau}_{{\mathcal L},q,a}(\rn)}
\le C(E)2^{-Ek},
\]
provided $M_0\in(1,\fz)$ is large enough.
\end{lemma}

\begin{proof}
Let us suppose that $M_0>2L+n$, where $L \in \N$ is chosen
so that
\begin{equation}\label{5.6}
L=\lfloor a+n+\alpha_1+\alpha_2+n\tau+1\rfloor.
\end{equation}
Notice that in this time
$a(X,D)f=a(X,D)( \sum_{i=k-3}^{k+3}\vz_i\ast f)$ for all
$f\in A^{w,\tau}_{{\mathcal L},q,a}(\rn)$.
If we go through a similar argument as we did for \eqref{5.5}
with the condition on $L$
replaced by \eqref{5.6},
we see that, for all $j\in\zz_+$ and $x,z\in\rn$,
\begin{equation}\label{5.7}
\left|\left[
\int_{\rn}\varphi_j(x-y)a(y,\cdot)e^{i\cdot y}\,dy
\right]^\land(z)\right|
\lesssim 2^{-2jL+k(4L-2M_0+n)}
\int_{\rn}\frac{|\tau_j(x-y)|}{(1+|y-z|^2)^L}\,dy.
\end{equation}
Indeed, we just need to replace \eqref{5.3} in the proof of
\eqref{5.5} by the following estimate,
for all $k\in\zz_+$, $\xi,y\in\rn$ and multi-indices $\az$, $\bz$
such that $\|\az\|_1+\|\bz\|_1=2L$,
\begin{gather*}
\lf|(1-\Delta_\xi)^L\left(\xi^{\az}\partial^{\bz}_ya(y,\xi)
\right)\r|\ls2^{2k(2L-M_0)}\chi_{B(0,2^{k+2})\setminus B(0,2^{k-2})}(\xi).
\end{gather*}
By \eqref{5.7},
we conclude that, for all $j\in\zz_+$ and $x,z\in\rn$,
\begin{eqnarray*}
&&\frac{|\varphi_j*(a(X,D)f)(x+z)|}{(1+2^j|z|)^a}\\
&&\hs\lesssim
\int_{\rn}\int_{\rn}
\frac{2^{-2jL+k(4L-2M_0+n)}|\tau_j(x+z-y)|}
{(1+2^j|z|)^a(1+|y-w|^2)^L}
\sum_{l=-3}^3|\varphi_{k+l}*f(w)|\,dy\,dw\\
&&\hs\lesssim
2^{-2jL+k(4L-2M_0+a+n)}\sum_{l=-3}^3
\sup_{w \in \rn}\frac{|\varphi_{k+l}*f(x+w)|}{(1+2^{k+l}|w|)^a}.
\end{eqnarray*}
Consequently,
we see that
\begin{equation}\label{5.11}
\frac{|\varphi_j(D)(a(X,D)f)(x+z)|}{(1+2^j|z|)^a}
\lesssim
2^{-2jL+k(4L-2M_0+a+n)}
\sup_{\substack{w \in \rn \\ l \in [-3,3] \cap \zz}}\frac{|\varphi_{k+l}(D)f(x+w)|}{(1+2^{k+l}|w|)^a}.
\end{equation}
Combining the estimate \eqref{5.11} and Lemma \ref{l2.3} then
induce the desired result, and hence completes the proof of Lemma \ref{l5.3}.
\end{proof}

In view of the atomic decomposition,
we have the following conclusion.

\begin{lemma}\label{l5.4}
Let $w$ be as in Theorem \ref{t5.1}.
Assume that $a \in S_{1,\mu}^0(\rn)$
can be expressed as
$a(X,D)=\Delta^{2M_0} \circ b(X,D)$
for some $b \in S_{1,\mu}^{-2M_0}(\rn)$.
Then
$a(X,D)$ is bounded on $A^{w,\tau}_{{\mathcal L},q,a}(\rn)$,
as long as $M_0$ is large.
\end{lemma}

\begin{proof}
For any $f \in A^{w,\tau}_{{\mathcal L},q,a}(\rn)$, by Theorem \ref{t4.1},
there exist a collection $\{{\mathfrak A}_{jk}\}_{j \in \zz_+, \, k \in \Z^n}$
of atoms
and a complex sequence
$\{\lambda_{jk}\}_{j \in \zz_+, \, k \in \Z^n}$
such that
$f=\sum_{j=0}^\infty
\sum_{k \in \Z^n}\lambda_{jk}{\mathfrak A}_{jk}$
in ${\mathcal S}'(\rn)$
and that
$\|\{\lambda_{jk}\}_{j \in \zz_+, \, k \in \Z^n}
\|_{a^{w,\tau}_{{\mathcal L},q,a}(\rn)}
\lesssim
\|f\|_{A^{w,\tau}_{{\mathcal L},q,a}(\rn)}.$
In the course of the proof of \cite[Theorem 3.1]{s09},
we have shown that atoms
$\{{\mathfrak A}_{jk}\}_{j \in \zz_+, \, k \in \Z^n}$
are transformed into molecules
$\{a(X,D){\mathfrak A}_{jk}\}_{j \in \zz_+, \, k \in \Z^n}$
satisfying the decay condition.
However, if $a(X,D)=\Delta^{2M_0} \circ b(X,D)$,
then atoms are transformed into molecules
with moment condition of order $2M_0$.
Therefore, via Theorem \ref{t4.1} by letting $L=2M_0$
then completes the proof of Lemma \ref{l5.4}.
\end{proof}

With Lemmas \ref{l5.2} through \ref{l5.4} in mind,
we prove Theorem \ref{t5.1}.

\begin{proof}[Proof of Theorem \ref{t5.1}]
We decompose $a(X,D)$ according to Lemma \ref{l5.1}.
We fix an integer $M_0$ large enough as in Lemmas \ref{l5.3} and \ref{l5.4}.
Let us write $A(X,D):= a(X,D) \circ (1+\Delta^{2M_0})^{-1}$
and $B(X,D):= \Delta^{2M_0} \circ a(X,D) \circ (1+\Delta^{2M_0})^{-1}$.

Let $\Phi$ and $\vz$ be as in \eqref{1.1} and \eqref{1.2}
satisfying that $\widehat\Phi(\xi)+\sum_{j\in\nn}
\widehat\vz(2^{-j}\xi)=1$ for all $\xi\in\rn$. Then by the Calder\'on
reproducing formula, we know that, for all $f\in
A^{w,\tau}_{{\mathcal L},q,a}(\rn)$, $f=\Phi\ast f+\sum_{j\in\nn}\vz_j\ast f$
in $\cs'(\rn)$. Therefore,
we see that
\begin{eqnarray*}
a(X,D)f(x)
&&=\sum_{j=0}^\fz a(X,D)(\vz_j\ast f)(x)\\
&&=\sum_{j=0}^\fz \int_\rn a(x,\xi)\widehat\vz(2^{-j}\xi) \widehat{f}(\xi)
e^{ix\xi}\,d\xi\\
&&=:\sum_{j=0}^\fz \int_\rn a_j(x,\xi)\widehat{f}(\xi)
e^{ix\xi}\,d\xi=:\sum_{j=0}^\fz a_j(X,D)f(x)
\end{eqnarray*}
in $\cs'(\rn)$,
where $a_j(x,\xi):=a(x,\xi)\widehat\vz(2^{-j}\xi)$
for all $x,\xi\in\rn$,
and $a_j(X,D)$ is the related operator of $a_j(x,\xi)$.
It is easy to see that $a_j\in S^0_{1,\mu}(\rn)$ and supports in
the annulus $2^{j-2}\le|\xi|\le 2^{j+2}$.
Then by Lemmas \ref{l5.2} and \ref{l5.3},
we conclude that $A(X,D)$ is bounded on $A^{w,\tau}_{{\mathcal L},q,a}(\rn)$.
Meanwhile, Lemma \ref{l5.4} shows
that $B(X,D)$ is bounded on $A^{w,\tau}_{{\mathcal L},q,a}(\rn)$.
Consequently, it follows that $a(X,D)=A(X,D)+B(X,D)$ is bounded
on $A^{w,\tau}_{{\mathcal L},q,a}(\rn)$,
which completes the proof of Theorem \ref{t5.1}.
\end{proof}

Since molecules are mapped to molecules
by pseudo-differential operators
if we do not consider the moment condition,
we have the following conclusion.
We omit the details.

\begin{theorem}\label{t5.2}
Under the condition of Theorem {\rm \ref{t4.2}},
pseudo-differential operators
with symbol $S^0_{1,1}(\rn)$ are bounded
on $A_{\cl,q,a}^{w,\tau}(\rn)$.
\end{theorem}

\section{Embeddings}
\label{s6}

\subsection{Embedding into $C({\mathbb R}^n)$}

Here we give a sufficient condition
for which the function spaces are embedded
into $C({\mathbb R}^n)$.
In what follows, the \emph{space $C({\mathbb R}^n)$}
denotes the set of all continuous functions on $\rn$.
Notice that here, we do not require that the functions
of $C(\rn)$ are bounded.

\begin{theorem}\label{t6.1}
Let $q\in(0,\fz]$, $a\in(0,\fz)$ and $\tau\in[0,\fz)$.
Let $w \in \star-{\mathcal W}^{\alpha_3}_{\alpha_1,\alpha_2}$
with
$\alpha_1, \alpha_2, \alpha_3 \in [0,\infty)$
and a quasi-normed function space $\cl(\rn)$
satisfy $(\cl1)$ through $(\cl6)$ such that
\begin{equation}\label{6.1}
a+\gamma-\alpha_1-n\tau<0.
\end{equation}
Then
$A^{w,\tau}_{\cl,q,a}(\rn)$
is embedded into
$C({\mathbb R}^n)$.
\end{theorem}

\begin{proof}
By Remark \ref{r3.1}(ii), we see that it suffices to consider
$B^{w,\tau}_{\cl,\infty,a}(\rn)$,
into which $A^{w,\tau}_{\cl,q,a}(\rn)$ is embedded.
Also let us assume (\ref{3.27}).
Let us prove that
$B^{w,\tau}_{\cl,\infty,a}(\rn)$
is embedded into $C(\rn)$.
Let $x \in \R^n$ be fixed.
From the definition of the Peetre maximal operator,
we deduce that,
for all $f\in B^{w,\tau}_{\cl,q,a}(\rn)$, $j\in\zz_+$ and $y \in B(x,1)$,
\[
\sup_{w \in B(x,1)}
|\vz_j*f(w)|
\lesssim
2^{ja}
\sup_{z \in \rn}
\frac{|\vz_j*f(y+z)|}{(1+2^j|z|)^a}.
\]
If we consider the $\cl(\rn)$-quasi-norm of both sides,
then we obtain
\begin{equation*}
\sup_{z \in B(x,1)}|\vz_j*f(z)|
\lesssim_x
\frac{2^{ja}}{\|\chi_{B(x,2^{-j})}\|_{\cl(\rn)}}
\left\|\chi_{B(x,2^{-j})}(\vz_j*f)^*_a\right\|_{\cl(\rn)}.
\end{equation*}%
Notice that $w_j(x)=w(x,2^{-j}) \ge 2^{j\alpha_1}w(x,1)$
for all $j\in\zz_+$ and $x\in\rn$,
and hence from (W2) and \eqref{6.1},
it follows that
$$
\sup_{z \in B(x,1)}|\vz_j*f(z)|
\lesssim_x
2^{j(a+\gamma-\alpha_1-n\tau)}
\|f\|_{B^{w,\tau}_{\cl,\infty,a}(\rn)}\,.
$$
Since this implies that
\[
f=\Phi*f+\sum_{j=1}^\infty \vz_j*f
\]
converges uniformly over any ball with radius $1$,
$f$ is continuous,
which completes the proof of Theorem \ref{t6.1}.
\end{proof}

\subsection{Function spaces $A^{w,\tau}_{\cl,q,a}(\rn)$ for $\tau$ large}

The following theorem generalizes \cite[Theorem 1]{yy} and
explains
what happens if $\tau$ is too large.
\begin{theorem}\label{t6.2}
Let $\oz\in\cw_{\az_1,\az_2}^{\az_3}$
with $\az_1,\az_2,\az_3 \in[0,\fz)$.
Define
a new index $\wz \tau$ by
\begin{equation}\label{6.3}
\wz \tau:=
\limsup_{j \to \infty}
\left(\sup_{P\in\cq_j(\rn)}
\frac{1}{nj}\log_2\frac{1}{\|\chi_P\|_{\cl(\rn)}}
\right)
\end{equation}
and a new weight $\wz \oz$ by
\begin{equation*}
\wz \oz(x,2^{-j}) :=
2^{jn(\tau-\wz \tau)}\oz(x,2^{-j}),
\quad
x \in \rn, \, j \in \Z_+.
\end{equation*}%

Assume that $\tau$ and $\wz \tau$ satisfy
\begin{equation}\label{6.5}
\tau>\wz\tau \ge 0.
\end{equation}
Then

{\rm (i)}
$\wz w\in
\cw_{(\az_1-n(\tau-\wz \tau))_+,(\az_2+n(\tau-\wz \tau))_+}^{\az_3}$;

{\rm (ii)}
for all $q\in(0,\fz)$ and $a>\alpha_3+N_0$,
then $F^{w,\tau}_{\cl,q,a}(\rn)$ and
$B^{w,\tau}_{\cl,q,a}(\rn)$ coincide, respectively, with
$F^{\wz w}_{\fz,\fz,a}(\rn)$ and
$B^{\wz w}_{\fz,\fz,a}(\rn)$
with equivalent norms.
\end{theorem}

\begin{proof}
We only prove
$F^{w,\tau}_{\cl,q,a}(\rn)$ coincides with $F^{\wz w}_{\fz,\fz,a}(\rn)$.
The assertion (i) can be proved as in Example \ref{e2.1}(iii)
and the proof for the spaces $B^{w,\tau}_{\cl,q,a}(\rn)$
and $B^{\wz w}_{\fz,\fz,a}(\rn)$ is similar. To this end,
by the atomic decomposition of the pairs
$(F^{w,\tau}_{\cl,q,a}(\rn), f^{w,\tau}_{\cl,q,a}(\rn))$ and
$(F^{\wz w}_{\fz,\fz,a}(\rn), f^{\wz w}_{\fz,\fz,a}(\rn))$,
it suffices to show
that $f^{w,\tau}_{\cl,q,a}(\rn)=f^{\wz w}_{\fz,\fz,a}(\rn)$
with norm equivalence.
Recall that, for all $\lz=\{\lz_{jk}\}_{j\in\zz_+,k\in\zz^n}$,
\begin{eqnarray*}
&&\|\lz\|_{f^{w,\tau}_{\cl,q,a}(\rn)}\\
&&\hs=\sup_{P\in\cq(\rn)}\frac1{|P|^\tau}\lf\|\lf[\sum_{j=j_P\vee0}^\fz\lf(
\chi_P w_j
\sup_{y\in\rn}\frac{1}{(1+2^j|y|)^a}
\sum_{k\in\zz^n}|\lz_{jk}|\chi_{Q_{jk}}(\cdot+y)
\r)^q\r]^{1/q}\r\|_{\cl(\rn)}
\end{eqnarray*}
and
\begin{eqnarray}\label{6.6}
\|\lz\|_{f^{\wz w}_{\fz,\fz,a}(\rn)}
&&=
\sup_{x\in\rn, j\in\Z_+}
\wz w_j(x)
\sup_{y\in\rn}\frac{1}{(1+2^j|y|)^a}
\sum_{k\in\zz^n}|\lz_{jk}|\chi_{Q_{jk}}(x+y)\noz\\
&&=\sup_{(x,y)\in\R^{2n}, j\in \Z_+}
\wz w_j(x)\frac{1}{(1+2^j|y|)^a}
\sum_{k\in\zz^n}|\lz_{jk}|\chi_{Q_{jk}}(x+y).
\end{eqnarray}
By \eqref{6.6},
there exist $j_0\in\zz_+$, $k_0\in\zz^n$ and $x_0,y_0\in\rn$ such that
$$x_0+y_0\in Q_{j_0k_0}\quad\mathrm{and}\quad
\|\lz\|_{f^{\wz w}_{\fz,\fz,a}(\rn)}\sim
\wz w_{j_0}(x_0)
\frac{|\lz_{j_0k_0}|}{(1+2^{j_0}|y_0|)^a}.
$$
Then, a geometric observation shows that
there exists $P_0\in\cq(\rn)$
whose sidelength is half
of that of $Q_{j_0k_0}$
and
which satisfies $y_0+P_0\st Q_{j_0k_0}$.
Thus, for all $x\in P_0$, we have $|x-x_0|\ls 2^{-j_0}$
and hence
$$
w_{j_0}(x_0)\le w_{j_0}(x)\lf(1+2^{j_0}|x-x_0|\r)^{\az_3}
\ls w_{j_0}(x),
$$
which, together with the assumption on $\tau$, implies that
\begin{eqnarray*}
&&\|\lz\|_{f^{w,\tau}_{\cl,q,a}(\rn)}\\
&&\hs=\sup_{P\in\cq(\rn)}\frac1{|P|^\tau}\lf\|\lf[\sum_{j=j_P\vee0}^\fz\lf(
\chi_P w_j
\sup_{y\in\rn}\frac{1}{(1+2^j|y|)^a}
\sum_{k\in\zz^n}|\lz_{jk}|\chi_{Q_{jk}}(\cdot+y)
\r)^q\r]^{1/q}\r\|_{\cl(\rn)}\\
&&\hs\gs\frac1{|P_0|^\tau}\lf\|
\chi_{P_0} w_{j_0}
\frac{|\lz_{j_0k_0}|}{(1+2^j|y_0|)^a}\chi_{Q_{jk}}(\cdot+y_0)
\r\|_{\cl(\rn)}\\
&&\hs\gs\|\lz\|_{f^{\wz w}_{\fz,\fz,a}(\rn)}
\frac{2^{-j_0n(\tau-\wz \tau)}\|\chi_{P_0}\|_{\cl(\rn)}}{|P_0|^\tau}.
\end{eqnarray*}
Consequently,
\begin{equation}\label{6.7}
\|\lz\|_{f^{w,\tau}_{\cl,q,a}(\rn)}
\gs
\|\lz\|_{f^{\wz w}_{\fz,\fz,a}(\rn)}.
\end{equation}

To obtain the reverse inclusion, we calculate
\begin{eqnarray*}
\|\lz\|_{f^{w,\tau}_{\cl,q,a}(\rn)}&&=
\sup_{P\in\cq(\rn)}\frac1{|P|^\tau}\lf\|\lf(\sum_{j=j_P\vee0}^\fz\lf[
\chi_P w_j
\sup_{y\in\rn}\frac{\sum_{k\in\zz^n}|\lz_{jk}|\chi_{Q_{jk}}(\cdot+y)}{(1+2^j|y|)^a}
\r]^q\r)^{1/q}\r\|_{\cl(\rn)}\\
&&\le\|\lz\|_{f^{\wz w}_{\fz,\fz,a}(\rn)}
\sup_{P\in\cq(\rn)}\frac1{|P|^\tau}\lf\|\chi_P \lf[\sum_{j=j_P\vee0}^\fz\lf(
\frac {w_j}{\wz w_j}
\r)^q\r]^{1/q}\r\|_{\cl(\rn)}.
\end{eqnarray*}
If we use \eqref{6.5} and \eqref{6.6},
then we have
\begin{eqnarray*}
\|\lz\|_{f^{w,\tau}_{\cl,q,a}(\rn)}
&&\le\|\lz\|_{f^{\wz w}_{\fz,\fz,a}(\rn)}
\sup_{P\in\cq(\rn)}
\frac1{|P|^\tau}\lf\|\chi_P\lf[\sum_{j=j_P\vee0}^\fz
 2^{-jnq(\tau-\wz \tau)}
\r]^{1/q}\r\|_{\cl(\rn)}\\
&&\sim \|\lz\|_{f^{\wz w}_{\fz,\fz,a}(\rn)}
\sup_{P\in\cq(\rn)}
\frac{2^{-(j_P\vee0)n(\tau-\wz \tau)}}{|P|^\tau}\|\chi_P\|_{\cl(\rn)}.
\end{eqnarray*}
Since $\wz\tau \in [0,\infty)$ and \eqref{6.3} holds true,
we see that
\begin{eqnarray*}
\frac{2^{-(j_P\vee0)n(\tau-\wz \tau)}}{|P|^\tau}\|\chi_P\|_{\cl(\rn)}
&&\sim
\frac{2^{-j_Pn(\tau-\wz \tau)}}{|P|^\tau}\|\chi_P\|_{\cl(\rn)}\noz\\
&&=
2^{j_P n \wz\tau}\|\chi_P\|_{\cl(\rn)}
=
|P|^{-\wz\tau}\|\chi_P\|_{\cl(\rn)}
\ls 1.
\end{eqnarray*}%
Hence, we conclude that
\begin{equation}\label{6.9}
\|\lz\|_{f^{w,\tau}_{\cl,q,a}(\rn)}
\le\|\lz\|_{f^{\wz w}_{\fz,\fz,a}(\rn)}.
\end{equation}
Hence from \eqref{6.7} and \eqref{6.9},
we deduce
that $F^{w,\tau}_{\cl,q,a}(\rn)$ and $F^{\wz w}_{\fz,\fz,a}(\rn)$
coincide with equivalent norms, which completes the proof of Theorem \ref{t6.2}.
\end{proof}

\section{Characterizations via differences and oscillations\label{s7}}

In this section we are going to characterize function spaces
by means of differences and oscillations.
To this end, we need some key constructions
from Triebel \cite{t92}.

For any $M\in\nn$,
Triebel \cite[p. 173, Lemma 3.3.1]{t92} proved that there exist
two smooth functions $\vz$ and $\psi$ on $\rr$ with $\supp \vz\st(0,1)$,
$\supp \psi\st(0,1)$, $\int_\rr\vz(\tau)\,d\tau=1$ and
$\vz(t)-\frac12\vz(\frac t2)=\psi^{(M)}(t)$ for $t\in\rr$.
Let $\rho(x):=\prod_{\ell=1}^n\vz(x_\ell)$ for all $x=(x_1,\cdots, x_n)\in\rn$.
For all $j\in\zz_+$ and $x\in\rn$,
let
\[
T_j(x)
:=
\sum_{m'=1}^{M} \sum_{m=1}^{M}
\frac{(-1)^{M+m+m'+1}}{M!}
\begin{pmatrix}
{M}\\
m'
\end{pmatrix}
\begin{pmatrix}
{M}\\
m
\end{pmatrix}
m^{M}
(2^{-j}mm')^{-n}
\rho\left(\frac{x}{2^{-j}mm'}\right),
\]
where $\binom{M}{m}$ for $m\in\{1,\cdots, M\}$
denotes the \emph{binomial coefficient}.
For any $f \in \cs'(\rn)$, let
\begin{equation}\label{7.1}
f^j:= T_j*f\ \mbox{for all}\ j\in\zz_+,\
\mbox{and}\ f^{-1}:= 0.\end{equation}

{}From Theorem \ref{t3.1} and Triebel
\cite[pp.\,174-175, Proposition 3.3.2]{t92},
we immediately deduce the following useful conclusions,
the details of whose proofs are omitted.

\begin{proposition}\label{p7.1}
Let $\alpha_1, \alpha_2, \alpha_3, \tau \in [0,\infty)$ and $q\in(0,\,\fz]$
and let $w\in{\mathcal W}^{\alpha_3}_{\alpha_1,\alpha_2}$.
Choose $a\in(0,\fz)$ and $M \in \N$ such that
\begin{equation}\label{7.2}
M>\alpha_1 \vee (a+n\tau+\alpha_2).
\end{equation}
For $j\in\zz_+$, $f \in \cs'(\rn)$ and $x\in\rn$, let
$F(x,2^{-j}):=f^j(x)-f^{j-1}(x)$, where $\{f^j\}_{j=-1}^\fz$
is as in \eqref{7.1}. Then

{\rm(i)} $f\in B^{w,\tau}_{\cl,q,a}(\rn)$ if and only if
$F\in L_{\cl,q,a}^{w,\tau}(\rr_{\zz_+}^{n+1})$ and
$\|F\|_{L_{\cl,q,a}^{w,\tau}(\rr_{\zz_+}^{n+1})}<\fz$.
Moreover, $\|f\|_{B^{w,\tau}_{\cl,q,a}(\rn)}
\sim\|F\|_{L_{\cl,q,a}^{w,\tau}(\rr_{\zz_+}^{n+1})}$ with
the implicit positive constants
independent of $f$.

{\rm(ii)} $f\in \cn^{w,\tau}_{\cl,q,a}(\rn)$ if and only if
$F\in \cn_{\cl,q,a}^{w,\tau}(\rr_{\zz_+}^{n+1})$ and
$\|F\|_{\cn_{\cl,q,a}^{w,\tau}(\rr_{\zz_+}^{n+1})}<\fz$.
Moreover, $\|f\|_{\cn^{w,\tau}_{\cl,q,a}(\rn)}
\sim\|F\|_{\cn_{\cl,q,a}^{w,\tau}(\rr_{\zz_+}^{n+1})}$ with
the implicit positive constants
independent of $f$.

{\rm(iii)} $f\in F^{w,\tau}_{\cl,q,a}(\rn)$ if and only if
$F\in P_{\cl,q,a}^{w,\tau}(\rr_{\zz_+}^{n+1})$ and
$\|F\|_{P_{\cl,q,a}^{w,\tau}(\rr_{\zz_+}^{n+1})}<\fz$.
Moreover, $\|f\|_{F^{w,\tau}_{\cl,q,a}(\rn)}
\sim\|F\|_{P_{\cl,q,a}^{w,\tau}(\rr_{\zz_+}^{n+1})}$ with
the implicit positive constants
independent of $f$.

{\rm(iv)} $f\in \ce^{w,\tau}_{\cl,q,a}(\rn)$ if and only if
$F\in \ce_{\cl,q,a}^{w,\tau}(\rr_{\zz_+}^{n+1})$ and
$\|F\|_{\ce_{\cl,q,a}^{w,\tau}(\rr_{\zz_+}^{n+1})}<\fz$.
Moreover, $\|f\|_{\ce^{w,\tau}_{\cl,q,a}(\rn)}
\sim
\|F\|_{\ce_{\cl,q,a}^{w,\tau}(\rr_{\zz_+}^{n+1})}$ with
the implicit positive constants
independent of $f$.
\end{proposition}

\subsection{Characterization by differences}

In this section,
we characterize our function spaces in terms of differences.
For an arbitrary function $f$, we inductively define
$\Delta^M_h f$ with $M \in \N$ and $h \in \rn$ by
$$\Delta_h f := \Delta_h^1 f:= f-f(\cdot-h) \quad
{\rm and}\quad\Delta^{M}_h f:= \Delta_h(\Delta^{M-1}_h f),$$
and ${\rm J}^{(1)}_{a,w,\cl}(f)$ and
${\rm J}^{(2)}_{a,w,\cl}(f)$ with $a\in (0,\fz)$ and $w_0$ as in
\eqref{2.3}, respectively, by
\begin{equation*}
{\rm J}^{(1)}_{a,w,\cl}(f):=\sup_{P\in\cq,\,|P|\ge 1}\frac1{|P|^\tau}
\lf\|\chi_Pw_0\sup_{y\in\rn}\frac{|f(\cdot+y)|}{(1+|y|)^a}\r\|_{\cl(\rn)}
\end{equation*}%
or
\begin{equation*}
{\rm J}^{(2)}_{a,w,\cl}(f):=\sup_{P\in\cq}\frac1{|P|^\tau}
\lf\|\chi_Pw_0\sup_{y\in\rn}\frac{|f(\cdot+y)|}{(1+|y|)^a}\r\|_{\cl(\rn)}.
\end{equation*}%

In what follows, we denote by $\oint_E$
the \emph{average over a measurable set $E$ of $f$}.

\begin{theorem}\label{t7.1}
Let $a, \az_1, \alpha_2, \alpha_3, \tau\in[0,\fz)$, $u\in[1,\fz]$, $q\in(0,\,\fz]$
and $w\in\star-{\mathcal W}^{\alpha_3}_{\alpha_1,\alpha_2}$.
If $M\in\nn$, $\az_1\in (a,M)$ and \eqref{7.2} holds true, then there exists a positive
constant $\wz C:=C(M)$, depending on $M$,
such that, for all $f\in\cs'(\rn)\cap L^1_\loc(\rn)$,
the following hold true:

{\rm(i)}
\begin{eqnarray*}
{\rm I}_1:&&={\rm J}^{(1)}_{a,w,\cl}(f)+\left\|
\left\{\sup_{z \in \rn}\left[\oint_{|h| \le \wz{C}\,2^{-j}}
\frac{|\Delta_h^{M}f(\cdot+z)|^u}{(1+2^j|z|)^{au}}\,dh\right]^{1/u}
\right\}_{j \in \Z_+}\right\|_{\ell^q(\cl^w_\tau(\rn,\zz_+))}\\
&&\sim \|f\|_{B^{w,\tau}_{\cl,q,a}(\rn)}
\end{eqnarray*}
with the implicit positive constants
independent of $f$.

\rm{(ii)}
\begin{eqnarray*}
{\rm I}_2:=&&{\rm J}^{(1)}_{a,w,\cl}(f)
+\left\|\left\{\sup_{z \in \rn}\left[\oint_{|h| \le \wz{C}\,2^{-j}}
\frac{|\Delta_h^{M}f(\cdot+z)|^u}{(1+2^j|z|)^{au}}\,dh\right]^{1/u}
\right\}_{j \in \Z_+}\right\|_{\cl^w_\tau(\ell^q(\rn,\zz_+))}\\
\sim&&\|f\|_{F^{w,\tau}_{\cl,q,a}(\rn)}
\end{eqnarray*}
with the implicit positive constants
independent of $f$.

\rm{(iii)}
\begin{eqnarray*}
{\rm I}_3:=&&{\rm J}^{(2)}_{a,w,\cl}(f)
+\left\|\left\{\sup_{z \in \rn}\left[\oint_{|h| \le \wz{C}\,2^{-j}}
\frac{|\Delta_h^{M}f(\cdot+z)|^u}{(1+2^j|z|)^{au}}\,dh\right]^{1/u}
\right\}_{j \in \Z_+}\right\|_{\ell^q(\cn\cl^w_\tau(\rn,\zz_+))}\\
\sim&&\|f\|_{\cn^{w,\tau}_{\cl,q,a}(\rn)}
\end{eqnarray*}
with the implicit positive constants
independent of $f$.

\rm{(iv)}
\begin{eqnarray*}
{\rm I}_4:=&&{\rm J}^{(2)}_{a,w,\cl}(f)
+\left\|\left\{\sup_{z \in \rn}\left[\oint_{|h| \le \wz{C}\,2^{-j}}
\frac{|\Delta_h^{M}f(\cdot+z)|^u}{(1+2^j|z|)^{au}}\,dh\right]^{1/u}
\right\}_{j \in \Z_+}\right\|_{\ce\cl^w_\tau(\ell^q(\rn,\zz_+))}\\
\sim&&\|f\|_{\ce^{w,\tau}_{\cl,q,a}(\rn)}
\end{eqnarray*}
with the implicit positive constants
independent of $f$.
\end{theorem}

\begin{proof} We only prove (i), since the proofs of others are similar.
To this end, for any $f\in\cs'(\rn)\cap L^1_\loc(\rn)$,
since $\rho\in C^\fz_c(\rn)$
(see \cite[pp.\,174-175, Proposition 3.3.2]{t92}),
we conclude that,  for all $j\in\zz_+$ and $x\in\rn$,
\begin{align}\label{eq:120428-1}
f^j(x)
:=\sum_{m'=1}^{M} \sum_{m=1}^{M}
\frac{(-1)^{M+m+m'-1}}{M!}
\begin{pmatrix}
{M}\\
m'
\end{pmatrix}
\begin{pmatrix}
{M}\\
m
\end{pmatrix}
m^{M}
\int_{\rn}\rho(y)f(x-2^{-j}mm'y)\,dy
\end{align}
and hence
\begin{eqnarray*}
&&
f^j(x)-
f^{j+1}(x)\\
&&
=\sum_{m'=1}^{M} \sum_{m=1}^{M}
\frac{(-1)^{M+m+m'-1}}{M!}
\begin{pmatrix}
{M}\\
m'
\end{pmatrix}
\begin{pmatrix}
{M}\\
m
\end{pmatrix}
m^{M}
\int_{\rn}\rho(y)f(x-2^{-j}mm'y)\,dy\\
&&\quad-\sum_{m'=1}^{M} \sum_{m=1}^{M}
\frac{(-1)^{M+m+m'-1}}{M!}
\begin{pmatrix}
{M}\\
m'
\end{pmatrix}
\begin{pmatrix}
{M}\\
m
\end{pmatrix}
m^{M}
\int_{\rn}\rho(y)f(x-2^{-j-1}mm'y)\,dy\\
&&
=\sum_{m'=0}^{M} \sum_{m=1}^{M}
\frac{(-1)^{M+m+m'-1}}{M!}
\begin{pmatrix}
{M}\\
m'
\end{pmatrix}
\begin{pmatrix}
{M}\\
m
\end{pmatrix}
m^{M}
\int_{\rn}\rho(y)f(x-2^{-j}mm'y)\,dy\\
&&\quad-\sum_{m'=0}^{M} \sum_{m=1}^{M}
\frac{(-1)^{M+m+m'-1}}{M!}
\begin{pmatrix}
{M}\\
m'
\end{pmatrix}
\begin{pmatrix}
{M}\\
m
\end{pmatrix}
m^{M}
\int_{\rn}\rho(y)f(x-2^{-j-1}mm'y)\,dy\\
&&
=
\sum_{m=1}^{M}
\frac{(-1)^{M+m-1}}{M!}
\begin{pmatrix}
{M}\\
m
\end{pmatrix}
m^{M}
\int_{\rn}\rho(y)
\lf[\Delta_{2^{-j}my}^{M}f(x)-\Delta_{2^{-j-1}my}^{M}f(x)\r]\,dy.
\end{eqnarray*}

As a consequence,
we see a pointwise estimate that, for all $x\in\rn$ and $u\in[1,\fz]$,
\begin{equation}\label{7.49}
\sup_{z \in \rn}
\frac{|f^j(x+z)-f^{j+1}(x+z)|}{(1+2^j|z|)^a}
\lesssim
\sup_{z \in \rn}\left[
\oint_{|h| \le \wz{C}\,2^{-j}}
\frac{|\Delta_h^{M}f(x+z)|^u}{(1+2^j|z|)^{au}}\,dh
\right]^{1/u}.
\end{equation}

Meanwhile,  by $T_0 \in \cs(\rn)$ and
$(1+|u|)^a \le (1+|u+y|)^a(1+|y|)^a$ for all $u,y\in\rn$,
we see that, for all $x\in\rn$,
\begin{eqnarray}\label{7.5}
\sup_{y\in\rn}\frac{|f^0(x+y)|}{(1+|y|)^a}&&=
\sup_{y\in\rn}\frac 1{(1+|y|)^a}\lf|\int_\rn T_0(u)f(x+y-u)\,du\r|\noz\\
&&\le\sup_{y\in\rn}
\int_\rn\lf|T_0(u+y)\r|\frac{(1+|u|)^a}{(1+|y|)^a}
\frac{|f(x-u)|}{(1+|u|)^a}\,du\noz\\
&&\ls\sup_{u\in\rn}\frac{|f(x+u)|}{(1+|u|)^a}.
\end{eqnarray}

Combining (\ref{7.49}) and (\ref{7.5})
with Proposition \ref{p7.1} (here we need
to use the assumption \eqref{7.2}), we conclude that
\begin{eqnarray*}
{\rm I}_1\gs&&
\sup_{P\in\cq,|P|\ge1}\frac1{|P|^\tau}
\lf\|\chi_Pw_0\sup_{y\in\rn}\frac{|(f^0-f^{-1})(\cdot+y)|}{(1+|y|)^a}\r\|_{\cl(\rn)}\\
&&+\left\|
\left\{\sup_{z \in \rn}
\frac{|(f^j-f^{j-1})(\cdot+z)|}{(1+2^j|z|)^a}
\right\}_{j \in \Z_+}\right\|_{\ell^q(\cl^w_\tau(\rn,\zz_+))}
\sim \|f\|_{B^{w,\tau}_{\cl,q,a}(\rn)},
\end{eqnarray*}
which is desired.

To show the inverse inequality,
for any $f\in B^{w,\tau}_{\cl,q,a}\cap L^1_\loc(\rn)$,
since $\{T_j\}_{j\in\zz_+}$ is an approximation to the identity
(see \cite[pp.\,174-175, Proposition 3.3.2]{t92}),
if we fix $|h| \le \wz{C}\,2^{-j}$ and $z \in \rn$,
then by \cite[p.\,195,\,(3.5.3/7)]{t92},
we see that, for almost every $x\in\rn$,
\begin{eqnarray*}
\left[\oint_{|h|< 2^{-j}}|\Delta_h^{M}f(x+z)|^u\,dy\right]^{1/u}
&&\ls \sum_{l=1}^\fz\left\{|f_{j+l}(x+z)|+
\left[\oint_{B(x+z,C2^{-j})}|f_{j+l}(y)|^u\,dh\right]^{1/u}\right\}\\
&&\quad\hs+
\sup_{w \in B(x+z,C2^{-j})}
\left|\int_{\rn}D^\alpha T_0(y)f(w+2^{-j}y)\,dy\right|,
\end{eqnarray*}
where and in what follows, $f_{j}:=f^{j}-f^{j-1}$ for all $j\in\zz_+$.
Then
\begin{eqnarray}\label{7.6}
&&\left[\oint_{|h|< 2^{-j}}\frac{|\Delta_h^{M}f(x+z)|^u}{(1+2^j|z|)^{au}}\,dh\right]^{1/u}\nonumber\\
&&\quad\lesssim
\sum_{l=1}^\infty
\frac{|f_{j+l}(x+z)|+
\left[\oint_{B(x+z,C2^{-j})}|f_{j+l}(y)|^u\,dy\right]^{1/u}}{(1+2^j|z|)^a}\nonumber\\
&&\quad\quad+
\sup_{w \in B(x+z,C2^{-j})}
\frac{1}{(1+2^j|z|)^a}
\left|\int_{\rn}D^\alpha T_0(y)f(w+2^{-j}y)\,dy\right|.
\end{eqnarray}
For the second term on the right-hand side of \eqref{7.6}, we have
\begin{eqnarray*}
&&\sup_{z\in\rn}
\left[\sup_{w \in B(x+z,C2^{-j})}
\frac{1}{(1+2^j|z|)^a}
\left|\int_{\rn}D^\alpha T_0(y)f(w+2^{-j}y)\,dy\right|\right]\\
&&\hs=\sup_{z\in\rn}
\left[\sup_{w \in B(0,C2^{-j})}
\frac{|(D^\alpha T_0)_j*\widetilde f(x+z+w)|}{(1+2^j|z|)^a}\right]\\
&&\hs\le\sup_{z\in\rn}
\sup_{w \in B(0,C2^{-j})}
\frac{|(D^\alpha T_0)_j*\widetilde f(x+z+w)|}{(1+2^j|z+w|)^a}
\lf[\frac{1+2^j(|z|+|w|)}{1+2^j|z|}\r]^a\\
&&\hs\ls\sup_{z\in\rn}
\frac{|(D^\alpha T_0)_j*\widetilde f(x+z)|}{(1+2^j|z|)^a},
\end{eqnarray*}
where $\widetilde{f}:=f(-\cdot)$. This observation,
together with the fact that
\begin{eqnarray*}
\left\|\left\{\sup_{z\in\rn}\frac{|(D^\alpha T_0)_j*\widetilde
f(x+z)|}{(1+2^j|z|)^a}
\right\}_{j \in \Z_+}\right\|_{\ell^q(\cl^w_\tau(\rn,\zz_+))}
&\ls&
\|f\|_{B^{w,\tau}_{\cl,q,a}(\rn)},
\end{eqnarray*}
implies that
\begin{eqnarray*}
&&\left\|\left\{\sup_{z\in\rn}
\sup_{w \in B(x+z,C2^{-j})}
\frac{1}{(1+2^j|z|)^a}
\left|\int_{\rn}D^\alpha T_0(y)f(w+2^{-j}y)\,dy\right|
\right\}_{j \in \Z_+}\right\|_{\ell^q(\cl^w_\tau(\rn,\zz_+))}\\
&&\quad\ls
\|f\|_{B^{w,\tau}_{\cl,q,a}(\rn)}.
\end{eqnarray*}
For the first term on the right-hand side of \eqref{7.6}, we see that
\begin{eqnarray}\label{7.7}
&&\sup_{z\in\rn}\lf[\oint_{y\in B(x+z,2^{-j})}|f_{j+l}(y)|^u\,dy\r]^{1/u}\frac 1{(1+2^j|z|)^{a}}\nonumber\\
&&\quad\le\sup_{z\in\rn}
\left\{\sup_{y \in B(0,2^{-j})}
\frac{\left|f_{j+l}(x+z+y)\right|}{(1+2^{j+l}|z+y|)^{a}}
\lf[\frac{1+2^{j+l}(|z|+|y|)}{1+2^j|z|}\r]^a\right\}\nonumber\\
&&\quad\ls2^{la}\sup_{z\in\rn}
\frac{\left|f_{j+l}(x+z)\right|}{(1+2^{j+l}|z|)^a}.
\end{eqnarray}
Since $w\in\star-{\mathcal W}^{\alpha_3}_{\alpha_1,\alpha_2}$, we have
$w_j(x)\ls 2^{-l\az_1}w_{j+l}(x)$, which, together with
$\az_1>a$ and \eqref{7.7}, implies that
\begin{eqnarray*}
&&\left\|\left\{\sum_{l=1}^\fz\sup_{z\in\rn}
\lf[\oint_{y\in B(x+z,2^{-j})}|f_{j+l}(y)|^u\,dy\r]^{1/u}\frac1{(1+2^j|z|)^a}
\right\}_{j \in \Z_+}\right\|_{\ell^q(\cl^w_\tau(\rn,\zz_+))}\\
&&\hs\ls
\left\|\left\{\sum_{l=1}^\fz 2^{la}\sup_{z\in\rn}
\frac{\left|f_{j+l}(\cdot+z)\right|}{(1+2^{j+l}|z|)^a}
\right\}_{j \in \Z_+}\right\|_{\ell^q(\cl^w_\tau(\rn,\zz_+))}\\
&&\hs\ls\left\{\sum_{l=1}^\fz 2^{la{\wz\theta}}\sup_{P\in\cq}\frac1{|P|^\tau}
\left[\sum_{j=(0\vee j_P)}^\fz\left\|\chi_P w_j\sup_{z\in\rn}
\frac{\left|f_{j+l}(\cdot+z)\right|}{(1+2^{j+l}|z|)^a}\right\|_{\cl(\rn)}^q
\right]^{{\wz\theta}/q}\right\}^{1/{\wz\theta}}\\
&&\hs\ls\left\{\sum_{l=1}^\fz 2^{-l(\az_1-a){\wz\theta}}\sup_{P\in\cq}\frac1{|P|^\tau}
\left[\sum_{j=(0\vee j_P)}^\fz\left\|\chi_P w_{j+l}\sup_{z\in\rn}
\frac{\left|f_{j+l}(\cdot+z)\right|}{(1+2^{j+l}|z|)^a}\right\|_{\cl(\rn)}^q
\right]^{{\wz\theta}/q}\right\}^{1/{\wz\theta}}\\
&&\hs\ls\sup_{P\in\cq}\frac1{|P|^\tau}
\left\{\sum_{j=(0\vee j_P)}^\fz\left\|\chi_P w_{j}\sup_{z\in\rn}
\frac{\left|f_{j}(\cdot+z)\right|}{(1+2^{j}|z|)^a}\right\|_{\cl(\rn)}^q
\right\}^{1/q}\sim\|f\|_{B^{w,\tau}_{\cl,q,a}(\rn)},
\end{eqnarray*}
where we chose ${\wz\theta}\in(0,\min\{\theta,q\})$
and $\theta$ is as in $(\cl3)$.

Meanwhile, by virtue of (\ref{eq:120428-1}),
we see that, for all $x\in\rn$,
$$
f^0(x)
=\sum_{m'=1}^{M} \sum_{m=1}^{M}
\frac{(-1)^{M+m+m'-1}}{M!}
\begin{pmatrix}
{M}\\
m'
\end{pmatrix}
\begin{pmatrix}
{M}\\
m
\end{pmatrix}
m^{M}
\int_{\rn}\rho(y)f(x-mm'y)\,dy
$$
and
\[
f(x)
=\sum_{m=1}^{M}
\frac{(-1)^{M+m+0-1}}{M!}
\begin{pmatrix}
{M}\\
0
\end{pmatrix}
\begin{pmatrix}
{M}\\
m
\end{pmatrix}
m^{M}
\int_{\rn}\rho(y)f(x)\,dy,
\]
which implies that, for all $x\in\rn$,
\begin{eqnarray*}
|f(x)|&&=
\lf|\sum_{m=1}^{M}
\frac{(-1)^{M+m-1}}{M!}
\binom Mm
m^{M}
\int_{\rn}\rho(y)f(x)\,dy+f^0(x)-f^0(x)\r|\\
&&\ls\lf|\sum_{m'=0}^{M} \sum_{m=1}^{M}
\frac{(-1)^{M+m+m'-1}}{M!}
\binom M{m'}\binom Mm
m^{M}
\int_{\rn}\rho(y)f(x-m'my)\,dy\r|+|f^0(x)|\\
&&\ls
\lf|\sum_{m=1}^{M}
\frac{(-1)^{M+m-1}}{M!}
\binom Mm
m^{M}
\int_{\rn}\rho(y)
\Delta_{my}^{M}f(x)\,dy\r|+|f^0(x)|.
\end{eqnarray*}
From this, we deduce that
\begin{eqnarray}\label{7.8}
\sup_{y\in\rn}\frac{|f(x+y)|}{(1+|y|)^a}&&\ls
\sup_{y\in\rn}\left[\oint_{|h|\ls1}\frac{|\Delta_{h}^{M}f(x+y)|^u}
{(1+|y|)^{au}}\,dh\right]^{1/u}
+\sup_{y\in\rn}\frac{|f^0(x+y)|}{(1+|y|)^a},
\end{eqnarray}
which, together with the trivial inequality
$$
\left\|\sup_{y\in\rn}\frac{|f^0(\cdot+y)|}{(1+|y|)^a}\right\|_{\cl(\rn)}
\lesssim \|f\|_{B^{w,\tau}_{\cl,q,a}(\rn)},
$$
implies that
\begin{eqnarray*}
{\rm J}^{(1)}_{a,w,\cl}(f)&&\ls\left\|
\left\{\sup_{z \in \rn}\left[\oint_{|h| \le \wz{C}\,2^{-j}}
\frac{|\Delta_h^{M}f(\cdot+z)|^u}{(1+2^j|z|)^{au}}\,dh\right]^{1/u}
\right\}_{j \in \Z_+}\right\|_{\ell^q(\cl^w_\tau(\rn,\zz_+))}
+ \|f\|_{B^{w,\tau}_{\cl,q,a}(\rn)}\\
&&\ls\|f\|_{B^{w,\tau}_{\cl,q,a}(\rn)}.
\end{eqnarray*}
This finishes the proof of (i) and hence Theorem \ref{t7.1}.
\end{proof}

If we further assume \eqref{6.1} holds true,
from Theorems \ref{t6.1} and \ref{t7.1},
we immediately deduce the following conclusions.
We omit the details.

\begin{corollary}\label{c7.1}
Let $\alpha_1, \alpha_2, \alpha_3, \tau $, $a$, $q$
and $w$ be as in Theorem \ref{t7.1}.
Assume \eqref{6.1} and \eqref{7.2}. Let $\{{\rm J}_j\}_{j=1}^4$
be as in Theorem \ref{t7.1}.
Then the following hold true:

{\rm(i)}
$f\in B^{w,\tau}_{\cl,q,a}(\rn)$ if and only if
$f \in \cs'(\rn)\cap L^1_\loc(\rn)$
and ${\rm J}_1<\fz$;
moreover, $
{\rm J}_1\sim \|f\|_{B^{w,\tau}_{\cl,q,a}(\rn)}$
with the implicit positive constants
independent of $f$.

{\rm(ii)}
$f\in F^{w,\tau}_{\cl,q,a}(\rn)$ if and only if
$f \in \cs'(\rn)\cap L^1_\loc(\rn)$
and
${\rm J}_2<\fz$;
moreover, $
{\rm J}_2\sim \|f\|_{F^{w,\tau}_{\cl,q,a}(\rn)}$
with the implicit positive constants
independent of $f$.

{\rm(iii)}
$f\in \cn^{w,\tau}_{\cl,q,a}(\rn)$ if and only if
$f \in \cs'(\rn)\cap L^1_\loc(\rn)$
and
${\rm J}_3<\fz$;
moreover, $
{\rm J}_3\sim \|f\|_{\cn^{w,\tau}_{\cl,q,a}(\rn)}$
with the implicit positive constants
independent of $f$.

{\rm(iv)}
$f\in \ce^{w,\tau}_{\cl,q,a}(\rn)$ if and only if
$f \in \cs'(\rn)\cap L^1_\loc(\rn)$
and
${\rm J}_4<\fz$;
moreover, $
{\rm J}_4\sim \|f\|_{\ce^{w,\tau}_{\cl,q,a}(\rn)}$
with the implicit positive constants
independent of $f$.
\end{corollary}

By the Peetre maximal function
characterizations of the Besov space $B^s_{p,q}(\rn)$
and the Triebel-Lizorkin space $F^s_{p,q}(\rn)$ (see, for example, \cite{u10}),
we know that, if $q\in(0,\fz]$, $\cl(\rn)=L^p(\rn)$
and $w_j\equiv 2^{js}$, then
$B^{w,\tau}_{\cl,q,a}=B^s_{p,q}(\rn)$ for all $p\in(0,\fz]$ and $a\in(n/p,\fz)$,
and $F^{w,\tau}_{\cl,q,a}=F^s_{p,q}(\rn)$ for all $p\in(0,\fz)$
and $a\in(n/\min\{p,q\},\fz)$.
Then, applying Theorem \ref{t7.1} in this case, we have the following corollary.
In what follows,
for all measurable functions $f$, $a\in(0,\fz)$
and $x\in\rn$, we define the \emph{Peetre maximal function} of $f$ as
$$f^*_a(x):=\sup_{z\in\rn}\frac{|f(x+z)|}{(1+|z|)^a}.$$

\begin{corollary}\label{c7.2}
Let $M\in\nn$, $u\in[1,\fz]$ and $q\in(0,\fz]$.

{\rm(i)} Let $p\in(0,\fz)$, $a\in(n/\min\{p,q\}, M/2)$ and
$s\in(a,M-a)$. Then
there exists a positive constant $\wz C:=C(M)$, depending on $M$,
such that
$f\in F^s_{p,q}(\rn)$ if and only if
$f \in \cs'(\rn)\cap L^1_\loc(\rn)$ and
\begin{eqnarray*}
{\rm J}_1:=\|f^*_a\|_{L^p(\rn)}
+\left\| \left\|\left\{ 2^{js}\sup_{z \in \rn}
\left[\oint_{|h| \le \wz{C}\,2^{-j}}
\frac{|\Delta_h^{M}f(\cdot+z)|^u}{(1+2^j|z|)^{au}}\,dh\right]^{1/u}
\right\}_{j \in \Z_+}\right\|_{\ell^q(\zz_+)}\right\|_{L^p(\rn)}
\end{eqnarray*}
is finite.
Moreover, ${\rm J}_1$ is equivalent to $\|f\|_{F^s_{p,q}(\rn)}$
with the equivalent positive constants
independent of $f$.

{\rm (ii)} Let $p\in(0,\fz]$, $a\in(n/p,M/2)$ and $s\in(a,M-a)$. Then
there exists a positive constant $\wz C:=C(M)$, depending on $M$,
such that
$f\in B^s_{p,q}(\rn)$ if and only if
$f \in \cs'(\rn)\cap L^1_\loc(\rn)$ and
\begin{eqnarray*}
{\rm J}_2:=
\|f^*_a\|_{L^p(\rn)}
+\left\|\left\{ \left\|2^{js}\sup_{z \in \rn}
\left[\oint_{|h| \le \wz{C}\,2^{-j}}
\frac{|\Delta_h^{M}f(\cdot+z)|^u}{(1+2^j|z|)^{au}}
\,dh\right]^{1/u}\right\|_{L^p(\rn)}
\right\}_{j \in \Z_+}\right\|_{\ell^q(\zz_+)}
\end{eqnarray*}
is finite.
Moreover, ${\rm J}_2$ is equivalent to $\|f\|_{B^s_{p,q}(\rn)}$
with the equivalent positive constants
independent of $f$.
\end{corollary}

\begin{proof}
Recall that by \cite[Theorem 3.3.2]{st95} (see also
\cite[pp.\,33-34]{rs96}),
$F^s_{p,q}(\rn)\subset L^1_\loc(\rn)$ if and only if either
$p\in(0,1)$, $s\in[n(1/p-1),\fz)$ and $q\in(0,\fz]$,
or $p\in[1,\fz)$, $s\in(0,\fz)$ and $q\in(0,\fz]$, or
$p\in[1,\fz)$, $s=0$ and $q\in(0,2]$,
and $B^s_{p,q}(\rn)
\subset L^1_\loc(\rn)$ if and only if either
$p\in(0,\fz]$, $s\in(n\max(0,1/p-1),\fz)$ and $q\in(0,\fz]$,
or $p\in(0,1]$, $s=n(1/p-1)$ and $q\in(0,1]$, or
$p\in(1,\fz]$, $s=0$ and $q\in(0,\min(p,2)]$.
From this, the aforementioned Peetre maximal function
characterizations of Besov spaces $B^s_{p,q}(\rn)$
and Triebel-Lizorkin spaces $F^s_{p,q}(\rn)$,
and Theorem \ref{t7.1}, we immediately
deduce the conclusions of (i) and (ii),  which
completes the proof of Corollary \ref{c7.2}.
\end{proof}

We remark that the difference characterizations obtained in Corollary
\ref{c7.2} is a little different from the classical difference characterizations
of Besov and Triebel-Lizorkin spaces in \cite[Section 3.5.3]{t92}.
Indeed, Corollary \ref{c7.2} can be seen as the Peetre maximal function version
of \cite[Theorem 3.5.3]{t92} in the case that $u=\fz$. We also
remark that the condition
$a\in(n/p,M)$ and $s\in(a,\fz)$ is somehow necessary, since in the classical
case, the condition $s\in(n/p,\fz)$ is necessary; see, for example, \cite{cs06}.

\subsection{Characterization by oscillations}\label{s7.2}

In this section,
we characterize our function spaces in terms of oscillations.

Let ${\mathbb P}_M$ be the \emph{set of all polynomials with degree less than $M$}.
By convention ${\mathbb P}_{-1}$ stands for the \emph{space} $\{0\}$.
We define, for all $(x,t) \in \R^{n+1}_+$, that
\[
{\rm osc}_u^M f(x,t):=
\inf_{P \in {\mathbb P}_M}
\left[\frac{1}{|B(x,t)|}\int_{B(x,t)}|f(y)-P(y)|^u\,dy\right]^{1/u}.
\]
We invoke the following estimates
from \cite{t92}.

\begin{lemma}
For any $f\in\cs'(\rn)$, let $\{f^j\}_{j=-1}^\fz$
be as in \eqref{7.1}.
Then there exists a positive constant $C$ such that the following estimates hold true:

{\rm (i)} for all $j\in\nn$ and $x\in\rn$,
\begin{equation}\label{7.9}
|f^j(x)-f^{j-1}(x)|\le C {\rm osc}^{M}_u f(x,\,2^{-j});
\end{equation}

{\rm (ii)} for all $j\in\zz_+$, $x\in\rn$ and $y\in B(x,2^{-j})$,
\begin{equation}\label{7.10}
\left|f^j(x)-\sum_{\|\alpha\|_1 \le M-1}
\frac{1}{\alpha!}D^\alpha f^j(x)(y-x)^\alpha\right|
\le C 2^{-jM}\sup_{z \in B(x,\,2^{-j})}
\sum_{\|\alpha\|_1=M}|D^\alpha f^j(z)|.
\end{equation}
\end{lemma}

\begin{proof}
Estimates \eqref{7.9} and \eqref{7.10}
appear, respectively, in \cite[p.\,188]{t92} and  \cite[p.\,182]{t92}.
\end{proof}

\begin{theorem}\label{t7.2}
Let $a, \alpha_1, \alpha_2, \alpha_3, \tau \in [0,\fz)$,
$u\in[1,\fz]$, $q\in(0,\,\fz]$
and $w\in\star-{\mathcal W}^{\alpha_3}_{\alpha_1,\alpha_2}$.
If $M\in\nn$, $\az_1\in (a,M)$ and \eqref{7.2} holds true, then,
for all $f \in \cs'(\rn)\cap L^1_\loc(\rn)$,
the following hold true:

{\rm(i)}
\begin{eqnarray*}
{\rm H}_1:={\rm J}^{(1)}_{a,w,\cl}(f)+\left\|\left\{\sup_{z \in \rn}
\frac{{\rm osc}_u^{M}f(\cdot+z,2^{-j})}{(1+2^j|z|)^a}
\right\}_{j \in \Z_+}\right\|_{\ell^q(\cl^w_\tau(\rn,\zz_+))}
\sim \|f\|_{B^{w,\tau}_{\cl,q,a}(\rn)}
\end{eqnarray*}
with the implicit positive constants
independent of $f$.

{\rm(ii)}
\begin{eqnarray*}
{\rm H}_2:={\rm J}^{(1)}_{a,w,\cl}(f)+\left\|\left\{\sup_{z \in \rn}
\frac{{\rm osc}_u^{M}f(\cdot+z,2^{-j})}{(1+2^j|z|)^a}
\right\}_{j \in \Z_+}\right\|_{\cl^w_\tau(\ell^q(\rn,\zz_+))}
\sim \|f\|_{F^{w,\tau}_{\cl,q,a}(\rn)}
\end{eqnarray*}
with the implicit positive constants
independent of $f$.

{\rm(iii)}
\begin{eqnarray*}
{\rm H}_3:={\rm J}^{(2)}_{a,w,\cl}(f)+\left\|\left\{\sup_{z \in \rn}
\frac{{\rm osc}_u^{M}f(\cdot+z,2^{-j})}{(1+2^j|z|)^a}
\right\}_{j \in \Z_+}\right\|_{\ell^q(\cn\cl^w_\tau(\rn,\zz_+))}
\sim \|f\|_{\cn^{w,\tau}_{\cl,q,a}(\rn)}
\end{eqnarray*}
with the implicit positive constants
independent of $f$.

{\rm(iv)}
\begin{eqnarray*}
{\rm H}_4:={\rm J}^{(2)}_{a,w,\cl}(f)+\left\|\left\{\sup_{z \in \rn}
\frac{{\rm osc}_u^{M}f(\cdot+z,2^{-j})}{(1+2^j|z|)^a}
\right\}_{j \in \Z_+}\right\|_{\ce\cl^w_\tau(\ell^q(\rn,\zz_+))}
\sim \|f\|_{\ce^{w,\tau}_{\cl,q,a}(\rn)}
\end{eqnarray*}
with the implicit positive constants
independent of $f$.
\end{theorem}

\begin{proof}We only prove (ii) since the proofs of others are similar.

By virtue of \eqref{7.5} and \eqref{7.9}, we have
\begin{eqnarray*}
{\rm H}_2\gs&&
\sup_{P\in\cq,|P|\ge1}\frac1{|P|^\tau}
\lf\|\chi_Pw_0\sup_{y\in\rn}\frac{|(f^0-f^{-1})(\cdot+y)|}{(1+|y|)^a}\r\|_{\cl(\rn)}\\
&&+\left\|
\left\{\sup_{z \in \rn}
\frac{|(f^j-f^{j-1})(\cdot+z)|}{(1+2^j|z|)^a}
\right\}_{j \in \Z_+}\right\|_{\cl^w_\tau(\ell^q(\rn,\zz_+))}
\sim \|f\|_{F^{w,\tau}_{\cl,q,a}(\rn)}
\end{eqnarray*}

For the reverse inequality, by \eqref{7.8} and Theorem \ref{t7.1}(ii),
we conclude that
\begin{eqnarray*}
&&\sup_{P\in\cq,|P|\ge1}\frac1{|P|^\tau}
\lf\|\chi_Pw_0\sup_{y\in\rn}\frac{|f(\cdot+y)|}{(1+|y|)^a}\r\|_{\cl(\rn)}
\ls\|f\|_{F^{w,\tau}_{\cl,q,a}(\rn)}.
\end{eqnarray*}
Therefore, we only need to prove that
\begin{eqnarray*}
\left\|\left\{\sup_{z \in \rn}
\frac{{\rm osc}_u^{M}f(\cdot+z,2^{-j})}{(1+2^j|z|)^a}
\right\}_{j \in \Z_+}\right\|_{\cl^w_\tau(\ell^q(\rn,\zz_+))}
&\ls&
\|f\|_{F^{w,\tau}_{\cl,q,a}(\rn)}.
\end{eqnarray*}
We use the estimate \cite[p.\,188,\,(11)]{t92}
with $k_0$ replaced by $T_0$.

Recall that the following estimate can be found
in \cite[p.\,188,\,(11)]{t92}:
\begin{eqnarray*}
&&{\rm osc}_u^{M}f(x+z,2^{-j})\\
&&\quad\lesssim
\sum_{l=1}^\infty\oint_{y\in B(x+z,2^{-j})}|f_{j+l}(y)|\,dy
+
\sup_{w \in B(x+z,C2^{-j})}
\left|\int_{\rn}D^\alpha T_0(y)f(w+2^{-j}y)\,dy\right|,
\end{eqnarray*}
where $C$ is a positive constant.
Consequently,
\begin{eqnarray}\label{7.11}
&&\frac{{\rm osc}_u^{M}f(x+z,2^{-j})}{(1+2^j|z|)^a}\nonumber\\
&&\quad\lesssim
\sum_{l=1}^\infty
\frac{\sup_{y\in B(x+z,2^{-j})}|f_{j+l}(y)|}{(1+2^j|z|)^a}\nonumber\\
&&\quad\quad
+\sup_{w \in B(x+z,C2^{-j})}
\frac{1}{(1+2^j|z|)^a}
\left|\int_{\rn}D^\alpha T_0(y)f(w+2^{-j}y)\,dy\right|.
\end{eqnarray}
Then by an argument similar to that used in the proof of Theorem \ref{t7.1},
for the second term on the right-hand side of \eqref{7.11}, we see that
\begin{eqnarray*}
&&\left\|\left\{\sup_{z\in\rn}
\sup_{w \in B(x+z,C2^{-j})}
\frac{1}{(1+2^j|z|)^a}
\left|\int_{\rn}D^\alpha T_0(y)f(w+2^{-j}y)
\,dy\right|\right\}_{j \in \Z_+}\right\|_{\cl^w_\tau(\ell^q(\rn,\zz_+))}\\
&&\quad\ls
\|f\|_{F^{w,\tau}_{\cl,q,a}(\rn)},
\end{eqnarray*}
We only need to consider the first term on the right-hand side of \eqref{7.11}.
Indeed, by $w\in\star-{\mathcal W}^{\alpha_3}_{\alpha_1,\alpha_2}$,
we have
$w_j(x)\ls 2^{-l\az_1}w_{j+l}(x)$, which, together with
$\az_1>a$ and \eqref{7.7}, implies that
\begin{eqnarray*}
&&\left\|\left\{\sum_{l=1}^\fz\sup_{z\in\rn}
\frac{\sup_{y\in B(x+z,2^{-j})}|f_{j+l}(y)|}{(1+2^j|z|)^a}
\right\}_{j \in \Z_+}\right\|_{\cl^w_\tau(\ell^q(\rn,\zz_+))}\\
&&\hs\ls
\left\|\left\{\sum_{l=1}^\fz2^{la}\sup_{z\in\rn}
\frac{\left|f_{j+l}(x+z)\right|}{(1+2^{j+l}|z|)^a}
\right\}_{j \in \Z_+}\right\|_{\cl^w_\tau(\ell^q(\rn,\zz_+))}\\
&&\hs\ls\left\{\sum_{l=1}^\fz 2^{la{\wz\theta}}\sup_{P\in\cq}\frac1{|P|^\tau}
\left\|\left(\sum_{j=(0\vee j_P)}^\fz\chi_P \lf[w_j\sup_{z\in\rn}
\frac{\left|f_{j+l}(x+z)\right|}{(1+2^{j+l}|z|)^a}\r]^q
\right)^{1/q}\right\|_{\cl(\rn)}^{\wz\theta}\right\}^{1/{\wz\theta}}\\
&&\hs\ls\left\{\sum_{l=1}^\fz 2^{-l(\az_1-a){\wz\theta}}\sup_{P\in\cq}\frac1{|P|^\tau}
\left\|\left(\sum_{j=(0\vee j_P)}^\fz\chi_P \lf[w_{j+l}\sup_{z\in\rn}
\frac{\left|f_{j+l}(x+z)\right|}{(1+2^{j+l}|z|)^a}\r]^q
\right)^{1/q}\right\|_{\cl(\rn)}^{\wz\theta}\right\}^{1/{\wz\theta}}\\
&&\hs\ls\sup_{P\in\cq}\frac1{|P|^\tau}
\left\|\left\{\sum_{j=(0\vee j_P)}^\fz\chi_P \lf[w_{j}\sup_{z\in\rn}
\frac{\left|f_{j}(x+z)\right|}{(1+2^{j}|z|)^a}\r]^q
\right\}^{1/q}\right\|_{\cl(\rn)}\sim\|f\|_{F^{w,\tau}_{\cl,q,a}(\rn)},
\end{eqnarray*}
where we chose $\wz\tz\in(0,\min\{\tz,q\})$.
This finishes the proof of Theorem \ref{t7.2}.
\end{proof}

If we further assume that \eqref{6.1} holds true,
then from Theorems \ref{t6.1} and \ref{t7.2},
we immediately deduce the following conclusions.
We omit the details.

\begin{corollary}\label{c7.3}
Let $\alpha_1, \alpha_2, \alpha_3, \tau $, $a$, $q$
and $w$ be as in Theorem \ref{t7.2}.
Assume that \eqref{6.1} and \eqref{7.2} hold true.
Let $\{{\rm H}_j\}_{j=1}^4$
be as in Theorem \ref{t7.2}.
Then the following hold true:

{\rm(i)}
$f\in B^{w,\tau}_{\cl,q,a}(\rn)$ if and only if
$f \in \cs'(\rn)\cap L^1_\loc(\rn)$
and ${\rm H}_1<\fz$;
moreover, $
{\rm H}_1\sim \|f\|_{B^{w,\tau}_{\cl,q,a}(\rn)}$
with the implicit positive constants
independent of $f$.

{\rm(ii)}
$f\in F^{w,\tau}_{\cl,q,a}(\rn)$ if and only if
$f \in \cs'(\rn)\cap L^1_\loc(\rn)$
and
${\rm H}_2<\fz$;
moreover, $
{\rm H}_2\sim \|f\|_{F^{w,\tau}_{\cl,q,a}(\rn)}$
with the implicit positive constants
independent of $f$.

{\rm(iii)}
$f\in \cn^{w,\tau}_{\cl,q,a}(\rn)$ if and only if
$f \in \cs'(\rn)\cap L^1_\loc(\rn)$
and
${\rm H}_3<\fz$;
moreover, $
{\rm H}_3\sim \|f\|_{\cn^{w,\tau}_{\cl,q,a}(\rn)}$
with the implicit positive constants
independent of $f$.

{\rm(iv)}
$f\in \ce^{w,\tau}_{\cl,q,a}(\rn)$ if and only if
$f \in \cs'(\rn)\cap L^1_\loc(\rn)$
and
${\rm H}_4<\fz$;
moreover, $
{\rm H}_4\sim \|f\|_{\ce^{w,\tau}_{\cl,q,a}(\rn)}$
with the implicit positive constants
independent of $f$.
\end{corollary}

Again, applying the Peetre maximal function
characterizations of
the spaces $B^s_{p,q}(\rn)$ and $F^s_{p,q}(\rn)$
(see, for example, \cite{u10}),
and Theorem \ref{t7.2},
we have the following corollary. Its proof is similar to that of
Corollary \ref{c7.2}. We omit the details.

\begin{corollary}\label{c7.4}
Let $M\in\nn$, $u\in[1,\fz]$ and $q\in(0,\fz]$.

{\rm(i)} Let $p\in(0,\fz)$, $a\in(n/\min\{p,q\},M)$
and $s\in(a,M-a)$. Then
$f\in F^s_{p,q}(\rn)$ if and only if
$f \in \cs'(\rn)\cap L^1_\loc(\rn)$ and
\begin{eqnarray*}
{\rm K}_1:=\|f^*_a\|_{L^p(\rn)}+
\left\|\left\|\left\{2^{js}\sup_{z \in \rn}
\frac{{\rm osc}_u^{M}f(\cdot+z,2^{-j})}{(1+2^j|z|)^a}
\right\}_{j \in \Z_+}\right\|_{\ell^q(\zz_+)}\right\|_{L^p(\rn)}<\fz.
\end{eqnarray*}
Moreover, ${\rm K}_1$ is equivalent to $\|f\|_{F^s_{p,q}(\rn)}$
with the equivalent positive constants
independent of $f$.

{\rm (ii)} Let $p\in(0,\fz]$, $a\in(n/p,M)$ and $s\in(a,M-a)$. Then
$f\in B^s_{p,q}(\rn)$ if and only if
$f \in\cs'(\rn)\cap L^1_\loc(\rn)$ and
\begin{eqnarray*}
{\rm K}_2:=\|f^*_a\|_{L^p(\rn)}+
\left\|\left\{\left\|2^{js}\sup_{z \in \rn}
\frac{{\rm osc}_u^{M}f(\cdot+z,2^{-j})}{(1+2^j|z|)^a}
\right\|_{L^p(\rn)}\right\}_{j \in \Z_+}\right\|_{\ell^q(\zz_+)}<\fz.
\end{eqnarray*}
Moreover, ${\rm K}_2$ is equivalent to $\|f\|_{B^s_{p,q}(\rn)}$
with the equivalent positive constants
independent of $f$.
\end{corollary}

Again, Corollary \ref{c7.4} can be seen as the Peetre maximal function version
of \cite[Theorem 3.5.1]{t92} in the case that $u\in[1,\fz]$.

\section{Isomorphisms between spaces}
\label{s8}

In this section, under some additional assumptions on $\cl(\rn)$,
we establish some isomorphisms between the considered spaces $\ala$.
First, in subsection \ref{s8.1},
we prove that if the parameter $a$ ia sufficiently large,
then the space $\ala$ coincides with the space $\alw$,
which is independent of the parameter $a$.
In subsection \ref{s8.30},
we give some further assumptions on $\cl(\rn)$ which ensure that
$\cl(\rn)$ coincides with $\ce_{\cl, 2,a}^{0,0}(\rn)$.
Finally, in subsection \ref{s8.2},
under some additional assumptions on $\cl(\rn)$,
we prove that the spaces $\ela$ and $\fla$ coincide.

\subsection{The role of the new parameter $a$}
\label{s8.1}

The new parameter $a$, which we added,
seems not to play any significant role.
We now consider some conditions
to remove the parameter $a$ from the definition of $\ala$.

Here we consider the following conditions.
\begin{assumption}
Let $\eta_{j,R}(x):= 2^{jn}(1+2^j|x|)^{-R}$
for $j \in \Z_+, \, R \gg 1$ and $x \in \rn$.

$(\cl7)$
There exist $R \gg 1$,
$r\in(0,\fz)$ and a positive constant $C(R,r)$,
depending on $R$ and $r$, such that,
for all $f\in\cl(\rn)$ and $j\in\zz_+$,
\[
\|w_j(\eta_{j,R}*|f|^r)^{1/r}\|_{\cl(\rn)}
\le C(R,r)
\|w_jf\|_{\cl(\rn)}.
\]

$(\cl7^\star)$
There exist $r\in(0,\fz)$ and a positive constant $C(r)$,
depending on $r$, such that,
for all $f\in\cl(\rn)$ and $j\in\zz_+$,
\[
\|w_j M(|f|^r)^{1/r}\|_{\cl(\rn)}
\le C(r)
\|w_j f\|_{\cl(\rn)}.
\]

$(\cl8)$ Let $q\in(0,\fz]$.
There exist $R \gg 1$, $r\in(0,\fz)$
and a positive constant $C(R,r,q)$,
depending on $R,r$ and $q$, such that,
for all $\{f_j\}_{j \in \N} \subset \cl(\rn)$,
\[
\|\{w_j(\eta_{j,R}*|f_j|^r)^{1/r}\}_{j \in \Z_+}\|_{\cl(\ell^q(\rn,\zz_+))}
\le C(R,r,q)
\|\{w_jf_j\}_{j \in \Z_+}\|_{\cl(\ell^q(\rn,\zz_+))}.
\]

$(\cl8^\star)$
Let $q\in(0,\fz]$.
There exist $r\in(0,\fz)$
and a positive constant $C(r,q)$,
depending on $r$ and $q$, such that,
for all $\{f_j\}_{j \in \N} \subset \cl(\rn)$,
\[
\|\{w_jM[|f_j|^r]^{1/r}\}_{j \in \Z_+}\|_{\cl(\ell^q(\rn,\zz_+))}
\le C(r,q)
\|\{w_jf_j\}_{j \in \Z_+}\|_{\cl(\ell^q(\rn,\zz_+))}.
\]
\end{assumption}

We now claim that in most cases
the parameter $a$ is auxiliary
by proving the following theorem.

\begin{theorem}\label{t8.1}
Let $\alpha_1, \alpha_2, \alpha_3, \tau \in [0,\infty)$,
$a\in (N_0+\alpha_3,\infty)$
 and $q\in(0,\,\fz]$,
where $N_0$ is as in $(\cl6)$.
Let $w\in{\mathcal W}^{\alpha_3}_{\alpha_1,\alpha_2}$,
$\tau\in[0,\fz)$ and $q\in(0,\,\fz]$.
Assume that $\Phi,\,\vz\in\cs(\rn)$ satisfy, respectively,
\eqref{1.1} and \eqref{1.2}.

{\rm (i)}
Assume that $(\cl7)$ holds true and, in addition, $a \gg 1$.
Then,
\begin{eqnarray*}
\|f\|_{B^{w,\tau}_{\cl,q,a}(\rn)}
&&\sim
\left\|\left\{\left[\int_{\rn}
\frac{2^{jn}|\vz_j*f(y)|^r}{(1+2^j|\cdot-y|)^{ar}}
\,dy\right]^{1/r}
\right\}_{j \in \Z_+}\right\|_{\ell^q(\cl^w_\tau(\rn,\Z_+))}\\
&&\sim\|\{\vz_j*f\}_{j \in \Z_+}\|_{\ell^q(\cl^w_\tau(\rn,\Z_+))}
\end{eqnarray*}
and
\begin{eqnarray*}
\|f\|_{\cn^{w,\tau}_{\cl,q,a}(\rn)}
&&\sim
\left\|
\left\{\left[\int_{\rn}
\frac{2^{jn}|\vz_j*f(y)|^r}{(1+2^j|\cdot-y|)^{ar}}
\,dy\right]^{1/r}
\right\}_{j \in \Z_+}\right\|_{\ell^q(\cn\cl^w_\tau(\rn,\Z_+))}\\
&&\sim
\|\{\vz_j*f\}_{j \in \Z_+}\|_{\ell^q(\cn\cl^w_\tau(\rn,\Z_+))}
\end{eqnarray*}
with the implicit positive constants independent of $f$.
In particular, if $(\cl7^\star)$ holds true,
then the above equivalences hold true.

{\rm (ii)}
Assume that $(\cl8)$ holds true and, in addition, $a \gg 1$.
Then
\begin{eqnarray}\label{9.x3}
\|f\|_{F^{w,\tau}_{\cl,q,a}(\rn)}
&&\sim
\left\|
\left\{\left[\int_{\rn}
\frac{2^{jn}|\vz_j*f(y)|^r}{(1+2^j|\cdot-y|)^{ar}}
\,dy\right]^{1/r}
\right\}_{j \in \Z_+}\right\|_{\cl^w_\tau(\ell^q(\rn,\Z_+))}\noz\\
&&\sim
\|\{\vz_j*f\}_{j \in \Z_+}\|_{\cl^w_\tau(\ell^q(\rn,\Z_+))}
\end{eqnarray}
and
\begin{eqnarray*}
\|f\|_{\ce^{w,\tau}_{\cl,q,a}(\rn)}
&&\sim
\left\|
\left\{\left[\int_{\rn}
\frac{2^{jn}|\vz_j*f(y)|^r}{(1+2^j|\cdot-y|)^{ar}}
\,dy\right]^{1/r}
\right\}_{j \in \Z_+}\right\|_{\ce\cl^w_\tau(\ell^q(\rn,\Z_+))}\\
&&\sim
\|\{\vz_j*f\}_{j \in \Z_+}\|_{\ce\cl^w_\tau(\ell^q(\rn,\Z_+))}
\end{eqnarray*}
with the implicit positive constants independent of $f$.
In particular, if $(\cl8^\star)$ holds true,
then the above equivalences hold true.
\end{theorem}

Motivated by Theorem \ref{t8.1},
let us define
\begin{eqnarray*}
&&\|f\|_{B^{w,\tau}_{\cl,q}(\rn)}
:=
\|\{\vz_j*f\}_{j \in \Z_+}\|_{\ell^q(\cl^w_\tau(\rn,\Z_+))}\\
&&\|f\|_{\cn^{w,\tau}_{\cl,q}(\rn)}
:=
\|\{\vz_j*f\}_{j \in \Z_+}\|_{\ell^q(\cn\cl^w_\tau(\rn,\Z_+))}\\
&&\|f\|_{F^{w,\tau}_{\cl,q}(\rn)}
:=
\|\{\vz_j*f\}_{j \in \Z_+}\|_{\cl^w_\tau(\ell^q(\rn,\Z_+))}\\
&&\|f\|_{\ce^{w,\tau}_{\cl,q}(\rn)}
:=
\|\{\vz_j*f\}_{j \in \Z_+}\|_{\ce\cl^w_\tau(\ell^q(\rn,\Z_+))}
\end{eqnarray*}
for all $f \in \cs'(\rn)$
as long as the assumptions of Theorem \ref{t8.1} are fulfilled.

\begin{lemma}\label{l8.1}
Let $\alpha_1, \alpha_2, \alpha_3, \tau \in [0,\infty)$, $a\in (N_0+\alpha_3,\infty)$,
$q\in(0,\,\fz]$ and $\varepsilon\in(0,\fz)$.
Assume that $\Phi,\,\vz\in\cs(\rn)$ satisfy, respectively,
\eqref{1.1} and \eqref{1.2}.

{\rm (i)} Then,
for all $f\in\cs'(\rn)$
\begin{eqnarray}
\label{8.1}
&&\|f\|_{B^{w,\tau}_{\cl,q,a}(\rn)}
\gtrsim
\left\|
\left\{\left[\int_{\rn}
\frac{2^{jn}|\vz_j*f(y)|^r}{(1+2^j|\cdot-y|)^{ar+n+\varepsilon}}\,dy
\right]^{1/r}
\right\}_{j \in \Z_+}\right\|_{\ell^q(\cl^w_\tau(\rn,\Z_+))},\\
\label{8.2}
&&\|f\|_{B^{w,\tau}_{\cl,q,a}(\rn)}
\lesssim
\left\|
\left\{\left[\int_{\rn}
\frac{2^{jn}|\vz_j*f(y)|^r}{(1+2^j|\cdot-y|)^{ar}}\,dy
\right]^{1/r}
\right\}_{j \in \Z_+}\right\|_{\ell^q(\cl^w_\tau(\rn,\Z_+))},\\
\label{8.3}
&&\|f\|_{\cn^{w,\tau}_{\cl,q,a}(\rn)}
\gtrsim
\left\|
\left\{\left[\int_{\rn}
\frac{2^{jn}|\vz_j*f(y)|^r}{(1+2^j|\cdot-y|)^{ar+n+\varepsilon}}\,dy
\right]^{1/r}
\right\}_{j \in \Z_+}\right\|_{\ell^q(\cn\cl^w_\tau(\rn,\Z_+))}
\end{eqnarray}
and
\begin{eqnarray}
\label{8.4}
&&\|f\|_{\cn^{w,\tau}_{\cl,q,a}(\rn)}
\lesssim
\left\|
\left\{\left[\int_{\rn}
\frac{2^{jn}|\vz_j*f(y)|^r}{(1+2^j|\cdot-y|)^{ar}}\,dy
\right]^{1/r}
\right\}_{j \in \Z_+}\right\|_{\ell^q(\cn\cl^w_\tau(\rn,\Z_+))},
\end{eqnarray}
where $\vz_0$ is understood as $\Phi$
and the implicit positive constants independent of $f$.

{\rm (ii)} Then,
for all $f\in\cs'(\rn)$
\begin{eqnarray}
\label{8.5}
&&\|f\|_{F^{w,\tau}_{\cl,q,a}(\rn)}
\gtrsim
\left\|
\left\{\left[\int_{\rn}
\frac{2^{jn}|\vz_j*f(y)|^r}{(1+2^j|\cdot-y|)^{ar+n+\varepsilon}}\,dy
\right]^{1/r}
\right\}_{j \in \Z_+}\right\|_{\cl^w_\tau(\ell^q(\rn,\Z_+))},\\
\label{8.6}
&&\|f\|_{F^{w,\tau}_{\cl,q,a}(\rn)}
\lesssim
\left\|
\left\{\left[\int_{\rn}
\frac{2^{jn}|\vz_j*f(y)|^r}{(1+2^j|\cdot-y|)^{ar}}\,dy
\right]^{1/r}
\right\}_{j \in \Z_+}\right\|_{\cl^w_\tau(\ell^q(\rn,\Z_+))},\\
\label{8.7}
&&\|f\|_{\ce^{w,\tau}_{\cl,q,a}(\rn)}
\gtrsim
\left\|
\left\{\left[\int_{\rn}
\frac{2^{jn}|\vz_j*f(y)|^r}{(1+2^j|\cdot-y|)^{ar+n+\varepsilon}}\,dy
\right]^{1/r}
\right\}_{j \in \Z_+}\right\|_{\ce\cl^w_\tau(\ell^q(\rn,\Z_+))}
\end{eqnarray}
and
\begin{eqnarray}
\label{8.8}
&&\|f\|_{\ce^{w,\tau}_{\cl,q,a}(\rn)}
\lesssim
\left\|
\left\{\left[\int_{\rn}
\frac{2^{jn}|\vz_j*f(y)|^r}{(1+2^j|\cdot-y|)^{ar}}\,dy
\right]^{1/r}
\right\}_{j \in \Z_+}\right\|_{\ce\cl^w_\tau(\ell^q(\rn,\Z_+))},
\end{eqnarray}
where $\vz_0$ is understood as $\Phi$
and the implicit positive constants independent of $f$.
\end{lemma}

\begin{proof}
Estimates \eqref{8.1}, \eqref{8.3}, \eqref{8.5} and \eqref{8.7}
are immediately deduced from the definition,
while \eqref{8.2}, \eqref{8.4}, \eqref{8.6} and \eqref{8.8}
depend on the following estimate:
By \cite[(2.29)]{u10},
we see that, for all $t\in[1,2]$, $N \gg 1$, $r\in(0,\fz)$, $\ell\in\nn$ and $x\in\rn$
\begin{equation*}
\lf[(\phi_{2^{-\ell}t}^*f)_a(x)\r]^r
\lesssim
\sum_{k=0}^\infty 2^{-kNr}2^{(k+\ell)n}
\int_\rn\frac{|((\phi_{k+\ell})_t*f)(y)|^r}{(1+2^\ell|x-y|)^{ar}}\,dy.
\end{equation*}%
In particular, when $l=0$, for all $x\in\rn$, we have
\begin{equation}\label{8.10}
(\phi_t^*f)_a(x)
\ls
\lf[\sum_{k=0}^\fz2^{-kNr}2^{kn}
\int_\rn\frac{|((\phi_{k})_t*f)(y)|^r}
{(1+|x-y|)^{ar}}\,dy\r]^{1/r}.
\end{equation}
If we combine Lemma \ref{l2.3} and \eqref{8.10},
then we obtain the desired result,
which completes the proof of Lemma \ref{l8.1}.
\end{proof}

The heart of the matter of the proof of Theorem \ref{t8.1}
is to prove the following dilation estimate.
The next lemma translates
the assumptions $(\cl7)$ and $(\cl8)$ into the one
of our function spaces.

\begin{lemma}\label{l8.2}
Let $\{F_j\}_{j\in\zz_+}$ be a
sequence of positive measurable functions on $\rn$.

{\rm (i)}
If $(\cl7)$ holds true, then
\begin{eqnarray*}
&&
\|\{\left(\eta_{j,2R}*[F_j{}^r]\right)^{1/r}\}_{j \in \Z_+}
\|_{\ell^q(\cl^w_\tau(\rn,\Z_+))}
\lesssim
\|\{F_j\}_{j \in\Z_+}\|_{\ell^q(\cl^w_\tau(\rn,\Z_+))}
\end{eqnarray*}%
and
\begin{eqnarray*}
&&
\|\{\left(\eta_{j,2R}*[F_j{}^r]\right)^{1/r}\}_{j \in \Z_+}
\|_{\ell^q(\cn\cl^w_\tau(\rn,\Z_+))}
\lesssim
\|\{F_j\}_{j \in\Z_+}\|_{\ell^q(\cn\cl^w_\tau(\rn,\Z_+))}
\end{eqnarray*}%
with the implicit positive constants independent of $\{F_j\}_{j\in\zz_+}$.

{\rm (ii)} If $(\cl8)$ holds true, then
\begin{eqnarray}
\label{8.13}
\|\{\left(\eta_{j,2R}*[F_j{}^r]\right)^{1/r}\}_{j \in \Z_+}
\|_{\cl^w_\tau(\ell^q(\rn,\Z_+))}
\lesssim
\|\{F_j\}_{j \in\Z_+}\|_{\cl^w_\tau(\ell^q(\rn,\Z_+))}\end{eqnarray}
and
\begin{eqnarray*}
\|\{\left(\eta_{j,2R}*[F_j{}^r]\right)^{1/r}\}_{j \in \Z_+}
\|_{\ce\cl^w_\tau(\ell^q(\rn,\Z_+))}
\lesssim
\|\{F_j\}_{j \in\Z_+}\|_{\ce\cl^w_\tau(\ell^q(\rn,\Z_+))}
\end{eqnarray*}
with the implicit positive constants independent of $\{F_j\}_{j\in\zz_+}$.
\end{lemma}

\begin{proof}
Due to similarity,
we only prove \eqref{8.13}.

For all sequences $F=\{F_j\}_{j\in\zz_+}$ of positive measurable functions on $\rn$,
define
$$\|F\|:=\|\{F_j\}_{j\in\zz_+}\|_{\cl_\tau^w(\ell^q(\rn,\zz_+))}.$$
Then, $\|\cdot\|$ is still a quasi-norm.
By the Aoki-Rolewicz theorem (see \cite{a42,r57}),
we know that there exists a quasi-norm $\qn\cdot\qn$ and $\wz\tz\in(0,1]$
such that, for all sequences $F$ and $G$, $\|F\|\sim\qn F\qn$ and
$$\qn F+G\qn^{\wz\tz}\le\qn F\qn^{\wz\tz}+\qn G\qn^{\wz\tz}. $$
Therefore, we see that
\begin{eqnarray}\label{8.15}
&&\hspace{-0.2cm}\left\|\lf\{
\left[\sum_{l=0}^\infty\eta_{k,2R}*\left( G_{k,l}\right)^{r}\right]^{1/r}\r\}_{k\in\zz_+}
\right\|_{\cl^w_\tau(\ell^q(\rn,\Z_+))}^{\wz\theta}\noz\\
&&\sim\left|\!\lf|\!\lf|\lf\{
\left[\sum_{l=0}^\infty\eta_{k,2R}*\left( G_{k,l}\right)^{r}\right]^{1/r}\r\}_{k\in\zz_+}
\right|\!\r|\!\r|^{\wz\theta}\noz\\
&&\lesssim
\sum_{l=0}^\infty
\left|\!\lf|\!\lf|\lf\{\left[\eta_{k,2R}*G_{k,l}{}^{r}\right]^{1/r}\r\}_{k\in\zz_+}
\right|\!\r|\!\r|^{\wz\theta}
\sim
\sum_{l=0}^\infty
\lf\|\lf\{\left[\eta_{k,2R}*G_{k,l}{}^{r}\right]^{1/r}\r\}_{k\in\zz_+}
\r\|_{\cl^w_\tau(\ell^q(\rn,\Z_+))}^{\wz\theta}\quad\quad
\end{eqnarray}
for all $\{G_{k,l}\}_{k,l \in \Z_+}$ of positive measurable functions.

We fix a dyadic cube $P$.
Our goal is to prove
\begin{equation}\label{8.16}
{\rm I}:=
\left\|\lf(\sum_{k=j_P\vee0}^\fz
\chi_Pw_k{}^q\left[\eta_{k,2R}*(F_k{}^r)\right]^{q/r}
\r)^{1/q}
\right\|_{\cl(\rn)}
\lesssim
|P|^\tau\|\{F_j\}_{j \in\Z_+}\|_{\cl^w_\tau(\ell^q(\rn,\Z_+))}
\end{equation}
with the implicit positive constant independent of
$\{F_j\}_{j \in\Z_+}$ and $P$.

By using \eqref{8.15},
we conclude that
\begin{eqnarray*}
{\rm I}
\ls
\left\{
\sum_{m \in \Z^n}
\left[
\left\|\lf(\sum_{k=j_P\vee0}^\fz
\chi_Pw_k{}^q[
\eta_{k,2R}*\left(\chi_{\ell(P)m+P}F_k\right)^r]^{q/r}
\r)^{1/q}
\right\|_{\cl(\rn)}\right]^{\min(\theta,q,r)}
\right\}^{\frac{1}{\min(\theta,q,r)}}.
\end{eqnarray*}
A geometric observation shows that
\[
\frac12|m|\ell(P) \le |x-y| \le 2n|m|\ell(P),
\]
whenever $x \in P$ and $y \in \ell(P)m+P$
with $|m| \ge 2$.
Consequently,
for all $m \in \Z^n$ and $x\in\rn$,
we have
\begin{eqnarray*}
\eta_{k,2R}\ast(\chi_{\ell(P)m+P}F_k)^r(x)
&&=\int_{\ell(P)m+P}2^{kn}(1+2^k|x-y|)^{-R}(1+2^k|x-y|)^{-R}F_k(y)^r\,dy\\
&&\lesssim
\frac{1}{|m|^R}
\int_{\ell(P)m+P}2^{kn}[1+2^k|m|\ell(P)]^{-R}F_k(y)^r\,dy\\
&&\lesssim
\frac{1}{|m|^{R}}
\eta_{j_P,R}*[\chi_{\ell(P)m+P}F_k{}^r](x).
\end{eqnarray*}
By this and $(\cl8)$,
we further conclude that
\begin{eqnarray*}
{\rm I}&&\lesssim
\left\{
\sum_{m \in \Z^n}
\left[
\left\|\lf(\sum_{k=j_P\vee0}^\fz
\left[\eta_{k,2R}*(\chi_{\ell(P)m+P}F_k)^r\right]^{q/r}
\r)^{1/q}
\right\|_{\cl(\rn)}\right]^{\min(\theta,q,r)}
\right\}^{\frac{1}{\min(\theta,q,r)}}\\
&&\lesssim
|P|^\tau\|\{F_j\}_{j \in\Z_+}\|_{\cl^w_\tau(\ell^q(\rn,\Z_+))}.
\end{eqnarray*}
Thus, \eqref{8.16} holds true,
which completes the proof of Lemma \ref{l8.2}.
\end{proof}

\begin{proof}[Proof of Theorem \ref{t8.1}]
Due to similarity,
we only prove the estimates for $F^{w,\tau}_{\cl,q,a}(\rn)$.

By Lemma \ref{l8.1},
we have
\begin{equation}\label{8.17}
\|f\|_{F^{w,\tau}_{\cl,q,a}(\rn)}
\lesssim
\left\|
\left\{\left[\int_{\rn}
\frac{2^{jn}|\vz_j*f(y)|^r}{(1+2^j|\cdot-y|)^{ar}}
\,dy\right]^{1/r}
\right\}_{j \in \Z_+}\right\|_{\cl^w_\tau(\ell^q(\rn,\Z_+))}.
\end{equation}
Observe that the right-hand side of \eqref{8.17}
is just
$$
\left\|
\left\{\left(\eta_{j,ar}*[|\vz_j*f(\cdot)|^r]\right)^{1/r}
\right\}_{j \in \Z_+}\right\|_{\cl^w_\tau(\ell^q(\rn,\Z_+))}.$$
By Lemma \ref{l8.2},
we see that
\begin{equation}\label{8.18}
\left\|
\left\{\left[\int_{\rn}
\frac{2^{jn}|\vz_j*f(y)|^r\,dy}{(1+2^j|\cdot-y|)^{ar}}
\right]^{1/r}
\right\}_{j \in \Z_+}\right\|_{\cl^w_\tau(\ell^q(\rn,\Z_+))}
\lesssim
\|f\|_{F^{w,\tau}_{\cl,q}(\rn)}.
\end{equation}
Also, it follows trivially, from the definition
of $F^{w,\tau}_{\cl,q,a}(\rn)$, that
\begin{equation}\label{8.19}
\|f\|_{F^{w,\tau}_{\cl,q}(\rn)}
\le
\|f\|_{F^{w,\tau}_{\cl,q,a}(\rn)}.
\end{equation}
Combining \eqref{8.17}, \eqref{8.18} and \eqref{8.19},
we obtain \eqref{9.x3},
which completes
the proof of Theorem \ref{t8.1}.
\end{proof}

\begin{proposition}\label{p8.1}
Let $q\in[1,\fz]$.
Assume that
$\theta=1$ in the assumption $(\cl3)$
and, additionally,
there exist some $M\in(0,\fz)$ and a positive constant $C(M)$,
depending on $M$, such that,
for all $f\in\cl(\rn)$ and $x\in\rn$,
\begin{equation}\label{8.20}
\|f(\cdot-x)\|_{\cl(\rn)} \le C(M) (1+|x|)^M\|f\|_{\cl(\rn)}.
\end{equation}
Then,
whenever $a \gg 1$,
\begin{eqnarray*}
\|f\|_{B^{w,\tau}_{\cl,q,a}(\rn)}
\sim
\left\|
\left\{\left[\int_{\rn}
\frac{2^{jn}|\vz_j*f(y)|^r}{(1+2^j|\cdot-y|)^{ar}}\,dy
\right]^{1/r}
\right\}_{j \in \Z_+}\right\|_{\ell^q(\cl^w_\tau(\rn,\Z_+))}
\sim
\|f\|_{B^{w,\tau}_{\cl,q}(\rn)}
\end{eqnarray*}
and
\begin{eqnarray*}
\|f\|_{\cn^{w,\tau}_{\cl,q,a}(\rn)}
\sim
\left\|
\left\{\left[\int_{\rn}
\frac{2^{jn}|\vz_j*f(y)|^r}{(1+2^j|\cdot-y|)^{ar}}\,dy
\right]^{1/r}
\right\}_{j \in \Z_+}\right\|_{\ell^q(\cn\cl^w_\tau(\rn,\Z_+))}
\sim
\|f\|_{\cn^{w,\tau}_{\cl,q}(\rn)}
\end{eqnarray*}
with the implicit constants independent of $f$.
\end{proposition}

It is not clear whether the counterpart
of Proposition \ref{p8.1}
for $\ce^{w,\tau}_{\cl,q,a}(\rn)$ and
$F^{w,\tau}_{\cl,q,a}(\rn)$ is available or not.

\begin{proof}[Proof of Proposition \ref{p8.1}]
We concentrate on the $B$-scale,
the proof for the $\cn$-scale being similar.
By Theorem \ref{t8.1}, we see that
\begin{eqnarray*}
&&\|f\|_{B^{w,\tau}_{\cl,q,a}(\rn)}\\
&&\hs\ls
\sup_{P\in\mathcal{Q}(\rn)}\frac1{|P|^\tau}
\lf\{\sum_{k=j_P\vee0}^\fz
\left\|\chi_P\left[w_k\left(\int_{\rn}
\frac{2^{kn}|\vz_k*f(y)|^r}{(1+2^k|\cdot-y|)^{ar+n+\varepsilon}}\,dy
\right)^{1/r}\right]\right\|_{\cl(\rn)}^q
\r\}^{1/q}.
\end{eqnarray*}
Now that $\theta=1$,
we are in the position of using the triangle inequality
to have
\begin{eqnarray*}
\|f\|_{B^{w,\tau}_{\cl,q,a}(\rn)}
\ls
\sup_{P\in\mathcal{Q}(\rn)}\frac1{|P|^\tau}
\lf\{\sum_{k=j_P\vee0}^\fz
\|\chi_Pw_k[\vz_k*f]\|_{\cl(\rn)}^q
\r\}^{1/q}
\end{eqnarray*}
whenever $a \gg 1$.
The reverse inequality being trivial,
we obtain the desired estimates,
which completes the proof of Proposition \ref{p8.1}.
\end{proof}

To conclude this section,
with Theorems \ref{t4.3} and \ref{t8.1} proved,
we have already obtained the biorthogonal wavelet decompositions
of Morrey spaces;
see also Section \ref{s10.1} below.

\subsection{Identification of the space $\cl(\rn)$}
\label{s8.30}

The following lemma is a natural extension
with $|\cdot|$ in the definition of $\|f\|_{\cl(\rn)}$
replaced by $\ell^2(\Z)$.
In this subsection,
we \emph{always assume} that
$\theta=1$ in $(\cl3)$
and that, for some finite positive constant $C(E)$,
depending on $E$, but not on $f$, such that,
for any $f \in \cl(\rn)$ and any set $E$ of finite measure,
\begin{equation}\label{9.x2}
\int_E|f(x)|\,dx \le C(E)\|f\|_{\cl(\rn)}.
\end{equation}
In this case $\cl(\rn)$ is a Banach space of functions
and the \emph{dual space} $\cl'(\rn)$ can be defined.

\begin{theorem}\label{t8.30}
Let $\cl$ be as above, $\psi,\vz\in\cs(\rn)$ satisfy,
respectively, \eqref{1.1} and \eqref{1.2},
and $N\in\nn$.
Suppose that $a\in(N,\fz)$
and that
\begin{equation}\label{eq:8.30}
(1+|x|)^{-N} \in \cl(\rn) \cap \cl'(\rn).
\end{equation}
Assume, in addition, that
there exists a positive constant $C$ such that,
for any finite sequence $\{\varepsilon_k\}_{k=1}^{k_0}\st\{-1,1\}$,
$f \in \cl(\rn)$ and $g \in \cl'(\rn)$,
\begin{equation}\label{eq:8.31}
\left\{\begin{array}{l}
\left\|\psi*f+\dis\sum_{k=1}^{k_0} \varepsilon_k\varphi_k*f\right\|_{\cl(\rn)}
\le C\|f\|_{\cl(\rn)}, \\
\left\|\psi*g+\dis\sum_{k=1}^{k_0} \varepsilon_k\varphi_k*g\right\|_{\cl'(\rn)}
\le C\|g\|_{\cl'(\rn)}.
\end{array}\r.
\end{equation}
Then,
$\cl(\rn)$ and $\cl'(\rn)$
are embedded into $\cs'(\rn)$,
and $\cl(\rn)$ and $\ce^{0,0}_{\cl,2,a}(\rn)$ coincide.
\end{theorem}

\begin{proof}
The fact that
$\cl(\rn)$ and $\cl'(\rn)$
are embedded into $\cs'(\rn)$
is a simple consequence
of (\ref{9.x2}) and \eqref{eq:8.30}.
By using the Rademacher sequence
$\{r_j\}_{j=1}^\infty$,
we obtain
\begin{eqnarray*}
\left\|\left(
\sum_{j=1}^\infty |\varphi_j*f|^2\right)^{1/2}\right\|_{\cl(\rn)}
&&=
\lim_{k_0 \to \infty}
\left\|\left(
\sum_{j=1}^{k_0} |\varphi_j*f|^2\right)^{1/2}\right\|_{\cl(\rn)}\\
&&\lesssim
\lim_{k_0 \to \infty}
\left\|\sum_{j=1}^{k_0}\int_0^1 |r_j(t)\varphi_j*f|\,dt\right\|_{\cl(\rn)},
\end{eqnarray*}
which, together with the assumption $a>N$, Theorem \ref{t8.1} and \eqref{eq:8.31},
implies that
\[
\|f\|_{\ce^{0,0}_{\cl,2,a}(\rn)}\sim
\left\|\left(|\psi*f|^2+
\sum_{j=1}^\infty |\varphi_j*f|^2\right)^{1/2}\right\|_{\cl(\rn)}
\lesssim \|f\|_{\cl(\rn)}.
\]

If we fix $g \in C^\infty_{\rm c}(\rn)$,
we then see that
\[
\int_{\rn}f(x)g(x)\,dx
=
\int_{\rn}\psi*f(x)\psi*g(x)\,dx
+
\sum_{j=1}^\infty
\int_{\rn}\varphi_j*f(x)\varphi_j*g(x)\,dx.
\]
From Theorem \ref{t8.1},
the H\"{o}lder inequality and the duality,
we deduce that
\[
\|f\|_{\cl(\rn)}
\lesssim
\sup\left\{
\|f\|_{\ce^{0,0}_{\cl,2,a}(\rn)}
\|g\|_{\ce^{0,0}_{\cl',2,a}(\rn)}
\,:\ g \in C^\infty_{\rm c}(\rn), \ \|g\|_{\cl'(\rn)}=1\right\}.
\]
Since we have proved that $\cl'(\rn)$
is embedded into $\ce^{0,0}_{\cl',2,a}(\rn)$,
by the second estimate of (\ref{eq:8.31}),
we conclude that
\[
\|f\|_{\cl(\rn)}
\lesssim
\|f\|_{\ce^{0,0}_{\cl,2,a}(\rn)}.
\]
Thus, the reverse inequality was proved,
which completes the proof of Theorem \ref{t8.30}.
\end{proof}

Let $\cl(\rn)$ be a Banach space of functions
and define
\[
\cl^{p}(\rn):=\{f:{\mathbb R}^n \to {\mathbb C}\,:\,
f\mbox{ is measurable and }|f|^p \in \cl(\rn) \}
\]
for $p\in(0,\infty)$,
and
$\|f\|_{\cl^p(\rn)}:=\|\,|f|^p\,\|_{\cl(\rn)}^{1/p}$
for all $f \in \cl^p(\rn)$.
A criterion for (\ref{eq:8.31})
is given in the book \cite{crmape11}.
Here we invoke the following fact.

\begin{proposition}\label{p9.31}
Let $\cl(\rn)$ be a Banach space of functions
such that $\cl^{p}(\rn)$
is a Banach space of functions
and that
the maximal operator $M$ is bounded on $(\cl^{p}(\rn))'$
for some $p\in(1,\fz)$.

Assume, in addition, that
${\mathcal Z}$ is a set of pairs
of positive measurable functions $(f,g)$
such that,
for all $p_0 \in (1,\infty)$ and $w \in A_{p_0}(\rn)$,
\begin{equation}\label{eq:8.32}
\int_{\rn}[f(x)]^{p_0}w(x)\,dx
\lesssim_{A_{p_0}(w)}
\int_{\rn}[g(x)]^{p_0}w(x)\,dx
\end{equation}
with the implicit positive constant depending on the weight constant
$A_{p_0}(w)$ of the weight $w$, but not on $(f,g)$.
Then
$\|f\|_{\cl(\rn)} \lesssim \|g\|_{\cl(\rn)}$
holds true for all $(f,g) \in {\mathcal Z}$
with the implicit positive constant independent on $(f,g)$.

\end{proposition}

A direct consequence of this proposition
is a criterion of $(\ref{eq:8.31})$.

\begin{theorem}\label{t8.32}
Let $\cl(\rn)$ be a Banach space of functions
such that $\cl^{p}(\rn)$ and $(\cl')^{p}(\rn)$
are Banach spaces of functions
and that
the maximal operator $M$ is bounded
on $(\cl^{p}(\rn))'$ and $((\cl')^{p}(\rn))'$
for some $p \in (1,\infty)$.
Then $(\ref{eq:8.31})$ holds true.
In particular,
if $a>N$
and $(1+|x|)^{-N} \in \cl(\rn) \cap \cl'(\rn)$,
then
$\cl(\rn)$ and $\cl'(\rn)$
are embedded into $\cs'(\rn)$,
and $\cl(\rn)$ and $\ce^{0,0}_{\cl,2,a}(\rn)$
coincide.
\end{theorem}

\begin{proof}
We have only to check (\ref{eq:8.31}).
Let
\[
{\mathcal Z}=
\left\{
\left(\psi*f+\sum_{k=1}^N \varepsilon_k\varphi_k*f,f\right)\,:\,
f \in \cl(\rn), \,
N \in \N, \,
\{\varepsilon_k\}_{k \in {\mathbb N}} \subset \{-1,1\}
\right\}.
\]
Then (\ref{eq:8.32}) holds true
according to the well-known Calder\'{o}n-Zygmund theory
(see \cite[Chapter 7]{du}, for example).
Thus, (\ref{eq:8.31}) holds true,
which completes the proof of Theorem \ref{t8.32}.
\end{proof}

\subsection{$F$-spaces and $\ce$-spaces}
\label{s8.2}

As we have seen in \cite{syy},
when $\cl(\rn)$ is the Morrey space
$\cm^p_q(\rn)$,
we have
$\ce_{\cl,q,a}^{s,\tau}(\rn)=F_{\cl,q,a}^{s,\tau}(\rn)$
with norm equivalence.
The same thing happens
under some mild assumptions
$(\ref{8.21})$ and $(\ref{8.22})$ below.
\begin{theorem}\label{t8.4}
Let $a\in (N_0+\alpha_3,\infty)$, $q\in(0,\fz]$ and $s\in\rr$.
Assume that $\cl(\rn)$ satisfies the assumption $(\cl8)$
and that there exist
positive constants $C$ and $\tau_0$ such that, for all $P\in\cq(\rn)$,
\begin{equation}\label{8.21}
C^{-1}\|\chi_P\|_{\cl(\rn)}\le|P|^{\tau_0}\le C\|\chi_P\|_{\cl(\rn)}.
\end{equation}
Then for all $\tau\in [0,\tau_0)$,
$\ce_{\cl,q,a}^{s,\tau}(\rn)=F_{\cl,q,a}^{s,\tau}(\rn)$
with equivalent norms.
\end{theorem}

\begin{proof}
By the definition of the norms
$\|\cdot\|_{\ce_{\cl,q,a}^{s,\tau}(\rn)}$
and
$\|\cdot\|_{F_{\cl,q,a}^{s,\tau}(\rn)}$, we need only to show that
\begin{equation}\label{8.222}
F^{s,\tau}_{\cl,q,a}(\rn)\hookrightarrow\ce^{s,\tau}_{\cl,q,a}(\rn).
\end{equation}
In view of the atomic decomposition theorem
(see Theorem \ref{t4.1}),
instead of proving (\ref{8.222}) directly,
we can reduce the matters
to the level of sequence spaces.
So we have only to prove
\begin{equation*}
f^{s,\tau}_{\cl,q,a}(\rn)\hookrightarrow e^{s,\tau}_{\cl,q,a}(\rn).
\end{equation*}%

To this end, by $(\cl8)$,
\begin{eqnarray*}
\|\lambda\|_{e^{s,\tau}_{\cl,q,a}(\rn)}
&&:=
\sup_{P\in\cq(\rn)}\frac1{|P|^\tau}\lf\|\lf[\sum_{j=0}^\fz\lf(
\chi_P 2^{js}
\sup_{y\in\rn}\frac{\sum_{k \in \Z^n}|\lambda_{jk}
|\chi_{Q_{jk}}(\cdot+y)}{(1+2^j|y|)^a}
\r)^q\r]^{1/q}\r\|_{\cl(\rn)}\\
&&\sim
\sup_{P\in\cq(\rn)}\frac1{|P|^\tau}\lf\|\lf[\sum_{j=0}^\fz\lf(
\chi_P 2^{js}
\sum_{k \in \Z^n}|\lambda_{jk}|\chi_{Q_{jk}}
\r)^q\r]^{1/q}\r\|_{\cl(\rn)}.
\end{eqnarray*}
Similarly, by $(\cl8)$, we also conclude that
\begin{eqnarray*}
\|\lambda\|_{f^{s,\tau}_{\cl,q,a}(\rn)}
&&:=
\sup_{P\in\cq(\rn)}\frac1{|P|^\tau}\lf\|\lf[\sum_{j=j_P\vee0}^\fz\lf(
\chi_P 2^{js}
\sup_{y\in\rn}\frac{\sum_{k \in \Z^n}|\lambda_{jk}
|\chi_{Q_{jk}}(\cdot+y)}{(1+2^j|y|)^a}
\r)^q\r]^{1/q}\r\|_{\cl(\rn)}\\
&&\sim
\sup_{P\in\cq(\rn)}\frac1{|P|^\tau}\lf\|\lf[\sum_{j=j_P\vee0}^\fz\lf(
\chi_P 2^{js}
\sum_{k \in \Z^n}|\lambda_{jk}|\chi_{Q_{jk}}
\r)^q\r]^{1/q}\r\|_{\cl(\rn)}.
\end{eqnarray*}
Then, it suffices to show that, for all dyadic cubes $P$
with $j_P \ge 1$,
$$\mi:=\frac1{|P|^\tau}\lf\|\lf[\sum_{j=0}^{j_P-1}\lf(
\chi_P 2^{js}
\sum_{k \in \Z^n}|\lambda_{jk}|\chi_{Q_{jk}}
\r)^q\r]^{1/q}\r\|_{\cl(\rn)}\ls\|\lz\|_{f^{s,\tau}_{\cl,q,a}(\rn)}.$$
For all $j\in\{0,\cdots,j_P-1\}$, there exists a unique $k\in\zz^n$ such that
$P\cap Q_{jk}\neq\emptyset$. Set $\lz_j:=\lz_{jk}$ and $Q_j:= Q_{jk}$,
then for all $j\in\{0,\cdots,j_P-1\}$,
\begin{eqnarray*}
\frac{2^{js}|\lz_j|}{|Q_j|^{\tau-\tau_0}}
&&\sim
\frac{\lf\|2^{js}|\lambda_j|\chi_{Q_j}\r\|_{\cl(\rn)}}{|Q_j|^\tau}\\
&&\ls
\frac1{|Q_j|^\tau}\lf\|\lf[\sum_{i=j}^{\fz}\lf(
\chi_{Q_j} 2^{is}
\sum_{k \in \Z^n}|\lambda_{ik}|\chi_{Q_{ik}}
\r)^q\r]^{1/q}\r\|_{\cl(\rn)}
\ls\|\lz\|_{f^{s,\tau}_{\cl,q,a}(\rn)},
\end{eqnarray*}
which implies that
\begin{eqnarray*}
\mi&&=\frac1{|P|^\tau}\lf\|\lf[\sum_{j=0}^{j_P-1}\lf(
\chi_P 2^{js}
\sum_{k \in \Z^n}|\lambda_{jk}|\chi_{Q_{jk}}
\r)^q\r]^{1/q}\r\|_{\cl(\rn)}
\ls\frac{\|\chi_P\|_{\cl(\rn)}}{|P|^\tau}\lf(\sum_{j=0}^{j_P-1}
2^{jsq}|\lambda_{j}|^q\r)^{1/q}
\\&&\hs\ls
\|\lz\|_{f^{s,\tau}_{\cl,q,a}(\rn)}|P|^{\tau_0-\tau}
\lf[\sum_{j=0}^{j_P-1}|Q_j|^{q(\tau-\tau_0)}\r]^{1/q}
\ls\|\lz\|_{f^{s,\tau}_{\cl,q,a}(\rn)}.
\end{eqnarray*}
This finishes the proof of Theorem \ref{t8.4}.
\end{proof}

The following is a variant of Theorem \ref{t8.4}.
\begin{theorem}\label{t8.5}
Let $\tau\in [0,\infty)$ and $q\in(0,\fz]$.
Assume that there exist
a positive constant $A$
and a positive constant $C(A)$, depending on $A$,
such that, for all $P\in\cq(\rn)$ and $k\in\zz_+$,
\begin{equation}\label{8.22}
\|\chi_P w_{j_P-k}\|_{\cl(\rn)} \le C(A) 2^{-Ak}\|\chi_{2^kP}w_{j_P-k}\|_{\cl(\rn)}
\end{equation}
and that $(\cl8)$ holds true.
Then
$\ce_{\cl,q,a}^{w,\tau}(\rn)=F_{\cl,q,a}^{w,\tau}(\rn)$
with equivalent norms for all $\tau\in[0,A)$.
\end{theorem}

\begin{proof}
By the definition, we have only to show that
$$F^{w,\tau}_{\cl,q,a}(\rn)\hookrightarrow\ce^{w,\tau}_{\cl,q,a}(\rn).$$
To this end, by $(\cl8)$,
\begin{eqnarray*}
\|\lambda\|_{e^{w,\tau}_{\cl,q,a}(\rn)}
&&:=
\sup_{P\in\cq(\rn)}\frac1{|P|^\tau}\lf\|\lf[\sum_{j=0}^\fz\lf(
\chi_P w_j
\sup_{y\in\rn}\frac{\sum_{k \in \Z^n}|\lambda_{jk}|
\chi_{Q_{jk}}(\cdot+y)}{(1+2^j|y|)^a}
\r)^q\r]^{1/q}\r\|_{\cl(\rn)}\\
&&\sim\sup_{P\in\cq(\rn)}\frac1{|P|^\tau}\lf\|\lf[\sum_{j=0}^\fz\lf(
\chi_P w_j
\sum_{k \in \Z^n}|\lambda_{jk}|\chi_{Q_{jk}}
\r)^q\r]^{1/q}\r\|_{\cl(\rn)}.
\end{eqnarray*}
Similarly, by $(\cl8)$, we also conclude that
\begin{eqnarray*}
\|\lambda\|_{F^{w,\tau}_{\cl,q,a}(\rn)}
&&:=
\sup_{P\in\cq(\rn)}\frac1{|P|^\tau}\lf\|\lf[\sum_{j=j_P\vee0}^\fz\lf(
\chi_P w_j
\sup_{y\in\rn}\frac{\sum_{k \in \Z^n}|\lambda_{jk}|
\chi_{Q_{jk}}(\cdot+y)}{(1+2^j|y|)^a}
\r)^q\r]^{1/q}\r\|_{\cl(\rn)}\\
&&\sim\sup_{P\in\cq(\rn)}\frac1{|P|^\tau}\lf\|\lf[\sum_{j=j_P\vee0}^\fz\lf(
\chi_P w_j
\sum_{k \in \Z^n}|\lambda_{jk}|\chi_{Q_{jk}}
\r)^q\r]^{1/q}\r\|_{\cl(\rn)}.
\end{eqnarray*}
Then, it suffices to show that, for all dyadic cubes $P$
with $j_P \ge 1$,
$$\mi:=\frac1{|P|^\tau}\lf\|\lf[\sum_{j=0}^{j_P-1}\lf(
\chi_P w_j
\sum_{k \in \Z^n}|\lambda_{jk}|\chi_{Q_{jk}}
\r)^q\r]^{1/q}\r\|_{\cl(\rn)}\ls\|\lz\|_{f^{w,\tau}_{\cl,q,a}(\rn)}.$$
For all $j\in\{0,\cdots,j_P-1\}$, there exists a unique $k\in\zz^n$ such that
$P\cap Q_{jk}\neq\emptyset$. Set $\lz_j:=\lz_{jk}$ and $Q_j:= Q_{jk}$,
then for all $j\in\{0,\cdots,j_P-1\}$,
\begin{eqnarray*}
&&\frac1{|Q_j|^\tau}\lf\|
w_j
\lambda_j\chi_{Q_j}
\r\|_{\cl(\rn)}\noz\\
&&\hs\le
\frac1{|Q_j|^\tau}\lf\|\lf[\sum_{i=j}^{\fz}\lf(
\chi_{Q_j} w_i
\sum_{k \in \Z^n}|\lambda_{ik}|\chi_{Q_{ik}}
\r)^q\r]^{1/q}\r\|_{\cl(\rn)}\le\|\lz\|_{f^{w,\tau}_{\cl,q,a}(\rn)}.
\end{eqnarray*}

Assume $q \in[ 1,\fz]$ for the moment.
Then by the assumption $q\in[1,\fz]$ and the
triangle inequality of $\|\cdot\|_{\cl(\rn)}^\tz$, we see that
\begin{eqnarray*}
\mi&&=\frac1{|P|^\tau}\lf\|\lf[\sum_{j=0}^{j_P-1}\lf(
\chi_P w_j
\sum_{k \in \Z^n}|\lambda_{jk}|\chi_{Q_{jk}}
\r)^q\r]^{1/q}\r\|_{\cl(\rn)}\\
&&
\hs\le
\frac1{|P|^\tau}\lf\|\sum_{j=0}^{j_P-1}
\chi_P w_j
|\lambda_{j}|\chi_{Q_{j}}
\r\|_{\cl(\rn)}\le
\frac1{|P|^\tau}\lf[\sum_{j=0}^{j_P-1}\lf\|
\chi_P w_j\lambda_{j}\chi_{Q_{j}}
\r\|_{\cl(\rn)}^\tz\r]^{1/\tz}.
\end{eqnarray*}
If we use the assumption \eqref{8.22},
then we see that
\begin{eqnarray*}
\mi\le
\|\lz\|_{f^{w,\tau}_{\cl,q,a}(\rn)}\frac1{|P|^\tau}
\lf[\sum_{j=0}^{j_P-1}2^{-jA\theta}|Q_j|^{\tau\tz}\r]^{1/\tz}
\ls\|\lz\|_{f^{w,\tau}_{\cl,q,a}(\rn)}.
\end{eqnarray*}
If $q \in (0,1)$,
since $\cl^{1/q}(\rn)$ is still a quasi-normed space of functions,
by the Aoki-Rolewicz theorem (see \cite{a42,r57}),
there exist an equivalent quasi-norm $\qn\cdot\qn$ and $\wz\tz\in(0,1]$ such that,
for all $f,g\in\cl^{1/q}(\rn)$,
\[\left\{
\begin{array}{l}
\|f\|_{\cl^{1/q}(\rn)}\sim\qn f\qn\\
\qn f+g\qn^{\wz\tz}\le\qn f\qn^{\wz\tz}+\qn g\qn^{\wz\tz}.
\end{array}\r.\]
Form this, it follows that
\begin{eqnarray*}
\mi^{\wz\tz}
&&\ls\frac1{|P|^{\tau\wz\tz}}\sum_{j=0}^{j_P-1}\lf|\!\lf|\!\lf|
\chi_P w_j\lambda_{j}\chi_{Q_{j}}
\r|\!\r|\!\r|^{\wz\tz}\sim\frac1{|P|^{\tau\wz\tz}}\sum_{j=0}^{j_P-1}\lf\|
\chi_P w_j\lambda_{j}\chi_{Q_{j}}
\r\|_{\cl(\rn)}^{\wz\tz}\\
&&\ls\frac1{|P|^{\tau\wz\tz}}\sum_{j=0}^{j_P-1}2^{-jA\wz\tz}\lf\|
w_j\lambda_{j}\chi_{Q_{j}}
\r\|_{\cl(\rn)}^{\wz\tz}\\
&&\ls\|\lz\|_{f^{w,\tau}_{\cl,q,a}(\rn)}^{\wz\tz}
\frac1{|P|^{\tau\wz\tz}}\sum_{j=0}^{j_P-1}2^{-jA\wz\tz}
\lf|Q_{j}\r|^{\tau\wz\tz}\ls\|\lz\|_{f^{w,\tau}_{\cl,q,a}(\rn)}^{\wz\tz},\\
\end{eqnarray*}
which completes the proof of Theorem \ref{t8.5}.
\end{proof}

\begin{remark}
In many examples (see Section \ref{s10}),
it is not so hard to show that $(\ref{8.21})$ holds true.
\end{remark}

The following theorem generalizes \cite[Theorem 1.1]{syy}.
Recall that $\cl(\rn)$ carries the parameter $N_0$
from $(\cl6)$.

\begin{theorem}\label{t8.6}
Let $\oz\in\cw_{\az_1,\az_2}^{\az_3}$ with $\az_1, \az_2, \az_3 \in [0,\infty)$.

{\rm (i)} Assume $\tau\in(0,\fz)$, $q\in(0,\fz)$ and that $(\cl7)$ holds true.
If $a\gg1$, then $\cn_{\cl,q,a}^{w,\tau}(\rn)$
is a proper subspace of $B_{\cl,q,a}^{w,\tau}(\rn)$.

{\rm (ii)} If $a\in(0,\fz)$ and $\tau\in[0,\fz)$, then $\cn_{\cl,\fz,a}^{w,\tau}(\rn)
=B_{\cl,\fz,a}^{w,\tau}(\rn)$
with equivalent norms.
\end{theorem}

\begin{proof}
Since {\rm (ii)} is immediately deduced from the definition, we only prove (i).
By $(\cl7)$ and Theorems \ref{t4.1} and \ref{t8.1}, we see that
\begin{eqnarray*}
&&\|\lz\|_{b^{w,\tau}_{\cl,q,a}(\rn)}\\
&&\hs=
\sup_{P\in\cq(\rn)}\frac1{|P|^\tau}\lf\{\sum_{j=j_P\vee0}^\fz\lf\|
\chi_P w_j
\sup_{y\in\rn}\frac{1}{(1+2^j|y|)^a}
\sum_{k\in\zz^n}|\lz_{jk}|\chi_{Q_{jk}}(\cdot+y)
\r\|_{\cl(\rn)}^q\r\}^{1/q}\\
&&\hs\sim
\sup_{P\in\cq(\rn)}\frac1{|P|^\tau}\lf\{\sum_{j=j_P\vee0}^\fz\lf\|
\chi_P w_j
\sum_{k\in\zz^n}|\lz_{jk}|\chi_{Q_{jk}}
\r\|_{\cl(\rn)}^q\r\}^{1/q}
\end{eqnarray*}
and
\begin{eqnarray*}
&&\|\lz\|_{n^{w,\tau}_{\cl,q,a}(\rn)}\\
&&\hs=
\lf\{\sum_{j=j_P\vee0}^\fz\sup_{P\in\cq(\rn)}\frac1{|P|^{\tau q}}\lf\|
\chi_P w_j
\sup_{y\in\rn}\frac{1}{(1+2^j|y|)^a}
\sum_{k\in\zz^n}|\lz_{jk}|\chi_{Q_{jk}}(\cdot+y)
\r\|_{\cl(\rn)}^q\r\}^{1/q}\\
&&\hs\sim
\lf\{\sum_{j=j_P\vee0}^\fz\sup_{P\in\cq(\rn)}\frac1{|P|^{\tau q}}\lf\|
\chi_P w_j
\sum_{k\in\zz^n}|\lz_{jk}|\chi_{Q_{jk}}
\r\|_{\cl(\rn)}^q\r\}^{1/q}.
\end{eqnarray*}

We abbreviate
$$
Q_{j(1,1,1,\cdots,1)}:=
{\overbrace{[2^{-j},2^{1-j})\times\cdots\times[2^{-j},2^{1-j})}^{n\ \mathrm{times}}}
$$
to $R_j$
for all $j\in\zz$ and set
\begin{equation*}
\lz_Q:=\lf\{
\begin{array}{ll}
\|w_j\chi_{R_j}\|_{\cl(\rn)}^{-1}|R_j|^\tau,\
&\ Q=R_j\ \mathrm{for\ some}\ j\in\zz;\\
0,
&\ Q\neq R_j\ \mathrm{for\ any}\ j\in\zz.
\end{array}
\r.
\end{equation*}
Then we have
\begin{eqnarray*}
\|\lz\|_{b^{w,\tau}_{\cl,q,a}(\rn)}\sim
\sup_{P\in\cq(\rn)}\frac1{|P|^\tau}\lf\{\sum_{j=j_P\vee0}^\fz\lf\|
\chi_{P\cap R_j} w_j\lz_{R_j}
\r\|_{\cl(\rn)}^q\r\}^{1/q}.
\end{eqnarray*}
In order that the inner summand is not zero, there are there possibilities:
(a) P contains $\{R_k, R_{k+1},\cdots\}$;
(b) $P$ agrees with $R_k$ for some $k\in\zz$;
(c) $P$ is a proper subset of $R_k$ for some $k\in\zz$. The last possibility (c) dose
not yield the supremum, while the first case (a) can be covered by the second
case (b). Hence it follows that
\begin{eqnarray}\label{8.24}
\|\lz\|_{b^{w,\tau}_{\cl,q,a}(\rn)}
&&\sim
\sup_{k\in\zz}\frac1{|R_k|^\tau}\lf\{\sum_{j=k\vee0}^\fz\lf\|
\chi_{R_k\cap R_j} w_j\lz_{R_j}
\r\|_{\cl(\rn)}^q\r\}^{1/q}\noz\\
&&\sim\sup_{k\in\zz}\frac1{|R_k|^\tau}\lf\|
\chi_{R_k} w_k\lz_{R_k}
\r\|_{\cl(\rn)}\sim1.
\end{eqnarray}
Meanwhile, keeping in mind
that $q$ is finite,
we have
\begin{equation}\label{8.25}
\|\lz\|_{n^{w,\tau}_{\cl,q,a}(\rn)}
\ge
\lf\{\sum_{j=0}^\fz\sup_{k\in\zz}\frac1{|R_k|^{\tau q}}\lf\|
\chi_{R_k\cap R_j} w_j\lz_{R_j}
\r\|_{\cl(\rn)}^q\r\}^{1/q}=\fz.
\end{equation}
This,
together with Theorem 4.1,
the atomic decomposition of
$(B^{w,\tau}_{\cl,q,a}(\rn), b^{w,\tau}_{\cl,q,a}(\rn))$ and
$(\cn^{w,\tau}_{\cl,q,a}(\rn), n^{w,\tau}_{\cl,q,a}(\rn))$,
\eqref{8.24} and \eqref{8.25},
then completes the proof of Theorem \ref{t8.6}.
\end{proof}

\section{Homogeneous spaces}
\label{s9}

What we have been doing so far
can be extended to the homogeneous cases.
Here we give definitions
and state theorems but the proofs are omitted.

Following Triebel \cite{t83}, we let
$$\cs_\infty(\rn):=\left\{\varphi\in\cs(\rn):\ \int_\rn
\varphi(x)x^\gamma\,dx=0\ \mbox{for all multi-indices}\ \gamma\in
\Z_+^n\r\}$$ and consider $\cs_\fz(\rn)$ as a subspace
of $\cs(\rn)$, including the topology. Write $\cs'_\infty(\rn)$ to
denote the {\it topological dual of} $\cs_\infty(\rn)$, namely, the set of
all continuous linear functionals on $\cs_\fz(\rn)$. We also endow
$\cs'_\infty(\rn)$ with the weak-$\ast$ topology. Let
$\mathcal{P}(\rn)$ be the {\it set of all polynomials on} $\rn$. It is
well known that $\cs'_\fz(\rn)=\cs'(\rn)/\mathcal{P}(\rn)$ as
topological spaces (see, for example, \cite[Proposition 8.1]{ysy-textbook}).

To develop a theory of homogeneous spaces,
we need to modify the class of weights. Let
$\R^{n+1}_{\Z}:=\{(x,t) \in \R^{n+1}_+:\ \log_2 t \in \Z\}$.

\begin{definition}\label{d9.1}
Let $\alpha_1, \alpha_2, \alpha_3 \in [0,\infty)$.
Then define the \emph{class}
$\dot{\mathcal W}_{\alpha_1,\alpha_2}^{\alpha_3}$
of weights
as the set of all measurable functions
$w:\R^{n+1}_{\Z}\to (0,\infty)$
satisfying the following conditions:

{\rm (i)}
Condition (H-W1){\rm:} There exists a positive constant $C$ such that,
for all $x \in \rn$ and $j, \nu \in \Z$
with $j \ge \nu$,
$$
C^{-1}
2^{-(j-\nu)\alpha_1}
w(x,2^{-\nu})
\le
w(x,2^{-j})
\le C
2^{-(\nu-j)\alpha_2}
w(x,2^{-\nu}).
$$

{\rm (ii)}
Condition (H-W2){\rm:} There exists a positive constant $C$ such that,
for all $x,y \in \rn$ and $j \in \Z$,
$$w_j(x)
\le
C
w(y,2^{-j})
\left(1+2^j|x-y|\right)^{\alpha_3}.
$$

The \emph{class}
$\star-\dot{\mathcal W}_{\alpha_1,\alpha_2}^{\alpha_3}$
is defined by making similar modifications to Definition \ref{d3.3}.
\end{definition}

As we did for the inhomogeneous case,
we write $w_j(x):= w(x,2^{-j})$
for $x \in \rn$ and $j \in\Z$.

\begin{definition}\label{d9.2}
Let $q\in(0,\fz]$ and $\tau\in[0,\fz)$.
Suppose, in addition, that
$w \in \dot{\mathcal W}^{\alpha_3}_{\alpha_1,\alpha_2}$
with $\alpha_1, \alpha_2, \alpha_3 \in [0,\infty)$.

{\rm (i)}
The {\it space}
$\ell^q(\cl^w_\tau(\rn,\zz))$
is defined to be the {space of all
sequences $G:=\{g_j\}_{j\in\zz}$ of measurable
functions on $\rn$ such that}
\begin{equation}\label{9.1}
\|G\|_{\ell^q(\cl^w_\tau(\rn,\zz))}
:= \sup_{P\in\mathcal{Q}(\rn)}\frac1{|P|^\tau}
\|\{\chi_P w_jg_j\}_{j=j_P}^\infty\|_{\cl^w(\ell^q(\rn,\zz))}
<\fz.
\end{equation}

{\rm (ii)}
The {\it space}
$\ell^q(\cn\cl^w_\tau(\rn,\zz))$
is defined to be the {space of all
sequences $G:=\{g_j\}_{j\in\zz}$ of measurable
functions on $\rn$ such that}
\begin{equation}\label{9.2}
\|G\|_{\ell^q(\cn\cl^w_\tau(\rn,\zz))}
:=
\lf\{\sum_{j=0}^\fz
\sup_{P\in\mathcal{Q}(\rn)}
\left(\frac{\|\chi_P w_jg_j\|_{\cl(\rn)}}{|P|^{\tau}}\right)^q\r\}^{1/q}<\fz.
\end{equation}

{\rm (iii)}
The {\it space}
$\cl^w_\tau(\ell^q(\rn,\zz))$ is defined to be the
{space of all sequences $G:=\{g_j\}_{j\in\zz}$ of
measurable functions on $\rn$ such that}
\begin{equation}\label{9.3}
\|G\|_{\cl^w_\tau(\ell^q(\rn,\zz))}
:= \sup_{P\in\mathcal{Q}(\rn)}\frac1{|P|^\tau}
\|\{\chi_P w_jg_j\}_{j=j_P}^\infty\|_{\ell^q(\cl^w(\rn,\zz))}<\fz.
\end{equation}

{\rm (iv)}
The {\it space}
$\ce\cl^w_\tau(\ell^q(\rn,\zz))$ is defined to be the
{space of all sequences $G:=\{g_j\}_{j\in\zz}$ of
measurable functions on $\rn$ such that}
\begin{equation}\label{9.30}
\|G\|_{\ce\cl^w_\tau(\ell^q(\rn,\zz))}
:= \sup_{P\in\mathcal{Q}(\rn)}\frac1{|P|^\tau}
\|\{\chi_P w_jg_j\}_{j=-\infty}^\infty\|_{\ell^q(\cl^w(\rn,\zz))}<\fz.
\end{equation}

When $q=\infty$, a natural modification is made
in \eqref{9.1} through \eqref{9.30}
and $\tau$ is omitted in the notation when $\tau=0$.
\end{definition}

\subsection{Homogeneous Besov-type
and Triebel-Lizorkin-type spaces}
\label{s9.1}

Based upon the inhomogeneous case,
we present the following definitions.

\begin{definition}\label{d9.4}
Let $a\in (0,\fz)$,
$\alpha_1, \alpha_2, \alpha_3, \tau\in[0,\fz)$, $q\in(0,\,\fz]$
and $w \in \dot{\cw}_{\alpha_1,\alpha_2}^{\alpha_3}$.
Assume also that $\cl(\rn)$ is a quasi-normed
space satisfying $(\cl1)$ through $(\cl4)$ and
that $\vz\in\cs_\fz(\rn)$ satisfies \eqref{1.2}.
For all $f\in\cs'_\fz(\rn)$, $x\in\rn$ and $j \in {\mathbb Z}$, let
\begin{equation}\label{9.x1}
(\vz_j^*f)_a(x):=
\sup_{y\in\rn}\frac{|\vz_j\ast f(x+y)|}{(1+2^j|y|)^a}.
\end{equation}

{\rm (i)} The {\it homogeneous generalized Besov-type space
$\dot{B}^{w,\tau}_{\cl,q,a}(\rn)$}
is defined to be the space of
all $f\in\cs'_\fz(\rn)$ such that
\begin{equation*}
\|f\|_{\dot{B}^{w,\tau}_{\cl,q,a}(\rn)}:=
\lf\|\lf\{(\vz_j^\ast f)_a\r\}_{j\in\zz}\r\|_{\ell^q(\cl_\tau^w(\rn,\zz))}<\fz.
\end{equation*}%

{\rm (ii)} The {\it homogeneous generalized Besov-Morrey space
$\dot{\cn}^{w,\tau}_{\cl,q,a}(\rn)$}
is defined to be the space of
all $f\in\cs'_\fz(\rn)$ such that
\begin{equation*}
\|f\|_{\dot{\cn}^{w,\tau}_{\cl,q,a}(\rn)}:=
\lf\|\lf\{(\vz_j^\ast f)_a\r\}_{j\in\zz}\r\|_{\ell^q(\cn\cl_\tau^w(\rn,\zz))}<\fz.
\end{equation*}%

{\rm (iii)} The {\it homogeneous generalized Triebel-Lizorkin-type space
$\dot{F}^{w,\tau}_{\cl,q,a}(\rn)$}
is defined to be the space of all $f\in\cs'_\fz(\rn)$ such that
\begin{equation*}
\|f\|_{\dot{F}^{w,\tau}_{\cl,q,a}(\rn)}:=
\lf\|\lf\{(\vz_j^\ast f)_a\r\}_{j\in\zz}\r\|_{\cl_\tau^w(\ell^q(\rn,\zz))}<\fz.
\end{equation*}%

{\rm (iv)} The {\it homogeneous generalized Triebel-Lizorkin-Morrey space
$\dot{\ce}^{w,\tau}_{\cl,q,a}(\rn)$}
is defined to be the space of all $f\in\cs'_\fz(\rn)$ such that
\begin{equation*}
\|f\|_{\dot{\ce}^{w,\tau}_{\cl,q,a}(\rn)}:=
\lf\|\lf\{(\vz_j^\ast f)_a\r\}_{j\in\zz}\r\|_{\ce\cl_\tau^w(\ell^q(\rn,\zz))}<\fz.
\end{equation*}%

{\rm (v)}
Denote by $\dot{\ca}^{w,\tau}_{\cl,q,a}(\rn)$
one of the above spaces.
\end{definition}

\begin{example}
One of the advantages of introducing the class
$\dot{\mathcal W}^{\alpha_3}_{\alpha_1,\alpha_2}$
is that the intersection space of these function spaces still falls under this scope.
Indeed, let $\alpha_1, \alpha_2, \alpha_3,
\beta_1, \beta_2,$ $\beta_3, \tau \in [0,\infty)$,
$q, q_1,q_2\in(0,\,\fz]$,
$w\in\dot{\mathcal W}^{\alpha_3}_{\alpha_1,\alpha_2}$
and
$w'\in\dot{\mathcal W}^{\beta_3}_{\beta_1,\beta_2}$.
Then it is easy to see
$
\dot{\mathcal A}^{w,\tau}_{\cl,q_1,a}(\rn)
\cap
\dot{\mathcal A}^{w',\tau}_{\cl,q_1,a}(\rn)
=
\dot{\mathcal A}^{w+w',\tau}_{\cl,q_1,a}(\rn).
$
\end{example}
The following lemma is immediately deduced from the definitions
(c.\,f. Lemma \ref{l3.1}).
\begin{lemma}\label{l9.1}
Let $\alpha_1, \alpha_2, \alpha_3, \tau \in [0,\infty)$,
$q, q_1,q_2\in(0,\,\fz]$
and $w\in\dot{\mathcal W}^{\alpha_3}_{\alpha_1,\alpha_2}$.
Then
\begin{eqnarray*}
&&\dot{B}^{w,\tau}_{\cl,q_1,a}(\rn)
\hookrightarrow
\dot{B}^{w,\tau}_{\cl,q_2,a}(\rn), \\
&&\dot{\cn}^{w,\tau}_{\cl,q_1,a}(\rn)
\hookrightarrow
\dot{\cn}^{w,\tau}_{\cl,q_2,a}(\rn),\\
&&\dot{F}^{w,\tau}_{\cl,q_1,a}(\rn)
\hookrightarrow
\dot{F}^{w,\tau}_{\cl,q_2,a}(\rn),\\
&&\dot{\ce}^{w,\tau}_{\cl,q_1,a}(\rn)
\hookrightarrow
\dot{\ce}^{w,\tau}_{\cl,q_2,a}(\rn)
\end{eqnarray*}
and
\begin{eqnarray*}
\dot{F}^{w,\tau}_{\cl,q,a}(\rn)
\hookrightarrow
\dot{\cn}^{w,\tau}_{\cl,\infty,a}(\rn)
\end{eqnarray*}
in the sense of continuous embeddings.
\end{lemma}

The next theorem is a homogeneous counterpart
of Theorem \ref{t3.3}.
\begin{theorem}\label{t9.1}
Let $\alpha_1, \alpha_2, \alpha_3, \tau \in [0,\infty)$, $q\in(0,\,\fz]$
and $w\in\dot{\mathcal W}^{\alpha_3}_{\alpha_1,\alpha_2}$.
Then the spaces
$\dot{B}^{w,\tau}_{\cl,q,a}(\rn)$
and
$\dot{F}^{w,\tau}_{\cl,q,a}(\rn)$
are continuously embedded into $\cs_\fz'(\rn)$.
\end{theorem}

\begin{proof}
In view of Lemma \ref{l9.1},
we have only to prove that
\[
\dot{B}^{w,\tau}_{\cl,\infty,a}(\rn) \hookrightarrow \cs_\infty'(\rn).
\]
Suppose that $\Phi$ satisfies \eqref{1.1} and
that $\widehat{\Phi}$ equals to $1$ near a neighborhood of the origin.
We write
$\vz(\cdot):=\Phi(\cdot)-2^{-n}\Phi(2^{-1}\cdot)$
and define $L_1(f):= f-\Phi*f$ for all $f \in \cs'_\infty(\rn)$.
Then by Theorem \ref{t3.3},
we have
$L_1(\dot{B}^{w,\tau}_{\cl,q,a}(\rn))
\hookrightarrow \cs'(\rn)
\hookrightarrow \cs'_\infty(\rn)$.
Therefore,
we need to prove that
\[
L_2(f):=\sum_{j=-\infty}^{0}\vz_j*f
\]
converges in $\cs'_\infty(\rn)$
and that $L_2$ is a continuous operator
from $\dot{B}^{w,\tau}_{\cl,\infty,a}(\rn)$
to $\cs'_\infty(\rn)$.

Notice that, for all $j\in\zz$ and $x\in\rn$,
\[
|\partial^\alpha(\vz_j*f)(x)|
\lesssim
2^{j|\alpha|}(\vz_j^*f)_a(x).
\]
Consequently, for any $\kappa \in \cs_\fz(\rn)$,
we have
\begin{eqnarray*}
\left|\int_{\rn}\kappa(x)\partial^\alpha(\vz_j*f)(x)\,dx\right|
\le
\int_{\rn}|\kappa(x)\partial^\alpha(\vz_j*f)(x)|\,dx
\le
2^{j|\alpha|}
\int_{\rn}|\kappa(x)|(\vz_j^*f)_a(x)\,dx.
\end{eqnarray*}
Now we use the condition (H-W2) to conclude that
\begin{eqnarray*}
\left|\int_{\rn}\kappa(x)\partial^\alpha(\vz_j*f)(x)\,dx\right|
&&\le
2^{j(|\alpha|-\alpha_1)}
\int_{\rn}\frac{|\kappa(x)|}{w(x,1)}
w_j(x)(\vz_j^*f)_a(x)\,dx\\
&&\le
2^{j(|\alpha|-\alpha_1)}
\int_{\rn}
\frac{1}{(1+|x|)^M}
w_j(x)(\vz_j^*f)_a(x)\,dx\\
&&=
2^{j(|\alpha|-\alpha_1)}
\sum_{k \in \Z^n}
\int_{Q_{jk}}
\frac{1}{(1+|x|)^M}
w_j(x)(\vz_j^*f)_a(x)\,dx\\
&&\ls
2^{j(|\alpha|-\alpha_1-M)}
\sum_{k \in \Z^n}
(|k|+1)^{-M}
\int_{Q_{jk}}
w_j(x)(\vz_j^*f)_a(x)\,dx.
\end{eqnarray*}
By $(\cl6)$ and (H-W2), we further see that
\begin{eqnarray*}
\left|\int_{\rn}\kappa(x)\partial^\alpha(\vz_j*f)(x)\,dx\right|
&&\ls
2^{j(|\alpha|-\alpha_1-\delta_0)}
\sum_{k \in \Z^n}
(|k|+1)^{-M+\delta_0}
\|w_j(\vz_j^*f)_a\|_{\cl(\rn)}\\
&&\ls
2^{j(|\alpha|-\alpha_1-\delta_0)}
\|f\|_{\dot{B}^{w,\tau}_{\cl,\infty,a}(\rn)}.
\end{eqnarray*}
Therefore, the summation defining $L_2(f)$ converges
in $\cs'_\infty(\rn)$,
which completes the proof of Theorem \ref{t9.1}.
\end{proof}

We remark that these homogeneous spaces have many similar properties to those
in Sections 4 through 9 of their inhomogeneous counterparts,
which will be formulated below. However, similar to the classical
homogeneous Besov spaces and Triebel-Lizorkin spaces, (see \cite[p.\,238]{t83}),
some of the most striking features of the spaces $B^{w,\tau}_{\cl,q,a}(\rn)$,
$F^{w,\tau}_{\cl,q,a}(\rn)$, $\cn^{w,\tau}_{\cl,q,a}(\rn)$ and $\ce^{w,\tau}_{\cl,q,a}(\rn)$
have no counterparts, such as the boundedness of  pointwise multipliers in Section 5.
Thus, we cannot expect to find counterparts
of the results in Section 5.

\subsection{Characterizations}

We have the following counterparts of Theorem \ref{t3.1}.

\begin{theorem}\label{t9.2}
Let $a, \alpha_1, \alpha_2, \alpha_3, \tau, q$, $w$
and $\cl(\rn)$ be as in Definition \ref{d9.4}.
Assume that $\psi \in {\mathcal S}_\fz(\rn)$
satisfies that
\begin{equation*}
\widehat{\psi}(\xi)\ne 0 \mbox{ if }\
\frac{\varepsilon}{2}<|\xi|<2\varepsilon.
\end{equation*}
for some $\varepsilon\in(0,\fz)$.
Let $\psi_j(\cdot):= 2^{jn}\psi(2^j\cdot)$
for all $j \in {\zz}$ and $\{(\psi_j^\ast f)_a\}_{j\in\zz}$ be as
in \eqref{9.x1} with $\varphi$ replaced by $\psi$.
Then
\begin{eqnarray*}
&&\|f\|_{\dot{B}^{w,\tau}_{\cl,q,a}(\rn)}
\sim
\lf\|\lf\{(\psi_j^\ast f)_a\r\}_{j\in\zz}\r\|_{\ell^q(\cl_\tau^w(\rn,\zz))},\\
&&\|f\|_{\dot{\cn}^{w,\tau}_{\cl,q,a}(\rn)}
\sim
\lf\|\lf\{(\psi_j^\ast f)_a\r\}_{j\in\zz}\r\|_{\ell^q(\cn\cl_\tau^w(\rn,\zz))},\\
&&\|f\|_{\dot{F}^{w,\tau}_{\cl,q,a}(\rn)}
\sim
\lf\|\lf\{(\psi_j^\ast f)_a\r\}_{j\in\zz}\r\|_{\cl_\tau^w(\ell^q(\rn,\zz))}
\end{eqnarray*}
and
\begin{eqnarray*}
&&\|f\|_{\dot{\ce}^{w,\tau}_{\cl,q,a}(\rn)}
\sim
\lf\|\lf\{(\psi_j^\ast f)_a\r\}_{j\in\zz}\r\|_{\ce\cl_\tau^w(\ell^q(\rn,\zz))}
\end{eqnarray*}
with implicit equivalent positive constants independent of $f$.
\end{theorem}

We also characterize
these function spaces in terms of local means
(see Corollary \ref{c3.1}).
\begin{corollary}\label{c9.1}
Under the notation of Theorem {\rm \ref{t9.2}}, let
\[
{\mathfrak M}f(x,2^{-j})
:=
\sup_\psi|\psi_{j}*f(x)|
\]
for all $(x,2^{-j}) \in \R^{n+1}_\Z$ and $f \in \cs'_\fz(\rn)$,
where the supremum is taken over all $\psi$
in $\cs_\fz(\rn)$
satisfying that
\[
\sum_{|\alpha| \le N}
\sup_{x \in \rn}
(1+|x|)^N |\partial^\alpha \psi(x)|
\le 1
\]
and that, for some $\varepsilon\in(0,\fz)$,
\[
\int_{\rn}\xi^{\alpha}\widehat{\psi}(\xi)\,d\xi=0, \quad
\widehat{\psi}(\xi)\ne 0 \mbox{ if }\
\frac{\varepsilon}{2}<|\xi|<2\varepsilon.
\]
Then, if $N$ is large enough, then
for all $f \in \cs'_\fz(\rn)$
\begin{eqnarray*}
&&\|f\|_{\dot{B}^{w,\tau}_{\cl,q,a}(\rn)}
\sim
\lf\|\lf\{{\mathfrak M} f(\cdot, 2^{-j})\r\}_{j\in\zz}\r\|_{\ell^q(\cl_\tau^w(\rn,\zz))},\\
&&\|f\|_{\dot{\cn}^{w,\tau}_{\cl,q,a}(\rn)}
\sim
\lf\|\lf\{{\mathfrak M} f(\cdot, 2^{-j})\r\}_{j\in\zz}\r\|_{\ell^q(\cn\cl_\tau^w(\rn,\zz))},\\
&&\|f\|_{\dot{F}^{w,\tau}_{\cl,q,a}(\rn)}
\sim
\lf\|\lf\{{\mathfrak M} f(\cdot, 2^{-j})\r\}_{j\in\zz}\r\|_{\cl_\tau^w(\ell^q(\rn,\zz))}
\end{eqnarray*}
and
\begin{eqnarray*}
&&\|f\|_{\dot{\ce}^{w,\tau}_{\cl,q,a}(\rn)}
\sim
\lf\|\lf\{{\mathfrak M} f(\cdot, 2^{-j})\r\}_{j\in\zz}\r\|_{\ce\cl_\tau^w(\ell^q(\rn,\zz))}
\end{eqnarray*}
with implicit equivalent positive constants independent of $f$.
\end{corollary}

\subsection{Atomic decompositions}

Now we place ourselves once again
in the setting of a quasi-normed space $\cl(\rn)$
satisfying only $(\cl1)$ through $(\cl6)$.
Now we are going to consider the atomic decompositions
of these spaces in Definition \ref{d9.4}.

\begin{definition}[c.\,f. Definition {\rm \ref{d4.1}}]\label{d9.5}
Let $K \in \Z_+$ and $L \in \Z_+ \cup\{-1\}$.

{\rm (i)}
Let $Q \in \cq(\rn)$.
A \emph{$(K,L)$-atom {\rm(}for $\dot{A}^{s,\tau}_{\cl,q,a}(\rn)${\rm)}}
supported near a cube $Q$ is a $C^K(\rn)$-function $a$ satisfying
\begin{eqnarray*}
\mbox{\rm (the support condition)}&\quad&\supp (a) \subset 3Q,\\
\mbox{\rm (the size condition)}&&
\|\partial^\alpha a\|_{L^\infty(\rn)} \le |Q|^{-\|\alpha\|_1/n},\\
\mbox{\rm (the moment condition)}&&
\int_{\rn}x^\beta a(x)\,dx=0
\end{eqnarray*}
for all multiindices $\alpha$ and $\beta$
satisfying
$\|\alpha\|_1 \le K$ and $\|\beta\|_1 \le L$.
Here the moment condition with $L=-1$
is understood as vacant condition.

{\rm (ii)}
A set $\{a_{jk}\}_{j \in \zz, \, k \in \Z^n}$
of $C^K(\rn)$-functions is called a \emph{collection of
$(K,L)$-atoms {\rm(}for $\dot{A}^{s,\tau}_{\cl,q,a}(\rn)${\rm)}}, if
each $a_{jk}$ is a $(K,L)$-atom
supported near $Q_{jk}$.
\end{definition}

\begin{definition}[c.\,f. Definition {\rm \ref{d4.2}}]\label{d9.6}
Let $K \in \zz_+$, $L \in \Z_+ \cup\{-1\}$ and $N \gg 1$.

{\rm (i)}
Let $Q \in \cq(\rn)$.
A \emph{$(K,L)$-molecule {\rm(}for $\dot{A}^{s,\tau}_{\cl,q,a}(\rn)${\rm)}}
supported near a cube $Q$
is a $C^K(\rn)$-function ${\mathfrak M}$ satisfying
\begin{eqnarray*}
\mbox{\rm (the decay condition)}&\quad &
|\partial^\alpha {\mathfrak M}(x)| \le
\left(1+\frac{|x-c_Q|}{\ell(Q)}\right)^{-N}\ {\rm for\ all}\ x\in\rn,\\
\mbox{\rm (the moment condition)}&&
\int_{\rn}x^\beta {\mathfrak M}(x)\,dx=0
\end{eqnarray*}
for all multiindices $\alpha$ and $\beta$
satisfying
$\|\alpha\|_1 \le K$ and $\|\beta\|_1 \le L$.
Here $c_Q$ and $\ell(Q)$ denote, respectively,
the center and the side length of $Q$, and
the moment condition with $L=-1$
is understood as vacant condition.

{\rm (ii)}
A collection $\{{\mathfrak M}_{jk}\}_{j \in \zz, \, k \in \Z^n}$
of $C^K(\rn)$-functions is called a \emph{collection of
$(K,L)$-molecules {\rm(}for $\dot{A}^{s,\tau}_{\cl,q,a}(\rn)${\rm)}}, if
each ${\mathfrak M}_{jk}$ is a $(K,L)$-molecule
supported near $Q_{jk}$.
\end{definition}

In what follows, for a function $F$ on
$\rr_{\zz}^{n+1}$,
we define
$$
\|F\|_{L_{\cl,q,a}^{w,\tau}(\rr_{\zz}^{n+1})}:
=\lf\|\lf\{\sup_{y\in\rn} \frac{|F(y,2^{-j})|}{(1+2^j|\cdot-y|)^a}
\r\}_{j\in\zz}\r\|_{\ell^q(\cl_\tau^w(\rn,\zz))},
$$
$$
\|F\|_{\cn_{\cl,q,a}^{w,\tau}(\rr_{\zz}^{n+1})}:
=\lf\|\lf\{\sup_{y\in\rn} \frac{|F(y,2^{-j})|}{(1+2^j|\cdot-y|)^a}
\r\}_{j\in\zz}\r\|_{\ell^q(\cn\cl_\tau^w(\rn,\zz))}
$$
$$
\|F\|_{F_{\cl,q,a}^{w,\tau}(\rr_{\zz}^{n+1})}:
=\lf\|\lf\{\sup_{y\in\rn} \frac{|F(y,2^{-j})|}{(1+2^j|\cdot-y|)^a}
\r\}_{j\in\zz}\r\|_{\cl_\tau^w(\ell^q(\rn,\zz))}
$$
and
$$
\|F\|_{\ce_{\cl,q,a}^{w,\tau}(\rr_{\zz}^{n+1})}:
=\lf\|\lf\{\sup_{y\in\rn} \frac{|F(y,2^{-j})|}{(1+2^j|\cdot-y|)^a}
\r\}_{j\in\zz}\r\|_{\ce\cl_\tau^w(\ell^q(\rn,\zz))}.
$$

\begin{definition}[c.\,f. Definition {\rm \ref{d4.3}}]\label{d9.7}
Let $\alpha_1, \alpha_2, \alpha_3, \tau \in [0,\infty)$ and $q\in(0,\,\fz]$.
Suppose that $w\in\dot{\mathcal W}^{\alpha_3}_{\alpha_1,\alpha_2}$.
Assume that $\Phi,\,\vz\in\cs(\rn)$ satisfy, respectively,
\eqref{1.1} and \eqref{1.2}.
Define $\Lambda:\R^{n+1}_\Z \to \C$
by setting, for all $(x,2^{-j}) \in \R^{n+1}_\Z$,
$$
\Lambda(x,2^{-j})
:=
\sum_{m \in \Z^n}\lambda_{j m}\chi_{Q_{j m}}(x),
$$
when
$\lambda:=\{\lambda_{jm}\}_{j \in \zz, \, m \in \Z^n}$,
a doubly-indexed complex sequence, is given.

{\rm (i)} The {\it  homogeneous sequence space
$\dot{b}^{w,\tau}_{\cl,q,a}(\rn)$}
is defined to be the space of
all $\lambda$ such that
$
\|\lambda\|_{\dot{b}^{w,\tau}_{\cl,q,a}(\rn)}:=
\|\Lambda\|_{\dot{L}_{\cl,q,a}^{w,\tau}(\rr_{\zz}^{n+1})}<\fz.
$

{\rm (ii)} The {\it  homogeneous sequence space
$\dot{n}^{w,\tau}_{\cl,q,a}(\rn)$}
is defined to be the space of
all $\lambda$ such that
$
\|\lambda\|_{\dot{n}^{w,\tau}_{\cl,q,a}(\rn)}:=
\|\Lambda\|_{\dot{\cn}_{\cl,q,a}^{w,\tau}(\rr_{\zz}^{n+1})}<\fz.
$

{\rm (iii)} The {\it  homogeneous sequence space
$\dot{f}^{w,\tau}_{\cl,q,a}(\rn)$}
is defined to be the space of all $\lambda$ such that
$
\|\lambda\|_{\dot{f}^{w,\tau}_{\cl,q,a}(\rn)}:=
\|\Lambda\|_{\dot{F}_{\cl,q,a}^{w,\tau}(\rr_{\zz}^{n+1})}<\fz.
$

{\rm (iv)} The {\it  homogeneous sequence space
$\dot{e}^{w,\tau}_{\cl,q,a}(\rn)$}
is defined to be the space of all $\lambda$ such that
$
\|\lambda\|_{\dot{e}^{w,\tau}_{\cl,q,a}(\rn)}:=
\|\Lambda\|_{\dot{\ce}_{\cl,q,a}^{w,\tau}(\rr_{\zz}^{n+1})}<\fz.
$
\end{definition}

As we did for inhomogeneous spaces,
we present the following definition.

\begin{definition}[c.\,f. Definition \ref{d4.2}]\label{d9.8}
Let $X$ be a \emph{function space} embedded continuously into $\cs_\fz'(\rn)$
and ${\mathcal X}$ a \emph{quasi-normed space} of sequences.
The pair $(X,{\mathcal X})$ is called to admit the \emph{atomic decomposition}
if it satisfies the following two conditions:

{\rm (i)}
For any $f \in X$,
there exist a collection of atoms,
$\{a_{jk}\}_{j \in \Z, \, k \in \Z^n}$,
and a sequence
$\{\lambda_{jk}\}_{j \in \Z, \, k \in \Z^n}$
such that
$f=\sum_{j=-\infty}^\infty
\sum_{k \in \Z^n}\lambda_{jk}a_{jk}
$
holds true in ${\mathcal S}'_\fz(\rn)$
and that
$$
\|\{\lambda_{jk}\}_{j \in \Z, \, k \in \Z^n}\|_{{\mathcal X}}
\lesssim
\|f\|_X.
$$

{\rm (ii)}
Suppose that a collection of atoms,
$\{a_{jk}\}_{j \in \Z, \, k \in \Z^n}$,
and a sequence
$\{\lambda_{jk}\}_{j \in \Z, \, k \in \Z^n}$
such that
$
\|\{\lambda_{jk}\}_{j \in \Z, \, k \in \Z^n}\|_{{\mathcal X}}<\infty.
$
Then
the series
$
f := \sum_{j=-\infty}^\infty
\sum_{k \in \Z^n}\lambda_{jk}a_{jk}
$
converges in $\cs_\fz'(\rn)$
and satisfies that
$
\|f\|_X
\lesssim
\|\{\lambda_{jk}\}_{j \in \zz, \, k \in \Z^n}\|_{{\mathcal X}}.
$

In analogy
one says that a pair $(X,{\mathcal X})$ admits the \emph{molecular decomposition}.
\end{definition}

The following result follows from a way
similar to the inhomogeneous case
(see the proof of Theorem \ref{t4.1}).

\begin{theorem}\label{t9.4}
Let $\alpha_1, \alpha_2, \alpha_3, \tau \in [0,\infty)$ and $q\in(0,\,\fz]$.
Suppose that $w\in\dot{\mathcal W}^{\alpha_3}_{\alpha_1,\alpha_2}$
and that \eqref{3.35}, \eqref{4.3}, \eqref{4.4} and \eqref{4.5} hold true.
Then the pair
$(\dot{A}^{w,\tau}_{\cl,q,a}(\rn),\dot{a}^{w,\tau}_{\cl,q,a}(\rn))$
admits the atomic decomposition.
\end{theorem}

In principle,
the proof of Theorem \ref{t9.4}
is analogous to that of Theorem \ref{t4.1};
We just need to modify the related proofs.
Among them, an attention is necessary to
prove the following counterpart of Lemma \ref{l4.2}.

\begin{lemma}\label{l9.2}
Let $\alpha_1, \alpha_2, \alpha_3 \in [0,\infty)$
and $w \in \dot{\mathcal W}^{\alpha_3}_{\alpha_1,\alpha_2}$.
Assume that
$K \in \Z_+$ and $L \in \Z_+$
satisfy {\rm \eqref{4.3}}, \eqref{4.4} and \eqref{4.5}.
Let
$\lambda:=\{\lambda_{jk}\}_{j\in\zz, \, k \in {\mathbb Z}^n}
\in \dot{b}^{w,\tau}_{\cl,\infty,a}(\rn)$
and $\{{\mathfrak M}_{jk}\}_{j\in\zz, \, k \in {\mathbb Z}^n}$
be a family of molecules. Then
$
f=
\sum_{j=-\infty}^\infty
\sum_{k \in {\mathbb Z}^n}\lambda_{jk}{\mathfrak M}_{jk}
$
converges in ${\mathcal S}_\fz'(\rn)$.
\end{lemma}

\begin{proof}
Let $\varphi \in \cs_\fz(\rn)$ be a test function.
Lemma \ref{l4.2} shows
$
f_+:=
\sum_{j=1}^\infty
\sum_{k \in {\mathbb Z}^n}\lambda_{jk}{\mathfrak M}_{jk}
$
converges in ${\mathcal S}_\fz'(\rn)$.
So we need to prove
$
f_-:=
\sum_{j=-\infty}^0
\sum_{k \in {\mathbb Z}^n}\lambda_{jk}{\mathfrak M}_{jk}
$
converges in ${\mathcal S}_\fz'(\rn)$.

Let $M \gg 1$ be arbitrary.
From Lemma \ref{l2.4},
the definition of the molecules
and the fact that $\vz \in \cs_\infty(\rn)$,
it follows
that for all $j \le 0$ and $k\in\zz^n$,
\begin{align*}
\left|\int_{\rn}
{\mathfrak M}_{jk}(x)\varphi(x)\,dx\right|
\lesssim
2^{j(M+1)}(1+2^{-j}|k|)^{-N}.
\end{align*}
By $(\cl6)$, we conclude that
\begin{align*}
\left|\int_{\rn}
{\mathfrak M}_{jk}(x)\varphi(x)\,dx\right|
\lesssim
2^{j(M+1-\gamma)}(1+2^{-j}|k|)^{-N}(1+|k|)^{+\delta_0}
\|\chi_{Q_{jk}}\|_{\cl(\rn)}.
\end{align*}
Consequently, we see that
\begin{align*}
\left|\lambda_{jk}\int_{\rn}
{\mathfrak M}_{jk}(x)\varphi(x)\,dx\right|
\lesssim
2^{j(M+1-\gamma-\alpha_1)}
(1+|k|)^{-N+\alpha_3+\delta_0}
\|\lambda\|_{b^{w,\tau}_{\cl,\infty,a}(\rn)}.
\end{align*}
So by the assumption,
this inequality is summable over $j\le0$
and $k \in \Z^n$,
which completes the proof of Lemma \ref{l9.2}.
\end{proof}

The homogeneous version of Theorem \ref{t4.2},
which is the regular case of decompositions,
reads as below, whose proof is similar to that of Theorem \ref{t4.2}.
We omit the details.

\begin{theorem}\label{t14.2}
Let $K \in \zz_+$, $L=-1$,
$\alpha_1, \alpha_2, \alpha_3, \tau \in [0,\infty)$ and $q\in(0,\,\fz]$.
Suppose that $w\in\star-\dot{\mathcal W}^{\alpha_3}_{\alpha_1,\alpha_2}$.
Assume, in addition, that \eqref{3.35}, $\eqref{4.4}$,
$\eqref{4.21}$ and $\eqref{4.22}$ hold true, namely, $a\in (N_0+\alpha_3,\infty)$.
Then the pair
$(\dot{A}^{w,\tau}_{\cl,q,a}(\rn),\dot{a}^{w,\tau}_{\cl,q,a}(\rn))$
admits the atomic / molecular decompositions.
\end{theorem}

\subsection{Boundedness of operators}

We first focus on the counterpart of Theorem \ref{f-m}. To this end,
for $\ell\in\mathbb{N}$ and $\alpha\in\rr$, let
$m\in C^{\ell}(\rr^n\backslash\{0\})$  satisfy that,
for all $|\sigma|\leq \ell$,
\begin{equation}\label{fm1n}
\sup_{R\in(0,\infty)}\left[R^{-n+2\alpha+2|\sigma|}\int_{R\leq|\xi|<2R}|
\partial_{\xi}^{\sigma}m(\xi)|^2\,d\xi\right]\leq A_{\sigma}<\infty.
\end{equation}
The \emph{Fourier multiplier} $T_m$ is defined by setting,
for all $f\in\mathcal{S}_\fz(\rr^n)$,
$\widehat{(T_mf)}:=m\,\widehat{f}$.

We remark that when $\alpha=0$, the condition \eqref{fm1n} is
just the classical H\"ormander condition
(see, for example, \cite[p.\,263]{st93}).
One typical example satisfying \eqref{fm1n} with $\alpha=0$ is the
kernels of Riesz transforms $R_j$ given by $\widehat{(R_jf)}(\xi):=
-i\frac{\xi_i}{|\xi|}\widehat{f}(\xi)$ for all $\xi\in\rn\setminus\{0\}$ and $j\in\{1,\cdots,n\}$.
When $\alpha\neq 0$, a typical example satisfying \eqref{fm1n} for
any $\ell\in\mathbb{N}$ is given by $m(\xi):=|\xi|^{-\alpha}$ for $\xi\in\rn\setminus\{0\}$;
another example is the symbol of a differential operator $\partial^{\sigma}$
of order $\alpha:=\sigma_1+\cdots+\sigma_n$ with $\sigma:=
(\sigma_1,\cdots,\sigma_n)\in\mathbb{Z}^n_+$.

It was proved in \cite{yyz12} that the Fourier multiplier $T_m$ is
bounded on some Besov-type and Triebel-Lizorkin-type spaces for suitable indices.

Let $m$ be as in \eqref{fm1n} and $K$ its inverse
Fourier transform.
To obtain the boundedness of $T_m$ on the spaces $\dot B^{w,\tau}_{\cl,q,a}(\rn)$ and
$\dot F^{w,\tau}_{\cl,q,a}(\rn)$, we need the following conclusion, which is
\cite[Lemma 3.1]{yyz12}.

\begin{lemma}\label{fm-l3.1n}
$K\in\cs'_\fz(\rn).$
\end{lemma}

The next lemma comes from \cite[Lemma 4.1]{c}; see also \cite[Lemma 3.2]{yyz12}.

\begin{lemma}\label{fm-l3.2n}
Let $\psi$ be a Schwartz function on $\rr^n$ satisfy \eqref{1.2}.
Assume, in addition, that $m$ satisfies \eqref{fm1n}.
If $a\in(0,\fz)$ and $\ell>a+n/2$, then there exists a
positive constant $C$ such that, for all $j\in\zz$,
$$\int_{\rr^n}\left(1+{2^j|z|}\right)^a|(K\ast\psi_j)(z)|
\, dz\leq C 2^{-j\alpha}.$$
\end{lemma}

Next we show that, via a suitable way, $T_m$ can also be defined
on the whole spaces ${\dot F}^{w,\tau}_{\cl,q,a}(\rn)$
and ${\dot B}^{w,\tau}_{\cl,q,a}(\rn)$.
Let $\vz$ be a Schwartz function on $\rr^n$ satisfy \eqref{1.2}.
Then there exists $\vz^\dagger\in\cs(\rn)$ satisfying \eqref{1.2}
such that
\begin{equation}\label{fm-3.1n}
\sum_{i\in\zz} \vz^\dagger_i*\vz_i=\dz_0
\end{equation}
in $\cs'_\fz(\rn)$.
For any $f\in{\dot F}^{w,\tau}_{\cl,q,a}(\rn)$ or ${\dot B}^{w,\tau}_{\cl,q,a}(\rn)$,
we define a linear functional $T_mf$ on $\cs_\fz(\rn)$
by setting, for all $\phi\in\cs_\fz(\rn)$,
\begin{equation}\label{fm-3.2n}
\langle T_mf,\phi\rangle:=\sum_{i\in\zz}f\ast\vz_i^\dagger\ast\vz_i\ast\phi\ast K(0)
\end{equation}
as long as the right-hand side converges.
In this sense, we say $T_mf\in\cs'_\fz(\rn)$.
The following result shows that $T_mf$ in \eqref{fm-3.2n}
is well defined.

\begin{lemma}\label{fm-l3.3n}
Let $\ell\in(n/2,\fz)$, $\az\in\rr$, $a\in (0,\fz)$,
$\alpha_1, \alpha_2, \alpha_3, \tau\in[0,\fz)$, $q\in(0,\,\fz]$,
$w \in \dot\cw_{\alpha_1,\alpha_2}^{\alpha_3}$ and $f\in{\dot F}^{w,\tau}_{\cl,q,a}(\rn)$
or ${\dot B}^{w,\tau}_{\cl,q,a}(\rn)$. Then
the series in \eqref{fm-3.2n} is convergent
and the sum in the right-hand side of \eqref{fm-3.2n} is independent of the choices of the pair
$(\vz^\dagger,\vz)$. Moreover, $T_mf\in \cs'_\fz(\rn)$.
\end{lemma}

\begin{proof}
By similarity, we only consider $f\in {\dot F}^{w,\tau}_{\cl,q,a}(\rn)$.
Let $(\psi^\dagger,\psi)$ be another pair of functions satisfying
\eqref{fm-3.1n}.
Since $\phi\in\cs_\fz(\rn)$, by the Calder\'on reproducing formula,
we know that
$$\phi=\sum_{j\in\zz}\psi_j^\dagger\ast\psi_j\ast\phi$$
in $\cs_\fz(\rn)$. Thus,
\begin{eqnarray*}
\sum_{i\in\zz}f\ast\vz_i^\dagger\ast\vz_i\ast\phi\ast K(0)
&&=\sum_{i\in\zz}f\ast\vz_i^\dagger\ast\vz_i\ast
\lf(\sum_{j\in\zz}\psi_j^\dagger\ast\psi_j\ast\phi\r)\ast K(0)\\
&&=\sum_{i\in\zz}\sum_{j=i-1}^{i+1}
f\ast\vz_i^\dagger\ast\vz_i\ast\psi_j^\dagger\ast\psi_j\ast\phi\ast K(0),
\end{eqnarray*}
where the last equality follows from the fact that
$\vz_i\ast\psi_j=0$ if $|i-j|\geq2$.

Similar to the argument in Lemma \ref{fm-l3.3},
we see that
\begin{eqnarray*}
&&\sum_{i=0}^\fz|f\ast\vz_i\ast\vz_i^\dagger\ast\psi_i\ast
\psi_i^\dagger\ast\phi\ast K(0)|
\ls\|f\|_{\dot F_{\cl,q,a}^{w,\tau}(\rn)},
\end{eqnarray*}
where $a$ is an arbitrary positive number.
When $i<0$, notice that
\begin{eqnarray*}
&&\int_\rn |\vz_i\ast f(y-z)||\vz_i(-y)|\,dy \\
&&\hs\ls \sum_{k\in\zz^n}
\frac{2^{in}}{(1+2^i|2^{-i}k|)^a}\int_{Q_{ik}} |\vz_i\ast f(y-z)|\,dy\\
&&\hs\ls \sum_{k\in\zz^n}
\frac{2^{in-i\az_1}(1+2^i|z|)^{\az_3}}{(1+2^i|2^{-i}k|)^{a-\az_3}}
\inf_{y\in Q_{ik}} \omega(y-z,2^{-i})\int_{Q_{ik}} |\vz_i\ast f(y-z)|\,dy\\
&&\hs\ls \sum_{k\in\zz^n}
\frac{2^{-i\az_1}(1+2^i|z|)^{\az_3}}{(1+2^i|2^{-i}k|)^{a-\az_3}}
\inf_{y\in Q_{ik}} \{\omega(y-z,2^{-i})|\vz_i^*f(y-z)|\}\\
&&\hs\ls 2^{in-i\az_1}(1+2^i|z|)^{\az_3}
2^{-in\tau}\|f\|_{\dot A^{w,\tau}_{\cl,q,a}(\rn)},
\end{eqnarray*}
which, together with the fact that, for $M$ sufficiently large,
$$|\psi_i\ast \phi(y-z)|\ls 2^{iM} \frac{2^{in}}{(1+2^i|y-z|)^{n+M}}$$
and Lemma \ref{fm-l3.2n}, further implies that
\begin{eqnarray*}
&&\sum_{i<0}|f\ast\vz_i\ast\vz_i^\dagger\ast\psi_i\ast
\psi_i^\dagger\ast\phi\ast K(0)|\\
&&\hs=\sum_{i<0}\int_\rn|f\ast\vz_i\ast\vz_i^\dagger(-z)\psi_i\ast\psi_i^\dagger\ast\phi\ast K(z)|\,dz\\
&&\hs\ls \sum_{i<0} 2^{in-i\az_1}
2^{-in\tau}\|f\|_{\dot F^{w,\tau}_{\cl,q,a}(\rn)}\int_\rn(1+2^i|z|)^{\az_3}
|\psi_i\ast\psi_i^\dagger\ast\phi\ast K(z)|\,dz\\
&&\hs\ls \sum_{i<0} 2^{in-i\az_1}2^{iM}
2^{-in\tau}\|f\|_{\dot F^{w,\tau}_{\cl,q,a}(\rn)}\int_{\rn}\int_\rn
\frac{2^{in}(1+2^i|z|)^{\az_3}}{(1+2^i|y-z|)^{n+M}}
|\psi_i^\dagger\ast K(y)|\,dy\,dz\\
&&\hs\ls \sum_{i<0} 2^{2in+iM-i\az_1}
\|f\|_{\dot F_{\cl,q,a}^{w,\tau}(\rn)}
\ls\|f\|_{\dot F_{\cl,q,a}^{w,\tau}(\rn)},
\end{eqnarray*}
where we chose $M>\az_1-2n$.

Similar to the previous arguments, we see that
$$\lf|\sum_{i\in\zz}\sum_{j=i-1}^{i+1}
f\ast\vz_i^\dagger\ast\vz_i\ast\psi_j^\dagger\ast\psi_j\ast\phi\ast K(0)\r|
\ls\|f\|_{\dot F_{\cl,q,a}^{w,\tau}(\rn)}.$$
Thus, $T_mf$ in \eqref{fm-3.2n} is independent of the choices of the pair $(\vz^\dagger,\vz)$.
Moreover, the previous argument also implies that $T_mf\in \cs'_\fz(\rn)$,
which completes the proof of Lemma \ref{fm-l3.3n}.
\end{proof}

Then, by Lemma \ref{fm-l3.2n},
we immediately have the following lemma and we omit the details here.

\begin{lemma}\label{fm-l3.4n}
Let $\alpha\in\rr$, $a\in(0,\fz)$, $\ell\in\nn$ and
$\vz,\,\psi\in\cs_\fz(\rn)$ satisfy \eqref{1.2}.
Assume that $m$
satisfies \eqref{fm1n} and $f\in \cs'_\fz(\rn)$ such that $T_mf\in \cs'_\fz(\rn)$.
If $\ell>a+n/2$,
then there exists a positive constants $C$ such that,
for all $x,\,y\in\rr^n$ and $j\in\zz$,
$$|(T_mf\ast\psi_j)(y)|\leq C 2^{-j\alpha}\left(1+{2^j|x-y|}\right)^a(\vz_j^*f)_a(x).$$
\end{lemma}

\begin{theorem}\label{f-mn}
Let $\az\in\rr$, $a\in (0,\fz)$,
$\alpha_1, \alpha_2, \alpha_3, \tau\in[0,\fz)$, $q\in(0,\,\fz]$,
$w \in \dot\cw_{\alpha_1,\alpha_2}^{\alpha_3}$ and $\wz w(x,2^{-j})=2^{j\az} w(x,2^{-j})$
for all $x\in\rn$ and $j\in\zz$.
Suppose that $m$ satisfies \eqref{fm1n} with $\ell\in\mathbb{N}$ and $\ell>a+n/2$,
then there exists a positive constant $C_1$ such that, for all
$ f\in {\dot F}^{w,\tau}_{\cl,q,a}(\rr^n)$,
$\|T_mf\|_{{\dot F}^{\wz w,\tau}_{\cl, q,a}(\rr^n)}\leq
C_1\|f\|_{{\dot F}^{w,\tau}_{\cl, q,a}(\rr^n)}$
and a positive constant $C_2$ such that, for all
$f\in {\dot B}^{\wz w,\tau}_{\cl, q,a}(\rr^n)$,
$\|T_mf\|_{{\dot B}^{\wz w,\tau}_{\cl, q,a}(\rr^n)}\leq
C_2\|f\|_{{\dot B}^{w,\tau}_{\cl, q,a}(\rr^n)}.$
Similar assertions hold true
for $\dot\ce^{w,\tau}_{\cl, q,a}(\rr^n)$
and $\dot\cn^{w,\tau}_{\cl, q,a}(\rr^n)$.
\end{theorem}

\begin{proof}
By Lemma \ref{fm-l3.4n}, we see that,
if $\ell>a+n/2$, then for all $j\in\zz$ and $x\in\rn$,
$$2^{j\alpha}\left(\psi_j^{\ast}(T_mf)\right)_{a}(x)\lesssim
(\vz_j^*f)_a(x).$$
Then by the definitions of ${\dot F}^{w,\tau}_{\cl,q,a}(\rn)$
and ${\dot B}^{w,\tau}_{\cl,q,a}(\rn)$, we immediately conclude the desired
conclusions, which completes the proof of Theorem \ref{f-mn}.
\end{proof}

The following is an analogy to Theorem \ref{t3.2},
which can be proven similarly.
We omit the details.

\begin{theorem}\label{t9.5}
Let $s \in [0,\infty)$, $a>\alpha_3+N_0$,
$\az_1, \az_2, \az_3, \tau\in[0,\fz)$, $q\in(0,\fz]$ and
$w \in \dot{\mathcal W}^{\alpha_3}_{\alpha_1,\alpha_2}$.
Set $w^*(x,2^{-j}):= 2^{js}w_j(x)$
for all $x \in \rn$ and $j \in\zz$.
Then the lift operator $(-\Delta)^{s/2}$
is bounded from $\dot{A}^{w,\tau}_{\cl,q,a}(\rn)$
to $\dot{A}^{w^*,\tau}_{\cl,q,a}(\rn)$.
\end{theorem}

We consider the class $\dot{S}^0_{1,\mu}(\rn)$ with $\mu\in [0,1)$.
Recall that a function $a$ is said to belong
to a \emph{class $\dot{S}^m_{1,\mu}(\rn)$}
of $C^\infty(\R^n_x\times\R^n_\xi)$-functions if
$$\sup_{x,\xi \in \rn}
|\xi|^{-m-\|{\vec{\alpha}}\|_1-\mu\|{\vec{\beta}}\|_1}
|\partial^{\vec{\beta}}_x\partial^{\vec{\alpha}}_\xi a(x,\xi)|
\lesssim_{\vec{\alpha},\vec{\beta}}1$$
for all multiindices ${\vec{\alpha}}$ and ${\vec{\beta}}$.
One defines
$$
a(X,D)(f)(x):= \int_\rn a(x,\xi)\hat{f}(\xi)
e^{ix\cdot\xi}\,d\xi
$$
for all $f \in \cs_\fz(\rn)$ and $x\in\rn$.
Theorem \ref{t5.1} has a following counterpart, whose proof
is similar and omitted.

\begin{theorem}\label{t15.1}
Let a weight $w \in \dot{\mathcal W}^{\alpha_3}_{\alpha_1,\alpha_2}$
with $\alpha_1, \alpha_2, \alpha_3 \in [0,\infty)$
and a quasi-normed function space $\cl(\rn)$
satisfy $(\cl1)$ through $(\cl6)$.
Let $\mu\in[0,1)$,
$\tau\in(0,\fz)$ and $q\in (0,\infty]$.
Assume, in addition, that {\rm \eqref{3.35}} holds true, that is,
$a\in (N_0+\alpha_3,\infty)$,
where $N_0$ is as in $(\cl6)$.
Then the pseudo-differential operators
with symbol $\dot{S}_{1,\mu}^0(\rn)$
are bounded on $\dot{A}^{w,\tau}_{\cl,q,a}(\rn)$.
\end{theorem}

\subsection{Function spaces $\dot{A}^{w,\tau}_{\cl,q,a}(\rn)$ for $\tau$ large}

Now we have the following counterpart
for Theorem \ref{t6.2}
\begin{theorem}\label{t16.2}
Let $\oz\in\dot{\cw}_{\az_1,\az_2}^{\az_3}$
with $\az_1,\az_2,\az_3 \ge 0$.
Define
a new index $\wz \tau$ by
\begin{equation*}
\wz \tau:=
\limsup_{j \to \infty}
\left(\sup_{P\in\cq_j(\rn)}
\frac{1}{nj}\log_2\frac{1}{\|\chi_P\|_{\cl(\rn)}}
\right)
\end{equation*}
and a new weight $\wz \oz$ by
\begin{equation*}
\wz \oz(x,2^{-j}) :=
2^{jn(\tau-\wz \tau)}\oz(x,2^{-j}),
\quad
x \in \rn, \, j \in \Z.
\end{equation*}

Assume that $\tau$ and $\wz \tau$ satisfy
\begin{equation*}
\tau>\wz\tau \ge 0.
\end{equation*}
Then

{\rm (i)}
$\wz w\in
\dot{\cw}_{(\az_1-n(\tau-\wz \tau))_+,
(\az_2+n(\tau-\wz \tau))_+}^{\az_3}$;

{\rm (ii)}
for all $q\in(0,\fz)$ and $a>\alpha_3+N_0$,
$\dot{F}^{w,\tau}_{\cl,q,a}(\rn)$ and $\dot{B}^{w,\tau}_{\cl,q,a}(\rn)$
coincide, respectively, with $\dot{F}^{\wz w}_{\fz,\fz,a}(\rn)$ and
$\dot{B}^{\wz w}_{\fz,\fz,a}(\rn)$
with equivalent norms.
\end{theorem}

\subsection{Characterizations via differences and oscillations}

We can extend Theorems \ref{t7.1} and \ref{t7.2}
to homogeneous spaces as follows, whose proofs are omitted by similarity.

\begin{theorem}\label{t17.1}
Let $a,\, \az_1,\, \alpha_2,\, \alpha_3,\, \tau\in[0,\fz)$, $u\in[1,\fz]$, $q\in(0,\,\fz]$
and $w\in\star-{\mathcal W}^{\alpha_3}_{\alpha_1,\alpha_2}$.
If $M\in\nn$, $\az_1\in (a,M)$ and \eqref{7.2} holds true, then there exists a positive
constant $\wz C:=C(M)$, depending on $M$,
such that, for all $f\in\cs_\fz'(\rn)\cap L^1_\loc(\rn)$,
the following hold true:

{\rm(i)}
\begin{eqnarray*}
\left\|
\left\{\sup_{z \in \rn}\left[\oint_{|h| \le \wz{C}\,2^{-j}}
\frac{|\Delta_h^{M}f(\cdot+z)|^u}{(1+2^j|z|)^{au}}\,dh\right]^{1/u}
\right\}_{j \in \Z}\right\|_{\ell^q(\cl^w_\tau(\rn,\zz))}
\sim \|f\|_{\dot{B}^{w,\tau}_{\cl,q,a}(\rn)}
\end{eqnarray*}
with the implicit positive constants
independent of $f$.

\rm{(ii)}
\begin{eqnarray*}
\left\|\left\{\sup_{z \in \rn}\left[\oint_{|h| \le \wz{C}\,2^{-j}}
\frac{|\Delta_h^{M}f(\cdot+z)|^u}{(1+2^j|z|)^{au}}\,dh\right]^{1/u}
\right\}_{j \in \Z}\right\|_{\cl^w_\tau(\ell^q(\rn,\zz))}
\sim
\|f\|_{\dot{F}^{w,\tau}_{\cl,q,a}(\rn)}
\end{eqnarray*}
with the implicit positive constants
independent of $f$.

\rm{(iii)}
\begin{eqnarray*}
\left\|\left\{\sup_{z \in \rn}\left[\oint_{|h| \le \wz{C}\,2^{-j}}
\frac{|\Delta_h^{M}f(\cdot+z)|^u}{(1+2^j|z|)^{au}}\,dh\right]^{1/u}
\right\}_{j \in \Z}\right\|_{\ell^q(\cn\cl^w_\tau(\rn,\zz))}
\sim
\|f\|_{\dot{\cn}^{w,\tau}_{\cl,q,a}(\rn)}
\end{eqnarray*}
with the implicit positive constants
independent of $f$.

\rm{(iv)}
\begin{eqnarray*}
\left\|\left\{\sup_{z \in \rn}\left[\oint_{|h| \le \wz{C}\,2^{-j}}
\frac{|\Delta_h^{M}f(\cdot+z)|^u}{(1+2^j|z|)^{au}}\,dh\right]^{1/u}
\right\}_{j \in \Z}\right\|_{\ce\cl^w_\tau(\ell^q(\rn,\zz))}
\sim
\|f\|_{\dot{\ce}^{w,\tau}_{\cl,q,a}(\rn)}
\end{eqnarray*}
with the implicit positive constants
independent of $f$.
\end{theorem}

\begin{theorem}\label{t17.2}
Let $a,\, \alpha_1,\, \alpha_2,\, \alpha_3,\, \tau \in [0,\fz)$,
$u\in[1,\fz]$, $q\in(0,\,\fz]$
and $w\in\star-\dot{\mathcal W}^{\alpha_3}_{\alpha_1,\alpha_2}$.
If $M\in\nn$, $\az_1\in (a,M)$ and \eqref{7.2} holds true, then,
for all $f \in \cs'_\fz(\rn)\cap L^1_\loc(\rn)$,
the following hold true:

{\rm(i)}
\begin{eqnarray*}
\left\|\left\{\sup_{z \in \rn}
\frac{{\rm osc}_u^{M}f(\cdot+z,2^{-j})}{(1+2^j|z|)^a}
\right\}_{j \in \Z}\right\|_{\ell^q(\cl^w_\tau(\rn,\zz))}
\sim \|f\|_{\dot{B}^{w,\tau}_{\cl,q,a}(\rn)}
\end{eqnarray*}
with the implicit positive constants
independent of $f$.

{\rm(ii)}
\begin{eqnarray*}
\left\|\left\{\sup_{z \in \rn}
\frac{{\rm osc}_u^{M}f(\cdot+z,2^{-j})}{(1+2^j|z|)^a}
\right\}_{j \in \Z}\right\|_{\cl^w_\tau(\ell^q(\rn,\zz))}
\sim \|f\|_{\dot{F}^{w,\tau}_{\cl,q,a}(\rn)}
\end{eqnarray*}
with the implicit positive constants
independent of $f$.

{\rm(iii)}
\begin{eqnarray*}
\left\|\left\{\sup_{z \in \rn}
\frac{{\rm osc}_u^{M}f(\cdot+z,2^{-j})}{(1+2^j|z|)^a}
\right\}_{j \in \Z}\right\|_{\ell^q(\cn\cl^w_\tau(\rn,\zz))}
\sim \|f\|_{\dot{\cn}^{w,\tau}_{\cl,q,a}(\rn)}
\end{eqnarray*}
with the implicit positive constants
independent of $f$.

{\rm(iv)}
\begin{eqnarray*}
\left\|\left\{\sup_{z \in \rn}
\frac{{\rm osc}_u^{M}f(\cdot+z,2^{-j})}{(1+2^j|z|)^a}
\right\}_{j \in \Z}\right\|_{\ce\cl^w_\tau(\ell^q(\rn,\zz))}
\sim \|f\|_{\dot{\ce}^{w,\tau}_{\cl,q,a}(\rn)}
\end{eqnarray*}
with the implicit positive constants
independent of $f$.
\end{theorem}

Next, we transplant Theorems \ref{t8.30} and \ref{t8.32} to the homogeneous
case. Again, since their proofs are similar, respectively, to the
inhomogeneous cases, we omit the details.

\begin{theorem}\label{t18.30}
Suppose that $a>N$
and that $(\ref{eq:8.30})$ is satisfied:
\[
(1+|x|)^{-N} \in \cl(\rn) \cap \cl'(\rn).
\]
Assume, in addition, that
there exists a positive constant $C$ such that,
for any finite sequence $\{\varepsilon_k\}_{k=-k_0}^{k_0}$
taking values $\{-1,1\}$,
\begin{equation}\label{eq:18.31}
\left\|\sum_{k=-k_0}^{k_0} \varepsilon_k\varphi_k*f\right\|_{\cl(\rn)}
\le C\|f\|_{\cl(\rn)}, \ \
\left\|\sum_{k=-k_0}^{k_0} \varepsilon_k\varphi_k*g\right\|_{\cl'(\rn)}
\le C\|g\|_{\cl'(\rn)}
\end{equation}
for all $f \in \cl(\rn)$ and $g \in \cl'(\rn)$. Then,
$\cl(\rn)$ and $\cl'(\rn)$
are embedded into $\cs_\fz'(\rn)$,
$\cl(\rn)$ and $\cl'(\rn)$
are embedded into $\cs_\fz'(\rn)$,
and $\cl(\rn)$ and $\dot{\ce}^{0,0}_{\cl,2,a}(\rn)$
coincide.
\end{theorem}

\begin{theorem}\label{t18.32}
Let $\cl(\rn)$ be a Banach space of functions
such that the spaces $\cl^{p}(\rn)$ and $(\cl')^{p}(\rn)$
are Banach spaces of functions
and that
the maximal operator $M$ is bounded
on $(\cl^{p}(\rn))'$ and $((\cl')^{p}(\rn))'$
for some $p \in (1,\infty)$.
Then $(\ref{eq:18.31})$ holds true.
In particular,
if $a>N$
and $(1+|x|)^{-N} \in \cl(\rn) \cap \cl'(\rn)$,
then
$\cl(\rn)$ and $\cl'(\rn)$
are embedded into $\cs_\fz'(\rn)$,
and $\cl(\rn)$ and $\dot{\ce}^{0,0}_{\cl,2,a}(\rn)$
coincide.
\end{theorem}

As a corollary $\cl(\rn)$ enjoys the following characterization.
\begin{corollary}
Let $\cl(\rn)$ be a Banach space of functions
such that $\cl^{p}(\rn)$ and $(\cl')^{p}(\rn)$
are Banach spaces of functions
and that
the maximal operator $M$ is bounded
on $(\cl^{p}(\rn))'$ and $((\cl')^{p}(\rn))'$
for some $p \in (1,\infty)$.
If $a>N$
and $(1+|x|)^{-N} \in \cl(\rn) \cap \cl'(\rn)$,
then
\begin{eqnarray*}
\|f\|_{\cl(\rn)}&&\sim
\left\|\left\{\sup_{z \in \rn}\left[\oint_{|h| \le \wz{C}\,2^{-j}}
\frac{|\Delta_h^{M}f(\cdot+z)|^u}{(1+2^j|z|)^{au}}\,dh\right]^{1/u}
\right\}_{j \in \Z}\right\|_{\ce\cl^1_0(\ell^2(\rn,\zz))}\\
&&\sim\left\|\left\{\sup_{z \in \rn}
\frac{{\rm osc}_u^{M}f(\cdot+z,2^{-j})}{(1+2^j|z|)^a}
\right\}_{j \in \Z}\right\|_{\ce\cl^1_0(\ell^2(\rn,\zz))}
\end{eqnarray*}
with the implicit positive constants
independent of $f \in \cl(\rn)$.
\end{corollary}

\section{Applications and examples}
\label{s10}

Now we present some examples for ${\mathcal L}(\rn)$
and survey what has been obtained recently.

\subsection{Morrey spaces}
\label{s10.1}

\paragraph{Morrey spaces}

Now, to begin with,
we consider the case
when ${\mathcal L}(\R^n):={\mathcal M}^p_u(\R^n)$,
the Morrey space.
Recall that definition was given in Example \ref{e5.11}.
Besov-Morrey spaces and Triebel-Lizorkin-Morrey spaces
are function spaces whose norms are obtained
by replacing the $L^p$-norms
with the Morrey norms.
More precisely,
the \emph{Besov-Morrey norm} $\|\cdot\|_{{\mathcal N}_{pqr}^s(\rn)}$ is given by
\begin{equation*}
\rVert f \lVert_{\mathcal{N}_{pqr}^s(\rn)}
:=
\rVert \Phi*f \lVert_{\mathcal{M}^p_q(\rn)}
+
\left(
\sum_{j=1}^\infty
2^{jsr}\rVert \varphi_j*f \lVert_{\mathcal{M}^p_q(\rn)}^r
\right)^{1/r}
\end{equation*}%
and the \emph{Triebel-Lizorkin-Morrey norm}
$\|\cdot\|_{{\mathcal E}_{pqr}^s(\rn)}$ is given by
\begin{equation*}
\rVert f \lVert_{\mathcal{E}_{pqr}^s(\rn)}
:=
\rVert\Phi*f \lVert_{\mathcal{M}^p_q(\rn)}
+
\left\| \left(
\sum_{j=1}^\infty
2^{jsr} |\varphi_j*f|^r
\right)^{1/r}\right\|_{\mathcal{M}^p_q(\rn)}
\end{equation*}%
for $0<q \le p<\infty, \, r\in(0,\infty]$ and $s \in {\mathbb R}$,
where $\Phi$ and $\vz$ are, respectively, as in \eqref{1.1} and \eqref{1.2},
and $\vz_j(\cdot)=2^{jn}\vz(2^j\cdot)$ for all $j\in\nn$.
The \emph{spaces} ${\mathcal N}_{pqr}^s(\rn)$ and ${\mathcal E}_{pqr}^s(\rn)$
denote the set of all $f \in {\mathcal S}'(\rn)$
such that the norms
$\rVert f \lVert_{\mathcal{N}_{pqr}^s(\rn)}$
and
$\rVert f \lVert_{\mathcal{E}_{pqr}^s(\rn)}$
are finite, respectively.
Let $\ca_{pqr}^s(\R^n)$ denote
either one of ${\mathcal N}_{pqr}^s(\rn)$
and ${\mathcal E}_{pqr}^s(\rn)$.
Write
\begin{gather*}
B^{w,\tau}_{p,u,q,a}(\rn)
:=
B^{w,\tau}_{{\mathcal M}^p_u,q,a}(\rn)\hs
\mbox{and}\hs
F^{w,\tau}_{p,u,q,a}(\rn)
:=
F^{w,\tau}_{{\mathcal M}^p_u,q,a}(\rn).
\end{gather*}
Then, if we let
$
w_j(x):=2^{js}
 \ (x \in \rn, \, j \in \Z_+)
$
with $s \in \R$, then
it is easy to show that
$\cn^s_{p,u,q,a}(\rn):=\cn^{s,0}_{p,u,q,a}(\rn)$
coincides with
${\mathcal N}^{s}_{p,u,q}(\rn)$
when $a>\frac{n}{\min(1,u)}$
and that
$F^s_{p,u,q,a}(\rn):=F^{s,0}_{p,u,q,a}(\rn)$
coincides with ${\mathcal E}^{s}_{p,u,q}(\rn)$
when $a>\frac{n}{\min(1,u,q)}$.
Indeed, this is just a matter of applying
the Plancherel-Polya-Nikolskij inequality
(Lemma \ref{l1.1})
and the maximal inequalities
obtained in \cite{st07,tx}.
These function spaces are dealt in \cite{st07,tx}.

Observe that $(\cl1)$ through $(\cl6)$ hold true in this case.

There exists another point of view
of these function spaces.
Recall that the function space
$A^{s,\tau}_{p,q}(\rn)$, defined by \eqref{3.5},
originated from \cite{yy1,yy2,yy3}.
The following is known,
which Theorem \ref{t8.6} in the present paper extends.
\begin{proposition}[{\rm \cite[Theorem 1.1]{ysy}}]\label{p10.1}
Let $s \in {\mathbb R}$.

{\rm (i)}
If $0<p<u<\infty$ and $q\in(0,\infty)$,
then ${\mathcal N}_{upq}^s(\rn)$ is a proper
subset of $B_{p,q}^{s,\frac{1}{p}-\frac{1}{u}}(\rn)$.

{\rm (ii)}
If $0<p<u<\infty$ and $q=\infty$, then
${\mathcal N}_{upq}^s(\rn)
=B_{p,q}^{s,\frac{1}{p}-\frac{1}{u}}(\rn)$ with equivalent norms.

{\rm (iii)}
If $0<p\le u<\infty$ and $q\in(0,\infty]$, then
${\mathcal E}_{upq}^s(\rn)
=F_{p,q}^{s,\frac{1}{p}-\frac{1}{u}}(\rn)$ with equivalent norms.

An analogy for homogeneous spaces
is also true.
\end{proposition}
Other related spaces are inhomogeneous \emph{Hardy-Morrey spaces}
$h\cm^p_q(\rn)$,
whose \emph{norm} is given by
\begin{equation*}
\|f\|_{h{\mathcal M}^p_q(\rn)}
:=
\left\|
\sup_{0<t \le 1}|t^{-n}\Phi(t^{-1}\cdot)*f|
\right\|_{{\mathcal M}^p_q(\rn)}
\end{equation*}%
for all $f \in \cs'(\rn)$ and $0<q \le p<\infty$,
where $\Phi$ is as in \eqref{1.1}.

Now in this example
\eqref{2.10} actually reads
\begin{equation*}
\cl(\rn):={\mathcal M}^p_q(\rn), \quad
\theta:=\min\{1,q\}, \quad
N_0:=\frac{n}{p}+1, \quad
\gamma:=\frac{n}{p}, \quad
\delta:=0
\end{equation*}%
and $\cl(\rn)$ satisfies $(\cl1)$ through $(\cl6)$ and $(\cl8)$ (see \cite{st05,tx}).
Meanwhile \eqref{2.11} actually reads as
\begin{equation*}
w_j(x) := 1\mbox{ for all }x\in\rn \mbox{ and }j\in\zz_+, \quad
\alpha_1=\alpha_2=\alpha_3=0.
\end{equation*}%
Hence, \eqref{2.12} is replaced by
\begin{equation*}
\tau \in [0,\infty), \, q\in (0,\infty], \, a>\frac{n}{p}+1.
\end{equation*}%

We refer to \cite{jw09,jw10,ky,s08-2,s09,st07,sw}
for more details and applications
of Hardy-Morrey spaces,
Besov-Morrey spaces and Triebel-Lizorkin-Morrey spaces.
Indeed, in \cite{ky,s08-2,st07},
Besov-Morrey spaces and its applications are investigated;
Triebel-Lizorkin-Morrey spaces
are dealt in \cite{s08-2,s09,st07};
Hardy-Morrey spaces are
defined and considered in
\cite{jw09,jw10,s09,sw}
and Hardy-Morrey spaces are applied to PDE
in \cite{jw10}.
We also refer to \cite{ist}
for more related results about Besov-Morrey spaces
and Triebel-Lizorkin-Morrey spaces,
where the authors covered weighted settings.

\paragraph{Generalized Morrey spaces}

We can also consider generalized Morrey spaces.
Let $p\in(0,\infty)$ and
$\phi:\,(0,\infty)\to(0,\infty)$ be a suitable function.
For a function $f$ locally in $L^p(\rn)$, we set
\begin{equation*}
\|f\|_{{\cal M}_{\phi,p}(\rn)}
:=
\sup_{Q \in {\cal D}(\rn)}
\phi(\ell(Q))\left[\frac1{|Q|}\int_{Q}|f(x)|^p\,dx\right]^{\frac{1}{p}},
\end{equation*}%
where $\ell(Q)$ denotes the \emph{side-length} of the cube $Q$.
The \emph{generalized Morrey space} ${\cal M}_{\phi,p}(\rn)$
is defined to be the space of all functions $f$ locally in $L^p(\rn)$
such that $\|f\|_{{\cal M}_{\phi,p}(\rn)}<\fz$.
Let $\cl(\rn):={{\cal M}_{\phi,p}(\rn)}$.
Observe that $(\cl1)$ through $(\cl6)$ are true under a suitable condition
on $\phi$. At least $(\cl1)$ through $(\cl5)$ hold true
without assuming any condition on $\phi$.
Morrey-Campanato spaces with growth function $\phi$ were first
introduced by Spanne \cite{sp65,sp66} and Peetre \cite{p66},
which treat singular integrals of convolution type.
In 1991, Mizuhara \cite{m91} studied the boundedness of
the Hardy-Littlewood maximal operator
on Morrey spaces with growth function $\phi$.
Later in 1994, Nakai \cite{n94} considered the boundedness of
singular integral (with non-convolution kernel),
and fractional integral operators
on Morrey spaces with growth function $\phi$.
In \cite{n04}, Nakai started to define the space ${\cal M}_{\phi,p}(\rn)$.
Later,
this type of function spaces was used in \cite{KNS,n94,sst10}.
We refer to \cite{n08} for more details
of this type of function spaces.
In \cite[p.\, 445]{n00},
Nakai has proven the following (see \cite[(10.6)]{sst11} as well).

\begin{proposition}
Let $p\in(0,\infty)$
and
$\phi:(0,\infty) \to (0,\infty)$
be an arbitrary function.
Then
there exists a function $\phi^*:(0,\infty) \to (0,\infty)$
such that
\begin{equation}\label{10.8}
\mbox{$\phi^*(t)$ is nondecreasing and $[\phi^*(t)]^pt^{-n}$ is nonincreasing,}
\end{equation}
and that ${\cal M}_{\phi,p}(\rn)$ and ${\cal M}_{\phi^*,p}(\rn)$
coincide.
\end{proposition}

We rephrase $(\cl8)$ by using \eqref{10.8} as follows.

\begin{proposition}[{\rm \cite[Theorem 2.5]{s08}}]
\label{p10.3}
Suppose that $\phi:(0,\infty) \to (0,\infty)$
is an increasing function.
Assume that $\phi:(0,\infty) \to (0,\infty)$ satisfies
\begin{equation}\label{10.9}
\int_r^\infty
\phi(t)\,\frac{dt}{t}
\sim
\phi(r)
\end{equation}
for all $r\in(0,\fz)$.
Then, for all $u\in(1,\infty]$ and sequences of measurable functions $\{f_j\}_{j=1}^\fz$,
\begin{equation*}
\lf\|
\lf(\sum_{j=1}^\fz [Mf_j]{}^u\r)^\frac{1}{u}
\r\|_{\cm_{\phi,p}(\rn)}
\sim
\lf\|
\lf(\sum_{j=1}^\fz |f_j|^u\r)^\frac{1}{u}
\r\|_{\cm_{\phi,p}(\rn)}
\end{equation*}%
with the implicit equivalent positive constants independent of $\{f_j\}_{j=1}^\fz$.
\end{proposition}

\begin{remark}\label{r10.1}
In {\rm \cite{s08}},
it was actually assumed that
\begin{equation}\label{10.11}
\int_r^\infty
\phi(t)\,\frac{dt}{t}
\ls
\phi(r) \mbox{ for all }r\in(0,\fz).
\end{equation}
However, under the assumption {\rm \eqref{10.8}},
the conditions {\rm \eqref{10.9}} and {\rm \eqref{10.11}}
are mutually equivalent.
\end{remark}

Now in this example
\eqref{2.10} actually reads as
\begin{equation*}
\cl(\rn):={\cal M}_{\phi,p}(\rn), \quad
\theta:=1, \quad
N_0:=\frac{n}{p}+1, \quad
\gamma:=\frac{n}{p}, \quad
\delta:=0
\end{equation*}%
and $\cl(\rn)$ satisfies $(\cl8)$ by Proposition \ref{p10.3}
and also $(\cl1)$ through $(\cl6)$.
While \eqref{2.11} actually reads
\begin{equation*}
w_j(x) := 1\mbox{ for all }x\in\rn \mbox{ and }j\in\zz_+, \quad
\alpha_1=\alpha_2=\alpha_3=0.
\end{equation*}%
Hence, \eqref{2.12} is replaced by
\begin{equation*}
\tau \in [0,\infty), \, q\in (0,\infty], \, a>\frac{n}{p}+1.
\end{equation*}%

\subsection{Orlicz spaces}\label{s10.2}

Now let us recall the definition
of Orlicz spaces which were given in Example \ref{e5.11}.

The proof of the following estimate can be found in \cite{cgmmp06}.
\begin{lemma}\label{l10.1}
If a Young function $\Phi$ satisfies
$$
{\rm (Doubling \ condition)}\,
\sup_{t>0}\frac{\Phi(2t)}{\Phi(t)}<\infty, \quad
\mbox{\rm ($\nabla_2$-condition)}\,
\inf_{t>0}\frac{\Phi(2t)}{\Phi(t)}>2,
$$
then for all $u\in(1,\fz]$ and sequences
of measurable functions $\{f_j\}_{j=1}^\infty$,
\begin{equation}\label{10.15}
\left\| \left(\sum_{j=1}^\fz
[Mf_j]{}^u \right)^{\frac{1}{u}}\right\|_{L^\Phi(\rn)}
\sim
\left\|\left(\sum_{j=1}^\fz|f_j|^u \right)^{\frac{1}{u}}\right\|_{L^\Phi(\rn)}
\end{equation}
with the implicit equivalent positive constants independent of $\{f_j\}_{j=1}^\fz$.
\end{lemma}

Thus, by Lemma \ref{l10.1}, $L^\Phi(\rn)$ satisfies $(\cl8)$.
Now in this example $\cl(\rn):=L^\Phi(\rn)$ also satisfy
$(\cl1)$ through $(\cl6)$ with the parameters
\eqref{2.10} and \eqref{2.11}
actually read as
\begin{equation*}
\cl(\rn):=L^\Phi(\rn), \quad \theta:=1, \quad N_0:=n+1, \quad \gamma:=n, \quad \delta:=0.
\end{equation*}
Indeed, since $\Phi$ is a Young function,
we have
\[
\int_{\rn}\Phi(2^{jn}\chi_{Q_{j0}}(x))\,dx
=
2^{-jn}\Phi(2^{jn})
\ge 1.
\]
Consequently
$\|\chi_{Q_{j0}}\|_{L^\Phi(\rn)} \ge 2^{-jn}$.
Meanwhile as before,
\begin{equation*}
w_j(x) := 1\mbox{ for all }x\in\rn \mbox{ and }j\in\zz_+, \quad
\alpha_1=\alpha_2=\alpha_3=0.
\end{equation*}
Hence \eqref{2.12} now stands for
\begin{equation*}
\tau \in [0,\infty), \, q\in (0,\infty], \, a>n+1.
\end{equation*}

This example can be generalized somehow.
Given a Young function $\Phi$,
define the \emph{mean Luxemburg norm }of $f$ on a cube $Q\in\cq(\rn)$ by
\begin{equation*}
\|f\|_{\Phi,\,Q}
:=
\inf\left\{\lambda>0:\,
\frac{1}{|Q|}\int_{Q}\Phi\left(\frac{|f(x)|}{\lambda}\right)\,dx\le 1\r\}.
\end{equation*}%
When
$\Phi(t):= t^p$ for all $t\in(0,\fz)$ with
$p\in[1,\infty)$,
\begin{equation*}
\|f\|_{\Phi,\,Q}
=
\left(\frac{1}{|Q|}\int_{Q}|f(x)|^p\,dx\r)^{1/p},
\end{equation*}%
that is, the mean Luxemburg norm coincides with the (normalized) $L^p$ norm.
The \emph{Orlicz-Morrey space}
$\cl^{\Phi,\,\phi}(\R^n)$
consists of all locally integrable functions $f$ on $\R^n$
for which the \emph{norm}
\begin{equation*}
\|f\|_{\cl^{\Phi,\,\phi}(\rn)}
:=
\sup_{Q\in\cq(\rn)}\phi(\ell(Q))\|f\|_{\Phi,\,Q}
\end{equation*}%
is finite.
As is written in \cite[Section 1]{SaSuTa4},
we can assume without loss of generality
that $t \mapsto \phi(t)$ and $t \mapsto t^n\phi(t)^{-1}$
are both increasing.

Using \cite[Proposition 2.17]{SaSuTa4},
we extend \cite{Ki1,Ki2}
and \cite[Proposition 2.17]{SaSuTa4} to the vector-valued
version.
In the next proposition,
we shall establish that $(\cl8)$ holds true
provided that
\[
\int_{1}^{t}
\Phi\left(\frac{t}{s}\right)\,ds
\le
\Phi(Ct) \quad(t\in (0,\fz))
\]
for some positive constant $C$ and for all $t\in(1,\fz)$.

\begin{proposition}\label{p10.4}
Let $q\in(0,\infty]$.
Let $\Phi$ be a normalized Young function.
Then the following are equivalent:

{\rm (i)}
The maximal operator $M$
is locally bounded in the norm determined by $\Phi$,
that is, there exists a positive constant $C$ such that,
for all cubes $Q\in\cq(\rn)$,
$$
\|M(g\chi_Q)\|_{\Phi,\,Q}
\le C
\|g\|_{\Phi,\,Q}.
$$

{\rm (ii)}
The function space $\cl(\rn):=\cl^{\Phi,\,\phi}(\rn)$ satisfies $(\cl8)$
with some $0<r<q$ and $w \equiv 1$.
Namely,
there exist $R \gg 1$ and $r\in(0,\fz)$ such that
\[
\|\{(\eta_{j,R}*|f_j|^r)^{1/r}
\}_{j \in \Z_+}\|_{\cl^{\Phi,\,\phi}(\ell^q(\rn,\zz_+))}
\lesssim
\|\{f_j\}_{j \in \Z_+}
\|_{\cl^{\Phi,\,\phi}(\ell^q(\rn,\zz_+))}
\]
holds true
for all $\{f_j\}_{j \in \N} \subset \cl^{\Phi,\,\phi}(\rn)$,
where the implicit positive constant is independent of
$\{f_j\}_{j \in \N}$.

{\rm (iii)}
The function $\Phi$ satisfies that, for some positive constant $C$ and all $t\in(1,\fz)$,
$$
\int_{1}^{t}
\frac{t}{s}\Phi'(s)\,ds
\le
\Phi(Ct).
$$

{\rm (iv)}
The function $\Phi$ satisfies that, for some positive constant $C$ and all $t\in(1,\fz)$,
$$
\int_{1}^{t}
\Phi\left(\frac{t}{s}\right)\,ds
\le
\Phi(Ct).
$$
\end{proposition}
Therefore, a result similar to Besov-Morrey spaces
and Triebel-Lizorkin-Morrey spaces can be obtained
as before.

\begin{proof}[Proof of Proposition \ref{p10.4}]
The proof is based upon a minor modification
of the known results.
However, the proof not being found in the literatures,
we outline the proof here.
In \cite[Proposition 2.17]{SaSuTa4} we have shown
that (i), (iii) and (iv) are mutually equivalent.
It is clear that (ii) implies (i).
Therefore, we need to prove that (iv) implies (ii).
In \cite[Claim 5.1]{SaSuTa4} we also have shown
that the space $\cl^{\Phi,\,\phi}(\rn)$ remains
the same if we change the value $\Phi(t)$ with $t \le 1$.
Therefore, we can and do assume
\[
\int_0^{t}
\frac{t}{s}\Phi'(s)\,ds
\le
\Phi(Ct)
\]
for all $t\in(0,\fz)$.
Consequently,
\begin{eqnarray*}
&&\int_{\R^n}
\Phi\left(
\left\{\sum_{j=1}^\infty\lf[ M(|f_j|^r)(x)\r]^{q/r}\right\}^{1/q}
\right)\,dx\\
&&\quad=
\int_{0}^\infty
\Phi'(t)\left|\left\{x \in \rn\,:\,
\left(\sum_{j=1}^\infty\lf[ M(|f_j|^r)(x)\r]^{q/r}\right)^{1/q}>t
\right\}\right|\,dt\\
&&\quad\ls
\int_{\R^n}
\int_0^\infty
\frac{\Phi'(t)}{t}\chi_{\left\{x \in \rn\,:\,
\left[\sum_{j=1}^\infty |f_j(x)|^{q}\right]^{1/q}>\frac{t}{2}
\right\}}(x)
\left[\sum_{j=1}^\infty |f_j(x)|^{q}\right]^{1/q}\,dt\\
&&\quad\ls
\int_{\R^n}
\Phi\left(
C_0\left[\sum_{j=1}^\infty |f_j(x)|^{q}\right]^{1/q}
\right)\,dx
\end{eqnarray*}
for some positive constant $C_0$.
This implies that,
whenever
$$
\left\|\left[\sum_{j=1}^\infty |f_j|^{q}\right]^{1/q}\right\|_{L^{\Phi,\phi}(\rn)}
\le \frac{1}{C_0},
$$
we have
\[
\int_{\R^n}
\Phi\left(
\left\{\sum_{j=1}^\infty \lf[M(|f_j|^r)(x)\r]^{q/r}\right\}^{1/q}
\right)\,dx \le 1.
\]
From the definition of the Orlicz-norm $\|\cdot\|_{L^{\Phi,\phi}(\rn)}$, we have
\eqref{10.15}. Once we obtain \eqref{10.15},
we can go through the same argument as
\cite[Theorem 2.4]{st05}.
We omit the details, which completes the proof of Proposition \ref{p10.4}.
\end{proof}

Now in this example,
if we assume the conditions of Proposition \ref{p10.4}, then
$(\cl1)$ through $(\cl6)$ hold true with
 the conditions on the parameters
\eqref{2.10} and \eqref{2.11}
actually read as
\begin{equation*}
\cl(\rn):=L^{\Phi,\phi}(\rn), \quad \theta:=1, \quad N_0:=n+1, \quad \gamma:=n,
\quad \delta:=0.
\end{equation*}%
Indeed, since $\Phi$ is a Young function,
again we have
\[
2^{jn}\int_{\rn}\Phi(\chi_{Q_{j0}}(x)/\lambda)\,dx
=
\Phi(\lambda^{-1})
\]
for $\lambda>0$.
Consequently
$\|\chi_{Q_{j0}}\|_{\Phi,Q_{j0}}=1/\Phi^{-1}(1)$
and hence
\[
\phi(2^{-j})\|\chi_{Q_{j0}}\|_{\Phi,Q_{j0}}
=
\phi(2^{-j})
=
\phi(2^{-j})2^{jn}2^{-jn}
\ge
\phi(1)2^{-jn}.
\]
Here we invoked an assumption
that $\phi(t)t^{-n}$ is a decreasing function.

Since ${\mathcal L}^{\Phi,\phi}(\rn)$ satisfies $({\mathcal L}8)$,
we obtain
$M\chi_{[-1,1]^n} \in {\mathcal L}^{\Phi,\phi}(\rn)$,
showing that $N_0:=n$ will do in this setting.

Meanwhile as before,
\begin{equation*}
w_j(x) := 1\mbox{ for all }x\in\rn \mbox{ and }j\in\zz_+, \quad
\alpha_1=\alpha_2=\alpha_3=0.
\end{equation*}%
Hence \eqref{2.12} now stands for
\begin{equation*}
\tau \in [0,\infty), \, q\in (0,\infty], \, a>n+1.
\end{equation*}

Finally, we remark that Orlicz spaces
are examples to which the results
in Subsection 9.2 apply.

\subsection{Herz spaces}
\label{s10.3}

Let $p,q \in(0, \infty]$ and $\alpha \in \R$.
We let $Q_0:= [-1,1]^n$ and $C_j:= [-2^j,2^j]^n \setminus [-2^{j-1},2^{j-1}]^n$
for all $j \in {\mathbb N}$.
Define the \emph{inhomogeneous Herz space $K_{p,q}^\alpha(\rn)$} to be
the set of all measurable functions $f$ for which the \emph{norm}
\[
\| f \|_{K_{p,q}^\az(\rn)}
:=
\| \chi_{Q_0} \cdot f \|_{L^p(\rn)}
+
\left(
\sum_{j=1}^\infty
2^{jq\alpha}
\| \chi_{C_j}  f \|_{L^p(\rn)}^q
\right)^\frac{1}{q}
\]
is finite,
where we modify naturally
the definition above when $p=\fz$ or $q=\infty$.

 The following is shown by Izuki
\cite{Iz}, which is $(\cl8)$ of this case. A complete
theory of Herz-type spaces was given in \cite{lyh}.

\begin{proposition}\label{p10.5}
Let $p\in(1,\infty)$,
$q,u \in(0, \infty]$
and $\alpha \in (-1/p,1/p')$.
Then, for all sequences of measurable functions $\{f_j\}_{j=1}^\infty$,
\[
\left\|
\left(\sum_{j=1}^\infty [Mf_j]{}^u\right)^{\frac{1}{u}}
\right\|_{K_{p,q}^\alpha(\rn)}
\sim
\left\|
\left(\sum_{j=1}^\infty |f_j|^u\right)^{\frac{1}{u}}
\right\|_{K_{p,q}^\alpha(\rn)}
\]
with the implicit equivalent positive constants independent of $\{f_j\}_{j=1}^\fz$.
\end{proposition}

Now in this example $(\cl1)$ through $(\cl6)$ hold true with
the parameters in
\eqref{2.10}, \eqref{2.11} and \eqref{2.12}
actually satisfying that
\begin{equation*}
\cl(\rn):=K_{p,q}^\alpha(\rn), \quad
\theta:=\min(1,p,q), \quad
N_0:=\frac{n}{q}+1+\max(\alpha,0), \quad
\gamma:=\frac{n}{p}+\alpha, \quad
\delta:=0,
\end{equation*}%
\begin{equation*}
w_j(x) := 1\mbox{ for all }x\in\rn \mbox{ and }j\in\zz_+, \quad
\alpha_1=\alpha_2=\alpha_3=0,
\end{equation*}%
\begin{equation*}
\tau \in [0,\infty), \, q\in (0,\infty], \, a\in({n}/{q}+1,\fz)
\end{equation*}%
respectively.
By virtue of Proposition \ref{p10.5},
we know that $(\cl8)$ holds true as well.

Therefore, again a result similar to Besov-Morrey spaces
and Triebel-Lizorkin spaces can be obtained
for $K_{p,q}^\az(\rn)$ with $p,q \in (0,\infty]$ and $\az\in\rr$
as before.
Homogeneous counterpart
of the above is available.
Define the \emph{homogeneous Herz space $\dot{K}^\az_{p,q}(\rn)$} to be
the set of all measurable functions $f$ for which the \emph{norm}
\[
\| f \|_{\dot K^\az_{p,q}(\rn)}
:=
\left[
\sum_{j=-\infty}^\infty
\| 2^{jq\az}\chi_{C_j}  f \|_{L^p(\rn)}^q
\right]^\frac{1}{q}
\]
is finite,
where we modify naturally
the definition above when $p=\fz$ or $q=\infty$.

An analogous result is available
but we do not go into the detail.

\subsection{Variable Lebesgue spaces}
\label{s10.4}

Starting from the  recent work by Diening \cite{d1},
there exist a series of results
of the theory of variable function spaces.
Let $p(\cdot):\rn \to (0,\infty)$ be a measurable function
such that
$
0<\inf_{x \in \rn}p(x)\le\sup_{x \in \rn}p(x)<\infty.
$
The \emph{space} $L^{p(\cdot)}(\rn)$,
the Lebesgue space with variable exponent $p(\cdot)$,
is defined
as the set of all measurable functions $f$
for which the quantity
$
\int_{\rn}|\varepsilon f(x)|^{p(x)}\,dx
$
is finite for some $\varepsilon\in(0,\fz)$.
We let
$$
\|f\|_{L^{p(\cdot)}(\rn)}
:=
\inf\left\{
\lambda>0 \,:\, \int_{\rn}\left[\frac{|f(x)|}{\lambda}\right]^{p(x)}\,dx
\le 1
\right\}
$$
for such a function $f$.
As a special case of the theory of Nakano and Luxemberg
\cite{l55,n50,n51},
we see $(L^{p(\cdot)}(\rn),\|\cdot\|_{L^{p(\cdot)}(\rn)})$
is a quasi-normed space.
It is customary to let
$
p_+:=\sup_{x \in \rn}p(x)
$
and
$
p_-:=\inf_{x \in \rn}p(x).
$

The following was shown in \cite{cfmp06}
and hence we have $(\cl8)$ for $L^{p(\cdot)}(\rn)$.
\begin{proposition}\label{p10.6}
Suppose that $p(\cdot):\rn \to (0,\infty)$ is a function
satisfying
\begin{gather}\label{10.25}
1<p_-:=\inf_{x \in \rn}p(x)
\le p_+:=\sup_{x \in \rn}p(x)<\infty,\\
\mbox{{\rm (the log-H\"{o}lder continuity)}}\quad
  |p(x)-p(y)|\ls\frac{1}{\log(1/|x-y|)}\nonumber\\
  \hspace{5cm}\quad\text{for\ all}\quad
  |x-y|\le\frac12, \label{10.26}
  \\
\mbox{{\rm (the decay condition)}}\quad
  |p(x)-p(y)|\ls\frac{1}{\log(e+|x|)}
  \quad\text{for\ all}\quad
  |y|\ge|x|.\label{10.27}
\end{gather}
Let $u \in (1,\infty]$.
Then, for all sequences of measurable functions $\{f_j\}_{j=1}^\infty$,
\[
\left\|
\left(\sum_{j=1}^\infty[ Mf_j]{}^{u}\right)^{\frac{1}{u}}
\right\|_{L^{p(\cdot)}(\rn)}
\sim
\left\|
\left(\sum_{j=1}^\infty |f_j|^{u}\right)^{\frac{1}{u}}
\right\|_{L^{p(\cdot)}(\rn)}
\]
with the implicit equivalent positive constants independent of $\{f_j\}_{j=1}^\fz$.
\end{proposition}

Now in this example $(\cl1)$ through $(\cl6)$ hold true with the parameters in
\eqref{2.10}, \eqref{2.11} and \eqref{2.12}
actually satisfies
\begin{equation*}
\cl(\rn):=L^{p(\cdot)}(\rn), \quad
\theta:=\min(1,p_-), \quad
N_0:=\frac{n}{p_-}+1, \quad
\gamma:=\frac{n}{p_-}, \quad
\delta:=0,
\end{equation*}%
\begin{equation*}
w_j(x) := 1\mbox{ for all }x\in\rn \mbox{ and }j\in\zz_+, \quad
\alpha_1=\alpha_2=\alpha_3=0,
\end{equation*}%
\begin{equation*}
\tau \in [0,\infty), \, q\in (0,\infty], \, a>\frac{n}{p_-}+1,
\end{equation*}%
respectively. Also, by virtue of Proposition \ref{p10.6},
we have $(\cl8)$ as well.
For the sake of simplicity,
let us write
$A_{p(\cdot),q}^{s,\tau}(\rn)$
instead of
$A_{L^{p(\cdot)}(\rn),q,a}^{s,\tau}(\rn)$.

The function space $A_{p(\cdot),q}^{s}(\rn)$
is well investigated and we have the following proposition,
for example.
\begin{proposition}[{\rm \cite{ns10}}]\label{p10.7}
Let $f \in {\mathcal S}'(\rn)$
and $p(\cdot)$ satisfy {\rm \eqref{10.25}},
{\rm \eqref{10.26}} and {\rm \eqref{10.27}}.
Then, the following are equivalent:

{\rm (i)}
$f$ belongs to the local Hardy space
$h^{p(\cdot)}(\rn)$
with variable exponent $p(\cdot)$,
that is,
$$
\|f\|_{h^{p(\cdot)}(\rn)}
:=
\left\|
\sup_{0<t \le 1}|t^{-n}\Phi(t^{-1}\cdot)*f|
\right\|_{L^{p(\cdot)}(\rn)}<\infty;
$$

{\rm (ii)}
$f$ satisfies
\[
\|f\|_{F^0_{p(\cdot),2}(\rn)}
:=
\|\Phi*f\|_{L^{p(\cdot)}(\rn)}
+
\left\|
\left(\sum_{j=1}^\infty |\vz_j*f|^2\right)^{1/2}
\right\|_{L^{p(\cdot)}(\rn)}<\infty.
\]
\end{proposition}

By virtue of Lemma \ref{l1.1}, Theorem \ref{t8.1},
Propositions \ref{p10.6} and \ref{p10.7},
we have the following.
\begin{proposition}\label{p10.8}
The function space $h^{p(\cdot)}(\rn)$ coincides with
$F_{p(\cdot),2,a}^{0,0}(\rn)$,
whenever $a \gg 1$.
\end{proposition}

Recall that
Besov/Triebel-Lizorkin spaces with
variable exponent
date back to the works
by Almeida and H\"ast\"o
\cite{AH}
and
Diening, H\"ast\"o and Roudenko
\cite{DHR}.
Xu investigated fundamental properties
of $A^{s}_{p(\cdot),q}(\rn)$
\cite{x1,x2}.
Among others
Xu obtained the atomic decomposition results.
As for $A^{s}_{p(\cdot),q}(\rn)$,
in \cite{NS}, Noi and Sawano
have investigated the complex interpolation
of $F_{p_0(\cdot),q_0}^{s_0}(\rn)$
and
$F_{p_1(\cdot),q_1}^{s_1}(\rn)$.

Finally, as we have announced
in Section \ref{s1},
we show the unboundedness of the Hardy-Littlewood maximal operator
and the maximal operator $M_{r,\lz}$.
\begin{lemma}\label{l10.2}
The maximal operator $M_{r,\lz}$
is not bounded on $L^{1+\chi_{{\mathbb R}^n_+}}(\rn)$
for all $r\in(0,\fz)$ and $\lz\in(0,\fz)$.
In particular,
the Hardy-Littlewood maximal operator $M$
is not bounded on $L^{1+\chi_{{\mathbb R}^n_+}}(\rn)$.
\end{lemma}

\begin{proof}
Consider
$f_r(x):=\chi_{[-r,0]}(x_n)\chi_{[-1,1]^{n-1}}
(x_1,x_2,\cdots,x_{n-1})$
for all $x=(x_1,\cdots,x_n)\in\rn$.
Then, for all $x$ in the support of $f_r$, we have
\[
M_{r,\lz}f_r(x)\sim Mf_r(x)
\sim\chi_{[-r,r]}(x_n)\chi_{[-1,1]^{n-1}}
(x_1,x_2,\cdots,x_{n-1}).
\]
Hence $\|M_{r,\lz}f \|_{L^{1+\chi_{{\mathbb R}^n_+}}}\gtrsim r^{-1/2}$,
while $\|f\|_{L^{1+\chi_{{\mathbb R}^n_+}}} \sim r^{-1}$,
showing the unboundedness,
which completes the proof of Lemma \ref{l10.2}.
\end{proof}

Lebesgue spaces with variable exponent
date back first to the works by Orlicz and Nakano
\cite{or,n50,n51},
where the case $p_+<\infty$ is considered.
When $p_+ \le \infty$,
Sharapudinov considered $L^{p(\cdot)}([0,1])$
\cite{sh}
and then
Kov\'a\v cik
and
R\'akosn\'{i}k
extended the theory to domains
\cite{kr}.

\subsection{Amalgam spaces}
\label{s10.5}

Let $ p,\,q \in(0, \infty]$ and $s \in \R$.
Recall that $Q_{0z}:= z+[0,1]^n$ for $z \in \Z^n$,
the \emph{translation of the unit cube}.
For a Lebesgue locally integrable function $f$
we define
\[
\| f \|_{(L^p(\rn),l^q(\langle z \rangle^s))}
:=
\| \{(1+|z|)^s \cdot \| \chi_{Q_{0z}}f\|_{L^p(\rn)} \}_{z \in \Z^n}
\|_{l^q}.
\]
Now in this example $(\cl1)$ through $(\cl6)$ hold true with the parameters
\eqref{2.10}, \eqref{2.11} and \eqref{2.12}
actually reading as, respectively,
\begin{equation*}
\cl(\rn):=(L^p(\rn),l^q(\langle z \rangle^s)), \
\theta:=\min(1,p,q), \
N_0:=n+1+s, \quad
\gamma:=\frac{n}{p}, \
\delta:=\max(-s,0).
\end{equation*}%
\begin{equation*}
w_j(x) := 1\mbox{ for all }x\in\rn \mbox{ and }j\in\zz_+, \quad
\alpha_1=\alpha_2=\alpha_3=0.
\end{equation*}%
\begin{equation*}
\tau \in [0,\infty), \, q\in (0,\infty], \, a>n+1+s.
\end{equation*}%
The following is shown essentially in \cite{kntyy}.
Actually, in \cite{kntyy}
the boundedness of singular integral operators
is established.
Using the technique employed in \cite[p.\,498]{gr85},
we have the following.

\begin{proposition}\label{p10.9}
Let $q,u \in(1, \infty]$, $p\in(1,\fz)$ and $s \in \R$.
Then, for all sequences of measurable functions $\{f_j\}_{j=1}^\infty$,
\[
\left\|
\left(\sum_{j=1}^\infty [Mf_j]{}^u\right)^{\frac{1}{u}}
\right\|_{(L^p(\rn),l^q(\langle z \rangle^s))}
\sim
\left\|
\left(\sum_{j=1}^\infty |f_j|^u\right)^{\frac{1}{u}}
\right\|_{(L^p(\rn),l^q(\langle z \rangle^s))}
\]
with the implicit equivalent positive constants independent of $\{f_j\}_{j=1}^\fz$.
\end{proposition}

Therefore, $(\cl6)$ is available and
the results above can be applicable to amalgam spaces.
Remark that amalgam spaces can be used
to describe the range of the Fourier transform;
see \cite{sy06} for details.
\subsection{Multiplier spaces}
\label{s10.6}

There is another variant of Morrey spaces.
\begin{definition}\label{d10.1}
For $r\in[0,\frac{n}{2})$, the \emph{space} $\overset{.}{X}_{r}(\rn)$ is defined
as the space of all functions
$f \in L_{\rm loc}^{2}\left( \mathbb{R}^{n}\right) $
that satisfy the following inequality\textrm{:}
\begin{equation*}
\left\Vert f\right\Vert _{\overset{.}{X}_{r}(\rn)}
:=
\sup\left\{\left\Vert f g \right\Vert_{L^{2}(\rn)}<\infty
\,:\ \left\Vert g\right\Vert_{\overset{.}{H^r}(\rn)}\leq 1
\right\}<\fz,
\end{equation*}
where $\overset{.}{H^r}\left( \mathbb{R}^{n}\right)$ stands for the
completion of the space $\mathcal{D}\left( \mathbb{R}^{n}\right) $ with
respect to the \emph{norm} $\left\Vert u\right\Vert _{\overset{.}{H^r}(\rn)}
:=\Vert( -\Delta) ^{\frac{r}{2}}u\Vert _{L^{2}(\rn)}.$
\end{definition}

We refer to \cite{M} for the reference of this field which contains a vast
amount of researches of the multiplier spaces. Here and below we place
ourselves in the setting of ${\mathbb{R}}^{n}$ with $n\geq 3$.

We characterize this norm in terms
of the $\overset{.}{H^r}(\rn)$-capacity and wavelets.
Here we present the definition of capacity (see \cite{MS,M}).
Denote by $\mathcal{K}$ the \emph{set of all compact sets
in $\R^n$}.

\begin{definition}[{\rm \cite{M}}]\label{d10.2}
Let $r\in[0,\frac{n}{2})$ and $e\in \mathcal{K}$.
The \emph{quantity $\mathrm{cap}(e,\overset{.}{H^r}(\rn)) $} stands
for the $\overset{.}{H^r}$-capacity,
which is defined by
\begin{equation*}
{\rm cap }\left( e,\overset{.}{H^r}(\rn)\right)
:=
\inf \left\{ \left\Vert
u\right\Vert _{\overset{.}{H^r}\left( \mathbb{R}^{n}\right) }^{2}:u\in
\mathcal{D}\left( \mathbb{R}^{n}\right) ,\text{ \ }u\geq 1\text{ \ on \ }%
e\right\} .
\end{equation*}
\end{definition}

Let us set $\frac1u:=\frac12-\frac{r}{n}$,
that is,
$u=\frac{2n}{n-2r}$.
Notice that by the Sobolev embedding theorem,
we have
\[
|e|^\frac1u=\|\chi_e\|_{L^u(\rn)} \le \|u\|_{L^u(\rn)} \ls \|u\|_{\overset{.}{H^r}(\rn)}
\]
for all $u \in \mathcal{D}\left( \mathbb{R}^{n}\right)$.
Consequently, we have
\begin{equation}\label{10.34}
{\rm cap }\left( e,\overset{.}{H^r}(\rn)\right)
\ge |e|^{\frac{n-2r}{n}}.
\end{equation}

Having clarified the definition of capacity,
let us now formulate our main result.
In what follows, we choose a system
$\{\psi_{\varepsilon,jk}
\}_{\varepsilon=1,2,\cdots,2^n-1, \, j \in {\mathbb Z}, \,
k \in {\mathbb Z}^n}$
so that it forms a complete orthonormal basis of $L^2(\rn)$ and that
\[
\psi_{\varepsilon,jk}(x)
=
\psi_{\varepsilon}(2^jx-k)
\]
for all $j \in \Z$, $k \in \Z^n$ and  $x \in \rn$.

\begin{proposition}[{\rm \cite{gs10,M}}]\label{p10.10}
Let $r\in[0,\frac{n}{2})$ and let
$f \in L^2_{\rm loc}(\rn) \cap \cs'(\rn)$.
Then the following
are equivalent\textrm{:}

{\rm (i)}
$f\in \overset{.}{X}_{r}(\rn).$

{\rm(ii)}
The function $f$ can be expanded as follows\textrm{:}
In the topology of $\cs'(\rn)$,
\[
f=\sum\limits_{\varepsilon=1}^{2^{n}-1} \sum\limits_{(j,k)
\in \mathbb{Z\times Z}^{n}}
 \lambda _{\varepsilon ,jk}\psi_{\varepsilon,jk}\qquad in \quad \cs'(\rn),
\]
where $\{\lambda _{\varepsilon ,jk}
\}_{\varepsilon=1,2,\cdots,2^n-1, \, (j,k)
\in \mathbb{Z\times Z}^{n}}$
satisfies that
\begin{equation*}
\sum\limits_{\varepsilon=1}^{2^{n}-1} \sum\limits_{(j,k)
\in \mathbb{Z\times Z}^{n}}\left\vert \lambda _{\varepsilon ,jk
}\right\vert ^{2}\int_{e}\left\vert \psi _{\varepsilon ,jk
}(x)\right\vert ^{2}M\chi_e(x)^{4/5}\,dx\leq
(C_1)^2\,{\rm cap}\left( e,\overset{.}{H^r}(\rn)\right)
\end{equation*}
for $e\in \mathcal{K}$.

{\rm(iii)}
Assume in addition $n \ge 3$ here.
The function $f$ can be expanded as follows\textrm{:}
In the topology of $\cs'(\rn)$,
\[
f=\sum\limits_{\varepsilon=1}^{2^{n}-1} \sum\limits_{(j,k)
\in \mathbb{Z\times Z}^{n}}
 \lambda_{\varepsilon,jk}\psi_{\varepsilon,jk}\qquad in \quad \cs'(\rn),
\]
where $\{\lambda _{\varepsilon ,jk}
\}_{\varepsilon=1,2,\cdots,2^n-1, \, \left(j,k \right)
\in \mathbb{Z\times Z}^{n}}$
satisfies that
\begin{equation*}
\sum\limits_{\varepsilon=1}^{2^{n}-1} \sum\limits_{(j,k)
\in \mathbb{Z\times Z}^{n}}\left\vert \lambda _{\varepsilon ,jk
}\right\vert^{2}\int_{\R^n}\left\vert \psi_{\varepsilon,jk}(x)
\right\vert^{2}\,dx\leq
(C_2)^2\,{\rm cap}\left( e,\overset{.}{H^r}(\rn)\right)
\end{equation*}
for $e\in \mathcal{K}$.

Furthermore,
the smallest values of $C_1$ and $C_2$ are both
equivalent to
$\|f\|_{\dot{X}_r(\rn)}$.
\end{proposition}

To show that this function space
falls under the scope of our theory,
let us set
\begin{eqnarray*}
&&\|F\|_{\dot{X}_r(\rn)}^{(1)}
:=
\sup_{e\in \mathcal{K}}
\left\{
\frac{1}{{\rm cap}\left(e,\overset{.}{H^r}(\rn)\right)}
\int_{\R^n}\left\vert F(x) \right\vert^{2}\,dx
\right\}^{1/2}
\end{eqnarray*}
and
\begin{eqnarray*}
&&\|F\|_{\dot{X}_r(\rn)}^{(2)}
:=
\sup_{e\in \mathcal{K}}
\left\{
\frac{1}{{\rm cap}\left(e,\overset{.}{H^r}(\rn)\right)}
\int_{e}\left\vert F(x) \right\vert^{2}M\chi_e(x)^{4/5}\,dx
\right\}^{1/2}.
\end{eqnarray*}
The \emph{space} $\dot{X}_{r}^{(i)}(\rn), \, i\in\{1,2\}$, denotes
the set of all measurable function $F:\rn \to \C$
for which $\|F\|_{\dot{X}_{r}(\rn)}^{(i)}<\infty$.

The following lemma,
which can be used to checking $(\cl6)$, is known.

\begin{lemma}[{\rm \cite[Lemma 2.1]{gs10}}]
\label{l10.3} Let $e$ be a compact set and $\kappa\in(0,\fz)$.
Define $E_{\kappa }=\{x\in {%
\mathbb{R}}^{n}\,:\,M\chi _{e}(x)>\kappa \}$. Then
\begin{equation*}
\text{\rm cap}\left( \overline{E_{\kappa }},\overset{.}{H^r}(\rn)\right)
\ls\kappa^{-2}\text{\rm cap}\left( e,\overset{.}{H^r}(\rn)\right)
\end{equation*}
\end{lemma}

By \eqref{10.34} and Lemma \ref{l10.3},
$(\cl1)$ through $(\cl6)$ hold true with
the condition \eqref{2.10}
actually reading as
\begin{equation*}
\cl(\rn):=\dot{X}_{r}^{(i)}(\rn) \mbox{ for } \, i\in\{1,2\}, \quad
\theta:=1, \quad
N_0:=n+1, \quad
\gamma:=2, \quad
\delta:=0
\end{equation*}%
In this case the condition \eqref{2.11} on $w$ is trivial:
\begin{equation*}
w_j(x):= 1 {\rm\ for\ all\ } j\in\zz_+ \mbox{ and } x\in\rn,
\alpha_1=\alpha_2=\alpha_3=0.
\end{equation*}%
Consequently, \eqref{2.12} reads as
\begin{equation*}
\tau \in [0,\infty), \, q\in (0,\infty], \, a>n+1.
\end{equation*}%

In view of Proposition \ref{p10.10}
we give the following proposition.

\begin{definition}\label{d10.3}
For any given sequence
$\lambda:=\{\lambda_{jk}\}_{j \in \Z_+, k \in {\mathbb Z}^n}$, let
\begin{eqnarray*}
\|\lambda\|_{\dot{X}_r(\rn)}^{(1)}
:=\|\lambda\|_{\dot{b}_{\dot{X}_{r}^{(1)}(\rn),2}^{0,0}}, \quad
\|\lambda\|_{\dot{X}_r(\rn)}^{(2)}
:=\|\lambda\|_{\dot{b}_{\dot{X}_{r}^{(2)}(\rn),2}^{0,0}}.
\end{eqnarray*}
The \emph{space $\dot{X}_{r}^{(i)}(\rn)$} for $i\in\{1,2\}$ is the set of all
sequences $\lambda:=\{\lambda_{jk}\}_{j \in \zz_+, \, k \in \Z^n}$
for which $\|\lambda\|_{\dot{X}_r(\rn)}^{(i)}$ is finite.
\end{definition}

In \cite{gs10},
essentially, we have shown the following conclusions.

\begin{proposition}\label{p10.11}
Let $r\in(0,\frac n2)$.

{\rm (i)}
If $n \ge 3$,
then
$(\dot{X}_r(\rn),\dot{X}_{r}^{(1)}(\rn))$
admits the atomic / molecular decompositions.

{\rm (ii)}
If $n \ge 1$,
then
$(\dot{X}_r(\rn),\dot{X}_{r}^{(2)}(\rn))$
admits the atomic / molecular decompositions.
\end{proposition}
However, due to Proposition \ref{p8.1},
this can be improved as follows.

\begin{proposition}\label{p10.12}
Let $r\in(0,\frac n2)$ and $n \ge 1$.
Then
$(\dot{X}_r(\rn),\dot{X}_{r}^{(1)}(\rn))$
admits the atomic / molecular decompositions.
\end{proposition}

\subsection{$\dot{B}_\sigma(\rn)$ spaces}
\label{s10.7}

The next example also falls under the scope
of our generalized Triebel-Lizorkin type spaces.

\begin{definition}\label{d10.4}
Let $\sigma\in[0,\fz)$, $p\in[1,\fz]$ and $\lambda\in[-\frac{n}{p},0]$.
The \emph{space}
$\dot{B}_\sigma(L_{p,\lambda})(\rn)$ is defined
as the space of all $f\in L^p_{\rm loc}(\rn)$
for which the \emph{norm}
\begin{equation*}
\|f\|_{\dot{B}_\sigma(L_{p,\lambda})(\rn)}
:=
\sup\left\{
\frac{1}{r^\sigma|Q|^{\frac{\lambda}{n}+\frac{1}{p}}}
\|f\|_{L^p(Q)}\,:\,r\in(0,\fz), \,
Q \in {\mathcal D}(Q(0,r))
\right\}
\end{equation*}%
is finite.
\end{definition}

Now in this example $(\cl1)$ through $(\cl6)$ hold true with the parameters
in \eqref{2.10} and \eqref{2.11}
actually reading as
\begin{equation*}
\cl(\rn):=\dot{B}_\sigma(L_{p,\lambda})(\rn), \quad
\theta:=1, \quad
N_0:=-\lambda+1, \quad
\gamma:=-\lambda, \quad
\delta:=0
\end{equation*}%
and
\begin{equation*}
w_j(x) := 1 {\rm \ for\ all \ }j\in\zz_+ \mbox{ and } x\in\rn, \quad
\alpha_1=\alpha_2=\alpha_3=0,
\end{equation*}%
respectively. Hence \eqref{2.12} now stands for
\begin{equation*}
\tau \in [0,\infty), \, q\in (0,\infty], \, a>-\lambda+1.
\end{equation*}%
We remark that
$\dot{B}^{\sigma}(\rn)$-spaces have been introduced recently
to unify $\lambda$-central Morrey spaces,
$\lambda$-central mean oscillation spaces
and usual Morrey-Campanato spaces \cite{mn:1}.
Recall that in Lemma \ref{l1.1} we have defined
$Q(0,r)$.
We refer \cite{kmns} for further generalizations
of this field.
\begin{definition}[{\rm \cite{mns}}]\label{d10.5}
Let
$  p\in(1,\fz), \, \sigma\in(0,\fz), \,
\lz\in[-\frac{n}{p},-\sigma)$
and $\varphi$ satisfy {\rm\eqref{1.1}} and {\rm\eqref{1.2}}.
Given $f \in {\mathcal S}'(\rn)$,
set
$$
\|f\|_{\dot{B}_\sigma(L_{p,\lambda}^D)(\rn)}
:=
\sup_{\substack{r\in(0,\fz)\\ Q \in \cq(\rn),\ Q \subset Q(0,r)}}
\frac{1}{r^\sigma|Q|^{\frac{\lambda}{n}+\frac{1}{p}}}
\left\|
\left(\sum_{j=-\log_2\ell(Q)}^\infty|\vz_j*f|^2\right)^{\frac{1}{2}}
\right\|_{L^p(Q)}.
$$
The \emph{space} $\dot{B}_\sigma(L_{p,\lambda}^D)(\rn)$
denotes the space of all $f \in {\mathcal S}'(\rn)$ for which
$\|f\|_{\dot{B}_\sigma(L_{p,\lambda}^D)(\rn)}$ is finite.
\end{definition}

\begin{lemma}[\cite{mns}]\label{l10.20}
Let $p\in(1,\fz), \, u\in(1,\infty], \, \sigma\in [0,\infty)$ and $\lambda\in(-\fz,0)$.
Assume, in addition, that $\sigma+\lambda<0$.
Then
\begin{equation*}
\left\|\left(
\sum_{j=1}^\infty [Mf_j]{}^u\right)^{\frac1u}
\right\|_{\dot{B}_\sigma(L_{p,\lambda})(\rn)}
\sim
\left\|\left(
\sum_{j=1}^\infty |f_j|^u\right)^{\frac1u}
\right\|_{\dot{B}_\sigma(L_{p,\lambda})(\rn)}
\end{equation*}%
with the implicit equivalent positive constants independent of
$\{f_j\}_{j=1}^\infty \subset \dot{B}_\sigma(L_{p,\lambda})(\rn)$.
\end{lemma}

\begin{proposition}[{\rm \cite{mns}}]\label{p10.13}
Let $p\in(1,\fz)$, $\sigma\in(0,\fz)$ and
$\lambda\in[-\frac{n}{p},-\sigma)$.
Then
$$\dot{B}_\sigma(L_{p,\lambda}^D)(\rn)\ and\
\dot{B}_\sigma(L_{p,\lambda})(\rn)$$
coincide.
More precisely, the following hold true{\rm\,:}

{\rm (i)}
$\dot{B}_\sigma(L_{p,\lambda})(\rn) \hookrightarrow {\mathcal S}'(\rn)$
in the sense of continuous embedding.

{\rm (ii)}
$\dot{B}_\sigma(L_{p,\lambda}^D)(\rn) \hookrightarrow
{\mathcal S}'(\rn) \cap L^p_{\rm loc}(\rn)$
in the sense of continuous embedding.

{\rm (iii)}
$f \in \dot{B}_\sigma(L_{p,\lambda})(\rn)$
if and only if
$f\in \dot{B}_\sigma(L_{p,\lambda}^D)(\rn)$
and the norms are mutually equivalent.

{\rm (iv)}
Different choices of $\varphi$
 yield
the equivalent norms
in the definition of $\|\cdot\|_{\dot{B}_\sigma(L_{p,\lambda}^D)(\rn)}$.
\end{proposition}

The atomic decomposition
of $\dot{B}_\sigma(\rn)$ is as follows:
First we introduce the sequence space.

\begin{definition}\label{d10.6}
Let $\sigma \in [0,\infty), \, p\in[1,\infty]$
and $\lambda\in[-\frac{n}{p},0]$.
The \emph{sequence space} $\dot{b}_\sigma(L^D_{p,\lambda})(\rn)$
is defined to be the space of all $\lambda:=\{\lambda_{jk}\}_{j \in \Z_+, \, k \in \Z^n}$
such that
\[
\|\lambda\|_{\dot{b}_\sigma(L^D_{p,\lambda})(\rn)}
:=
\sup_{\substack{r\in(0,\fz)\\ Q \in \cq(\rn),\ Q \subset Q(0,r)}}
\frac{1}{r^\sigma|Q|^{\frac{\lambda}{n}+\frac{1}{p}}}
\left\|
\sum_{j=-\log_2\ell(Q)}^\infty \lambda_{jk}\chi_{Q_{jk}}
\right\|_{L^p(Q)}<\fz.
\]
\end{definition}

In view of Theorem \ref{t5.1},
we have the following,
which is a direct corollary of Theorem \ref{t4.1}.
\begin{theorem}\label{t10.1}
The pair $(\dot{B}_\sigma(L^D_{p,\lambda})(\rn),
\dot{b}_\sigma(L^D_{p,\lambda})(\rn))$
admits the atomic / molecular decompositions.
\end{theorem}

\subsection{Generalized Campanato spaces}
\label{s10.8}

Returning to the variable exponent setting
described in Section \ref{s10.4},
we define $d_{p(\cdot)}$ to be
$$
  d_{p(\cdot)}:=
  \min\left\{d\in\Z_+:p_-(n+d+1)>n\right\}.
$$
Let the \emph{space} $L^{q}_{\rm comp}(\rn)$ be the set of all $L^q$-functions with compact support.
For a nonnegative integer $d$, let
$$
     L^{q,d}_{\rm comp}(\rn)
     := \left\{ f\in L^{q}_{\rm comp}(\rn):\
              \int_{\rn} f(x)x^{\alpha}\,dx=0,\ |\alpha|\le d \right\}.
$$
Likewise if $Q$ is a cube,
then we write
$$
     L^{q,d}(Q)
     := \left\{ f\in L^{q}(Q):\
              \int_Q f(x)x^{\alpha}\,dx=0,\ |\alpha|\le d \right\},
$$
where the \emph{space} $L^q(Q)$ is a closed subspace of functions in $L^q(\rn)$
having support in $Q$.

Recall that ${\mathcal P}_d(\rn)$ is the \emph{set of all polynomials
having degree at most $d$}.
For a locally integrable function $f$,
a cube $Q$ and a nonnegative integer $d$,
there exists a unique polynomial $P\in{\mathcal P}_d(\rn)$
such that, for all $q\in{\mathcal P}_d(\rn)$,
$$\int_Q(f(x)-P(x))q(x)\,dx=0.$$
Denote this \emph{unique polynomial} $P$ by $P^d_Qf$.
It follows immediately from the definition
that $P^d_Q g=g$ if $g \in {\mathcal P}_d(\rn)$.

We postulate on $\phi:\R^{n+1}_+ \to (0,\infty)$
the following conditions:

{\rm (A1)}
There exist positive constants $M_1$ and $M_2$ such that
$$
M_1 \le \frac{\phi(x,2r)}{\phi(x,r)} \le M_2
\quad (x \in \rn, \, r\in(0,\fz))
$$
holds true. (Doubling condition)

{\rm (A2)}
There exist positive constants $M_3$ and $M_4$ such that
$$
M_3 \le \frac{\phi(x,r)}{\phi(y,r)} \le M_4 \quad
(x,y \in \rn, \, r\in(0,\fz), \, |x-y| \le r)
$$
holds true. (Compatibility condition)

{\rm (A3)}
There exists a positive constant $M_5$ such that
$$
\int_0^r \frac{\phi(x,t)}{t}\,dt \le M_5\phi(x,r)
\quad (x \in \rn, \, r\in(0,\fz))
$$
holds true. ($\nabla_2$-condition)

{\rm (A4)}
There exists a positive constant $M_6$ such that
$
\int_r^\infty \frac{\phi(x,t)}{t^{d+2}}\,dt
\le M_6
\frac{\phi(x,r)}{r^{d+1}}
$
for some integer $d \in [0,\infty)$. ($\Delta_2$-condition)

{\rm (A5)}
$
\sup_{x \in \rn}\phi(x,1)<\infty.
$
(Uniform condition)

Here the constants $M_1,M_2,\cdots,M_6$
need to be specified for later considerations.

Notice that the Morrey-Campanato space with variable growth function $\phi(x,r)$
was first introduced by Nakai \cite{n93,n06} by using an idea originally from
\cite{ny85}. In \cite{n94}, Nakai established the boundedness of
the Hardy-Littlewood maximal operator, singular integral (of Calder\'on-Zygmund type),
and fractional integral operators on Morrey spaces
with variable growth function $\phi(x,r)$.

Recently, Nakai and Sawano considered
a more generalized version in \cite{ns10}.

Let us say that $\phi:\cq(\rn) \to (0,\infty)$
is a \emph{nice function}, if there exists $b \in (0,1)$
such that, for all cubes $Q \in \cq(\rn)$,
\[
\frac{1}{\phi(Q)}
\left[\frac{1}{|Q|}\int_Q|f(x)-P^d_Qf(x)|^q\,dx\right]^{\frac1q}
>b
\]
for some $f \in \cl_{q,\phi,d}(\R^n)$ with norm $1$.
In \cite[Lemma 6.1]{ns10}, we showed that $\phi$ can be assumed to be nice.
Actually,
there exists a nice function
$\phi^\dagger$ such that
$\cl_{q,\phi,d}(\R^n)$
and
$\cl_{q,\phi^\dagger,d}(\R^n)$
coincide as a set
and the norms are mutually equivalent
\cite[Lemma 6.1]{ns10}.

\begin{definition}[{\rm \cite{ns10}}]\label{d10.7}
Let $\phi:\R^{n+1}_+ \to (0,\infty)$ be a function,
which is not necessarily nice,
and $f \in L^q_{{\rm loc}}(\rn)$.
Define, when $q\in(1,\fz)$,
\begin{align*}
     \|f\|_{{\mathcal L}_{q,\phi,d}(\rn)}
       & := \sup_{(x,t) \in \R^{n+1}_+}
\frac 1{\phi(x,t)}\left\{ \frac 1{|Q(x,t)|}
          \int_{Q(x,t)} |f(y)-P^d_{Q(x,t)} f(y)|^q \,dy
\right\}^{1/q},\
\end{align*}
and, when $q=\infty$,
\begin{align*}
     \|f\|_{{\mathcal L}_{q,\phi,d}(\rn)}
       & := \sup_{(x,t) \in \R^{n+1}_+} \frac 1{\phi(x,t)}
          \|f-P^d_{Q(x,t)} f\|_{L^\infty(Q(x,t))}.
\end{align*}
Then the \emph{Campanato space} ${\mathcal L}_{q,\phi,d}(\rn)$
is defined to be the set of all $f$ such that
$\|f\|_{{\mathcal L}_{q,\phi,d}(\rn)}$ is finite.
\end{definition}

\begin{definition}[{\rm \cite{ns10}}]\label{d10.8}
Let $q \in[1, \infty]$, $\vz$ satisfies \eqref{1.2} and $\phi:\R^{n+1}_+ \to (0,\infty)$
a function.
A distribution $f \in \cs'(\rn)$ is said
to belong to the \emph{space} $\mathcal{L}^D_{q,\phi}(\rn)$,
if
\[
\|f\|_{\mathcal{L}^D_{q,\phi}(\rn)}
:=
\sup_{(x,t) \in \R^{n+1}_{\Z}}
\frac{1}{\phi(x,t)}
\left\{\frac{1}{|Q(x,t)|}
\int_{Q(x,t)}|\varphi_{(\log_2t^{-1})}*f(y)|^q\,dy
\right\}^{\frac{1}{q}}<\infty.
\]
\end{definition}

\begin{proposition}[\cite{ns10}]\label{p10.14}
Assume {\rm (A1)} through {\rm (A5)}.
Then

{\rm (i)}
The spaces
$\mathcal{L}^D_{q,{\phi}}(\rn)$ and $\mathcal{L}_{q,{\phi},d}(\rn)$
coincide.
More precisely, the following hold true{\rm:}

{\rm (a)}
Let $f \in \mathcal{L}^D_{q,\phi}(\rn)$.
Then there exists $P \in {\mathcal P}(\rn)$
such that
$f-P \in \mathcal{L}_{q,\phi,d}(\rn)$.
In this case,
$\|f-P\|_{{\mathcal L}_{q,\phi,d}(\rn)} \ls \|f\|_{{\mathcal L}^D_{q,\phi}(\rn)}$
with the implicit positive constant independent of $f$.

{\rm (b)}
If $f \in \mathcal{L}_{q,\phi,d}(\rn)$,
then $f \in \mathcal{L}^D_{q,\phi}(\rn)$
and
$\|f\|_{{\mathcal L}^D_{q,\phi}(\rn)} \ls \|f\|_{{\mathcal L}_{q,\phi,d}(\rn)}$
with the implicit positive constant independent of $f$.

In particular,
the definition of the function space
$\mathcal{L}^D_{q,{\phi}}(\rn)$
is independent of the admissible choices of $\vz$ {\rm:}
Any $\vz \in \cs(\rn)$ does the job
in the definition of $\mathcal{L}^D_{q,{\phi}}(\rn)$
as long as
$\chi_{Q(0,1)} \le \widehat\vz \le \chi_{Q(0,2)}$.

{\rm (ii)}
The function space $\mathcal{L}^D_{q,{\phi}}(\rn)$
is independent of $q$.
\end{proposition}

In view of Definition \ref{d10.8},
if we assume that $\phi$ satisfies (A1) through (A5),
then we have the following proposition.

\begin{proposition}\label{p10.15}
Let $\vz$ satisfies \eqref{1.2}.
If  $\phi:\cq(\rn) \to (0,\infty)$ satisfies {\rm (A1)} through {\rm(A5)}, then
$$
     \|f\|_{\mathcal{L}^D_{\infty,\phi}(\rn)}
       \sim
\sup_{(x,t) \in \R^{n+1}_{\Z}} \frac 1{\phi(Q(x,t))}
    \sup_{y \in Q(x,t)}
\left\{\sup_{z \in \R^n}
\frac{|\varphi_{(\log_2t^{-1})}*f(y+z)|}{(1+t^{-1}|z|)^a}\right\},
$$
whenever $a \gg 1$,
with the implicit equivalent positive constants independent of $f$.
\end{proposition}

To proof Proposition \ref{p10.15}, we just need to check \eqref{8.20} by using $(A1)$ and $(A2)$.
We omit the details.

\begin{definition}\label{d10.9}
Define
\begin{eqnarray*}
&&     \|\lambda\|_{l^D_{\infty,\phi}(\rn)}\\
  &&\hs     :=
\sup_{(x,t) \in \R^{n+1}_{\Z}} \frac{1}{\phi(Q(x,t))}
    \sup_{y \in Q(x,t)}
\left\{\sup_{z \in \R^n}\frac{1}{(1+t^{-1}|z|)^a}
\sum_{k \in \Z^n}|\lambda_{(\log_2t^{-1})k}|\chi_{Q_{(\log_2t^{-1})k}}\right\}.
\end{eqnarray*}
\end{definition}

Now in this example $(\cl1)$ through $(\cl6)$ hold true with the parameters in
\eqref{2.10} and \eqref{2.11}
satisfying the following conditions:
\begin{equation*}
\cl(\rn):=L^\infty(\rn), \quad
\theta:=1, \quad
N_0:=0, \quad
\gamma:=0, \quad
\delta:=0
\end{equation*}%
and
$w(x,t) := \frac 1{\phi(Q(x,t))}$
{\rm\ for\ all\ } $x\in\rn$ and $t\in(0,\fz),$
$\alpha_1=\log_2 M_1{}^{-1},
\alpha_2=\log_2 M_2,
\alpha_3=\log_2 \frac{M_2}{M_1},$
respectively.
Furthermore, unlike the proceeding examples,
we choose
\begin{equation*}
\tau=0, \, q=\infty, \, a>N_0+\log_2 \frac{M_2}{M_1}.
\end{equation*}%
Therefore, ${\mathcal L}_{q,\phi,d}(\rn)$
and $\mathcal{L}^D_{q,\phi}(\rn)$
fall under the scope of our theory.

\begin{theorem}\label{t10.2}
Under the conditions $(A1)$ through $(A5)$,
the pair $(\mathcal{L}^D_{\infty,\phi}(\rn),l^D_{\infty,\phi}(\rn))$
admits the atomic / molecular decompositions.
\end{theorem}

Theorem \ref{t10.2} is just a consequence of Theorem \ref{t4.1}.
We omit the details.

\section*{Acknowledgement}

Yoshihiro Sawano is grateful
to Professor Mitsuo Izuki for his careful reading of the present paper.
Yoshihiro Sawano, Tino Ullrich and Dachun Yang would like to thank
Professor Winfried Sickel for his helpful comments and discussions
on the present paper.
Yoshihiro Sawano and Dachun Yang are thankful
to Professor Eiichi Nakai
for his several remarks on Section \ref{s10}.

\bigskip

Yiyu Liang, Dachun Yang (Corresponding author) and Wen Yuan

\medskip

School of Mathematical Sciences, Beijing Normal University,
Laboratory of Mathematics and Complex Systems, Ministry of
Education, Beijing 100875, People's Republic of China

\smallskip

{\it E-mails}: \texttt{yyliang@mail.bnu.edu.cn} (Y. Liang)

\hspace{1.55cm}\texttt{dcyang@bnu.edu.cn} (D. Yang)

\hspace{1.55cm}\texttt{wenyuan@bnu.edu.cn} (W. Yuan)

\bigskip

Yoshihiro Sawano

\medskip

Department of Mathematics, Kyoto University, Kyoto 606-8502, Japan

\smallskip

{\bf The present address}:
Department of Mathematics and Information Sciences,
Tokyo Metropolitan University,
Minami-Ohsawa 1-1, Hachioji-shi,
Tokyo 192-0397, Japan.

\smallskip

{\it E-mail}: \texttt{ysawano@tmu.ac.jp}

\bigskip

Tino Ullrich

\medskip

Hausdorff Center for Mathematics \& Institute for Numerical Simulation,
Endenicher Allee 60, 53115 Bonn, Germany

\smallskip

{\it E-mail}: \texttt{tino.ullrich@hcm.uni-bonn.de}

\end{document}